\documentclass[12pt,a4paper]{elsarticle}
\usepackage[nottoc]{tocbibind}
\usepackage{graphicx}
\usepackage{slashed}
\usepackage{enumerate}
\usepackage{amsmath}
\allowdisplaybreaks
\usepackage{mathtools,cancel}
\usepackage[bottom]{footmisc}
\usepackage[scr=rsfso]{mathalfa}
\usepackage[hidelinks]{hyperref}
{

}
\usepackage{amsfonts}
\usepackage{psfrag}
\usepackage{color}
\usepackage{mathtools}
\usepackage{pgf,tikz}
\usepackage{mathrsfs}
\usetikzlibrary{arrows}
\usepackage{fancyhdr}
\usepackage{amssymb}
\usepackage{slashed}
\usepackage{enumitem}

\usepackage[T1]{fontenc}
\usepackage[utf8]{inputenc}

\newtheorem{definition}{Definition}
\newtheorem{remark}{Remark}

\newtheorem{proposition}{Proposition}[section]

\newcommand{\Gammatop}{\Gamma_{\text{top}}}
\newtheorem{theorem}{Theorem}[section]

\newtheorem{lemma}{Lemma}[section]

\numberwithin{equation}{section}
\allowdisplaybreaks

\newenvironment{proof}{\smallskip\noindent\emph{Proof.}\hspace{1pt}}%
{\hspace{-5pt}{\nobreak\quad\nobreak\hfill\nobreak$\square$\vspace{8pt}%
		\par}\smallskip\goodbreak}
\newcommand{\bbGamma}{{\mathpalette\makebbGamma\relax}}
\newcommand{\makebbGamma}[2]{%
  \raisebox{\depth}{\scalebox{1}[-1]{$\mathsurround=0pt#1\mathbb{L}$}}%
}
\newcommand{\Yub}{\mathcal{Y}_{\ubar}}
\newcommand{\Lbar}{\underline{L}}
\newcommand{\Hbar}{\underline{H}}
\newcommand{\psig}{\psi_g}
\newcommand{\aln}{(\al \nabla)^i}
\newcommand{\ali}{a^{\frac{i}{2}}}
\newcommand{\scaletwoHprime}[1]{\lVert #1 \rVert_{L^2_{(sc)}(H_{u^{\prime}}^{(0,\ubar)}) }}
\newcommand{\twoSuprime}[1]{\lVert #1 \rVert_{L^2(S_{u^{\prime},\ubar})}}
\newcommand{\scaletwoHbarprime}[1]{\lVert #1 \rVert_{L^2_{(sc)}(\Hbar_{\ubar^{\prime}}^{(u_{\infty},u)}) }}
\newcommand{\tbeta}{\beta^R}
\newcommand{\nablaF}{\nabla^{i_5}}
\newcommand{\hnablaF}{\hnabla^{i_5}}
\newcommand{\tbetabar}{\betabar^R}
\newcommand{\hsp}{\hspace{.5mm}}
\newcommand{\omegabar}{\underline{\omega}}
\newcommand{\gslash}{\slashed{g}}

\newcommand{\dubarprime}{\hspace{.5mm} \text{d} \ubar^{\prime}}
\newcommand{\hnabla}{\widehat{\nabla}}
\newcommand{\intu}{\int_{u_{\infty}}^u}
\newcommand{\intubar}{\int_0^{\ubar}}

\newcommand{\F}{{{F}^P}_Q}
\newcommand{\chihat}{\hat{\chi}}

\newcommand{\upr}{\lvert u^\prime \rvert}
\newcommand{\al}{a^{\frac{1}{2}}}
\newcommand{\chibar}{\underline{\chi}}
\newcommand{\chibarhat}{\underline{\hat{\chi}}}
\newcommand{\ubar}{\underline{u}}

\newcommand{\sumit}{\sum_{i_1+i_2+i_3=i}}
\newcommand{\sumitm}{\sum_{i_1+i_2+i_3=i-1}}
\newcommand{\sumiF}{\sum_{i_1+i_2+i_3+i_4+i_5=i}}
\newcommand{\sumiFi}{\sum_{i_1+i_2+i_3+i_4+i_5=i-1}}
\newcommand{\sumif}{\sum_{i_1+i_2+i_3+i_4=i}}
\newcommand{\sumifi}{\sum_{i_1+i_2+i_3+i_4=i-1}}
\newcommand{\sumifim}{\sum_{i_1+i_2+i_3+i_4=i-2}}

\newcommand{\be}{\begin{equation}}
\newcommand{\ee}{\end{equation}}

\newcommand{\bm}{\begin{align}*}
\newcommand{\enm}{\end{align}*}

\newcommand{\scaleinfinitySu}[1]{\lVert{#1} \rVert_{L^\infty_{(sc)}(S_{u,\ubar})}}

\newcommand{\scaleinfinitySuprimeubarprime}[1]{\lVert{#1} \rVert_{L^\infty_{(sc)}(S_{u^{\prime},\ubar^\prime})}}

\newcommand{\tildetr}{\widetilde{\tr \chibar}}

\newcommand{\Y}{\mathrm{\Upsilon}}

\newcommand{\nablap}{\nabla^{i_1}\psi_g^{i_2}}
\newcommand{\nablapp}{\nabla^{i_1}\psi_g^{i_2+1}}

\newcommand{\twoSu}[1]{\lVert{#1} \rVert_{L^2(S_{u,\ubar})}}

\newcommand{\inftySu}[1]{\lVert{#1} \rVert_{L^{\infty}(S_{u,\ubar})}}
\newcommand{\oneSu}[1]{\lVert{#1}\rVert_{L^1(S_{u,\ubar})}}

\newcommand{\bespeq}{\begin{equation}\begin{split}}
\newcommand{\espeq}{\end{split}\end{equation}}
\newcommand{\scaletwoSu}[1]{\lVert{#1} \rVert_{L^2_{(sc)}(S_{u,\ubar})}}
\newcommand{\scaleoneSu}[1]{\lVert{#1} \rVert_{L^1_{(sc)}(S_{u,\ubar})}}

\newcommand{\scaletwoSuprime}[1]{\lVert{#1} \rVert_{L^2_{(sc)}(S_{u^\prime,\ubar})}}
\newcommand{\scaletwoSuzprime}[1]{\lVert{#1} \rVert_{L^2_{(sc)}(S_{u^\prime,0})}}

\newcommand{\scaleoneSuprimeubarprime}[1]{\lVert{#1} \rVert_{L^1_{(sc)}(S_{u^\prime,\ubar^\prime})}}
\newcommand{\scaletwoSuprimeubarprime}[1]{\lVert{#1} \rVert_{L^2_{(sc)}(S_{u^\prime,\ubar^\prime})}}

\newcommand{\scaletwoSuubarprime}[1]{\lVert{#1} \rVert_{L^2_{(sc)}(S_{u,\ubar^\prime})}}
\newcommand{\scaletwoSuzubarprime}[1]{\lVert{#1} \rVert_{L^2_{(sc)}(S_{u_{\infty},\ubar^\prime})}}

\newcommand{\LpSu}[1]{\lVert #1 \rVert_{L^p(\Suu)}}
\newcommand{\LinftySu}[1]{\lVert #1 \rVert_{L^{\infty}(\Suu)}}
\newcommand{\LpHu}[1]{\lVert #1 \rVert_{L^p(\Hu)}}
\newcommand{\LpHbaru}[1]{\lVert #1 \rVert_{L^p(\Hbu)}}

\newcommand{\alphabar}{\underline{\alpha}}
\newcommand{\betabar}{\underline{\beta}}
\newcommand{\etabar}{\underline{\eta}}

\newcommand{\scaletwoHzero}[1]{\lVert{#1} \rVert_{L^2_{(sc)}(H_{u_{\infty}}^{(0,\underline{u})})}}
\newcommand{\scaletwoHu}[1]{\lVert{#1} \rVert_{L^2_{(sc)}(H_u^{(0,\underline{u})})}}
\newcommand{\scaletwoHbarzero}[1]{\lVert{#1} \rVert_{L^2_{(sc)}(\underline{H}_{0}^{(u_{\infty},u)})}}
\newcommand{\scaletwoHbaru}[1]{\lVert{#1} \rVert_{L^2_{(sc)}(\underline{H}_{\underline{u}}^{(u_{\infty},u)})}}

\newcommand{\omegad}{\omega^{\dagger}}

\newcommand{\duprime}{\hspace{.5mm} \text{d}u^{\prime}}
\newcommand{\mubar}{\underline{\mu}}
\newcommand{\kappabar}{\underline{\kappa}}

\newcommand{\Suu}{S_{u,\ubar}}

\newcommand{\aupr}{\frac{a}{\upr^2}}
\newcommand{\haln}{(\al \hnabla)^i}

\newcommand{\nablat}{\nabla^{i_3}}
\newcommand{\hnablat}{\hnabla^{i_3}}
\newcommand{\nablaf}{\nabla^{i_4}}
\newcommand{\hnablaf}{\hnabla^{i_4}}
\newcommand{\Hodge}[1]{\prescript{*}{}{#1}}

\newcommand{\alphabarF}{\alphabar^F}
\newcommand{\rhoF}{\rho^F}
\newcommand{\sigmaF}{\sigma^F}
\newcommand{\alphaF}{\alpha^F}
\newcommand{\YM}{\mathcal{Y}}
\newcommand{\ub}{\underline{u}}

\newcommand{\Lb}{\underline{L}}

\def\a {\alpha}
\def\b {\beta}

\def\ub {\underline{u}}

\def\Lb {\underline{L}}
\def\Hb {\underline{H}}

\def\Hu{H_u^{(0,\underline{u})}}
\def\Hbu{\underline{H}_{\underline{u}}^{(u_{\infty},u)}}

\renewcommand{\div}{\mbox{div }}
\newcommand{\curl}{\mbox{curl }}

\newcommand{\tr}{\mbox{tr}}

\newcommand{\scaletwoHuprime}[1]{\lVert #1 \rVert_{L^2_{(sc)}(H_{u^{\prime}}^{(0,\ubar)}    ) }  }

\newcommand{\m}{\mu}
\newcommand{\n}{\nu}
\newcommand\restri[2]{{
		\left.\kern-\nulldelimiterspace 
		#1 
		\right|_{#2} 
}}

\definecolor{ffqqqq}{rgb}{1.,0.,0.}
\definecolor{uuuuuu}{rgb}{0.26666666666666666,0.26666666666666666,0.26666666666666666}

\makeatletter
\def\ps@pprintTitle{%
  \let\@oddhead\@empty
  \let\@evenhead\@empty
  \let\@oddfoot\@empty
  \let\@evenfoot\@oddfoot
}
\def\@author#1{\g@addto@macro\elsauthors{\normalsize%
    \def\baselinestretch{1}%
    \upshape\authorsep#1\unskip\textsuperscript{%
      \ifx\@fnmark\@empty\else\unskip\sep\@fnmark\let\sep=,\fi
      \ifx\@corref\@empty\else\unskip\sep\@corref\let\sep=,\fi
      }%
    \def\authorsep{\unskip,\space}%
    \global\let\@fnmark\@empty
    \global\let\@corref\@empty  
    \global\let\sep\@empty}%
    \@eadauthor={#1}
}
\makeatother
\begin{document}
\begin{frontmatter}
	
\title{Formation of trapped surfaces in the Einstein-Yang-Mills system}

\author{Nikolaos Athanasiou\fnref{fn2}}
\ead{n.athanasiou@uoc.gr}
\fntext[fn2]{University of Crete}

\author{Puskar Mondal\corref{cor1}\fnref{fn1}}
\ead{puskar_mondal@fas.harvard.edu}
\cortext[cor1]{Corresponding author}
\fntext[fn1]{Harvard University}

\author{Shing-Tung Yau\corref{cor1}\fnref{fn3}}
\ead{yau@maths.harvard.edu}
\fntext[fn3]{Harvard University, Tsinghua University}
\begin{abstract}
We prove a scale-invariant, semi-global existence result and a trapped surface formation result in the context of coupled Einstein-Yang-Mills theory, without symmetry assumptions. More precisely, we prove a scale-invariant semi-global existence theorem from past null-infinity and show that the focusing of the gravitational and/or chromoelectric-chromomagnetic waves could lead to the formation of a trapped surface. Adopting the signature for decay rates approach introduced in \cite{A19}, we develop a novel gauge (and scale) invariant hierarchy of non-linear estimates for the Yang-Mills curvature which, together with the estimates for the gravitational degrees of freedom, yields the desired semi-global existence result. Once semi-global existence has been established, the formation of a trapped surface follows from a standard ODE argument.    	
\par \noindent 
\end{abstract}
\end{frontmatter}
	\vspace{0.2cm}
	
	\setlength{\voffset}{-0in} \setlength{\textheight}{0.9\textheight}
	
	\setcounter{page}{1} \setcounter{equation}{0}
	\tableofcontents
	\section{Introduction}
	\noindent In the present work, we investigate the question of dynamical trapped surface formation for the Einstein-Yang-Mills system (E-Y-M system for short) for a $(3+1)-$dimensional Lorentzian manifold $\left(\mathcal{M},g\right)$ and a (matrix-valued) Yang-Mills curvature two-form ${F^P}_{Q\alpha\beta}$ (to be defined rigorously in Section \ref{section23}):	\begin{gather}
	 \label{Ein}   R_{\mu\nu} - \frac{1}{2}Rg_{\m\n}= \mathfrak{T}_{\m\n},\\
	 \label{Yan}
	 \mathfrak{T}_{\m\n} = \frac{1}{2} \left( {F^P}_{Q\m\alpha}{{F^Q}_{P\n}}^{\alpha} + {\Hodge{F}^P}_{Q\m\alpha}{{\Hodge{F}^Q}_{P\n}}^{\alpha} \right),
	\end{gather}
	where by ${\Hodge{F}^{P}}_Q$ we denote the Hodge dual of ${F^P}_Q$.

	\subsection{Historical background} \noindent Trapped surfaces have been an object of significant interest within the classical theory of General Relativity for almost sixty years. After Schwarzschild's discovery of his eponymous metric in 1915, it took almost twenty years before researchers came to realize the existence, within it, of a region $\mathcal{B}$ with the following surprising yet salient features: First of all, observers situated inside $\mathcal{B}$ cannot send signals to observers situated at an ideal \textit{conformal boundary at infinity}, called $\mathcal{I}^+$.
Furthermore, any observer located inside $\mathcal{B}$ lives only for finite proper time\footnote{Sbierski \cite{Sb} moreover showed that the termination of the observer's proper time manifests in a particularly ferocious way, as they, in fact, get torn apart by infinitely strong tidal forces.} (geodesic incompleteness). 
The characteristics of this region (which later came to be known as a black hole) took most of the researchers of the time aback. The consensus seemed to be that these observed phenomena have to be accidents; pathologies, only present because of the strong (spherical) symmetry inherent in the Schwarzschild solution and that, in general solutions to the Einstein equations\footnote{The meaning of this phrase was not rigorous at the time, as the setup for the initial value problem in General Relativity had not yet been discovered.}, such phenomena would not arise. However, in the 60's, this belief was spectacularly falsified by Roger Penrose through his celebrated incompleteness theorem\footnote{This theorem, in fact, was the main reason why he was awarded the Nobel Prize in Physics back in 2020, "for the discovery that black hole formation is a robust prediction of the general theory of relativity".}. It was Penrose \cite{P73} who introduced the notion of a \textit{trapped suface}, without which one cannot state the incompleteness theorem :

\begin{definition}
Given a $(3+1)$- dimensional Lorentzian manifold $(\mathcal{M}, g)$, a closed spacelike $2-$surface $S$ is caled \textbf{trapped} if the following two fundamental forms $\chi$ and $\chibar$ have everywhere pointwise negative expansions on $S$:

\[  \chi(X,Y) := g(D_X L, Y), \hspace{2mm} \chibar(X, Y) := g(D_X \Lbar, Y).    \]Here $D$ denotes the Levi-Civita connection of $g$, $L$ and $\Lbar$ denote a null basis of the 2-dimensional orthogonal complement of $T_p S$ in $T_p \mathcal{M}$, extended as smooth vector fields and $X, Y$ are arbitrary $S-$tangent vector fields. 
\end{definition}In other words, a surface is called trapped if both $\tr\chi$ and $\tr\chibar$ are pointwise negative everywhere on $S$. These traces signify the infinitesimal changes in area along the null generators normal to $S$, whence one can interpret trapped surfaces as closed, spacelike $2-$surfaces that infinitesimally decrease in area "along any possible future direction".

\vspace{3mm}

\par\noindent The incompleteness theorem is now presented.

\begin{theorem}[Penrose Incompleteness] Let $(\mathcal{M} ,g)$ be a spacetime containing a non-compact Cauchy hypersurface. If $(\mathcal{M}, g)$ moreover satisfies the null energy codition and contains a closed trapped surface, it is geodesically incomplete. 

\end{theorem}The existence of a trapped surface is a stable feature in the context of dynamics. Indeed, sufficiently small perturbations of Schwarzschild initial data must also contain such surfaces, by Cauchy stability. As such, incompleteness is not an accident, but rather a recurring theme in the dynamics of the Einstein equations. 

\vspace{3mm}

\par\noindent At the time, the existence of a trapped surface in a spacetime was too strong an assumption to begin with. In fact, the only way back then to guarantee its existence was to assume it at the level of initial data, but this itself can be a highly non-trivial question. The first trapped surface formation result at the level of initial data was given 
 in \cite{SY83}. This however, begged the question of whether trapped surfaces are dynamical objects, meaning whether they can be formed in evolution starting with data devoid of trapped surfaces. This problem bears high physical significance and serves as a \textit{test of reality} of black holes, in the following sense. The mathematical definition of a black hole region would be without physical meaning if it did not accurately capture what physicists perceive as black holes (this is more meaningful now than ever, as scientists recently succeeded in capturing the first-ever image of a black hole). Hence if "mathematical" black holes describe "physical" black holes, they should mathematically verify certain physical properties, one of which is dynamical formation. 

\vspace{3mm}The first results along this direction were obtained by Christodoulou for the Einstein equations coupled to a massless scalar field in spherical symmetry. Through a series of works \cite{C91}, \cite{C93}, \cite{C94} and \cite{C99}, Christodoulou managed to not only prove trapped surface formation, but to understand the picture of gravitational collapse in its entirety for the given model and under the given symmetry. The breakthrough in the absence of symmetry came in \cite{C09} by the same author. In this work, Christodoulou introduced a hierarchy of small and large components in the initial data which (almost) persists under the evolution of the Einstein equations. He termed his method the \textit{short pulse} method. After Christodoulou, the work \cite{Kl-Rod} by Klainerman-Rodnianski reduces the size of Christodoulou's work from about 600 to approximately 120 pages, by using a slightly different hierarchy. Moreover, it reduces the number of derivatives of curvature required to prove semi-global existence from two to one. A few years later, An  \cite{AnThesis} introduces the signature for decay rates $s_2$ on his way to proving an extension of \cite{Kl-Rod} from a finite region to a region close to past null infinity. In 2014, An and Luk \cite{AL17} prove the first \textit{scale-critical} trapped surface formation criterion for the vacuum equations in the absence of symmetry. While Christodoulou's data in \cite{C09} were  large in $H^1(\mathbb{R}^3)$, An and Luk give data which only have to be large in $H^{\frac{3}{2}}(\mathbb{R}^3)$, which is a scale-critical norm for the initial data. Taking advantage of the scale criticality in \cite{AL17}, An \cite{A17} constructs initial data that give rise not merely to trapped surfaces, but an \textit{apparent horizon}, a smooth 3-dimensional hypersurface consisting of marginally outer trapped surfaces. In 2019, An \cite{A19} produces a 55-page proof of trapped surface formation for the vacuum equations, making use of the signature for decay rates and obtaining an existence result from a region close to past null infinity. In \cite{AnAth}, An and the first author extend \cite{A19} to the case of the Einstein-Maxwell system. 

\subsection{Yang-Mills}
\noindent The Non-abelian Yang-Mills dynamics has long been an area of high interest and intense investigation \cite{eardley1982global, eardley1982global2,klainerman1995finite, satah, ghanem}. Classically, Yang-Mills fields correspond to special types of interacting waves and their rich non-linear structure  due to the non-abelian characteristics at the level of partial differential equations has led to development of sophisticated techniques. Of course, at the quantized level Yang-Mills theory plays the most important role in the contemporary high energy particle physics. More specifically, non-abelian Yang-Mills fields describe Strong and weak interactions, two of the four fundamental forces present in nature. The pure Yang-Mills theory while quantized, however, exhibits a non-trivial mass gap \cite{jaffe2006quantum, douglas2004report, mondal2023geometric} i.e., the excitation of the vacuum requires a non-zero energy and this fundamentally separates it from the U(1) Yang-Mills theory the electromagnetism where both classical and quantum spectra are gap-less. Classical Yang-Mills fields on the other hand, do not exhibit any spectral gap and therefore its existence is not expected. The scenario is substantially different while Yang-Mills fields are coupled to other sources. For example, the Yang-Mills fields that acquire mass at the classical level through interaction with scalar fields can exist in nature. For example, the carrier of the weak force $W$ and $Z$ bosons may exist classically due to the fact that they obtain mass through the spontaneous symmetry breaking (Higgs mechanism). In other words, while pure Yang-Mills fields can not exist in nature all by itself, coupling to other source field may lead to a different outcome. In this article, we want to investigate the dynamics of Yang-Mills field while coupled to gravity and possible condensation property.

Contrary to the Einstein-Maxwell system or Einstein-scalar field system, the Einstein-Yang-Mills equations are tremendously rich even in spherical symmetry and exhibit non-trivial dynamics. The numerical result of Bartnik \cite{bartnik1988particlelike} first showed the existence of a countable family of soliton-type solutions that are globally regular. Later Yau, Wasserman, and Smoller \cite{smoller1991smooth} rigorously proved the existence of such soliton-like solutions. However, such solutions were proven to be unstable against perturbations \cite{straumann1990instability}. \cite{smoller1993existence} first rigorously proved the existence of an infinite number of black hole solutions to the Einstein-Yang-Mills equations with gauge group SU(2) in spherical symmetry. The existence of these nontrivial solutions essentially unfolds the rich characteristics of the Einstein-Yang-Mills system. Due to the non-linear characteristics of the Yang-Mills fields, the fully coupled Einstein-Yang-Mills system is dynamically flexible i.e., both the possibility of the existence of regular solutions and the formation of singularities (naked or trapped behind a horizon) are open. This is precisely due to the fact that the non-linearity of Yang-Mills fields can counterbalance the non-linearity of gravity. The important feature, however, is that the nature of this balance is delicate and the possibility of collapse is open.

Motivated by these previous studies on spherically symmetric static Einstein-Yang-Mills system, one would naturally ask the following question "\textit{given suitable initial data, can one prove the formation of an Einstein-Yang-Mills black hole in an evolutionary manner?}". Moreover, it is desirable to consider a generic data i.e., remove any symmetry assumption on the spacetime. The first step towards answering such a question would be prove the formation of a trapped surface. In a recent pre-print \cite{M-Y}, the second and third authors provided a semi-global existence result for the characteristic initial value problem for the Einstein-Yang-Mills equations for generic initial data. A natural next step would be to investigate if one can form trapped surface in an evolutionary manner. A naive though would be to incorporate a certain `largeness' in the initial data. However, in the presence of gravity, this issue is far from obvious due to the possibility of a naked singularity. In addition, the gravity-Yang-Mills competition is delicate and one ought to identify the exact Yang-Mills curvature component (suitable gauge-invariant norm) on which to impose the largeness condition. The initial data should be chosen appropriately to be able to evolve the spacetime long enough to form a trapped surface. The most important issue that arises in the context of Yang-Mills theory is the choice of gauge for the Yang-Mills fields. Since Yang-Mills theory is a gauge theory after all, one ought to work in a particular choice of gauge (or equivalently descend to the `orbit space' of the theory). Unfortunately, there does not exist a global gauge choice in Yang-Mills theory (in topological terms, one can not find a single chart to cover the entire orbit space \cite{singer1978some}). The traditional choice of Lorentz gauge is known to develop finite time coordinate singularities in non-abelian theory contrary, to the linear Maxwell theory where such a breakdown is absent.  In fact, the geometry of the orbit space of the theory (i.e., the space of connections modulo the bundle automorphisms) has a non-trivial effect in the matter of gauge choice \cite{babelon1981riemannian, narasimhan1979geometry}. The positivity of sectional curvature \cite{babelon1981riemannian} of the orbit space leads to the development of the so-called `Gribov horizon'  \cite{moncrief1979gribov} which essentially indicates the breakdown of the so called `Coulomb' gauge.

 We circumvent the problem of Yang-Mills gauge choice in the current context by developing a novel gauge invariant framework for the Yang-Mills fields, where we write down the Yang-Mills equations in the double null gauge in terms of the fully gauge covariant derivative. By the virtue of the compatibility of the fiber metrics with the fully gauge covariant derivative, we are able to obtain the necessary gauge-invariant energy estimates \footnote{To our knowledge this gauge-invariant framework is not utilized in any previous work related to Yang-Mills equations.}. The true non-linear characteristics of the Yang-Mills theory manifest themselves in the higher-order energy estimates when we commute the fully gauge covariant derivatives. The estimates pass through to the orbit space (the true configuration space of the theory) by the virtue of its gauge-invariant property. The task, therefore, is to control these non-linear terms or, more precisely, preserve the smallness of appropriate weighted norms of them compared to the linear terms throughout the domain of existence. As it turns out, among the non-linearities present, the Yang-Mills non-liearities are the borderline ones, due to slow decay of Yang-Mills fields in the direction of past null infinity. In addition, we require control on certain additional null derivatives of the Yang-Mills curvature components. We discuss in due course how to handle this precisely. Once we obtain the necessary gauge invariant estimates, we can work in a particular collection of gauge choices to yield a semi-global existence result for the coupled theory. The latter part is standard and therefore we omit its discussion, focusing instead on the trapped surface formation mechanism.    

\subsection{Main Results}
\noindent The main contribution of this paper is the following pair of theorems:

\begin{theorem}{\textbf{(Semi-global existence result)}} \label{main1}Given $\mathcal{I} > 0$ and a positive number $N\geq 3$, there exists an $a_0= a_0(\mathcal{I})$ sufficiently large so that the following holds. Any smooth initial data set $(\chihat,\alpha^F)$ to the Einstein-Yang-Mills system \eqref{Ein}-\eqref{Yan} satisfying the following conditions:

\[1)\hspace{5mm} \sup_{0\leq \ubar\leq 1} \sum_{\substack{i \leq N+7\\ m\leq 3}}a^{-\frac{1}{2}} \lVert \nabla_4^m(\lvert u_{\infty} \rvert\nabla)^i \chihat \rVert_{L^\infty_{(S_{u_{\infty},\ubar})}} \leq \mathcal{I},   \]\[2)\hspace{5mm} \sup_{0\leq \ubar\leq 1} \sum_{\substack{i \leq N+7\\ m\leq 3}}a^{-\frac{1}{2}} \lVert \hnabla_4^m(\lvert u_{\infty} \rvert\hnabla)^i \alpha^F \rVert_{L^2_{(S_{u_{\infty},\ubar})}} \leq \mathcal{I},  \]\[ 3)\hspace{5mm} \text{The initial data on} \hspace{2mm} \ubar =0 \hspace{2mm}\text{are trivial (Minkowski)}, \]gives rise to a unique smooth solution to the Einstein-Yang-Mills system in the region \[ u_{\infty}\leq u \leq - \frac{a}{4}, \hspace{2mm} 0\leq \ubar \leq 1. \]
\end{theorem}
\begin{theorem}{\textbf{(Formation of trapped surfaces)}}\label{main2} With the notation of Theorem 1.1, if in addition to the assumptions of the previous Theorem, the initial data also satisfy the following isotropic condition:

\[ 4)\hspace{5mm} \int_0^1 \lvert u_{\infty}\rvert^2 \left(\lvert \chihat\rvert^2 + \lvert \alpha^F \rvert_{\gamma, \delta}^2\right)(u_{\infty}, \ubar^{\prime}, \theta^1, \theta^2) \dubarprime \geq a,\hspace{2mm}\text{uniformly for all}\hspace{2mm} (\theta^1, \theta^2) \in \mathbb{S}^2,\]then the spacetime arising as a solution from Theorem 1.1 has a trapped surface at $S_{-\frac{a}{4},1}$.
\end{theorem}

\subsection{Acknowledgements}
\noindent N. A. gratefully acknowledges support by an H.F.R.I. grant for postdoctoral researchers (3rd call, no. 7126). Moreover, most of this work was carried out while N.A. was based at the American College of Thessaloniki, which he also wishes to gratefully acknowledge. P.M. is supported by the Center of Mathematical Sciences and Applications at Harvard University.
	\section{Setup}
	
	\subsection{Construction of the double null gauge}

\noindent Denote by $L\mathcal{M}$ the \textit{frame bundle} of $\mathcal{M}$. We construct a double null gauge, meaning a smooth section of this bundle such that, through it, each point $p \in \mathcal{M}$ maps to a renormalized frame $(e_1, e_2, e_3, e_4) \in L\mathcal{M}$ with $g(e_3, e_4) =-2$, $g(e_A, e_B) = \delta_{AB}$ and $g(e_3, e_A) = g(e_4, e_A) =0$.
    
    \vspace{3mm}
    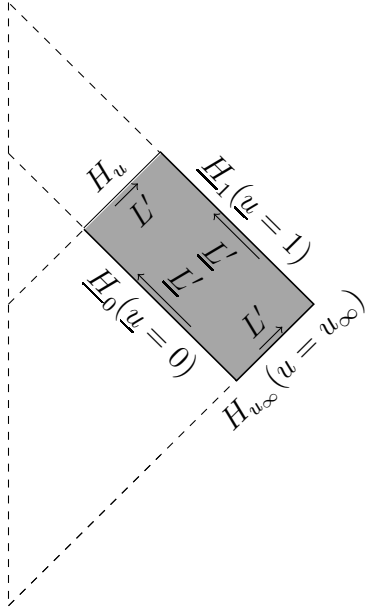
\begin{figure}
    \begin{center}
    \begin{tikzpicture}
		\draw [white](3,-1)-- node[midway, sloped, below,black]{$H_{u_{\infty}}(u=u_{\infty})$}(4,0);
		
		\draw [white](2,2)--node [midway,sloped,above,black] {$\Hb_1(\ub=1)$}(4,0);
		\draw [white](1,1)--node [midway,sloped, below,black] {$\Hb_{0}(\ub=0)$}(3,-1);
		\draw [dashed] (0, 4)--(0, -4);
		\draw [dashed] (0, -4)--(4,0)--(0,4);
		\draw [dashed] (0,0)--(2,2);
		\draw [dashed] (0,-4)--(2,-2);
		\draw [dashed] (0,2)--(3,-1);
		\draw [very thick] (1,1)--(3,-1)--(4,0)--(2,2)--(1,1);
		\fill [black!35!white]  (1,1)--(3,-1)--(4,0)--(2,2)--(1,1);
		\draw [white](1,1)-- node[midway,sloped,above,black]{$H_{u}$}(2,2);
		\draw [->] (3.3,-0.6)-- node[midway, sloped, above,black]{$L'$}(3.6,-0.3);
		\draw [->] (1.4,1.3)-- node[midway, sloped, below,black]{$L'$}(1.7,1.6);
		\draw [->] (3.3,0.6)-- node[midway, sloped, below,black]{$\Lb'$}(2.7,1.2);
		\draw [->] (2.4,-0.3)-- node[midway, sloped, above,black]{$\Lb'$}(1.7,0.4);
		\end{tikzpicture}
			\caption{Schematic depiction of the spacetime region of existence}
	\end{center}
	\end{figure}We begin with two null hypersurfaces $H_{u_{\infty}}, \Hbar_0$ and their intersection $S_{u_{\infty},0}$, a topological 2-sphere. For any point $q$ on this $2-$sphere, the tangent space $T_qS_{u_{\infty},u}$ is 2-dimensional and admits a $2-$dimensional orthogonal complement $T_q^{\perp}{S_{u_{\infty},0}}$, on which we can find two future-directed null vectors $L^{\prime}_q$ and $\Lbar^{\prime}_q$, normalized so that  \[ g(L^{\prime}_q, \Lbar^{\prime}_q) = -2. \]The pair $\begin{Bmatrix} L^{\prime}_q, \Lbar^{\prime}_q\end{Bmatrix}$ is uniquely determined up to a scaling factor $s >0$ \[\begin{Bmatrix} L^{\prime}_q, \Lbar^{\prime}_q\end{Bmatrix} \mapsto \begin{Bmatrix} s L^{\prime}_q, s^{-1}\Lbar^{\prime}_q\end{Bmatrix}.\]
	Starting from $q$ and initially tangent to $L^{\prime}_q$, a unique geodesic is sent out. Call this geodesic, $l_q$. We extend the vectorfield $L^{\prime}$ along $l_q$ by parallel transport: $D_{L^{\prime}}L^{\prime} =0$. It then follows by simple calculation that $l_q$ is null, so that $g(L^{\prime}, L^{\prime}) =0$ along $l_q$. Gathering the $\begin{Bmatrix} l_q \end{Bmatrix}$ together we get a null hypersurface $H_{u_{\infty}}$. The null hypersurface $\Hbar_0$ is obtained similarly.
	Note that, by construction, given a point $p$ on $H_{u_{\infty}}$ or $\Hbar_0$, in the corresponding tangent spaces, there is a preferred null vector $L^{\prime}_p$ or $\Lbar^{\prime}_p$.
	
	\vspace{3mm}
	We next choose a lapse function $\Omega$, which we define to be equal to 1 on $S_{u_{\infty},0}$ and then extend as a continuous function along both initial null hypersurfaces\footnote{Indeed, there is a gauge freedom in choosing $\Omega$ on the initial hypersurfaces.}. Define the vector fields
	
	\[ L := \Omega^2 L^{\prime} \hspace{2mm} \text{along}\hspace{2mm} H_{u_{\infty}} \hspace{2mm} \text{and} \hspace{2mm} \Lbar := \Omega^2 \Lbar^{\prime} \hspace{2mm} \text{along}
	\hspace{2mm} 
\Hbar_0.\]We use these vector fields to define two functions \[ \ubar \hspace{2mm} \text{on} \hspace{2mm} H_{u_{\infty}} \hspace{2mm} \text{satisfying} \hspace{2mm} L\ubar=1 \hspace{2mm} \text{on} \hspace{2mm}  H_{u_{\infty}} \hspace{2mm} \text{and} \hspace{2mm} \ubar =0 \hspace{2mm} \text{on} \hspace{2mm} S_{u_{\infty},0},\]\[ u \hspace{2mm} \text{on} \hspace{2mm} \Hbar_0 \hspace{2mm} \text{satisfying} \hspace{2mm} \Lbar u=1 \hspace{2mm} \text{on} \hspace{2mm}  \Hbar_0 \hspace{2mm} \text{and} \hspace{2mm} u =0 \hspace{2mm} \text{on} \hspace{2mm} S_{u_{\infty},0}.\]We now use these (so-called \textit{optical}) functions to proceed further with the construction. Let $S_{u_{\infty}, \ubar^{\prime}}$ be the embedded $2-$surface on $H_{u_{\infty}}$ on which $\ubar = \ubar^{\prime}$ and define $S_{u,0}$ similarly. At each point $p \in S_{u_{\infty}, \ubar^{\prime}}$ we have constructed a preferred null vector $L^{\prime}_p$. It follows that we can uniquely determine an incoming $g-$null vector $\Lbar^{\prime}_p$ satisfying $g(L^{\prime}_p, \Lbar^{\prime}_p)=-2$. Let $\underline{l}_p$ be the unique geodesic emanating from $p$ with tangent vector $\Lbar^{\prime}_p$. We extend the definition of $\Lbar^{\prime}$ along $\underline{l}_q$ by parallel transport, so that $D_{\underline{L}^{\prime}}\Lbar^{\prime}=0$. Gathering all the $\begin{Bmatrix} \underline{l}_p\end{Bmatrix}$ on $S_{u_{\infty}, \ubar^{\prime}}$, we thus obtain the null hypersurface $\Hbar_{\ubar^{\prime}}$. We obtain the null hypersurface $H_{u^{\prime}}$ in an analogous way and define $S_{u,\ubar}:= H_u\cap \Hbar_{\ubar}$. Having constructed the vector fields $L^{\prime}$ and $\Lbar^{\prime}$ in all of the spacetime region, we extend the definition of the lapse function $\Omega$ by requiring, at each point $p \in S_{u, \ubar}$ that \[ g(L^{\prime}_p, \Lbar^{\prime}_p) = -2 \restri{\Omega^{-2}}{p}. \]The incoming null hypersurfaces $\begin{Bmatrix} \Hbar_{\ubar} \end{Bmatrix}_{0\leq \ubar \leq 1}$ and outgoing null hypersurfaces $\begin{Bmatrix} H_u \end{Bmatrix}_{-\frac{a}{4} \leq u \leq u_{\infty}}$ along with their pairwise intersections $S_{u, \ubar}$ together define a \textit{double null foliation} on the spacetime. On a given $S_{u,\ubar}$, we have $g(\Omega L^{\prime}, \Omega \Lbar^{\prime}) = -2$ and hence the vectors \[e_3:= \Omega \Lbar^{\prime}, e_4:= \Omega L^{\prime}\] define a normalized null pair at each point on the sphere. We make the gauge choice $\Omega \equiv 1$ along both initial hypersurfaces.
	
\subsection{Choice of coordinates and expression of the metric}	
\noindent To define angular coordinates on each $S_{u,\ubar}$ in a smooth way, we begin by defining angular coordinates	on $S_{u_{\infty},0}$. Since this is a standard 2-sphere in Minkowki space, we can use the stereographic projection coordinates $(\theta^1, \hsp \theta^2)$ on $S_{u_{\infty},0}$. We first extend this coordinate to the whole of $\Hbar_0$ by insisting that $\slashed{\mathcal{L}}_{\Lbar}\theta^A = 0$ on $\Hbar_0$ for $A=1, \hsp 2$ and then to the whole spacetime by insisting that, for all $u$, $\slashed{\mathcal{L}}_L \theta^A =0$, where $L$ initially starts normal to some $S_{u,0}$. As such we have established a coordinate system $(u,\ubar, \theta^1, \theta^2)$ in a neighbourhood of the initial sphere. In these coordinates, the vectors $e_3, \hsp e_4$ become

\[e_3 = \Omega^{-1}\left( \frac{\partial}{\partial u} + b^A \hsp \frac{\partial}{\partial \theta^A}\right), \hspace{2mm} e_4 = \Omega^{-1} \frac{\partial}{\partial \ubar} \] and the metric now takes the following form:

\be g= -2\Omega^2 \left(\text{d}u \otimes \text{d}\ubar + \text{d}\ubar \otimes \text{d}u \right) + \gslash_{AB} \left(\text{d}\theta^A - b^A \text{d}u\right)\otimes\left(\text{d}\theta^B- b^B \text{d}u\right)  \ee It is a requirement that $b^A \equiv 0 $ on $\Hbar_0$. The section that maps $p\in \mathcal{M} \mapsto \left(\restri{\theta^1}{p},\restri{\theta^2}{p} \restri{e_3}{p}, \restri{e_4}{p}\right)$ is the double null gauge we wanted to construct.

\subsection{Yang-Mills Theory} \label{section23}
\noindent Now we define a Yang-Mills theory on $(\mathcal{M},g)$. We denote by $\mathfrak{P}$ a $C^{\infty}$ principal bundle with base a $3+1$ dimensional Lorentzian manifold $M$ and a Lie group $G$. We assume that $G$ is compact semi-simple (for physical purposes) and therefore admits a positive definite non-degenerate bi-invariant metric. Its Lie algebra $\mathfrak{g}$ by construction admits an adjoint invariant, positive definite scalar product denoted by  $\langle~,~\rangle$ which enjoys the property: for $A,B,C\in \mathfrak{g}$,
\begin{eqnarray}
\label{eq:adinvpp}
\langle[A,B],C\rangle=\langle A,[B,C]\rangle.
\end{eqnarray}
as a consequence of adjoint invariance.
A Yang-Mills connection is defined as a $1-$form $w$ on $\mathfrak{P}$ with values in $\mathfrak{g}$ endowed with compatibility properties. Its representative in a local trivialization of $\mathfrak{P}$ over $U\subset \mathcal{M}$, 
\begin{eqnarray}
\varphi: p\mapsto (x,a),~p\in \mathfrak{P},~x\in U,~a\in G
\end{eqnarray}
is the $1-$form $s^{*}w$ on $U$, where $s$ is the local section of $\mathfrak{P}$ corresponding canonically to the local trivialization  $s(x)=\varphi^{-1}(x,e)$, called a \textit{gauge}. Let $A_{1}$ and $A_{2}$ be representatives of $w$ in gauges $s_{1}$ and $s_{2}$ over $U_{1}\subset \mathcal{M}$ and $U_{2}\in \mathcal{M}$. In $U_{1}\cap U_{2}$, one has 
\begin{eqnarray}
\label{eq:gauge}
A_{1}=Ad(u^{-1}_{12})A_{2}+u_{12}\Theta_{MC},
\end{eqnarray}
where $\Theta_{MC}$ is the Maurer-Cartan form on $G$, and $u_{12}:U_{1}\cap U_{2}\to G$ generates the transformation between the two local trivializations: \begin{eqnarray}
s_{1}=R_{u_{12}}s_{2},
\end{eqnarray}
$R_{u_{12}}$ is the right translation on $\mathfrak{P}$ by $u_{12}$. Given the principal bundle $\mathfrak{P}\to \mathcal{M}$, a Yang-Mills potential $A$ on $\mathcal{M}$ is a section of the fibered tensor product $T^{*}M\otimes_{M}\mathfrak{P}_{Affine,\mathfrak{g}}$ where $\mathfrak{P}_{Affine,\mathfrak{g}}$ is the affine bundle with base $M$ and typical fibre $\mathfrak{g}$, associated to $\mathfrak{P}$ via relation (\ref{eq:gauge}). If $\widehat{A}$ is another Yang-Mills potential on $M$, then $A-\widehat{A}$ is a section of the tensor product of vector bundles $T^{*}M\otimes_{M}P_{Ad,\mathfrak{g}}$, where $\mathfrak{P}_{Ad,\mathfrak{g}}:=\mathfrak{P}\times _{Ad}\mathfrak{g}$ is the vector bundle associated to $\mathfrak{P}$ by the adjoint representation of $G$ on $\mathfrak{g}$. There is a inner product in the fibres of $\mathfrak{P}_{Ad,\mathfrak{g}}$, induced from that on $\mathfrak{g}$. The curvature $\mathbf{C}$ of the connection $w$ considered as a $1-$ form on $\mathfrak{P}$ is a $\mathfrak{g}$-valued $2-$form on $\mathfrak{P}$. Its representative in a gauge where $w$ is represented by $A$ is given by 
\begin{eqnarray}
F:=dA+[A,A],
\end{eqnarray}
and relation between two representatives $F_{1}$ and $F_{2}$ on $U_{1}\cap U_{2}$ is $F_{1}=Ad(u^{-1}_{12})F_{2}$ and therefore $F$ is a section of the vector bundle $\Lambda^{2}T^{*}M\otimes _{M}\mathfrak{P}_{Ad,\mathfrak{g}}$. For a section $\mathfrak{O}$ of the vector bundle $\otimes^{k}T^{*}M\otimes _{M}\mathfrak{P}_{Ad,\mathfrak{g}}$, a natural covariant derivative is defined as follows 
\begin{eqnarray}
\label{eq:covariant}
\widehat{D}\mathfrak{O}:=D\mathfrak{O}+[A,\mathfrak{O}],
\end{eqnarray}
where $D$ is the usual covariant derivative induced by the Lorentzian structure of $M$ and by construction $\widehat{D}\mathfrak{O}$ is a section of the vector bundle $\otimes^{k+1}T^{*}M\otimes _{M}\mathfrak{P}_{Ad,\mathfrak{g}}$. The Yang-Mills coupling constant $g_{YM}$ is set to 1. If the space of connections in a particular Sobolev class is denoted by $\mathcal{A}$ and $\mathcal{G}$ is the automorphism group of the bundle $\mathfrak{P}$, then the true configuration space (or the orbit space) of the theory is $\mathcal{A}/\mathcal{G}$\footnote{$\mathcal{A}/\mathcal{G}$ is an infinite dimensional manifold modulo certain additional criteria}. 

\noindent The classical Yang-Mills equations (in the absence of sources)
correspond to setting the natural (spacetime and gauge as defined in \ref{eq:covariant}) covariant divergence of this curvature
two-form $F$ to zero. By virtue of its definition in terms of the connection, this curvature also satisfies
the Bianchi identity that asserts the vanishing of its gauge covariant exterior derivative. Taken
together, these equations provide a geometrically natural nonlinear generalization of Maxwell's
equations (when the latter are written in terms of a `vector potential') and of course play a
fundamental role in modern elementary particle physics. If nontrivial bundles are considered
or nontrivial spacetime topologies are involved, then the foregoing so-called `local trivializations' of the bundles in question must be patched together to give global descriptions but, by virtue of the covariance of the formalism, there is a natural way of carrying out this patching procedure, at least over those regions of spacetime where the connections are well-defined. 

We choose a vector space $V$ and a matrix representation for the action of $G$ on $V$. For simplicity, let us confine our attention to real representations though in fact this restriction is inessential. We now consider vector bundles over spacetime (so-called ‘associated’ bundles) with standard fiber $\simeq V$ . Cross sections of such bundles would represent, in physical terms, multiplets of Higgs fields. To formulate field equations for such Higgs fields that are naturally covariant with respect to automorphisms of the associated vector bundle (i.e., with respect to gauge transformations acting on the Higgs fields) one needs a covariant derivative operator $\widehat{D}$ or connection defined on this bundle. Such an object is naturally induced from a ‘fundamental’ connection on the principal $G$-bundle described above and in turn induces (when expressed relative to a local trivialization) a one-form on (some local chart for) the base manifold with values in the chosen matrix representation for the Lie algebra $\mathfrak{g}$. Let us consider the dimension of the group $G$ to be $dim_{G}$ and since $\mathfrak{g}:=T_{e}G$, it has a natural vector space structure. Assume that the vector space $\mathfrak{g}$ has a basis $\{\chi_{A}\}_{A=1}^{dim_{G}}$ given by a set of $k\times k$ real valued matrices ($k$ being the dimension of the representation $V$ of the Lie algebra $\mathfrak{g}$). The connection $1-$form field is then defined to be 
\begin{eqnarray}
A:=A^{A}_{\mu}\chi_{A}dx^{\mu}=A^{A}_{\mu}(\chi_{A})^{P}_{Q}dx^{\mu}=A^{P}~_{Q\mu}dx^{\mu},~P,Q=1,2,3,...,k.
\end{eqnarray}
From now on by the connection 1-form field $A_{\mu}$, we will always mean $A^{P}~_{Q\mu}$. In the current setting $A\in \Omega^{1}(M;End(V))$, where $End(V)$ denotes the space of endomorphisms of the vector space $V$. The curvature of this connection is defined to be the Yang-Mills field $F\in \Omega^{2}(M;End(V))$
\begin{eqnarray}
F^{P}~_{Q\mu\nu}:=\partial_{\mu}A^{P}~_{Q\nu}-\partial_{\nu}A^{P}~_{Q\mu}+[A,A]^{P}~_{Q\mu\nu},
\end{eqnarray}
where the bracket is defined on the Lie algebra $\mathfrak{g}$ and given by the commutator of matrices under multiplication. The Yang-Mills coupling constant is set to unity. Since $G$ is compact, it admits a positive definite adjoint invariant metric on $\mathfrak{g}$. We choose a basis of $\mathfrak{g}$ such that this adjoint invariant metric takes the Cartesian form $\delta_{AB}$ and work with representations for which the bases satisfy 
\begin{eqnarray}
-tr(\chi_{A}\chi_{B})=(\chi_{A})^{P}_{Q}(\chi_{A})^{Q}_{P}=\delta_{AB}.
\end{eqnarray}

\noindent Under a gauge transformation by $\mathcal{U}$, the $\mathfrak{g}$-valued 1-form field $A$ transforms as 
\begin{eqnarray}
A_{\mu}\mapsto \mathcal{U}^{-1}A_{\mu}\mathcal{U}+\mathcal{U}\partial_{\mu}\mathcal{U}^{-1} 
\end{eqnarray}
and therefore $A_{\mu}$ is not a tensor in the sense that it is \textit{not} a $(1,1)$ section of the associated $V-$bundle over $\mathcal{D}_{u,\bar{u}}$. For any $\mathfrak{g}$-valued section $\mathcal{K}^{P}~_{Q\mu_{1}\mu_{2}\mu_{3}....\mu_{k}}$ of a vector bundle over $\mathcal{D}_{u,\bar{u}}$ that transforms as a tensor under the gauge transformation, the gauge covariant derivative is defined to be  
\begin{align}
\label{eq:gaugecov}
\nonumber &\hat{D}_{\alpha}\mathcal{K}^{P}~_{Q\mu_{1}\mu_{2}\mu_{3}....\mu_{k}}
:= D_{\alpha}\mathcal{K}^{P}~_{Q\mu_{1}\mu_{2}\mu_{3}....\mu_{k}}+A^{P}~_{R\alpha}\mathcal{K}^{R}~_{Q\mu_{1}\mu_{2}\mu_{3}....\mu_{k}}-A^{R}~_{Q\alpha}\mathcal{K}^{P}~_{R\mu_{1}\mu_{2}\mu_{3}....\mu_{k}}\\
=& D_{\alpha}\mathcal{K}^{P}~_{Q\mu_{1}\mu_{2}\mu_{3}....\mu_{k}}+[A
,\mathcal{K}]^{P}~_{Q\mu_{1}\mu_{2}\mu_{3}....\mu_{k}},
\end{align}
where $D_{\alpha}$ is the ordinary spacetime covariant derivative with respect to a Lorentzian metric on $\mathcal{D}_{u,\bar{u}}$. Even though the connection of the gauge bundle appears in the definition of the gauge covariant derivative, we will never make explicit use of it in the current context, but rather work with the fully gauge covariant derivative $\hat{D}$. More specifically, in our analysis, we  will encounter the commutator of the fully gauge covariant derivative which yields Riemann curvature and Yang-Mills curvature components. In other words, using the fully gauge covariant derivative, we do not encounter any connection terms, allowing us to obtain estimates in a gauge-invariant way. The commutator of the fully gauge covariant derivative while acting on a $\mathfrak{g}$-valued section of a vector bundle $\mathcal{K}^{P}~_{Q\mu_{1}\mu_{2}\mu_{3}....\mu_{k}}$ (or a section of the mixed bundle) yields 
\begin{align} \nonumber 
[\hat{D}_{\alpha},\hat{D}_{\beta}]\mathcal{K}^{P}~_{Q\mu_{1}\mu_{2}\mu_{3}....\mu_{k}}=& F^{P}~_{R\alpha\beta}\mathcal{K}^{R}~_{Q\mu_{1}\mu_{2}\mu_{3}....\mu_{k}}-F^{R}~_{Q\alpha\beta}\mathcal{K}^{P}~_{R\mu_{1}\mu_{2}\mu_{3}....\mu_{k}}\\ -&\sum_{i}R^{\gamma}~_{\mu_{i}\alpha\beta}\mathcal{K}^{P}~_{Q\mu_{1}\mu_{2}\mu_{3}...\hat{\gamma}....\mu_{k}},
\end{align}
where $\hat{\gamma}$ indicates the removal of the index $\hat{\mu}_{i}$ and replacing by $\gamma$. Note that $\hat{D}$ is compatible with both the metrics and therefore the commutator produces curvature of the mixed bundle (Yang-Mills curvature and spacetime curvature). Action of $\hat{D}$ is only well-defined on \textit{sections} of the mixed bundle.

	\subsection{The E-Y-M system expressed in the double null gauge}
	\noindent In this section we are going to express the Einstein-Yang-Mills system in the double null gauge given above. We begin by decomposing curvature components and Ricci coefficients with respect to the frame $(e_1, e_2,e_3,e_4)$. 
    Let $A, B$ take values in $\begin{Bmatrix} 1,2 \end{Bmatrix}$. We define the following (Weyl) null curvature components:

    \[ \alpha_{AB} := W(e_A, e_4, e_B, e_4), \hspace{2mm} \alphabar_{AB}:= W(e_A, e_3, e_B, e_3),       \] \[ \beta_A := \frac{1}{2} W(e_A, e_4 , e_3 ,e_4), \hspace{2mm} \betabar_A := \frac{1}{2} W(e_A, e_3, e_3, e_4),     \] \[  \rho := \frac{1}{4} W(e_3, e_4, e_3, e_4), \hspace{2mm} \sigma  = \frac{1}{4} \Hodge{W}(e_3, e_4, e_3 ,e_4).      \]Moreover, for reasons to be explained later on, we define the \textit{renormalized curvature components }

    \be \label{tbetatbetabar} \tbeta_A = \beta_A - \frac{1}{2}R_{A4}, \hspace{2mm} \tbetabar_A = \betabar_A + \frac{1}{2} R_{A3}.       \ee For the Ricci coefficients, we decompose as follows:

    \[  \chi_{AB} := g(D_A e_4, \hsp e_B), \hspace{2mm} \chibar_{AB} := g(D_A e_3, e_B),            \] \[   \eta _A := -\frac{1}{2} g(D_A e_3 , \hsp e_4), \hspace{2mm} \etabar_A := -\frac{1}{2} g(D_A e_4, e_3), \]\[  \omega := -\frac{1}{4} g(D_4 e_3 ,\hsp e_4),  \hspace{2mm} \omegabar := -\frac{1}{4} g(D_3 e_4, \hsp e_3),                \]\[ \zeta_A := \frac{1}{2} g(D_A e_4, \hsp e_3).  \]Moreover, if $\gamma$ denotes the induced metric on $S_{u,\ubar}$, we make the following further decomposition:

    \[ \chi = \chihat + \frac{1}{2} \tr\chi  \gamma, \hspace{2mm} \chibar = \chibarhat + \frac{1}{2} \tr\chibar \hsp \gamma.    \]

\noindent The Yang-Mills curvature 2-form $F$ is also decomposed into null components as follows:
	
	\begin{gather}
	    \alpha_{A}^F := \F(e_A, e_4), \hspace{2mm}
	    \alphabar_{A}^F := \F(e_A,e_3),\label{nullc1} \\ \rho^F := \frac{1}{2}\F(e_3,e_4), \hspace{2mm} \sigma^F := \frac{1}{2}\Hodge{\F}(e_3,e_4) = \F(e_1,e_2). \label{nullc2}
	\end{gather}
	
	\vspace{3mm}
	
	\par\noindent For the Yang-Mills energy-momentum tensor $\mathfrak{T}_{\mu\nu}$, there holds $g^{\alpha \beta}\mathfrak{T}_{\alpha \beta} = 0$, whence it follows that the scalar curvature vanishes and the Einstein equations reduce to $R_{\mu\nu}= \mathfrak{T}_{\mu\nu}$.

 \vspace{3mm}

 \par\noindent Before we are ready to present the equations, we introduce a few basic definitions. First of all, denote by $\nabla$ and $\hnabla$ the covariant derivative operators induced by $D$ and $\hat{D}$ on $S_{u,\ubar}$ respectively. Let $\nabla_3, \nabla_4$ denote the projections of the covariant derivatives $D_3$ and $D_4$ to $S_{u,\ubar}$ and define, in a similar fashion, $\hnabla_3$ and $ \hnabla_4$ using $\hat{D}$. For two $1-$forms $\phi_A^{1}, \hsp \phi_A^{2}$, we define 

 \[ (\phi_1 \hat{\otimes} \phi_2)_{AB} := \phi_A^1 \hsp \phi_B^2 + \phi_B^1 \hsp \phi_A^2 - \gamma_{AB}\hsp (\phi^1 \cdot \phi^2),   \]while for symmetric $2-$tensors $\phi_{AB}^1, \hsp \phi_{AB}^2$, we define 
 \[  (\phi^1 \wedge \phi^2)_{AB} := \slashed{\epsilon}^{AB} \hsp(\gamma^{-1})^{CD} \hsp\phi_{AB}^1 \hsp \phi_{CD}^2.    \]Here $\slashed{\epsilon}$ is the volume form associated with the metric $\gamma$. Moreover, by $\phi^1 \cdot \phi^2$ we mean an arbitrary contraction of the tensor product of $\phi^1$ and $\phi^2$ with respect to the metric $\gamma$. We also define suitable trace, divergence and curl operators. For totally symmetric tensors $\phi$, we define these operators as follows:
 \[  (\div \phi)_{A_1 \dots A_r}:= \nabla^{B}\phi_{B A_1 \dots A_r},   \] \[   (\curl \phi)_{A_1\dots A_r} := \slashed{\epsilon}^{BC}\nabla_B \phi_{C A_1\dots A_r}, 
   \] \[  (\tr\phi)_{A_1 \dots A_{r-1}} :=    (\gamma^{-1})^{BC} \phi_{BC A_1 \dots A_{r-1}}.        \]Be it noted that the operators $\widehat{\div}$ and $\widehat{\curl}$ can be defined similarly on sections of the mixed bundle. Furthermore, we introduce the $*$ operator on $1-$forms and $2-$tensors:
 \[ \Hodge{\phi}_A := \gamma_{AC}\slashed{\epsilon}^{CB}\hsp \phi_B,        \] \[ \Hodge{\phi}_{AB} := \gamma_{BD} \hsp \slashed{\epsilon}^{DC}\hsp \phi_{AC}.  \]
 Finally, on a $1-$form $\phi$, the operator $\nabla \hat{\otimes}$ is defined as follows:  \[(\nabla \hat{\otimes} \phi)_{A} := \nabla_B \phi_A + \nabla_A \phi_B - \gamma_{AB} \hsp \div \phi. \] 
Expressed in the double null gauge, the Einstein-Yang-Mills system turns into a system of null structure, null Bianchi and null Yang-Mills equations, which take the following form:	
	\begin{gather}
\nabla_{4}\tr\chi+\frac{1}{2}(\tr\chi)^{2}=-|\hat{\chi}|^{2}_{\gamma}-2\omega \tr\chi-\mathfrak{T}_{44}\\
\nabla_{4}\hat{\chi}+\tr\chi \hat{\chi}=-2\omega\hat{\chi}-\alpha\\
\nabla_{3}tr\underline{\chi}+\frac{1}{2}(tr\underline{\chi})^{2}=-|\hat{\underline{\chi}}|^{2}_{\gamma}-2\underline{\omega}tr\underline{\chi}-\mathfrak{T}_{33}\\
\nabla_{3}\hat{\underline{\chi}}+tr\underline{\chi}\hat{\underline{\chi}}=-2\underline{\omega}\hat{\underline{\chi}}-\underline{\alpha}\\
\label{eq:eta}
\nabla_{4}\eta_{a}=-\chi\cdot(\eta-\underline{\eta})-\beta-\frac{1}{2}\mathfrak{T}_{a4}\\
\nabla_{3}\underline{\eta}_{a}=-\underline{\chi}\cdot(\underline{\eta}-\eta)+\underline{\beta}+\frac{1}{2}\mathfrak{T}_{a3} \\
\nabla_{4}\underline{\omega}=2\omega\underline{\omega}+\frac{3}{4}|\eta-\underline{\eta}|^{2}-\frac{1}{4}(\eta-\underline{\eta})\cdot(\eta+\underline{\eta})-\frac{1}{8}|\eta+\underline{\eta}|^{2}+\frac{1}{2}\rho+\frac{1}{4}\mathfrak{T}_{43} \\
\nabla_{3}\omega=2\omega\underline{\omega}+\frac{3}{4}|\eta-\underline{\eta}|^{2}+\frac{1}{4}(\eta-\underline{\eta})\cdot(\eta+\underline{\eta})-\frac{1}{8}|\eta+\underline{\eta}|^{2}+\frac{1}{2}\rho+\frac{1}{4}\mathfrak{T}_{43} \\
\nabla_{4}tr\underline{\chi}+\frac{1}{2}\tr\chi tr\underline{\chi}=2\omega tr\underline{\chi}+2\text{div}\underline{\eta}+2|\underline{\eta}|^{2}_{\gamma}+2\rho-\hat{\chi}\cdot\hat{\underline{\chi}}\\
\nabla_{3}\tr\chi+\frac{1}{2}tr\underline{\chi}\tr\chi=2\underline{\omega}\tr\chi+2\text{div}\eta+2|\eta|^{2}+2\rho-\hat{\chi}\cdot \hat{\underline{\chi}}\\
\nabla_{4}\hat{\underline{\chi}}+\frac{1}{2}\tr\chi\hat{\underline{\chi}}=\nabla\hat{\otimes}\underline{\eta}+2\omega\hat{\underline{\chi}}-\frac{1}{2}tr\underline{\chi}\hat{\chi}+\underline{\eta}\hat{\otimes}\underline{\eta}+\hat{\mathfrak{T}}_{ab}\\
\nabla_{3}\hat{\chi}+\frac{1}{2}tr\underline{\chi}\hat{\chi}=\nabla\hat{\otimes}\eta+2\underline{\omega}\hat{\chi}-\frac{1}{2}\tr\chi\hat{\underline{\chi}}+\eta\hat{\otimes}\eta+\hat{\mathfrak{T}}_{ab}\\
\label{eq:1}
\text{div}\hat{\chi}=\frac{1}{2}\nabla \tr\chi-\frac{1}{2}(\eta-\underline{\eta})\cdot(\hat{\chi}-\frac{1}{2}\tr\chi\gamma)-\beta+\frac{1}{2}\mathfrak{T}(e_{4},\cdot)\\
\text{div}\hat{\underline{\chi}}=\frac{1}{2}\nabla tr\underline{\chi}-\frac{1}{2}(\underline{\eta}-\eta)\cdot(\hat{\underline{\chi}}-\frac{1}{2}tr\underline{\chi}\gamma)-\underline{\beta}+\frac{1}{2}\mathfrak{T}(e_{3},\cdot)\\
\text{curl}\eta=\hat{\underline{\chi}}\wedge\hat{\chi}+\sigma\epsilon=-\text{curl} \underline{\eta}\\
\label{eq:4}
K-\frac{1}{2}\hat{\chi}\cdot \hat{\underline{\chi}}+\frac{1}{4}\tr\chi tr\underline{\chi}=-\rho+\frac{1}{4}\mathfrak{T}_{43},\\
\nabla_{4}\zeta=2\nabla\omega+\chi\cdot(\etabar-\zeta)+2\omega(\zeta+\etabar)-\beta-\sigma^{F}\epsilon\cdot\alpha^{F}-\rho^{F}\cdot\alpha^{F},\\
\nabla_{3}\zeta=-2\nabla\omega-\chibar\cdot(\zeta+\eta)+2\omegabar(\zeta-\eta)-\betabar+\epsilon\sigma^{F}\cdot\alphabar^{F}-\rho^{F}\cdot\alphabar^{F}.
\end{gather}
	
	

	\vspace{3mm}
	
\par\noindent	The Bianchi equations then read as follows:
	
	\begin{equation}
	    \begin{split}
	        \nabla_3 \alpha + \frac{1}{2}\tr\chibar \alpha = &\nabla \hat{\otimes}\beta + 4 \omegabar \alpha - 3\left(\chihat \rho + \Hodge{\chihat} \sigma \right)+ (\zeta+4\eta)\hat{\otimes}\beta \\ & - D_4 R_{AB} + \frac{1}{2}\left(D_B R_{4A} + D_A R_{4B}\right) + \frac{1}{2}\left(D_4 R_{43} - D_3 R_{44}\right)\gslash_{AB},
	    \end{split}
	\end{equation}
	\begin{equation}
	        \nabla_4 \beta + 2 \tr\chi \beta = \text{div} \alpha -2 \omega \beta + \left(\eta - 2 \zeta\right)\cdot \alpha + \frac{1}{2} \left(D_4 R_{4A} - D_A R_{44}\right),
	\end{equation}
	\begin{equation}
	    \nabla_3 \beta + \tr\chibar \beta = \nabla \rho+ \Hodge{\nabla}\sigma + 2\omegabar \beta +2 \chihat \cdot \betabar + 3 \left(\eta \rho+ \Hodge{\eta}\sigma \right) + \frac{1}{2} \left( D_A R_{43} - D_4 R_{3A}\right),
	\end{equation}
	
	\begin{equation}
	    \nabla_4 \sigma + \frac{3}{2} \tr\chi \sigma = - \text{div} \Hodge{\beta} + \frac{1}{2}\hsp \chibarhat \cdot \Hodge{\alpha} - (\zeta+2 \etabar) \cdot \Hodge{\beta} - \frac{1}{4}\left( D_{\mu} R_{4\nu} - D_{\nu}R_{4\mu}\right){\epsilon^{\mu\nu}}_{34},
	\end{equation}
	\begin{equation}
	    \nabla_3 \sigma + \frac{3}{2}\tr\chibar \sigma = -\text{div}\Hodge{\betabar} + \frac{1}{2}\hsp\chihat \cdot \Hodge{\alphabar} -(\zeta + 2\eta)\cdot \Hodge{\betabar} +\frac{1}{4}\left( D_{\mu} R_{3\nu} -0 D_{\nu}R_{3\mu}\right){\epsilon^{\mu\nu}}_{34},
	\end{equation}
	\begin{equation}
	    \nabla_4 \rho + \frac{3}{2}\tr\chi \rho = \text{\div}\beta - \frac{1}{2}\chibarhat \cdot \alpha + (\zeta + 2\etabar) \cdot \beta - \frac{1}{4}\left(D_3 R_{44} - D_4 R_{43} \right), 
	\end{equation}
	\begin{equation}
	    \nabla_3 \rho + \frac{3}{2}\tr\chibar \rho = -\text{div}\betabar - \frac{1}{2}\chihat \cdot \alphabar +(\zeta-2\eta)\cdot \betabar + \frac{1}{4} \left(D_3 R_{34} - D_4 R_{33}\right), 
	\end{equation}
	
	\begin{equation}
	    \nabla_4 \betabar + \tr\chi \betabar = -\nabla \rho + \Hodge{\nabla} \sigma + 2 \omega \betabar + 2\chibarhat \cdot \beta -3\left(\etabar \rho - \Hodge{\etabar}\sigma \right) - \frac{1}{2}\left(D_A R_{43} - D_3R_{4A} \right), 
	\end{equation}
	\begin{equation}
	    \nabla_3 \betabar + 2 \tr\chibar \betabar = -\text{div}\alphabar -2 \omegabar \betabar + \etabar\cdot\alphabar +\frac{1}{2}\left(D_A R_{33} - D_3 R_{3A}\right),
	\end{equation}
	\begin{equation}
	    \begin{split}
	        \nabla_4 \alphabar + \frac{1}{2}\tr\chi \hsp \alphabar = &-\nabla \hat{\otimes}\betabar +4 \omega\hsp \alphabar -3\left(\chibarhat \rho - \Hodge{\chibarhat}\sigma \right) +\left(\zeta - 4\etabar\right)\hat{\otimes}\betabar \\ &- D_3 R_{AB} + \frac{1}{2}\left(D_A R_{3B} + D_B R_{3A}\right) + \frac{1}{2}\left(D_3 R_{34} - D_4 R_{43}\right) \slashed{g}_{AB}.
	    \end{split}
	\end{equation}
	Moreover, the Yang-Mills equations,
	
	\begin{equation}
	    \widehat{D}_{\m} {{F^P}}_{Q\n \lambda} +\widehat{D}_{\n} {{F^P}}_{Q\lambda \m}+\widehat{D}_{\lambda} {{F}^P}_{Q\m\n}=0, \hspace{2mm}  g^{\a\b}\widehat{D}_{\a}{F^P}_{Q\b\m}=0,
	\end{equation}
	take the following form:
	
\begin{equation}
    \hnabla_4 \alphabar^F +\frac{1}{2}\tr\chi\alphabar^F = - \hnabla \rho^F + \Hodge{\hnabla}\sigma^F -2 \Hodge{\etabar} \sigma^F - 2 \etabar \rho^F + 2 \omega \alphabar^F - \chibarhat \cdot \alpha^F,
\end{equation}
\begin{equation}
    \hnabla_3 \alpha^F + \frac{1}{2}\tr\chibar\alpha^F = \hnabla \rho^F + \Hodge{\hnabla}\sigma^F -2 \Hodge{\eta}\sigma^F+ 2 \eta \rho^F +2 \omegabar \alpha^F - \chihat \cdot \alphabar^F,
\end{equation}
\begin{equation}
    \hnabla_4 \rho^F = \widehat{\text{div}}\alpha^F -\tr\chi\rho^F -\left(\eta-\etabar\right)\cdot \alpha^F,
\end{equation}
\begin{equation}
    \hnabla_4 \sigma^F = -\widehat{\text{curl}}\alpha^F -\tr\chi\sigma^F + \left(\eta-\etabar\right) \cdot\Hodge{\alpha}^F,
\end{equation}
\begin{equation}
    \hnabla_3 \rho^F + \tr\chibar \rho^F = - \widehat{\text{div}} \alphabar^F + \left( \eta-\etabar \right)\cdot \alphabar^F,
\end{equation}
\begin{equation}
    \hnabla_3 \sigma^F + \tr\chibar \sigma^F = -\widehat{\text{curl}}\alphabar^F+\left(\eta-\etabar\right)\cdot\Hodge{\alphabar}^F.
\end{equation}
\begin{remark}[An important renormalization] As in \cite{AnAth}, the analysis will take place using the renormalized quantities $\tbeta$ and $\tbetabar$ introduced in \eqref{tbetatbetabar}. The benefit of using these, as also explained in \cite{AnAth}, is that the Bianchi equations can then be expressed in a way so
that the right-hand sides of the equations are directly expressible in terms of the Ricci coefficients
and the curvature and Yang-Mills components.

\end{remark}
\subsection{Integration}
Let $U$ be a coordinate patch on a $2-$sphere $S_{u,\ubar}$ and let $p_U$ be a partition of unity subordinate to $U$. For a function $\phi$, we define its integral on a $2-$sphere as well as on the null hypersurfaces $H_u$ and $\Hbar_{\ubar}$. 

\begin{equation}
    \int_{S_{u,\ubar}} \phi := \sum_{U} \int_{-\infty}^{\infty}\int_{-\infty}^{\infty} \phi \cdot p_U \cdot \sqrt{\text{det}\gamma}\hsp \text{d}\theta^1 \text{d}\theta^2,
\end{equation}
\begin{equation}
    \int_{\Hu} := \sum_{U} \int_{0}^{\ubar} \int_{-\infty}^{\infty}\int_{-\infty}^{\infty} \phi \cdot2 \hsp p_U \cdot \Omega \cdot \sqrt{\text{det}\gamma}\hsp \text{d}\theta^1 \text{d}\theta^2 \dubarprime,
\end{equation}
\begin{equation}
    \int_{\underline{H}_{\ubar}^{(u_{\infty},u)}} := \sum_{U} \int_{u_{\infty}}^{u} \int_{-\infty}^{\infty}\int_{-\infty}^{\infty} \phi \cdot2 \hsp p_U \cdot \Omega \cdot \sqrt{\text{det}g}\hsp \text{d}\theta^1 \text{d}\theta^2 \duprime,
\end{equation}For a spacetime region $\mathcal{D}_{u,\ubar} : =  \begin{Bmatrix} \left( u^{\prime}, \ubar^{\prime}, \theta^1, \theta^2 \right) \hsp \mid \hsp u_{\infty} \leq u^{\prime} \leq u, 0\leq \ubar^{\prime} \leq \ubar \end{Bmatrix}$, we define the spacetime integral

\begin{equation}
      \int_{\mathcal{D}_{u,\ubar}} \phi := \sum_{U} \int_{u_{\infty}}^{u} \int_{0}^{\ubar}  \int_{-\infty}^{\infty}\int_{-\infty}^{\infty} \phi \cdot  p_U \cdot \Omega^2 \cdot \sqrt{-\text{det}g}\hsp \text{d}\theta^1 \text{d}\theta^2 \dubarprime\duprime.
\end{equation}We proceed with the definition of $L^p$ norms $(1\leq p < \infty)$ for an arbitrary tensorfield $\phi$:

\begin{equation}
 \LpSu{\phi}^p :=   \int_{\Suu} \langle \phi, \phi \rangle_{\gamma}^{\frac{p}{2}}
\end{equation}
\begin{equation}
    \LpHu{\phi}^p := \int_{\Hu}\langle \phi, \phi \rangle_{\gamma}^{\frac{p}{2}}
\end{equation}
\begin{equation}
    \LpHbaru{\phi}^p := \int_{\Hbu}\langle \phi, \phi \rangle_{\gamma}^{\frac{p}{2}}.
\end{equation}For the case $p=\infty$, we separately define 
\begin{equation}
    \LinftySu{\phi} := \underset{(\theta^1, \theta^2) \in \Suu}{\text{sup}}\langle \phi, \phi \rangle_{\gamma}^{\frac{1}{2}}(\theta^1,\theta^2).
\end{equation}

\subsection{Signature for decay rates and scale-invariant norms}

\noindent Perhaps the most challenging aspect of trapped surface formation results, historically, has been the attempt to find initial data that are, in an appropriate sense, large (this is by necessity, as is implied by the monumental work of \cite{ChrKl}) but also small enough to allow for an existence result of a spacetime region that gives trapped surfaces the time they would require to form. The first such initial data set, in the absence of symmetries, was given by \cite{C09}. Later contributions include \cite{Kl-Rod}, \cite{AL17} and \cite{A17}. Moreover, one would have to construct norms which preserve, at least approximately, the hierarchy present in the initial data upon evolution of the Einstein equations. The signature for decay rates, which was first introduced in \cite{AnThesis}, is the tool we will use in the present paper to build \textit{scale-invariant norms}. These will be norms that, upon evolution of the initial data, remain bounded above by a uniform constant (with the exception of a few anomalous terms). For  another application of this framework, see \cite{AnAth}.

\vspace{3mm}

To each $\phi \in \begin{Bmatrix}
\alpha, \alphabar, \tbeta, \tbetabar, \rho, \sigma, \eta, \etabar, \chi, \chibar, \omega, \omegabar, \zeta,\gamma
\end{Bmatrix}$ we associate its \textit{signature for decay rates} $s_2(\phi)$:

\[ s_2(\phi) = 0\cdot N_4(\phi) + \frac{1}{2}N_A(\phi) + 1\cdot N_3(\phi)-1.\]Here $N_\alpha(\phi)$ $(\alpha = 1,2,3,4)$ denotes the number of times $e_\alpha$ appears in the definition of $\phi$. We also extend the definition of $s_2$ in the same way to the set of null components of the curvature $2-$form $F$, which we will schematically  denote by $\mathcal{Y}$. We get the following tables of signatures:

{\renewcommand{\arraystretch}{1.25}
\begin{center}
\begin{tabular}{||c || c c c c c c c c c c c c c c||} 
 \hline$\phi$ &
 $\alpha$ & $\alphabar$ & $\tbeta$ & $\tbetabar$ &$\rho$ &$\sigma$ & $\eta$ & $\etabar$ & $\chi$ &$\chibar$ & $\omega$ & $\omegabar$ & $\zeta$ & $\gamma$  \\ [1ex]  \hline\hline
 $s_2(\phi)$ & 0 & 2 & 0.5 & 1.5 & 1  & 1 & 0.5 & 0.5  &0 & 1&0 & 1&0.5 &0 \\ 
 \hline
\end{tabular}
\end{center}

\begin{center}
\begin{tabular}{||c || c c c c||} 
 \hline$\mathcal{Y}$ &
 $\alpha^F$ & $\alphabar^F$ & $\rho^F$ & $\sigma^F$  \\ [1ex]  \hline\hline
 $s_2(\mathcal{Y})$ & 0 & 1 & 0.5 & 0.5 \\ 
 \hline
\end{tabular}
\end{center}

\par\noindent Several properties of $s_2$ follow:

\[ s_2(\nabla_4 \phi) = s_2(\phi), \hspace{2mm} s_2(\hnabla_4 \mathcal{Y}) = s_2(\mathcal{Y}), \]\[s_2(\nabla \phi) = s_2(\phi) +\frac{1}{2}, \hspace{2mm} s_2(\hnabla \mathcal{Y}) = s_2(\mathcal{Y}) + \frac{1}{2},\]\[s_2(\nabla_3 \phi) = s_2(\phi) +1, \hspace{2mm} s_2(\hnabla_3 \mathcal{Y}) = s_2(\mathcal{Y}) + 1.\]Finally, perhaps the most important property of $s_2$ is \textit{signature conservation}: \be \label{sc} s_2(\phi_1 \cdot\phi_2) = s_2(\phi_1)+ s_2(\phi_2), \hspace{2mm} s_2(\phi \cdot \mathcal{Y}) = s_2(\phi)+ s_2(\mathcal{Y}) \hspace{2mm} \text{and}\hspace{2mm} s_2(\mathcal{Y}_1 \cdot \mathcal{Y}_2) = s_2(\mathcal{Y}_1)+ s_2(\mathcal{Y}_2). \ee This allows for the (almost)-preservation of the scale-invariant norms upon evolution, as we shall see. 

\vspace{3mm}

For any horizontal tensor-field $\phi$ and Yang-Mills component $\mathcal{Y}$, we define the following norms:

\begin{equation}
    \scaleinfinitySu{\phi \vee \mathcal{Y}} := a^{-s_2(\phi \vee \mathcal{Y})} \lvert u \rvert^{2s_2(\phi \vee \mathcal{Y})+1}\inftySu{\phi \vee \mathcal{Y}},
\end{equation}
\begin{equation}
       \scaletwoSu{\phi \vee \mathcal{Y}} := a^{-s_2(\phi \vee \mathcal{Y})} \lvert u \rvert^{2s_2(\phi \vee \mathcal{Y})}\twoSu{\phi \vee \mathcal{Y}},
\end{equation}
\begin{equation}
       \scaleoneSu{\phi \vee \mathcal{Y}} := a^{-s_2(\phi \vee \mathcal{Y})} \lvert u \rvert^{2s_2(\phi \vee \mathcal{Y})-1}\oneSu{\phi \vee \mathcal{Y}},
\end{equation}Notice the difference in the $u$-weights amongst the definitions. 

\vspace{3mm}

A crucial property of the above norms is the \textit{scale-invariant H\"older's inequalities} that they satisfy. For $\Y$ denoting an arbitrary $\phi$ or $\YM$, there hold:
\begin{equation}
\scaleoneSu{\Y_1 \cdot \Y_2} \leq \frac{1}{\lvert u \rvert} \scaletwoSu{\Y_1}\scaletwoSu{\Y_2},
\end{equation}
\begin{equation}
    \scaleoneSu{\Y_1 \cdot \Y_2} \leq \frac{1}{\lvert u \rvert} \scaleinfinitySu{\Y_1}\scaleoneSu{\Y_2},
\end{equation}
\begin{equation}
    \label{257}\scaletwoSu{\Y_1 \cdot \Y_2} \leq \frac{1}{\lvert u \rvert} \scaleinfinitySu{\Y_1}\scaletwoSu{\Y_2}.
\end{equation}Notice that this is possible partly thanks to the signature conservation property \eqref{sc}. In the region of study, the factor $\frac{1}{\lvert u\rvert}$ plays the role that $\delta^{\frac{1}{2}}$ plays in the definition of the corresponding norms\footnote{See Section 2.17 in [Kl-Rod].} in [Kl-Rod], namely that of measuring the \textit{smallness} of the nonlinear terms. The above inequalities are the primary tools that will be used to close the bootstrap argument required for the existence part.

\subsection{Norms}\label{Norms}
\noindent Let $N \geq 3$ be a natural number. Let $\psi_g \in \begin{Bmatrix} \omega, \tr\chi, \omegabar, \eta,\etabar \end{Bmatrix}$, $\Psi_u \in \begin{Bmatrix} \tbeta, \rho, \sigma,\tbetabar \end{Bmatrix}$ and $\Psi_{\ubar} \in \begin{Bmatrix} \rho, \sigma, \tbetabar, \alphabar \end{Bmatrix}$. Moreover, we will sometimes use $\Psi$ to denote an arbitrary $\Psi_u$ or a $\Psi_{\ubar}$. Also, define $\tildetr := \tr\chibar + \frac{2}{\lvert u\rvert}$ and let $\mathcal{Y}_{\ubar}\in \begin{Bmatrix}\rho^F, \sigma^F,\alphabar^F\end{Bmatrix}$. For $0\leq i \leq N$, we define

\begin{equation}
    \begin{split}
        \bbGamma_{i,\infty}(u,\ubar) := \frac{1}{\al}\scaleinfinitySu{\aln \chihat} &+ \scaleinfinitySu{\aln \psi_g} + \frac{\al}{\lvert u \rvert} \scaleinfinitySu{\aln \chibarhat}\\ &+ \frac{a}{\lvert u \rvert^2}\scaleinfinitySu{\aln \tr\chibar} + \frac{a}{\lvert u \rvert} \scaleinfinitySu{\aln \tildetr},
    \end{split}
 \end{equation}
   \be \mathcal{R}_{i, \infty}(u,\ubar) := \frac{1}{\al}\scaleinfinitySu{\aln \alpha} + \scaleinfinitySu{\aln \Psi_{\ubar}}, \ee
    \be \begin{split} \mathbb{YM}_{\hsp i,\infty}(u,\ubar) := &\frac{1}{\al}\scaleinfinitySu{(\al\hnabla)^i \alpha_F} + \scaleinfinitySu{(\al\hnabla)^i \mathcal{Y}_{\ubar}} 
.\end{split} \ee
For $0 \leq i \leq N-1$, define
\begin{equation}
\mathbb{YM}^{\hsp \mathcal{C}}_{j,\infty}(u,\ubar):=  \frac{1}{\al}\scaleinfinitySu{\haln \hnabla_4 \alphaF} + \frac{a}{\lvert u \rvert} \scaleinfinitySu{\haln \hnabla_3 \alphabarF}. \ee
Furthermore, for $0\leq i \leq N+3$ and $0\leq j \leq N+4$, we define

\begin{equation}
      \begin{split}
        \bbGamma_{j,2}(u,\ubar) := \frac{1}{\al}\scaletwoSu{\aln \chihat} &+ \scaletwoSu{\aln \psi_g} + \frac{\al}{\lvert u \rvert} \scaletwoSu{\aln \chibarhat}\\ &+ \frac{a}{\lvert u \rvert^2}\scaletwoSu{\aln \tr\chibar} + \frac{a}{\lvert u \rvert} \scaletwoSu{\aln \tildetr},
    \end{split}
\end{equation}
\begin{equation}
    \mathcal{R}_{i,2}(u,\ubar):= \frac{1}{\al}\scaletwoSu{\aln \alpha} + \scaletwoSu{\aln \Psi_{\ubar}}, \ee
    \be \mathbb{YM}_{\hsp j,2}(u,\ubar) := \frac{1}{\al}\scaletwoSu{(\al\hnabla)^j \alpha_F} + \scaletwoSu{(\al\hnabla)^j \mathcal{Y}_{\ubar}} .
\end{equation}For $0 \leq j \leq N+2$, define
\begin{equation}
\mathbb{YM}^{\hsp \mathcal{C}}_{j,2}(u,\ubar):=  \frac{1}{\al}\scaletwoSu{\haln \hnabla_4 \alphaF} + \frac{a}{\lvert u \rvert} \scaletwoSu{\haln \hnabla_3 \alphabarF}. \ee
Finally, for $0\leq i \leq N+4$ and $0\leq j \leq N+5$, we define the norms along the null hypersurfaces:

\begin{equation}
    \mathcal{R}_i(u,\ubar) := \frac{1}{\al}\scaletwoHu{\aln \alpha} + \scaletwoHu{\aln \Psi_{u}},
\end{equation}

\begin{equation}
    \underline{\mathcal{R}}_i(u,\ubar) := \frac{1}{\al}\scaletwoHbaru{\aln \tbeta} + \scaletwoHbaru{\aln \Psi_{\ubar}}
\end{equation}
\begin{equation}
\mathbb{YM}^{\mathcal{C}}_{j} := \frac{1}{\al}\scaletwoHu{\haln \hnabla_4 \alphaF},   
\end{equation}
\begin{equation}
\underline{\mathbb{YM}}^{\mathcal{C}}_{j} :=\intu \frac{a^3}{\upr^4}\scaletwoSuprimeubarprime{\haln \hnabla_3 \alphabarF}^2 \duprime,
\end{equation}
\begin{equation}
    \mathbb{YM}_j := \frac{1}{\al}\scaletwoHu{(\al)^{j-1} \hnabla^j \alpha^F} + \scaletwoHu{(\al)^{j-1} \hnabla^j(\rho^F, \sigma^F)}
\end{equation}  
\begin{equation}
    \underline{\mathbb{YM}}_j := \frac{1}{\al}\scaletwoHbaru{(\al)^{j-1} \hnabla^j (\rho^F, \sigma^F)} + \scaletwoHbaru{(\al)^{j-1} \hnabla^j \alphabar^F}
\end{equation}  
 We now set $\bbGamma_{i,\infty},\bbGamma_{i,2}, \mathcal{R}_{i,\infty}, \mathcal{R}_{i,2}, \mathbb{YM}_{\hsp i,\infty}, \mathbb{YM}_{\hsp i,2}$ to be the suprema over $(u,\ubar)$ of the corresponding norms. Finally, we define the total norms
 
 \begin{equation}
     \bbGamma := \sum_{0\leq i \leq N} \left( \bbGamma_{i,\infty}+\mathcal{R}_{i,\infty} + \mathbb{YM}_{\hsp i,\infty}\right) +\sum_{0\leq i \leq N+3}\mathcal{R}_{i,2} +\sum_{0\leq j \leq N+4} \left( \bbGamma_{j,2} +\mathbb{YM}_{j,2}\right) 
 \end{equation}
   \begin{equation}
       \mathcal{R}:= \sum_{0\leq i \leq N+4} \mathcal{R}_i + \underline{\mathcal{R}}_i, \hspace{2mm}\mathbb{YM}:= \sum_{0 \leq j \leq N+5} \mathbb{YM}_j + \underline{\mathbb{YM}}_j.
   \end{equation}
   
\noindent In addition, we also define the following top-order total norm for connection coefficients, that is to be controlled by means of separate elliptic estimates in terms of $\bbGamma, \mathcal{R}$ and $\mathbb{YM}$:
\begin{align}
 \nonumber \mathcal{O}_{N+5,2}:=&\frac{1}{a^{\frac{1}{2}}}\|(a^{\frac{1}{2}})^{N+4}\nabla^{N+5}\widehat{\chi}\|_{\mathcal{L}^{2}_{sc}(H^{(0,\underline{u})}_{u})}+\|(a^{\frac{1}{2}})^{N+4}\nabla^{N+5}(\tr\chi,\omega)\|_{\mathcal{L}^{2}_{sc}(H^{(0,\underline{u})}_{u})}\\\nonumber 
+&\|(a^{\frac{1}{2}})^{N+4}\nabla^{N+5}\eta\|_{\mathcal{L}^{2}_{sc}(\underline{H}^{(u_{\infty},u)}_{\underline{u}})}+\frac{a}{|u|}\|(a^{\frac{1}{2}})^{N+4}\nabla^{N+5}(\eta,\underline{\eta})\|_{\mathcal{L}^{2}_{sc}(H^{(0,\underline{u})}_{u})}\\\nonumber 
+&\|(a^{\frac{1}{2}})^{N+4}\nabla^{N+5}\underline{\omega}\|_{\mathcal{L}^{2}_{sc}(\underline{H}^{(u_{\infty},u)}_{\underline{u}})}+ \int_{u_{\infty}}^{u}\frac{a^{\frac{3}{2}}}{|u^{'}|^{3}}\|(a^{\frac{1}{2}})^{N+4}\nabla^{N+5}\widehat{\underline{\chi}}\|_{\mathcal{L}^{2}_{sc}(S_{u^{'},\underline{u}})}du^{'}\\
+&\int_{u_{\infty}}^{u}\frac{a^{2}}{|u^{'}|^{3}}\|(a^{\frac{1}{2}})^{N+4}\nabla^{N+5}\tr\underline{\chi}\|_{\mathcal{L}^{2}_{sc}(S_{u^{'},\underline{u}})}du^{'}.
\end{align}
\subsection{Commutation Formulae}
\noindent Use the definition of the gauge-covariant derivative (\ref{eq:gaugecov}) and project it onto the topological 2-sphere $S_{u,\ubar}$ to yield
\begin{eqnarray}
[\hnabla_{4},\hnabla_{B}]\mathcal{G}^{P}~_{QA_{1}A_{2}A_{3}\cdot\cdot\cdot\cdot A_{n}}=[\widehat{D}_{4},\widehat{D}_{B}]\mathcal{G}^{P}~_{QA_{1}A_{2}A_{3}\cdot\cdot\cdot\cdot A_{n}}\nonumber+(\nabla_{B}\log\Omega)\hnabla_{4}\mathcal{G}^{P}~_{QA_{1}A_{2}A_{3}\cdot\cdot\cdot\cdot A_{n}}\\\nonumber 
-\gamma^{CD}\chi_{BD}\hnabla_{C}\mathcal{G}^{P}~_{QA_{1}A_{2}A_{3}\cdot\cdot\cdot\cdot A_{n}}-\sum_{i=1}^{n}\gamma^{CD}\chi_{BD}\underline{\eta}_{A_{i}}\mathcal{G}^{P}~_{QA_{1}A_{2}A_{3}\cdot\cdot\hat{A}_{i}C\cdot\cdot A_{n}}\\\nonumber 
+\sum_{i=1}^{n}\gamma^{CD}\chi_{A_{i}B}\underline{\eta}_{D}\mathcal{G}^{P}~_{QA_{1}A_{2}A_{3}\cdot\cdot\hat{A}_{i}C\cdot\cdot A_{n}}
\end{eqnarray}
\begin{eqnarray}
[\widehat{D}_{4},\widehat{D}_{A}]\mathcal{G}^{P}~_{QA_{1}A_{2}\cdot\cdot\cdot\cdot A_{n}}=-\sum_{i}R(e_{C},e_{A_{i}},e_{4},e_{A})\mathcal{G}^{P}~_{QA_{1}\cdot\cdot\hat{A}_{i},\cdot\cdot A_{n}}
\nonumber+F^{P}~_{R4A}\mathcal{G}^{R}~_{QA_{1}A_{2}\cdot\cdot\cdot A_{n}}\\-F^{R}~_{Q4A}\mathcal{G}^{P}~_{RA_{1}A_{2}\cdot\cdot\cdot A_{n}}\nonumber +(\nabla_{A}\log\Omega)\hnabla_{4}\mathcal{G}^{P}~_{QA_{1}A_{2}\cdot\cdot\cdot\cdot A_{n}}.
\end{eqnarray}
Notice that the last term is redundant since it already appears in the previous expression. We need to take care of the curvature terms.

\begin{eqnarray}
[\hnabla_{4},\hnabla_{A}]\mathcal{G}\sim(\beta+\alpha^{F}\cdot(\rho^{F}+\sigma^{F}))\mathcal{G}+\alpha^{F}\mathcal{G}\nonumber+(\eta+\underline{\eta})\hnabla_{4}\mathcal{G}-\chi\hnabla\mathcal{G}+\chi\underline{\eta}\mathcal{G}.
\end{eqnarray}
Notice that if $\mathcal{G}$ is not a section of the vector bundle $~^{k}\otimes T^{*}S\otimes P_{Ad,\mathfrak{g}}$ ($P_{AD,\mathfrak{g}}$ is the vector bundle associated to $P$ by the adjoint representation of the gauge group $G$ on $\mathfrak{g}$), then $\alpha^{F}\mathcal{G}$ does not appear and the gauge covariant derivatives are simply ordinary covariant derivatives. For higher order commutation, we have the following lemma:\\
\begin{lemma} 
\label{commutation}
\textit{Suppose $\mathcal{G}$ is a section of the product vector bundle $~^{k}\otimes T^{*}\mathbb{S}^{2}\otimes P_{Ad,\rho}$, $k\geq 1$,  that satisfies $\hnabla_{4}\mathcal{G}=\mathcal{F}_{1}$ and $\hnabla_{4}\hnabla^{I}\mathcal{G}=\mathcal{F}^{I}_{1}$, then $\mathcal{F}^{I}_{1}$ verifies the following schematic expression:}
\begin{equation}
\begin{split}
\mathcal{F}^{I}_{1}\sim&\sum_{J_{1}+J_{2}+J_{3}+J_{4}=I-1}\nabla^{J_{1}}(\eta+\underline{\eta})^{J_{2}}\nabla^{J_{3}}\beta\hnabla^{J_{4}}\mathcal{G}\\+&\sum_{J_{1}+J_{2}+J_{3}+J_{4}+J_{5}=I-1}\nabla^{J_{1}}(\eta+\underline{\eta})^{J_{2}}\hnabla^{J_{3}}\alpha^{F}\hnabla^{J_{4}}(\rho^{F},\sigma^{F})\hnabla^{J_{5}}\mathcal{G}\\
+&\sum_{J_{1}+J_{2}+J_{3}+J_{4}=I-1}\nabla^{J_{1}}(\eta+\underline{\eta})^{J_{2}}\hnabla^{J_{3}}\alpha^{F}\hnabla^{J_{4}}\mathcal{G}+\sum_{J_{1}+J_{2}+J_{3}=I}\nabla^{J_{1}}(\eta+\underline{\eta})^{J_{2}}\hnabla^{J_{3}}\mathcal{F}_{1}\\+&\sum_{J_{1}+J_{2}+J_{3} +J_{4}=I}\nabla^{J_{1}}(\eta+\underline{\eta})^{J_{2}}\hat{\nabla}^{J_{3}}\chi\hnabla^{J_{4}}\mathcal{G}.
\end{split}
\end{equation}
Similarly, for $\hnabla_{3}\mathcal{G}=\mathcal{F}_{2}$, and $\hnabla_{3}\hnabla^{I}\mathcal{G}=\mathcal{F}^{I}_{2}$, 
\begin{equation}
\begin{split}
\mathcal{F}^{I}_{2}\sim &\sum_{J_{1}+J_{2}+J_{3}+J_{4}=I-1}\nabla^{J_{1}}(\eta+\underline{\eta})^{J_{2}}\nabla^{J_{3}}\underline{\beta}\hnabla^{J_{4}}\mathcal{G}\\+&\sum_{J_{1}+J_{2}+J_{3}+J_{4}+J_{5}=I-1}\nabla^{J_{1}}(\eta+\underline{\eta})^{J_{2}}\hnabla^{J_{3}}\underline{\alpha}^{F}\hnabla^{J_{4}}(\rho^{F},\sigma^{F})\hnabla^{J_{5}}\mathcal{G}\\
+&\sum_{J_{1}+J_{2}+J_{3}+J_{4}=I-1}\nabla^{J_{1}}(\eta+\underline{\eta})^{J_{2}}\hnabla^{J_{3}}\underline{\alpha}^{F}\hnabla^{J_{4}}\mathcal{G}+\sum_{J_{1}+J_{2}+J_{3}=I}\nabla^{J_{1}}(\eta+\underline{\eta})^{J_{2}}\hnabla^{J_{3}}\mathcal{F}_{2}\\+&\sum_{J_{1}+J_{2}+J_{3} +J_{4}=I}\nabla^{J_{1}}(\eta+\underline{\eta})^{J_{2}}\hat{\nabla}^{J_{3}}\underline{\chi}\hnabla^{J_{4}}\mathcal{G}.
\end{split}
\end{equation}
\end{lemma}
\begin{proof}
For $I=1$, this identity is clearly satisfied due to the calculations above. Assume, it holds for $J=I-1$ and show that it holds for $J=I$. We omit the proof and refer to \cite{AnAth}. 
\end{proof} 

\begin{remark}
In both of the expressions, one can replace $\beta$ ($\underline{\beta}$ resp.) by the re-normalized curvature $\beta^{R}_{A}:=\beta_{A}-\frac{1}{2}R_{4A}$ ($\underline{\beta}^{R}_{A}:=\underline{\beta}_{A}+\frac{1}{2}R_{3A}$ resp.) and remove the quadratic terms $\alpha^{F}(\rho^{F},\sigma^{F})$ ($\underline{\alpha}^{F}(\rho^{F},\sigma^{F})$ resp.). By moving the top derivatives of $\mathcal{G}$ multiplied by $\tr\underline{\chi}$ from the right hand side to the left hand side, one may also obtain 
\begin{equation}
\begin{split}
\mathcal{F}^{I}_{2}+\frac{I}{2}\tr\underline{\chi}\hat{\nabla}^{I}\mathcal{G}\sim&\sum_{J_{1}\nonumber+J_{2}+J_{3}+J_{4}=I-1}\nabla^{J_{1}}(\eta+\underline{\eta})^{J_{2}}\nabla^{J_{3}}\underline{\beta}^{R}\hnabla^{J_{4}}\mathcal{G}
\\+&\sum_{J_{1}+J_{2}+J_{3}+J_{4}=I-1}\nabla^{J_{1}}(\eta+\underline{\eta})^{J_{2}}\hnabla^{J_{3}}\underline{\alpha}^{F}\hnabla^{J_{4}}\mathcal{G}\\\nonumber+&\sum_{J_{1}+J_{2}\nonumber+J_{3}=I}\nabla^{J_{1}}(\eta+\underline{\eta})^{J_{2}}\hnabla^{J_{3}}\mathcal{F}_{2}\\ +&\sum_{J_{1}+J_{2}+J_{3} +J_{4}=I}\nabla^{J_{1}}(\eta+\underline{\eta})^{J_{2}}\hat{\nabla}^{J_{3}}\widehat{\underline{\chi}}\hnabla^{J_{4}}\mathcal{G}\\\nonumber+&\sum_{J_{1}+J_{2}+J_{3} +J_{4}=I-1}\nabla^{J_{1}}(\eta+\underline{\eta})^{J_{2}+1}\hat{\nabla}^{J_{3}}\tr\underline{\chi}\hnabla^{J_{4}}\mathcal{G}.
\end{split}
\end{equation}
Finally, if one takes into account the identities

\[ \beta^R_A = -(\text{div}\chihat)_A + \frac{1}{2}\nabla_A \tr\chi - \frac{1}{2}\left((\eta - \etabar)\cdot (\chihat - \frac{1}{2}\tr\chi \gamma) \right)_A - \frac{1}{2} R_{4A},\]
\[ \betabar^R_A = (\text{\div}\chibarhat)_A -\frac{1}{2}\nabla_A \tr\chibar -\frac{1}{2} \left( (\eta-\etabar)\cdot (\chibarhat-\frac{1}{2}\tr\chibar \gamma) \right)_A + \frac{1}{2}R_{3A},          \]we arrive at the following identities:
\begin{equation}
    \begin{split}
    \label{c1}    \mathcal{F}_1^I \sim& \sum_{J_1+J_2+J_3=I} \nabla^{J_1}(\eta+\etabar)^{J_2}\hnabla^{J_3}\mathcal{F}_1 \\&+\sum_{J_1+J_2+J_3 +J_4=I} \nabla^{J_1}(\eta+\etabar)^{J_2}\nabla^{J_3}(\chihat,\tr\chi)\hnabla^{J_4} \mathcal{G} \\ &+\sum_{J_1+J_2+J_3+J_4=I-1 }\nabla^{J_1} (\eta+\etabar)^{J_2+1}\nabla^{J_3}(\chihat,\tr\chi)\hnabla^{J_4}\mathcal{G}\\&+ \sum_{J_{1}+J_{2}+J_{3}+J_{4}+J_{5}=I-1}\nabla^{J_{1}}(\eta+\underline{\eta})^{J_{2}}\hnabla^{J_{3}}\alpha^{F}\hnabla^{J_{4}}(\rho^{F},\sigma^{F})\hnabla^{J_{5}}\mathcal{G} 
        \\&+ \underbrace{\sum_{J_1+J_2+J_3+J_4=I-1}\nabla^{J_1}(\eta+\etabar)^{J_2}\hnabla^{J_3}\alpha^F \hnabla^{J_4}\mathcal{G}}_{C_{1}}
    \end{split}
\end{equation}and 

\begin{equation}
    \begin{split}
   \label{c2}     \mathcal{F}_2^I + \frac{I}{2}\tr\chibar \hnabla^I\mathcal{G} &\sim \sum_{J_1+J_2+J_3=I}\nabla^{J_1}(\eta+\etabar)^{J_2}\hnabla^{J_3}\mathcal{F}_2\\ &+ \sum_{J_1+J_2+J_3+J_4=I}\nabla^{J_1}(\eta+\etabar)^{J_2}\nabla^{J_3}(\chibarhat,\hsp \tildetr)\hnabla^{J_4}\mathcal{G}\\&+\sum_{J_{1}+J_{2}+J_{3} +J_{4}=I-1}\nabla^{J_{1}}(\eta+\underline{\eta})^{J_{2} +1}\hat{\nabla}^{J_{3}}\tr\underline{\chi}\hnabla^{J_{4}}\mathcal{G}\\&+\sum_{J_1+J_2+J_3+J_4=I-1}\nabla^{J_1} (\eta+\etabar)^{J_2+1}\nabla^{J_3}(\chibarhat,\tr\chibar)\hnabla^{J_4}\mathcal{G}       \\&+ \sum_{J_{1}+J_{2}+J_{3}+J_{4}+J_{5}=I-1}\nabla^{J_{1}}(\eta+\underline{\eta})^{J_{2}}\hnabla^{J_{3}}\alphabar^{F}\hnabla^{J_{4}}(\rho^{F},\sigma^{F})\hnabla^{J_{5}}\mathcal{G}\\&+ \underbrace{\sum_{J_1+J_2+J_3+J_4=I-1}\nabla^{J_1}(\eta+\etabar)^{J_2}\hnabla^{J_3}\alphabar^F \hnabla^{J_4}\mathcal{G}}_{C_{2}}.
    \end{split}
\end{equation}
\end{remark} 

\begin{remark}
Note that the commutation terms $C_{1}$ and $C_{2}$ are produced because of the non-abelian (and hence non-linear) characteristic of the Yang-Mills theory. For any section of the bundle $\otimes^{k} T^{*}S_{u,\ubar}$, $C_{1}$ and $C_{2}$ drop out.
\end{remark}

\subsection{Gauge-invariant inequalities for scale-invariant norms}
\noindent Let the space-time gauge covariant derivative be denoted by $\widehat{D}$ and let its projection on the horizontal direction, i.e., on the topological $2-$sphere, be denoted by $\widehat{\nabla}$. More precisely $\widehat{D}_{\mu}:=D_{\mu}+[A_{\mu},\cdot]$ and $\widehat{\nabla}_{A}:=\nabla_{A}+[A(e_{A}),\cdot]$, where $\nabla$ is the horizontal component of the usual spacetime covariant derivative $D$. For $\mathcal{G}\in \Gamma(^{N}\otimes T^{*}S\otimes P_{AD,\mathfrak{g}})$, we define the gauge-invariant norm as follows:
\begin{eqnarray}
|\mathcal{G}|^{2}_{\gamma,\delta}:=\mathcal{G}^{P}~_{QA_{1}A_{2}A_{3}\cdot\cdot\cdot\cdot A_{N}}\mathcal{G}^{Q}~_{PB_{1}B_{2}B_{3}\cdot\cdot\cdot\cdot B_{N}}\gamma^{A_{1}B_{1}}\gamma^{A_{2}B_{2}}\gamma^{A_{3}B_{3}}\cdot\cdot\cdot\cdot \gamma^{A_{N}B_{N}},
\end{eqnarray}
where repeated indices $P$ and $Q$ represent the inner product with respect to the metric $\delta$ on the fibres of the vector bundle $P_{AD,\mathfrak{g}}$. Notice that due to the compatibility of the fibre metric $\delta$ with the connection $\widehat{\nabla}$, we have the following 
\begin{eqnarray}
\nabla |\mathcal{G}|^{2}_{\gamma,\delta}=2\langle \widehat{\nabla}\mathcal{G},\mathcal{G}\rangle_{\gamma,\delta},
\end{eqnarray}
where $\langle~,\rangle_{\gamma,\delta}$ denotes the joint inner product induced by the metrics $\gamma$ and $\delta$. Clearly $|\mathcal{G}|^{2}_{\gamma,\delta}=\langle\mathcal{G},\mathcal{G}\rangle_{\gamma,\delta}$. Now we proceed to define the $L^{p}$ ($1\leq p<\infty$) norms of an arbitrary section $\mathcal{G}\in \Gamma(^{N}\otimes T^{*}S\otimes P_{AD,\mathfrak{g}})$:
\begin{eqnarray}
||\mathcal{G}||^{p}_{L^{p}(S_{u,\underline{u}})}:=\int_{S_{u,\underline{u}}}\langle\mathcal{G},\mathcal{G}\rangle^{\frac{p}{2}}_{\gamma,\delta},~||\mathcal{G}||^{p}_{L^{p}(H_{u})}:=\int_{H_{u}}\langle\mathcal{G},\mathcal{G}\rangle^{\frac{p}{2}}_{\gamma,\delta},~||\mathcal{G}||^{p}_{L^{p}(\underline{H}_{\underline{u}})}:=\int_{\underline{H}_{\underline{u}}}\langle\mathcal{G},\mathcal{G}\rangle^{\frac{p}{2}}_{\gamma,\delta}.
\end{eqnarray}
In the case $p=\infty$, we define the $L^{\infty}$ norm as follows: 
\begin{eqnarray}
||\mathcal{G}||_{L^{\infty}(S_{u,\underline{u}})}:=\sup_{\theta\in S_{u,\underline{u}}}\langle\mathcal{G},\mathcal{G}\rangle^{\frac{1}{2}}_{\gamma,\delta}.
\end{eqnarray}
In order to establish Sobolev inequalities on $S_{u,\underline{u}}$ (note that these topological $2-$spheres are evolving in spacetime), we need to have estimates for the metric $\gamma$. These are carried out in Section \ref{s3}.

\subsection{Yang-Mills pair integration}
\noindent To close the energy estimates for curvature and thus establish semi-global existence, we ought to obtain energy estimates for the Yang-Mills fields. To that end, instead of using the Yang-Mills stress energy tensor to construct suitable currents, integrate over the bulk $\mathcal{D}_{u,\underline{u}}$, and control the error terms, we directly utilize the manifestly hyperbolic character of the null Yang-Mills equations. However, since the gauge covariant derivative appears in the equations instead of the ordinary covariant derivative, we need to be careful while performing a direct integration by parts procedure by means of the equations. As we shall see, we once again utilize the compatibility of the gauge covariant derivative with the metric on the fibres of the bundle $P_{AD,\mathfrak{g}}$. We briefly sketch the procedure here, even though we will perform the actual energy estimates later. First recall the following integration lemma:\\
\begin{lemma}
\label{integration}
\textit{Let $f:\mathcal{M}\to\mathbb
R$ be a gauge-invariant object on the spacetime. The following integration by parts identities hold for $f$:
\begin{eqnarray}
\int_{D_{u,\underline{u}}}\nabla_{4}f=\int_{\underline{H}_{\underline{u}}}f-\int_{\underline{H}_{\underline{u}_{0}}}f+\int_{D_{u,\underline{u}}}(2\omega-\tr\chi)f
\end{eqnarray}
and 
\begin{eqnarray}
\int_{D_{u,\underline{u}}}\nabla_{3}f=\int_{H_{u}}f-\int_{H_{u_{0}}}f+\int_{D_{u,\underline{u}}}(2\underline{\omega}-\tr\underline{\chi})f,
\end{eqnarray}
}
\end{lemma}
\begin{proof}
The proof is a simple consequence of the following integration by parts procedure:
\begin{equation}
\begin{split}
\int_{D_{u\underline{u}}}\nabla_{4}f=&\int_{u_{0}}^{u}du^{'}\int_{\underline{u}_{0}}^{\underline{u}}\left(\int_{S_{u^{'}\underline{u}^{'}}}\frac{\partial f}{\partial \underline{u}^{'}}\Omega \mu_{\gamma}\right)d\underline{u}^{'}\\\nonumber 
=&\int_{u_{0}}^{u}du^{'}\int_{\underline{u}_{0}}^{\underline{u}}\left\{\frac{d}{d\underline{u}^{'}}\int_{S_{u^{'}\underline{u}^{'}}}f\Omega \mu_{\gamma}-\int_{S_{u^{'}\underline{u}^{'}}}f(\frac{\partial\Omega}{\partial \underline{u}^{'}}+\frac{\Omega}{2}\tr_{\gamma}\partial_{u}\gamma)\mu_{\gamma}\right\}d\underline{u}^{'}\\\nonumber 
=&\int_{u_{0}}^{u}du^{'}\int_{S_{u^{'}\underline{u}}}f\Omega \mu_{\gamma}-\int_{u_{0}}^{u}\int_{S_{u^{'}\underline{u}_{0}}}f\Omega \mu_{\gamma}+\int_{D_{u\underline{u}}}f(2\omega-\tr\chi)\\\nonumber 
=&\int_{\underline{H}_{\underline{u}}}f-\int_{\underline{H}_{\underline{u}_{0}}}f+\int_{D_{u\underline{u}}}(2\omega-\tr\chi)f.
\end{split}
\end{equation}
The other part follows in a similar fashion.
\end{proof}
\begin{proposition}
\label{hyperbolic}
Null Yang-Mills equations are manifestly hyperbolic.
\end{proposition}
\begin{remark}
Note that this is natural since the  Yang-Mills equations are derivable from a Lagrangian and as such they are equipped with a canonical stress-energy tensor.
\end{remark}
\begin{proof}
The integration lemma can be utilized to prove the manifestly hyperbolic characteristics of the Yang-Mills equations while expressed in the double null coordinates. First consider the null triple $(\underline{\alpha}^{F},\rho^{F},\sigma^{F})$ and recall their gauge covariant evolution equations 
\begin{equation}
    \hnabla_4 \alphabar^F +\frac{1}{2}\tr\chi\alphabar^F = - \hnabla \rho^F + \Hodge{\hnabla}\sigma^F -2 \Hodge{\etabar} \sigma^F - 2 \etabar \rho^F + 2 \omega \alphabar^F - \chibarhat \cdot \alpha^F,
\end{equation}
\begin{equation}
    \hnabla_3 \rho^F + \tr\chibar \rho^F = - \widehat{\text{div}} \alphabar^F + \left( \eta-\etabar \right)\cdot \alphabar^F,
\end{equation}
\begin{equation}
    \hnabla_3 \sigma^F + \tr\chibar \sigma^F = -\widehat{\text{curl}}\alphabar^F+\left(\eta-\etabar\right)\cdot\Hodge{\alphabar}^F.
\end{equation}
Now define $f_{1}:=|\underline{\alpha}^{F}|^{2}_{\gamma,\delta},~f_{2}=|\rho^{F}|^{2}_{\gamma,\delta}$, and $f_{3}:=|\sigma^{F}|^{2}_{\gamma,\delta}$. With these definitions in mind, let us apply the integration lemma to $f_{1}$, $f_{2}$, and $f_{3}$ to yield
\begin{eqnarray}
\int_{D_{u,\underline{u}}}\nabla_{4}f_{1}+\int_{D_{u,\underline{u}}}\nabla_{3}f_{2}+\int_{D_{u,\underline{u}}}\nabla_{3}f_{3}=\int_{\underline{H}_{\underline{u}}}f_{1}+\int_{H_{u}}f_{2}+\int_{H_{u}}f_{3}-\int_{\underline{H}_{\underline{u}_{0}}}f_{1}\\\nonumber 
-\int_{H_{u_{0}}}f_{2}-\int_{H_{u_{0}}}f_{3}+\int_{D_{u,\underline{u}}}(2\omega-\tr\chi)f_{1}+\int_{D_{u,\underline{u}}}(2\underline{\omega}-\tr\underline{\chi})(f_{2}+f_{3}).
\end{eqnarray}
In order for these equations to exhibit a hyperbolic characteristic, the left hand side should simplify to terms that are algebraic in $\underline{\alpha}^{F}, \rho^{F},$ and $\sigma^{F}$ upon using the null evolution equations. Now we note the most important point: $f_{1}$, $f_{2}$, and $f_{3}$ are gauge-invariant objects and therefore we have the following as a consequence of the compatibility of the gauge covariant connection $\hat{\nabla}$ with the metrics $\gamma$ and $\delta$ of the fibres:
\begin{eqnarray}
\nabla_{4}f_{1}=2\langle \underline{\alpha}^{F},\widehat{\nabla}_{4}\underline{\alpha}^{F}\rangle_{\gamma,\delta},~\nabla_{3}f_{2}=2\langle \rho^{F},\widehat{\nabla}_{3}\rho^{F}\rangle_{\gamma,\delta},~\nabla_{3}f_{3}=2\langle \sigma^{F},\widehat{\nabla}_{3}\sigma^{F}\rangle_{\gamma,\delta}.
\end{eqnarray}
Now we only focus on the principal terms for  the hyperbolicity argument.

\begin{equation}
\begin{split}
\langle\underline{\alpha}^{F},\widehat{\nabla}_{4}\underline{\alpha}^{F}\rangle_{\gamma,\delta}&=\langle \underline{\alpha}^{F},-\hnabla \rho^F + \Hodge{\hnabla}\sigma^F+\cdot\cdot\cdot\cdot\rangle_{\gamma,\delta}\\
&=-\text{div}\langle\underline{\alpha}^{F},\rho^{F}\rangle_{\gamma,\delta}+\langle\widehat{\text{div}}\underline{\alpha}^{F},\rho^{F}\rangle_{\gamma,\delta}+\text{div}~^{*}\langle\underline{\alpha}^{F},\sigma^{F}\rangle_{\gamma,\delta}+\langle\widehat{\text{curl}}\underline{\alpha}^{F},\sigma^{F}\rangle +\text{l.o.t}
\end{split}
\end{equation} 

\begin{equation} 
\langle \rho^{F},\hnabla _{3}\rho^{F}\rangle_{\gamma,\delta}=\langle \rho^{F},-\widehat{\text{div}} \alphabar^F+\cdot\cdot\cdot\cdot\rangle_{\gamma,\delta}\end{equation} \begin{equation}
\langle \sigma^{F},\hnabla _{3}\sigma^{F}\rangle_{\gamma,\delta}=\langle \sigma^{F}, -\widehat{\text{curl}}\underline{\alpha}^{F}+\cdot\cdot\cdot\cdot\rangle_{\gamma,\delta}.
\end{equation}
Now upon addition, we have 
\begin{equation}
\begin{split}
&\langle\underline{\alpha}^{F},\widehat{\nabla}_{4}\underline{\alpha}^{F}\rangle_{\gamma,\delta}+\langle \rho^{F},\hnabla _{3}\rho^{F}\rangle_{\gamma,\delta}+\langle \sigma^{F},\hnabla _{3}\sigma^{F}\rangle_{\gamma,\delta}\\ 
=&-\text{div}\langle\underline{\alpha}^{F},\rho^{F}\rangle_{\gamma,\delta}+\langle\widehat{\text{div}}\underline{\alpha}^{F},\rho^{F}\rangle_{\gamma,\delta}+\text{div}~^{*}\langle\underline{\alpha}^{F},\sigma^{F}\rangle_{\gamma,\delta}+\langle\widehat{\text{curl}}\underline{\alpha}^{F},\sigma^{F}\rangle\\ 
&\hspace{5cm}-\langle \rho^{F},\widehat{\text{div}} \alphabar^F\rangle_{\gamma,\delta}-\langle \sigma^{F}, \widehat{\text{curl}}\underline{\alpha}^{F}\rangle_{\gamma,\delta}+\text{l.o.t}\\ 
=&-\text{div}\langle\underline{\alpha}^{F},\rho^{F}\rangle_{\gamma,\delta}+\text{div}~^{*}\langle\underline{\alpha}^{F},\sigma^{F}\rangle_{\gamma,\delta}+\text{l.o.t}
\end{split}
\end{equation}
which, upon integration over the topological $2-$sphere, yields terms that are algebraic in $\underline{\alpha}^{F}, \rho^{F}$, and $\sigma^{F}$. Here $\langle\underline{\alpha}^{F},\rho^{F}\rangle_{\gamma,\delta}=(\underline{\alpha}^{F})^{P}~_{Qab}(\rho^{F})^{Q}~_{P}~^{b}$ and $~^{*}\langle \underline{\alpha}^{F},\sigma^{F}\rangle_{\gamma,\delta}=\epsilon^{ca}(\underline{\alpha}^{F})^{P}~_{Qab}\rho^{Q}~_{P}~^{b}$. The most vital property that is utilized here is the compatibility of the connection $\hnabla$ with the inner product $\langle~,~\rangle_{\gamma,\delta}$ induced by the fibre metrics $\gamma$ and $\delta$ together with the Hodge structure present in the null Yang-Mills equations. Notice that nowhere in the procedure did we require explicit information about the Yang-Mills connection $1-$form $A^{P}~_{Q\mu}dx^{\mu}$. The remaining Yang-Mills null evolution equations may be utilized in a  similar  manner to obtain energy identities associated with the triple $\alpha^{F},\rho^{F},\sigma^{F}$. This concludes the proof of the hyperbolic characteristics of the null Yang-Mills equations.
\end{proof}
\section{Preliminary estimates} \label{s3}
\subsection{Preliminary bootstrap assumptions}
\noindent We will be employing a bootstrap argument to obtain a priori bounds on $\bbGamma$, $\mathcal{R}$ and $\mathbb{YM}$. Along the initial hypersurfaces $H_{u_{\infty}}$ and $\Hbar_0$, an analysis of the initial data using transport equations (see for example, [LukLocal]) yields
 
 \[ \bbGamma_0 + \mathcal{R}_0 + \mathbb{YM}_0 \lesssim \mathcal{I}. \]Our goal is to show that in the entire region \[ \mathcal{D}:= \begin{Bmatrix} (u,\ubar,\theta^1, \theta^2) \hspace{1mm} \mid \hspace{1mm} u_{\infty}\leq u \leq -\frac{a}{4}, \hsp 0\leq \ubar \leq 1 \end{Bmatrix}\] there exists a constant $c(\mathcal{I}) = \mathcal{I}^4+\mathcal{I}^2 +\mathcal{I}+1$ such that \[ \bbGamma + \mathcal{R} + \mathbb{YM} \lesssim c(\mathcal{I}). \] We assume, as a bootstrap assumption the following:
 
 \begin{equation}\label{bootstrap}
 \bbGamma \leq \Gamma, \hspace{2mm}\mathcal{R}\leq R, \hspace{2mm} \mathbb{YM} \leq M,
 \end{equation}where $\Gamma, \hsp R$ and $M$ are large so that 
 
 \[ \mathcal{I}^4 +\mathcal{I}^2 +\mathcal{I}+1 \ll \text{min}\begin{Bmatrix}
 \Gamma, \hsp R,\hsp M\end{Bmatrix}\]but also such that \[ (\Gamma+R+M)^{20}\leq a^{\frac{1}{16}}.\]

\subsection{Estimates on the metric components}
\vspace{3mm}
For the metric component $\Omega$, the induced metric $\gamma$ of $S_{u,\ubar}$ and  for the area of $S_{u,\ubar}$ the following propositions hold:

\begin{proposition}\label{31}
Under the assumptions of Theorem \ref{main1} and the bootstrap assumptions \eqref{bootstrap}, we have 
\[\lVert \Omega-1 \rVert_{L^{\infty}(S_{u,\ubar})} \lesssim \frac{\Gamma}{\lvert u \rvert}. \]

\end{proposition}
\begin{proposition}
Under the assumptions of Theorem \ref{main1} and the bootstrap assumptions \eqref{bootstrap},  there exist two constants $c$ and $C$ depending only on the initial data such that the bounds 
\[ c \leq \text{det}\hsp \gamma \leq C.\]and \[  \lvert \gamma_{AB} \rvert + \lvert \gamma^{-1}_{AB}\rvert \leq C  \]hold throughout the slab of existence $\mathcal{D}$.\label{32}
\end{proposition}

\begin{proposition}\label{33}
Under the assumptions of Theorem \ref{main1} and the bootstrap assumptions \eqref{bootstrap}, fix a point $(u, \theta)$ on the initial hypersurface $\Hbar_0$. Let $\Lambda(u)$ and $\lambda(u)$ be the largest and smallest eigenvalues of $\gamma^{-1}(u,0,\theta)\hsp \gamma(u,\ubar,\theta) $ respectively, along the outgoing null geodesics emanating from $(u,\theta)$. There holds 

\[    \lvert \Lambda(u)-1 \rvert + \lvert \lambda(u)-1 \rvert \lesssim \frac{1}{\al}. \]
\end{proposition}
\begin{proposition}\label{34}
Under the assumptions of Theorem \ref{main1} and the bootstrap assumptions \eqref{bootstrap}, for the area of the $2-$sphere $S_{u,\ubar}$, there holds  \[ \sup_{\ubar} \lvert \text{Area}(S_{u,\ubar} )- \text{Area}(S_{u,0}) \rvert \lesssim \frac{\Gamma^{\frac{1}{2}}}{\al}\lvert u \rvert^2.  \] 
\end{proposition}The proofs of Propositions \ref{31} to \ref{34} are the same as in [An]

\subsection{Estimates for transport equations}
We shall be using two fundamental  bounds on transport equations throughout this work. 

 \begin{proposition} \label{3.5}
Under the assumptions of Theorem \ref{main1} and the bootstrap assumptions \eqref{bootstrap}, the following hold for an arbitrary $\mathcal{G} \in \Gamma(^{N}\otimes T^{*}S\otimes P_{AD,\mathfrak{g}}):$
\begin{equation}
\scaletwoSu{\mathcal{G}} \lesssim \lVert \mathcal{G} \rVert_{L^2_{(sc)}(S_{u,\ubar^{\prime\prime}})} + \int_{\ubar^{\prime\prime}}^{\ubar}  \scaletwoSuubarprime{\hnabla_4 \mathcal{G}} \dubarprime 
\end{equation}

\begin{equation}
\scaletwoSu{\mathcal{G}} \lesssim \lVert \mathcal{G} \rVert_{L^2_{(sc)}(S_{u^{\prime\prime},\ubar})} + \int_{u^{\prime\prime}}^{u} \aupr  \scaletwoSuprime{\hnabla_4 \mathcal{G}} \duprime
\end{equation}
\end{proposition}
There are, however, cases which are  borderline and require more delicate control than what the above Proposition provides. These have to do with components $X$ satisfying an equation of the form $\hnabla_3 X= - \lambda \tr\chibar X + \dots$, where $\lambda >0$. Keeping in mind that $\tr\chibar$ is the worst Ricci coefficient in terms of peeling, one would hope to be able to get rid of its appearance and thus obtain stronger bounds regarding the peeling properties of $X$. The following weighted transport inequality achieves this.

\begin{proposition} \label{3.6}
Let $\mathcal{G}, \mathcal{H} \in \Gamma(^{N}\otimes T^{*}S\otimes P_{AD,\mathfrak{g}})$ and assume that the following equation holds:

\[ \hnabla_3 \mathcal{G} + \lambda_0 \hsp \tr\chibar \hsp \mathcal{G} = \mathcal{H}.\]\label{36}
\noindent Then, under the assumptions of Theorem \ref{main1} and the bootstrap assumptions \eqref{bootstrap}, the following is true:
\[  \lvert u \rvert^{\lambda_1}\twoSu{\mathcal{G}}  \lesssim \lvert u_{\infty}\rvert^{\lambda_1}\lVert \mathcal{G} \rVert_{L^2(S_{u_{\infty}, \ubar})} + \intu \lvert u^{\prime}\rvert^{\lambda_1 } \twoSuprime{\mathcal{H}}\duprime      \]
for $\lambda_{1}=2\lambda_{0}-1$.
\end{proposition}
\begin{proof}
The variation of area formula for a scalar function $f$ reads:

\begin{equation}
\underline{L} \int_{S_{u,\ubar}} f = \int_{S_{u,\ubar}} \Lbar f + \Omega \hsp \tr\chibar \hsp f = \int_{S_{u,\ubar}} \Omega \hsp \left(e_3(f) + \tr\chibar \hsp f \right).
\end{equation}Plugging in $f= \lvert u \rvert^{2 \lambda_1} \lvert \mathcal{G} \rvert_{\gamma,\delta}^2$ and exploiting the compatibility of the gauge-covariant derivative with the fibre metrics $\gamma, \delta$, we calculate:

\begin{equation}
\begin{split}
&\underline{L} \int_{S_{u,\ubar}}\lvert u \rvert^{2 \lambda_1} \lvert \mathcal{G} \rvert_{\gamma,\delta}^2 \\ = &\int_{S_{u,\ubar}} \Omega \left( -2 \hsp \lambda_1 \lvert u \rvert^{2 \hsp \lambda_1 -1} e_3(u) \lvert \mathcal{G} \rvert_{\gamma,\delta}^2 + 2 \lvert u \rvert^{2\hsp \lambda_1} \langle \mathcal{G}, \hnabla_3 \mathcal{G} \rangle_{\gamma, \hsp \delta} + \tr\chibar \hsp \lvert u \rvert^{2\lambda_1}\hsp \lvert \mathcal{G} \rvert_{\gamma,\delta}^2 \right) \\ =& \int_{S_{u,\ubar}} \Omega \left( 2 \lvert u \rvert^{2\lambda_1} \langle \mathcal{G}, \hnabla_3 \mathcal{G} + \lambda_0 \tr\chibar \mathcal{G} \rangle_{\gamma,\delta }\right) + \int_{S_{u,\ubar}} \Omega \lvert u \rvert^{2\lambda_1} \left( \frac{-2\lambda_1 e_3(u)}{\lvert u \rvert} + (1-2\lambda_0)\hsp \tr\chibar \right) \lvert \mathcal{G} \rvert_{\gamma,\delta}^2. \label{3.3}
\end{split}
\end{equation}Notice that

\begin{equation}
\begin{split}
&\frac{-2\lambda_1 e_3(u)}{\lvert u \rvert} + (1-2\lambda_0)\hsp \tr\chibar \\ = & \frac{-2\lambda_1 (\Omega^{-1}-1)}{\lvert u \rvert} +(1-2\lambda_0)(\tr\chibar + \frac{2}{\lvert u \rvert}) - \frac{2\lambda_1 + 2 - 4 \lambda_0}{\lvert u \rvert} \\ \leq& \frac{\Gamma}{\lvert u \rvert^2},
\end{split}
\end{equation}where we have used the bootstrap assumption $\inftySu{\tr\chibar + \frac{2}{\lvert u \rvert}} \leq\frac{\Gamma}{\lvert u \rvert^2}$ and the definition of $\lambda_1$. For the first term in the last line of \eqref{3.3} we then use Cauchy-Schwartz and for the second we apply Gr\"onwall's inequality to get:

\begin{equation}
\begin{split}
&\lvert u \rvert^{2 \lambda_1} \twoSu{ \mathcal{G} } \\ \lesssim &\mathrm{e}^{\Gamma \lVert u^{-2} \rVert_{L_u^1}}\left( \lvert u_{\infty}\rvert^{\lambda_1}\lVert \mathcal{G} \rVert_{L^2(S_{u_{\infty}, \ubar})} + \intu \lvert u^{\prime}\rvert^{\lambda_1 } \twoSuprime{\mathcal{H}}\duprime \right) \\ \lesssim &\lvert u_{\infty}\rvert^{\lambda_1}\lVert \mathcal{G} \rVert_{L^2(S_{u_{\infty}, \ubar})} + \intu \lvert u^{\prime}\rvert^{\lambda_1 } \twoSuprime{\mathcal{H}}\duprime, 
\end{split} 
\end{equation}where we have used the fact that $\mathrm{e}^{\Gamma \lVert u^{-2} \rVert_{L_u^1}} \lesssim \mathrm{e}^{\Gamma/ a} \lesssim  1$.

\end{proof}

\subsection{Sobolev embedding}
\noindent With the derived estimates for the metric $\gamma$, we can obtain a bound on the isoperimetric constant for a topological $2-$sphere $S$: 
\begin{eqnarray}
I(S):=\sup_{U\subset S,~\partial U\in C^{1}}\frac{\min\left\{Area(U),Area(U^{c})\right\}}{[Perimeter(\partial U)]^{2}}.
\end{eqnarray}
The following proposition yields an upper bound for $I(S)$.\\
\begin{proposition}\textit{Under the assumption on the initial data and the bootstrap assumption (2.10), the isoperimetric constant obeys the following bound 
\begin{eqnarray} \label{37}
I(S_{u,\underline{u}})\leq \frac{1}{\pi}
\end{eqnarray}
for $u\in [u_{\infty},-\frac{a}{4}]$ and $\underline{u}\in [0,1]$.
}\end{proposition}
\begin{proof}  Fix a $u$. For $U_{\ubar}$ a subset of $S_{u,\ubar}$, denote by $U_0 \subset S_{u,0}$ the backward image of $U_{\ubar}$ under the diffeomorphism generated by the equivariant vector field $L.$ Using Propositions \ref{32} and \ref{33} and their proof, we can obtain the estimates

\[      \frac{\text{Perimeter}(\partial U_{\ubar}) }{\text{Perimeter}(\partial U_0)}\geq \sqrt{\inf_{S_{u,0}}\lambda(\ubar) }  \]and

\[ \frac{\text{Area}(U_{\ubar})}{\text{Area}(U_0)} \leq \sup_{S_{u,0}}\frac{\det(\gamma_{\ubar})}{\det(\gamma_0)}, \hspace{2mm}\frac{\text{Area}(U_{\ubar}^c)}{\text{Area}(U_0^c)} \leq \sup_{S_{u,0}}\frac{\det(\gamma_{\ubar})}{\det(\gamma_0)}. \]The conclusion then follows from the fact that $I(S_{u,0}) = \frac{1}{2\pi}$ and the bounds from Propositions \ref{32}, \ref{33} .       \end{proof}
Throughout this work, we will be using an $L^2-L^{\infty}$ Sobolev estimate. To obtain it, utilizing the basic estimates above, we may proceed to write down the following gauge-invariant Sobolev inequalities for the topological $2-$ sphere $S$.\\
\begin{proposition} \textit{Let $(S,\gamma)$ be a Riemannian $2-$manifold with the isoperimetric constant $I(S)$. Then the following Sobolev embedding holds for any $\mathcal{G}\in \Gamma(^{N}\otimes T^{*}S\otimes P_{AD,\mathfrak{g}})$}
\begin{eqnarray}
\label{eq:Lp}
\left(Area(S)\right)^{-\frac{1}{p}}||\mathcal{G}||_{L^{p}(S)}\leq C_{p}\left(\max(1,I(S))\right)^{\frac{1}{2}}\left(||\widehat{\nabla}\mathcal{G}||_{L^{2}(S)}+Area(S)^{-\frac{1}{2}}||\mathcal{G}||_{L^{2}(S)}\right)
\end{eqnarray}
for any $p\in (2,\infty)$.
\end{proposition}
\begin{proof} We know that the standard Sobolev inequality with $p\in (2,\infty)$
\begin{eqnarray}
\label{eq:standard}
\left(Area(S)\right)^{-\frac{1}{p}}||f||_{L^{p}(S)}\leq C_{p}\left(\max(1,I(S))\right)^{\frac{1}{2}}\left(||\nabla f||_{L^{2}(S)}+Area(S)^{-\frac{1}{2}}||f||_{L^{2}(S)}\right)
\end{eqnarray}
holds for a function $f$. Now we set 
\begin{eqnarray}
f=\sqrt{\mathcal{G}^{P}~_{QA_{1}A_{2}A_{3}\cdot\cdot\cdot\cdot A_{N}}\mathcal{G}^{Q}~_{PB_{1}B_{2}B_{3}\cdot\cdot\cdot\cdot B_{N}}\gamma^{A_{1}B_{1}}\gamma^{A_{2}B_{2}}\gamma^{A_{3}B_{3}}\cdot\cdot\cdot\cdot \gamma^{A_{N}B_{N}}+\epsilon}
\end{eqnarray}
with $\epsilon>0$ to yield 
\begin{eqnarray}
||\nabla f||_{L^{2}(S)}=||\frac{\langle \mathcal{G},\widehat{\nabla}\mathcal{G}\nonumber\rangle_{\gamma,\delta}}{\sqrt{\mathcal{G}^{P}~_{QA_{1}A_{2}A_{3}\cdot\cdot\cdot\cdot A_{N}}\mathcal{G}^{Q}~_{PB_{1}B_{2}B_{3}\cdot\cdot\cdot\cdot B_{N}}\gamma^{A_{1}B_{1}}\gamma^{A_{2}B_{2}}\gamma^{A_{3}B_{3}}\cdot\cdot\cdot\cdot \gamma^{A_{N}B_{N}}+\epsilon}}||_{L^{2}(S)}\\\nonumber \leq ||\frac{|\mathcal{G}|_{\gamma,\delta}|\widehat{\nabla}\mathcal{G}|_{\gamma,\delta}}{\sqrt{\mathcal{G}^{P}~_{QA_{1}A_{2}A_{3}\cdot\cdot\cdot\cdot A_{N}}\mathcal{G}^{Q}~_{PB_{1}B_{2}B_{3}\cdot\cdot\cdot\cdot B_{N}}\gamma^{A_{1}B_{1}}\gamma^{A_{2}B_{2}}\gamma^{A_{3}B_{3}}\cdot\cdot\cdot\cdot \gamma^{A_{N}B_{N}}+\epsilon}}||_{L^{2}(S)},
\end{eqnarray}
since $f$ is a gauge-invariant object and the gauge covariant derivative is compatible with fibre metrics $\gamma$ and $\delta$. Substituting this inequality in (\ref{eq:standard}) and taking the limit $\epsilon\to 0$ yields the desired gauge-invariant Sobolev inequality. \end{proof}
\begin{proposition} \textit{Let $(S,\gamma)$ be a Riemannian $2-$manifold with the isoperimetric constant $I(S)$. Then the following Sobolev embedding holds for any $\mathcal{G}\in \Gamma(^{N}\otimes T^{*}S\otimes P_{AD,\mathfrak{g}})$}
\begin{eqnarray}
\label{eq:Linfty}
||\mathcal{G}||_{L^{\infty}(S)}\leq C_{p}\left(\max(1,I(S))\right)^{\frac{1}{2}}[Area(S)]^{\frac{1}{2}-\frac{1}{p}}\left(||\widehat{\nabla}\mathcal{G}||_{L^{p}(S)}+Area(S)^{-\frac{1}{2}}||\mathcal{G}||_{L^{p}(S)}\right)
\end{eqnarray}
for any $p\in (2,\infty)$.\end{proposition}
\begin{proof} A calculation similar to the previous one and the standard $L^{\infty}-L^{p}$ Sobolev inequality on the Riemannian manifold $(S,\gamma)$ yield the result. \end{proof}

The two inequalities above, together with Propositions \ref{31}-\ref{34}, allow us to control the $L^{2}$-norm of $\mathcal{G}$ in terms of its $H^{2}$-norm. Following the area estimates, we have $Area(S_{u,\underline{u}})\approx u^{2}$. Therefore, we obtain the following important inequality.\\ 
\begin{proposition} \label{Sobolev}Under the assumptions of Theorem \ref{main1} and the bootstrap assumptions \eqref{bootstrap}, the following Sobolev embedding holds for any $\mathcal{G}\in \Gamma(^{N}\otimes T^{*}S_{u,\underline{u}}\otimes P_{AD,\mathfrak{g}}):$
\begin{eqnarray}
||\mathcal{G}||_{L^{\infty}(S_{u,\underline{u}})}\lesssim \sum_{0\leq I\leq 2}|||u|^{I-1}\widehat{\nabla}^{I}\mathcal{G}||_{L^{2}(S_{u,\underline{u}})},
\end{eqnarray}
which, in scale-invariant  norms, reads
\begin{eqnarray}
||\mathcal{G}||_{{L}^{\infty}_{sc}(S_{u,\underline{u}})}\lesssim \sum_{0\leq I\leq 2}||(a^{\frac{1}{2}}\widehat{\nabla})^{I}\mathcal{G}||_{{L}^{2}_{sc}(S_{u,\underline{u}})}.
\end{eqnarray}
\end{proposition}
\begin{proof} Substitute $p=4$ in the previous proposition (\ref{eq:Linfty}) and estimate the right hand side by means of (\ref{eq:Lp}), the estimate on the area given by Proposition \ref{34} and on the isoperimetric constant of $S_{u,\underline{u}}$ given by Proposition \ref{37}. \end{proof}

\section{Scale-invariant estimates on the Ricci coefficients and Yang-Mills components}
\subsection{Preliminary estimates}

In this section we list some necessary and useful estimates, to be used extensively in the following sections. Let \[\Y \in \begin{Bmatrix} \frac{1}{\al}\alphaF, \rhoF, \sigmaF, \alphabarF\end{Bmatrix}, \hspace{2mm}  \Psi \in \begin{Bmatrix} \frac{\alpha}{\al},\tbeta, \rho, \sigma, \tbetabar, \alphabar \end{Bmatrix}, \hspace{2mm} \tilde{\Y} \in \begin{Bmatrix} \frac{1}{\al}\hnabla_4 \alphaF, \frac{a}{\lvert u \rvert} \hnabla_3 \alphabarF \end{Bmatrix}, \]\[ \psi_g  \in \begin{Bmatrix} \frac{a}{\lvert u \rvert^2}\tr\chibar, \frac{1}{\al}\chihat,\frac{\al}{\lvert u \rvert}\chibarhat, \frac{a}{\lvert u \rvert}\tildetr, \omegabar,\zeta, \etabar, \eta, \omega, \tr\chi\end{Bmatrix}. \]Finally, we note that by  $\nablap$ we mean an expression of the form \[ \nabla^{k_1}\psi_g \dots \nabla^{k_{i_2}} \psi_g,        \]where $k_1+\dots +k_{i_2} = i_1$.

\begin{proposition}
Under the assumptions of Theorem \ref{main1} and the bootstrap assumptions \eqref{bootstrap}, the following hold:
\[ \sum_{i_1+i_2\leq N+4} \scaletwoSu{(\al)^{i_1+i_2}\nabla^{i_1}\psi_g^{i_2}}\leq \lvert u \rvert,        \]\[ \sum_{i_1+i_2+i_3 \leq N+4} \scaletwoSu{(\al)^{i_1+i_2+i_3}\nabla^{i_1}\psi_g^{i_2}\nabla^{i_3}(\psi_g, \Y) } \leq \Gamma,   \] \[   \sum_{i_1+i_2+i_3 +i_4 \leq N+4} \scaletwoSu{(\al)^{i_1+i_2+i_3+i_4}\nabla^{i_1}\psi_g^{i_2}\hnabla^{i_3}(\psi_g, \Y)\hnablaf (\psi_g, \Y)   } \leq \frac{\Gamma^2}{\lvert u \rvert},     \]\[   \sum_{i_1+i_2+i_3 +i_4+i_5 \leq N+5} \scaletwoSu{(\al)^{i_1+i_2+i_3+i_4+i_5}\nabla^{i_1}\psi_g^{i_2}\hnabla^{i_3}(\psi_g, \Y)\hnablaf (\psi_g, \Y) \hnablaF (\psi_g, \Y)  } \leq \frac{\Gamma^3}{\lvert u \rvert^2},
 \] \[\sum_{i_1+i_2+i_3\leq N+3} \scaletwoSu{(\al)^{i_1+i_2+i_3} \nablap \nablat \Psi} \leq \Gamma, \], \[ \sum_{i_1+i_2+i_3 +i_4\leq N+3}  \scaletwoSu{ (\al)^{i_1+i_2+i_3+i_4} \nablap \hnablat(\psi_g, \Y)  \nablaf \Psi } \leq \frac{\Gamma^2}{\lvert u \rvert}, \]
\[ \sum_{i_1+i_2+i_3 +i_4+i_5 \leq N+3}  \scaletwoSu{ (\al)^{i_1+i_2+i_3+i_4+i_5} \nablap \hnablat(\psi_g, \Y)  \hnablaf(\psi_g,\Y)\nablaF \Psi } \leq \frac{\Gamma^3}{\lvert u \rvert^2}, \]\[\sum_{i_1+i_2+i_3\leq N+2} \scaletwoSu{(\al)^{i_1+i_2+i_3} \nablap \hnablat \tilde{\Y}} \leq \Gamma ,\]\[ \sum_{i_1+i_2+i_3+i_4 \leq N+2} \scaletwoSu{(\al)^{i_1+i_2+i_3+i_4} \nablap \hnablat(\psi_g, \Y, \Psi) \hnablaf \tilde{\Y}} \leq \frac{\Gamma^2}{\lvert u \rvert}.\]
\end{proposition}\begin{proof}
We only prove the first two statements and the rest follow in the same way (see also [An] for more details), taking into account the bootstrap assumptions \eqref{bootstrap} and the fact that these assumptions help us control up to $N+4$ derivatives of $(\psi_g, \Y)$, $N+3$ derivatives of $\Psi$ and $N+2$ derivatives of $\tilde{\Y}$. Working on the first inequality, we see that if $i_2=0$ then it trivially holds, since \[  \scaletwoSu{1}= \lvert u\rvert.  \] If $i_2 \geq 1$, we write $\nablap$ as  $\nabla^{k_1}\psi_g \dots \nabla^{k_{i_2}} \psi_g,      $ for $k_1+\dots +k_{i_2}= i_1$. Assume, without loss of generality, that $k_{i_2}$ is the largest among the coefficients.  We then write \[ (\al)^{i_1+i_2}\nablap = (\al)^{i_2} (\al \nabla)^{k_{i_2}}\psi_g \Pi_{\lambda=1}^{i_2-1}(\al \nabla)^{k_{\lambda}} \psi_g \]and proceed to bound $(\al \nabla)^{k_{i_2}}\psi_g$ in $L^2_{(sc)}$ and the rest of the terms in $L^{\infty}_{(sc)}$. By iterating \eqref{257} a total of $(i_2-1)$ times, we obtain

\begin{equation}
\begin{split}
&\frac{1}{\lvert u \rvert}\sum_{i_1+i_2 \leq N+4} \scaletwoSu{(\al)^{i_1+i_2}\nablap} \\ \leq & \frac{1}{\lvert u \rvert} \sum_{i_1+i_2 \leq N+4} \frac{(\al)^{i_2}}{\lvert u \rvert^{i_2-1}} \scaletwoSu{(\al \nabla)^{k_{i_2}}\psi_g}\Pi_{\lambda=1}^{i_2-1}\scaleinfinitySu{(\al \nabla)^{k_{\lambda}} \psi_g}\\ \leq &\left(\frac{\al \Gamma}{\lvert u \rvert}\right)^{i_2}\leq 1.
\end{split}
\end{equation}Moving on to the second inequality, if $i_2=0$, it trivially holds by the definition of $\Gamma$. If $i_2 \geq 1$, in the case where $i_1 \leq N+2$, we bound $(\al \nabla)^{i_4}(\psi_g, \Y) $ in $L^2_{(sc)}$ and the rest of the terms in $L^{\infty}_{(sc)}$, thus getting a bound of \[\frac{1}{\lvert u \rvert} \cdot \scaleinfinitySu{(\al)^{i_1+i_2} \nablap} \cdot \scaletwoSu{(\al \nabla)^{i_3}(\psi_g, \Y)}\leq\frac{(\al)^{i_2}\Gamma^{i_2+1}}{\lvert u \rvert^{i_2}}\leq \Gamma. \]If, on the other hand, $i_1 \geq N+3$, we bound $(\al \nabla)^{i_4}(\psi_g, \Y) $ in $L^{\infty}_{(sc)}$ and the rest of the expression in $L^{2}_{(sc)}$. Using the bound obtained in the first inequality of the proposition, we then arrive at a bound of $\frac{1}{\lvert u \rvert} \cdot \lvert u \rvert \cdot \Gamma = \Gamma$. The result follows.
\end{proof}

\subsection{Estimates on the Ricci coefficients}

\begin{proposition}\label{omegaprop}
Under the assumptions of Theorem \ref{main1} and the bootstrap assumptions \eqref{bootstrap}, there holds \[ \sum_{i\leq N+4} \scaletwoSu{\aln \omega} \lesssim \frac{\al}{\lvert u \rvert^{\frac{1}{2}}}\left(1+\hsp \underline{\mathcal{R}}[\rho]\right). \]
\end{proposition}

\begin{proof}
We begin by recalling that $\omega$ satisfies the schematic equation

\[\nabla_3 \omega =\frac{1}{2}\rho+ \psi_g \psi_g + \mathcal{Y}_{\ubar} \mathcal{Y}_{\ubar}.\]Crucially, the schematic product $\mathcal{Y}_{\ubar} \mathcal{Y}_{\ubar}$ is \textbf{gauge-invariant} (in this case, it equals $\lvert \rho^F \rvert^2 + \lvert \sigma^F \rvert^2$). Using the commutation formula \eqref{c2} and the notation of Section \ref{Norms}, we have, for a general $i$:

\begin{equation}
    \begin{split}
        \nabla_3\nabla^i\omega + \frac{i}2\tr\chibar \nabla^i \omega = &\nabla^i\rho + \sum_{i_1+i_2+i_3=i-1}\nabla^{i_1}\psi_g^{i_2+1}\nabla^{i_3}\rho+ \sum_{i_1+i_2+i_3+i_4=i} \nabla^{i_1}\psi_g^{i_2}\nabla^{i_3}\psi_g \nabla^{i_4}\psi_g  \\ &+ \sum_{i_1+i_2+i_3+i_4=i} \nabla^{i_1}\psi_g^{i_2}\hnabla^{i_3}\mathcal{Y}_{\ubar} \hnabla^{i_4}\mathcal{Y}_{\ubar} + \sum_{i_1+i_2+i_3+i_4=i} \nabla^{i_1}\psi_g^{i_2}\nabla^{i_3}(\chibarhat,\tildetr)\nabla^{i_4}\omega\\&+ \sum_{i_1+i_2+i_3+i_4=i-1}\nabla^{i_1}\psi_g^{i_2+1}\nabla^{i_3}\tr\chibar \nabla^{i_4}\omega \\&+ \sum_{i_1+i_2+i_3+i_4=i-2}\nabla^{i_1}\psi_g^{i_2+1}\nabla^{i_3}(\chibarhat,\tr\chibar)\nabla^{i_4}\omega \\&+\sum_{i_1+i_2+i_3+i_4+i_5=i-1}\nabla^{i_1}\psi_g^{i_2}\hnabla^{i_3}\mathcal{Y}_{\ubar}\hnabla^{i_4}\mathcal{Y}_{\ubar}\nabla^{i_5}\omega.
    \end{split}
\end{equation}Note that since $\omega$ is not a section of the vector bundle $~^{k}\otimes T^{*}S\otimes P_{Ad,\mathfrak{g}}$, the last term on the right-hand side of \eqref{c2} does not appear. Passing to scale-invariant norms, we get 
\begin{align}
     \nonumber     &\scaletwoSu{\aln \omega}\\  \nonumber \lesssim &\lVert (\al\nabla)^i\omega \rVert_{L^2_{(sc)}(S_{u_\infty},0)} + \intu \frac{a}{\upr^2}\scaletwoSuprime{\aln \rho}\duprime\\  \nonumber &+ \intu \sum_{i_1+i_2+i_3=i-1} \frac{a}{\upr^2}\scaletwoSuprime{(\al)^i \nabla^{i_1}\psi_g^{i_2+1}\nabla^{i_3}\rho}\duprime \\  \nonumber &+ \intu \sum_{i_1+i_2+i_3+i_4=i} \frac{a}{\upr^2}\scaletwoSuprime{(\al)^i \nabla^{i_1}\psi_g^{i_2}\nabla^{i_3} \psi_g \nabla^{i_4}\psi_g} \duprime \\  \nonumber &+ \intu \sum_{i_1+i_2+i_3+i_4=i} \frac{a}{\upr^2}\scaletwoSuprime{(\al)^i \nabla^{i_1}\psi_g^{i_2}\hnabla^{i_3} \mathcal{Y}_{\ubar} \hnabla^{i_4}\mathcal{Y}_{\ubar}} \duprime \\  \nonumber &+\intu \sum_{i_1+i_2+i_3+i_4=i} \frac{a}{\upr^2}\scaletwoSuprime{(\al)^i \nabla^{i_1}\psi_g^{i_2}\nabla^{i_3}(\chibarhat, \tildetr) \nabla^{i_4}\omega} \duprime \\  \nonumber &+\intu \sum_{i_1+i_2+i_3+i_4=i-1} \frac{a}{\upr^2}\scaletwoSuprime{(\al)^i \nabla^{i_1}\psi_g^{i_2+1}\nabla^{i_3}\tr\chibar \nabla^{i_4}\omega} \duprime \\  \nonumber &+\intu \sum_{i_1+i_2+i_3+i_4=i-2} \frac{a}{\upr^2}\scaletwoSuprime{(\al)^i \nabla^{i_1}\psi_g^{i_2+1}\nabla^{i_3}(\chibarhat, \tr\chibar) \nabla^{i_4}\omega} \duprime \\ &+\intu \sum_{i_1+i_2+i_3+i_4+i_5=i-1} \frac{a}{\upr^2}\scaletwoSuprime{(\al)^i \nabla^{i_1}\psi_g^{i_2}\hnabla^{i_3}\mathcal{Y}_{\ubar}\hnabla^{i_4}\mathcal{Y}_{\ubar} \nabla^{i_5}\omega} \duprime.
    \end{align} For $0\leq i \leq N+4$, the first term, by virtue of the fact that $\Omega\equiv 1$ initially, vanishes. The second term can be bounded, using H\"older's inequality, by $ \frac{\al}{\lvert u \rvert^{\frac{1}{2}}} \underline{\mathcal{R}}[\rho].$ The third, fourth and fifth terms can be bounded above by
\[\intu \frac{a}{\upr^2}\frac{\Gamma^2}{\upr}\duprime \leq \frac{a\hsp  \Gamma^2}{\lvert u \rvert^2}\leq \frac{\al}{\lvert u \rvert^{\frac{1}{2}}}.\]The sixth term is controlled as follows:

\begin{equation}
    \begin{split}
        &\intu \sum_{i_1+i_2+i_3+i_4=i} \frac{a}{\upr^2}\scaletwoSuprime{(\al)^i \nabla^{i_1}\psi_g^{i_2}\nabla^{i_3}(\chibarhat, \tildetr) \nabla^{i_4}\omega} \duprime \\ \lesssim &\intu \frac{a}{\upr^2} \frac{\upr}{\al} \frac{\Gamma^2}{\upr}\duprime
 \lesssim \frac{\al \hsp \Gamma^2}{\lvert u \rvert} \lesssim \frac{\al}{\lvert u \rvert^{\frac{1}{2}}}.  \end{split}
\end{equation}The seventh term is controlled as follows:

\begin{align}
     \nonumber     &\intu \sum_{i_1+i_2+i_3+i_4=i-1} \frac{a}{\upr^2}\scaletwoSuprime{(\al)^i \nabla^{i_1}\psi_g^{i_2+1}\nabla^{i_3}\tr\chibar \nabla^{i_4}\omega} \duprime \\ \lesssim &\intu \frac{a}{\upr^2} \frac{\upr^2}{a} \frac{\Gamma^3}{\upr^2}\duprime
 \lesssim \frac{\Gamma^3}{\lvert u \rvert} \lesssim \frac{\al}{\lvert u \rvert^{\frac{1}{2}}}.  \end{align} For the eighth and most borderline term, we estimate:

\begin{align}
       \nonumber   &\intu \sum_{i_1+i_2+i_3+i_4=i-2} \frac{a}{\upr^2}\scaletwoSuprime{(\al)^i \nabla^{i_1}\psi_g^{i_2+1}\nabla^{i_3}(\chibarhat,\tr\chibar) \nabla^{i_4}\omega} \duprime \\ \nonumber  =  &\intu \sum_{i_1+i_2+i_3+i_4=i-2} \frac{a^{\frac{3}{2}}}{\upr^2}\scaletwoSuprime{(\al)^{i-1} \nabla^{i_1}\psi_g^{i_2+1}\nabla^{i_3}(\chibarhat,\tr\chibar) \nabla^{i_4}\omega} \duprime \\ \lesssim &\intu \frac{a^{\frac{3}{2}}}{\upr^2} \frac{\upr^2}{a} \frac{\Gamma^3}{\upr^2}\duprime
 \lesssim \frac{\al\hsp \Gamma^3}{\lvert u \rvert} \lesssim \frac{\al}{\lvert u \rvert^{\frac{1}{2}}}.  \end{align}Finally, the ninth term can be estimated as follows:

\begin{align}
       \nonumber   &\intu \sum_{i_1+i_2+i_3+i_4+i_5=i-1} \frac{a}{\upr^2}\scaletwoSuprime{(\al)^i \nabla^{i_1}\psi_g^{i_2}\hnabla^{i_3}\mathcal{Y}_{\ubar}\hnabla^{i_4}\mathcal{Y}_{\ubar} \nabla^{i_5}\omega} \duprime\\\lesssim &\intu\frac{a^{\frac{3}{2}}}{\upr^2}\frac{\Gamma^3}{\upr^2}\duprime \lesssim  \frac{\al}{\lvert u \rvert^{\frac{1}{2}}}.
    \end{align}
The result follows.
\end{proof}

We move on to estimates for $\chibarhat$ and $\chihat$, which are the same as in [AnAth] or [An2019], because of the absence of matter terms in the corresponding structure equations.

\begin{proposition}\label{chihats}
Under the assumptions of Theorem \ref{main1} and the bootstrap assumptions \eqref{bootstrap}, there hold

\[ \sum_{i\leq N+4} \frac{\al}{\lvert u \rvert}\scaletwoSu{\aln \chibarhat} \lesssim 1, \hspace{2mm} \sum_{i\leq N+4}\frac{1}{\al}\scaletwoSu{\aln\chihat} \lesssim \mathcal{R}[\alpha]+1.\]
\end{proposition}

\begin{proof}
This is the same as in [An2019]. 
\end{proof}
\begin{remark}\label{remarkhat}
As will be shown later on, there holds \[ \sum_{0\leq i \leq N+3}\frac{1}{\al}\scaletwoSu{\aln \alpha}\lesssim 1.\] As a consequence, when $i\neq N+4$, the result of Proposition \ref{chihats} can be improved to 
\[ \sum_{0\leq i \leq N+3} \frac{1}{\al}\scaletwoSu{\aln \chihat} \lesssim 1. \]By Sobolev embedding, this also implies the pointwise estimate

\[ \sum_{0\leq i \leq N+1} \frac{1}{\al}\scaleinfinitySu{\aln \chihat}\lesssim 1.\]This will be useful later on, for example in Proposition \ref{trchiprop}.
\end{remark}
The estimates for $\omegabar$ are, in a sense, dual to those for $\omega$.
\begin{proposition}
Under the assumptions of Theorem \ref{main1} and the bootstrap assumptions \eqref{bootstrap}, there holds
\[ \sum_{i\leq N+4}\scaletwoSu{\aln \omegabar}\leq  \mathcal{R}[\rho]+1.\]
\end{proposition}

\begin{proof}
We have, schematically, 

\[ \nabla_4 \omegabar = \frac{1}{2}\rho + \psi_g \psi_g + \mathcal{Y}_{\ubar}\mathcal{Y}_{\ubar}.\] As before, the schematic product of the Yang-Mills components is gauge-invariant. Using the commutation formula \eqref{c1}, we have, for a general $i$:

\begin{equation}
    \begin{split}
        \nabla_4 \nabla^i \omegabar  = &\nabla^i\rho + \sum_{i_1+i_2+i_3=i-1}\nabla^{i_1}\psi_g^{i_2+1}\nabla^{i_3}\rho+ \sum_{i_1+i_2+i_3+i_4=i} \nabla^{i_1}\psi_g^{i_2}\nabla^{i_3}\psi_g \nabla^{i_4}\psi_g  \\ &+ \sum_{i_1+i_2+i_3+i_4=i} \nabla^{i_1}\psi_g^{i_2}\hnabla^{i_3}\mathcal{Y}_{\ubar} \hnabla^{i_4}\mathcal{Y}_{\ubar} +\sum_{i_1+i_2+i_3+i_4=i} \nabla^{i_1}\psi_g^{i_2} \nabla^{i_3}(\chihat,\tr\chi)\nabla^{i_4}\omegabar\\&+ \sum_{i_1+i_2+i_3+i_4=i-2}\nabla^{i_1}\psi_g^{i_2+1}\nabla^{i_3}(\chihat,\tr\chi)\nabla^{i_4}\omegabar \\&+ \sum_{i_1+i_2+i_3+i_4+i_5=i-1}\nabla^{i_1}\psi_g^{i_2}\hnabla^{i_3}\alpha^F \hnabla^{i_4}(\rho^F,\sigma^F)\nabla^{i_5}\omegabar. 
    \end{split}
\end{equation}
Passing to scale-invariant norms, we have, since $\lVert \aln \omegabar \rVert_{L^2_{(sc)}(S_{u,0})}$ vanishes,

\begin{equation}
    \begin{split}
        &\scaletwoSu{\aln\omegabar}\\ \lesssim &\intubar \scaletwoSuubarprime{\aln  \rho} \dubarprime \\+ &\intubar \sum_{i_1+i_2+i_3=i-1} \scaletwoSuubarprime{(\al)^i\nabla^{i_1}\psi_g^{i_2+1}\nabla^{i_3}\rho} \dubarprime \\+&\intubar \sumif \scaletwoSuubarprime{(\al)^i\nabla^{i_1}\psi_g^{i_2}\nabla^{i_3}\psi_g \nabla^{i_4}\psi_g}\dubarprime \\+&\intubar\sumif \scaletwoSuubarprime{(\al)^i\nabla^{i_1}\psi_g^{i_2}\hnabla^{i_3}\mathcal{Y}_{\ubar} \hnabla^{i_4}\mathcal{Y}_{\ubar}}\dubarprime\\+&\intubar\sumif \scaletwoSuubarprime{(\al)^i\nabla^{i_1}\psi_g^{i_2}\hnabla^{i_3}(\chihat,\tr\chi)\nabla^{i_4}\omegabar }\dubarprime \\+&\intubar\sumifim \scaletwoSuubarprime{(\al)^i\nabla^{i_1}\psi_g^{i_2+1}\nabla^{i_3}(\chihat,\tr\chi)\nabla^{i_4}\omegabar}\dubarprime\\+&\intubar\sumiFi \scaletwoSuubarprime{(\al)^i\nabla^{i_1}\psi_g^{i_2}\hnabla^{i_3}\alpha^F\hnabla^{i_4}(\rho^F,\sigma^F)\nabla^{i_5}\omegabar}\dubarprime.
    \end{split}
\end{equation}

For $0\leq i \leq N+4$, the first four terms are controlled as in Proposition \ref{omegaprop} and are bounded above by $\mathcal{R}[\rho]+1$. For the next terms, there holds
\begin{equation}
    \begin{split}
        &\intubar\sumif \scaletwoSuubarprime{(\al)^i\nabla^{i_1}\psi_g^{i_2}\hnabla^{i_3}(\chihat,\tr\chi)\nabla^{i_4}\omegabar }\dubarprime \\ \lesssim &\intubar \al \sumif \scaletwoSuubarprime{(\al)^i\nabla^{i_1}\psi_g^{i_2}\hnabla^{i_3}(\frac{\chihat}{\al},\frac{\tr\chi}{\al})\nabla^{i_4}\omegabar }\dubarprime \\ \lesssim &\frac{\al \Gamma^2}{\lvert u \rvert}\lesssim 1.
    \end{split}
\end{equation}Working similarly, there hold

\begin{equation}
    \intubar\sumifim \scaletwoSuubarprime{(\al)^i\nabla^{i_1}\psi_g^{i_2+1}\nabla^{i_3}(\chihat,\tr\chi)\nabla^{i_4}\omegabar}\dubarprime \lesssim \frac{a\hsp \Gamma^3}{\lvert u \rvert^2}\lesssim 1,
\end{equation}
\begin{equation}
    \intubar\sumiFi \scaletwoSuubarprime{(\al)^i\nabla^{i_1}\psi_g^{i_2}\hnabla^{i_3}\alpha^F\hnabla^{i_4}(\rho^F,\sigma^F)\nabla^{i_5}\omegabar}\dubarprime \lesssim \frac{\al\hsp  \Gamma^3}{\lvert u \rvert^2}\lesssim 1.
\end{equation}The claim follows.
\end{proof}
We move on to estimates for $\eta$.
\begin{proposition}\label{etabound}
Under the assumptions of Theorem \ref{main1} and the bootstrap assumptions \eqref{bootstrap}, there holds

\[ \sum_{0\leq i \leq N+4} \scaletwoSu{\aln\eta}\lesssim \mathcal{R}[\tbeta]+1.\]
\end{proposition}
\begin{proof}
We begin with the schematic structure equation for $\eta$:

\[ \nabla_4 \eta = \tbeta + (\chihat,\tr\chi)\cdot(\eta,\etabar) + \alpha^F \Yub,\]with the product of Yang-Mills curvature terms being once again gauge-invariant. Here, the $\Yub$ term is $(\rho^F, \sigma^F)$. Using the commutation formula \eqref{c1} for the $\nabla_4-$direction, we have

\begin{equation}
    \begin{split}
        \nabla_4\nabla^i \eta = &\nabla^i\tbeta + \sum_{i_1+i_2+i_3=i-1}\nabla^{i_1}\psi_g^{i_2+1}\nabla^{i_3}\tbeta +\sumif \nabla^{i_1}\psi_g^{i_2}\nabla^{i_3}(\chihat,\tr\chi)\nabla^{i_4}(\eta,\etabar)\\&+\sumif \nabla^{i_1}\psi_g^{i_2}\hnabla^{i_3}\alpha^F \hnabla^{i_4}\Yub + \sumiFi \nabla^{i_1}\psi_g^{i_2}\hnabla^{i_3}\alpha^F \hnabla^{i_4}(\rho^F, \sigma^F)\nabla^{i_5}\eta.
    \end{split}
\end{equation}Estimating in scale-invariant norms, we have 
\begin{equation}
    \begin{split}
    \scaletwoSu{\aln\eta} &\lesssim \intubar \scaletwoSuubarprime{\aln \tbeta} \dubarprime \\&+\intubar \sumitm \scaletwoSuubarprime{(\al)^i\nabla^{i_1}\psi_g^{i_2+1}\nabla^{i_3}\tbeta}\dubarprime \\&+\intubar \sumif \scaletwoSuubarprime{(\al)^i\nabla^{i_1}\psi_g^{i_2}\nabla^{i_3}(\chihat,\tr\chi)\nabla^{i_4}(\eta,\etabar) }\dubarprime \\&+ \intubar \sumif \scaletwoSuubarprime{(\al)^i \nabla^{i_1}\psi_g^{i_2}\hnabla^{i_3}\alpha^F \hnabla^{i_4}\Yub}\\&+ \intubar \sumiFi \scaletwoSuubarprime{(\al)^i \nabla^{i_1}\psi_g^{i_2}\hnabla^{i_3}\alpha^F\hnabla^{i_4}(\rho^F, \sigma^F)\nabla^{i_5}\eta}\dubarprime .
    \end{split}
\end{equation}For $0\leq i \leq N+4$, the first term is bounded by $\mathcal{R}[\tbeta]$. The second term is bounded by

\begin{equation}
    \begin{split}
        \intubar \sumitm \scaletwoSuubarprime{(\al)^i\nabla^{i_1}\psi_g^{i_2+1}\nabla^{i_3}\tbeta}\dubarprime \lesssim \frac{\Gamma^2}{\lvert u\rvert}\lesssim 1.
    \end{split}
\end{equation}Notice that, since $i_3\leq i-1\leq N+3$, we can bound $i_3$ derivatives of $\tbeta$ using the bootstrap assumption \eqref{bootstrap} on the total norm $\bbGamma$. For the third term there holds
\[ \intubar \sumif \scaletwoSuubarprime{(\al)^i\nabla^{i_1}\psi_g^{i_2}\nabla^{i_3}(\chihat,\tr\chi)\nabla^{i_4}(\eta,\etabar) }\dubarprime \lesssim \frac{\al\hsp \Gamma^2}{\lvert u \rvert}\lesssim 1. \]Here we have picked up an $\al$-term because of the existence of $\chihat$ in the expression. For the fourth term, we have \[  \intubar \sumif \scaletwoSuubarprime{(\al)^i \nabla^{i_1}\psi_g^{i_2}\hnabla^{i_3}\alpha^F \hnabla^{i_4}(\rho^F, \sigma^F)} \lesssim \frac{\al O^2}{\lvert u \rvert}\lesssim 1\]and finally, for the fifth term, we have 

\begin{equation}
    \begin{split}
        &\intubar \sumiFi \scaletwoSuubarprime{(\al)^i \nabla^{i_1}\psi_g^{i_2}\hnabla^{i_3}\alpha^F\hnabla^{i_4}(\rho^F, \sigma^F)\nabla^{i_5}\eta}\dubarprime \\= &\intubar a \sumiFi \bigg\lVert (\al)^{i-1}\nabla^{i_1}\psi_g^{i_2}\hnabla^{i_3}\left(\frac{\alpha^F}{\al}\right) \hnabla^{i_4}(\rho^F, \sigma^F)\nabla^{i_5}\eta \bigg\rVert_{L^2_{(sc)}(S_{u,\ubar^{\prime}})}\dubarprime \\ \lesssim& \frac{a \Gamma^3}{\lvert u \rvert^2}\lesssim 1.
    \end{split}
\end{equation}The result follows.
\end{proof}

\begin{proposition}
\label{trchiprop}
Under the assumptions of Theorem \ref{main1} and the bootstrap assumptions \eqref{bootstrap}, there holds

\[ \sum_{0\leq i \leq N+4} \scaletwoSu{\aln \tr\chi}\lesssim \mathcal{R}[\alpha] +\underline{\mathbb{YM}}[\rho^F, \sigma^F]+1. \]
\end{proposition}

\begin{proof}
We begin by recalling the schematic equation

\[ \nabla_4 \tr\chi = \lvert \chihat \rvert^2 + \psi_g\psi_g + \lvert \alpha^F \rvert^2. \]Commuting with $i$ angular derivatives using \eqref{c1}, we obtain

\begin{equation}
    \begin{split}
        \nabla_4 \nabla^i \tr\chi =& \sum_{i_1+i_2+i_3+i_4=i}\nabla^{i_1}\psi_g^{i_2}\nabla^{i_3}\chihat \nabla^{i_4}(\chihat, \tr\chi) +\sum_{i_1+i_2+i_3+i_4=i}\nabla^{i_1}\psi_g^{i_2}\nabla^{i_3}\psi_g \nabla^{i_4}\psi_g \\&+ \sumif \nabla^{i_1}\psi_g^{i_2}\hnabla^{i_3}\alpha^F \hnabla^{i_4}\alpha^F + \sumifim \nabla^{i_1}\psi_g^{i_2+1}\nabla^{i_3}(\chihat,\tr\chi)\nabla^{i_4}\psi_g \\ &+ \sumiFi \nabla^{i_1}\psi_g^{i_2}\hnabla^{i_3} \alpha^F \hnabla^{i_4}(\rho^F, \sigma^F) \nabla^{i_5}\psi_g.
    \end{split}
\end{equation}
Passing to scale-invariant norms, we have

\begin{equation}
    \begin{split}
        \scaletwoSu{\aln\tr\chi} \lesssim &\intubar\sum_{i_1+i_2+i_3+i_4=i} \scaletwoSuubarprime{(\al)^i\nabla^{i_1}\psi_g^{i_2}\nabla^{i_3}\chihat \nabla^{i_4}(\chihat, \tr\chi)} \dubarprime \\+& \intubar \sum_{i_1+i_2+i_3+i_4=i}\scaletwoSuubarprime{(\al)^i\nabla^{i_1}\psi_g^{i_2}\nabla^{i_3}\psi_g \nabla^{i_4}\psi_g}\dubarprime \\ +& \intubar \sum_{i_1+i_2+i_3+i_4=i}\scaletwoSuubarprime{(\al)^i\nabla^{i_1}\psi_g^{i_2}\hnabla^{i_3}\alpha^F \hnabla^{i_4}\alpha^F}\dubarprime \\+& \intubar \sum_{i_1+i_2+i_3+i_4=i-2}\scaletwoSuubarprime{(\al)^i\nabla^{i_1}\psi_g^{i_2+1}\nabla^{i_3}(\chihat,\tr\chi) \nabla^{i_4}\psi_g}\dubarprime\\ +& \intubar \sum_{i_1+\dots+i_5=i-1}\scaletwoSuubarprime{(\al)^i\nabla^{i_1}\psi_g^{i_2}\hnabla^{i_3}\alpha^F \hnabla^{i_4}(\rho^F,\sigma^F) \nabla^{i_5}\psi_g}\dubarprime  .  
    \end{split}
\end{equation}From the first term, the most dangerous case is when $i_4$ falls on $\chihat$, so we only give details for that. We distinguish two cases:

\begin{itemize}
    \item If in the term $\nabla^{i_1}\psi_g^{i_2}$ there exists some $\psi_g$ whose derivative is of order $>N+3$, we bound that term in $L^2_{(sc)}$ and the rest of the terms in $L^{\infty}_{(sc)}$. Notice, crucially, that from Remark \ref{remarkhat}, we can bound $\scaleinfinitySu{\aln \chihat} \lesssim \al$ for small $i$. As a consequence, we have the bound

    \[   \intubar\sum_{i_1+i_2+i_3+i_4=i} \scaletwoSuubarprime{(\al)^i\nabla^{i_1}\psi_g^{i_2}\nabla^{i_3}\chihat \nabla^{i_4}(\chihat, \tr\chi)} \dubarprime \lesssim \frac{\scaleinfinitySu{\aln \chihat}^2}{\lvert u \rvert}  \lesssim 1.    \]
    \item Otherwise, in the expression $\nabla^{i_3}\chihat \nabla^{i_4}\chihat$, at most one index $i_3, i_4$ is greater than $N+1$ (in which case, we cannot bound that term in $L^{\infty}_{(sc)}$). Say without loss of generality, that $i_3>N+1$. We bound $(\al)^{i_3-1}\nabla^{i_3}\chihat$ in $L^2_{(sc)}$ above by $\mathcal{R}[\alpha]+1$ and the rest of the terms in $L^{\infty}_{(sc)}$ above by $1$ (using the improvement mentioned in Remark \ref{remarkhat}), whence     \[   \intubar\sum_{i_1+i_2+i_3+i_4=i} \scaletwoSuubarprime{(\al)^i\nabla^{i_1}\psi_g^{i_2}\nabla^{i_3}\chihat \nabla^{i_4}(\chihat, \tr\chi)} \dubarprime  \lesssim \left(\mathcal{R}[\alpha]+1\right)\cdot1.    \]
\end{itemize}The second term is handled as in the previous propositions. For the third term, we can bound \[   \intubar\sum_{i_1+i_2+i_3+i_4=i} \scaletwoSuubarprime{(\al)^i\nabla^{i_1}\psi_g^{i_2}\hnabla^{i_3}\alpha^F \hnabla^{i_4}\alpha^F} \dubarprime  \lesssim \left(\underline{\mathbb{YM}}[\rho^F, \sigma^F]+1\right)\cdot1,    \]using the same idea as for $\chihat$ and the improvement given in Remark \ref{remarkF}. The last two terms are bounded above by $1$, as in previous Propositions.

\end{proof}
\begin{proposition}\label{trchibarbound}
Under the assumptions of Theorem \ref{main1} and the bootstrap assumptions \eqref{bootstrap}, the following estimates hold:

\[ \frac{a}{\lvert u \rvert} \sum_{0\leq i \leq N+4} \scaletwoSu{\aln \tildetr} \lesssim 1, \hspace{2mm} \frac{a}{\lvert u \rvert^2} \sum_{0\leq i \leq N+4} \scaletwoSu{\aln \tr\chibar} \lesssim 1.\]
\end{proposition}

\begin{proof}
Notice that $\tr\chibar$ satisfies the following structure equation:

\[ \nabla_3 \tr\chibar + \frac{1}{2}(\tr\chibar)^2 = - \lvert \chibarhat \rvert^2 - 2 \omegabar \tr\chibar -\lvert \alphabar^F \rvert^2 .\]Commuting with $i$ angular derivatives using \eqref{c2}, we get 

\begin{equation}\begin{split}
    \nabla_3\nabla^i \tr\chibar + \frac{i+1}{2}\tr\chibar \nabla^i \tr\chibar = &\sumif \nabla^{i_1}\psi_g^{i_2}\nabla^{i_3}\chibarhat \nabla^{i_4}\chibarhat \\+&\sumif \nabla^{i_1}\psi_g^{i_2}\nabla^{i_3}\omegabar \nabla^{i_4}\tr\chibar \\+&\sumif \nabla^{i_1}\psi_g^{i_2}\hnabla^{i_3}\alphabar^F \hnabla^{i_4}\alphabar^F\\+&\sumif \nabla^{i_1}\psi_g^{i_2}\nabla^{i_3}(\chibarhat,
    \tildetr)\nabla^{i_4}\tr\chibar \\ +&\sumifi \nabla^{i_1}\psi_g^{i_2+1}\nabla^{i_3}\tr\chibar \nabla^{i_4}\tr\chibar\\+& \sumifim \nabla^{i_1}\psi_g^{i_2+1}\nabla^{i_3}(\chibarhat, \tr\chibar) \nabla^{i_4}\tr\chibar \\ +& \sumiFi \nabla^{i_1}\psi_g^{i_2}\hnabla^{i_3}\alphabar^F \hnabla^{i_4}(\rho^F, \sigma^F) \nabla^{i_5}\tr\chibar := G_i. 
\end{split}    
\end{equation}
Passing to scale invariant norms and using the weighted transport inequality from Proposition \ref{3.6}, we obtain  
\begin{align}
\frac{a}{\lvert u \rvert^2}\scaletwoSu{\aln \tr\chibar} \lesssim\hsp &\frac{a}{\lvert u_{\infty}\rvert^2}\lVert \aln \tr\chibar \rVert_{L^2_{(sc)}(S_{u_{\infty},\ubar})} \\&+\intu \frac{a^2}{\upr^4}\scaletwoSuprime{\aln G_i}\duprime .
\end{align} We focus on $0\leq i \leq N+4$. For the first term in $G_i$, there holds \begin{equation}
        \intu  \frac{a^2}{\upr^4}\scaletwoSuprime{a^{\frac{i}{2}} \sumif \nablap \nabla^{i_3}\chibarhat \nabla^{i_4}\chibarhat}\duprime \lesssim \intu \frac{a}{\upr^2}\cdot \frac{\Gamma^2}{\upr} \duprime \lesssim 1.
\end{equation}The second and third terms are handled in the same way. For the fourth term, there holds

\begin{equation}
\begin{split}
        &\intu  \frac{a^2}{\upr^4}\scaletwoSuprime{a^{\frac{i}{2}} \sumif \nablap \nabla^{i_3}(\chibarhat,\tildetr) \nabla^{i_4}\tr\chibar}\duprime\\ \lesssim &\intu \frac{a^2}{\upr^4}\cdot \frac{\upr^2}{a}\frac{\upr}{\al} \cdot \frac{\Gamma^2}{\upr} \duprime \lesssim \frac{\al\hsp \Gamma^2}{\lvert u \rvert}\lesssim 1.
\end{split}
\end{equation}For the fifth term, we bound

\begin{equation}
\begin{split}
        &\intu  \frac{a^2}{\upr^4}\scaletwoSuprime{a^{\frac{i}{2}} \sumifi \nabla^{i_1}\psi_g^{i_2+1} \nabla^{i_3}\tr\chibar \nabla^{i_4}\tr\chibar}\duprime\\ \lesssim& \intu\frac{a^2}{\upr^4}\cdot \frac{\upr^4}{a^2}\cdot \frac{\Gamma^3}{\upr^2}\duprime\lesssim \frac{\Gamma^3}{\lvert u \rvert}\lesssim 1.
\end{split}
\end{equation}For the sixth term, we can bound \begin{equation}
\begin{split}
        &\intu  \frac{a^2}{\upr^4}\scaletwoSuprime{a^{\frac{i}{2}} \sumifim \nabla^{i_1}\psi_g^{i_2+1} \nabla^{i_3}(\chibarhat, \tr\chibar) \nabla^{i_4}\tr\chibar}\duprime\\ =&\intu  \frac{a^\frac{5}{2}}{\upr^4}\scaletwoSuprime{a^{\frac{i-1}{2}} \sum_{i_1+i_2+i_3+i_4+1=i-1} \nabla^{i_1}\psi_g^{i_2+1} \nabla^{i_3}(\chibarhat, \tr\chibar) \nabla^{i_4}\tr\chibar}\duprime \\ \lesssim& \intu\frac{a^{\frac{5}{2}}}{\upr^4}\cdot \frac{\upr^4}{a^2}\cdot \frac{\Gamma^3}{\upr^2}\duprime\lesssim \frac{\al \hsp \Gamma^3}{\lvert u \rvert}\lesssim 1.
\end{split}
\end{equation}Finally, there holds \begin{align}&\intu  \frac{a^2}{\upr^4}\scaletwoSuprime{a^{\frac{i}{2}} \sumiFi \nabla^{i_1}\psi_g^{i_2} \hnabla^{i_3}\alphabar^F \hnabla^{i_4}(\rho^F,\sigma^F)\nabla^{i_5}\tr\chibar}\duprime \\= &\intu \frac{a^{\frac{5}{2}}}{\upr^4}\scaletwoSuprime{a^{\frac{i-1}{2}} \sumiFi \nabla^{i_1}\psi_g^{i_2} \hnabla^{i_3}\alphabar^F \hnabla^{i_4}(\rho^F,\sigma^F)\nabla^{i_5}\tr\chibar}\duprime \\ \lesssim &\intu \frac{a^{\frac{5}{2}}}{\upr^4} \cdot \frac{\upr^2}{a}\cdot\frac{\Gamma^3}{\upr^2} \duprime \lesssim \frac{a^{\frac{3}{2}}\hsp \Gamma^3}{\lvert u \rvert^3} \lesssim 1. \end{align}Crucially, this implies that \begin{equation}
    \frac{a}{\lvert u \rvert^2}\sum_{0\leq i \leq N+2}\scaleinfinitySu{\aln \tr\chibar}\lesssim \mathcal{I}+1,
\end{equation}by Sobolev embedding. This will prove useful in the estimates for $\tildetr$.	For this component, we have the following structure equation:

\[ \nabla_3 \tildetr + \tr\chibar \tildetr = \frac{2}{\lvert u \rvert^2}(\Omega^{-1}-1)+\tildetr \tildetr + \psi_g \tr\chibar -\lvert \chibarhat \rvert^2 - \lvert \alphabar^F \rvert^2:= H_0.  \]
Commuting with $i$ angular derivatives, we arrive at

\begin{align}
    \nabla_3\nabla^i\tildetr + \frac{i+2}{2}\tr\chibar \nabla^i \tildetr &= \sum_{i_1+i_2+i_3=i}\nablap \nabla^{i_3}H_0 \\&+ \sumif \nablap\nabla^{i_3}(\chibarhat,\tildetr)\nabla^{i_4}\tildetr \\ &+\sumifi \nablapp \nabla^{i_3}\tr\chibar \nabla^{i_4}\tildetr \\&+\sumifim \nablapp \nabla^{i_3}(\chibarhat,\tr\chibar)\nabla^{i_4}\tildetr \\&+\sumiFi \nablap \hnabla^{i_3}\alphabarF \hnabla^{i_4}(\rhoF,\sigmaF)\nabla^{i_5}\tildetr := H_i.
\end{align}Estimating in scale-invariant norms, we have

    \[ \frac{a}{\lvert u \rvert} \scaletwoSu{\aln\tildetr} \lesssim \frac{a}{\lvert u_{\infty} \rvert} \lVert \aln \tildetr\rVert_{L^2_{(sc)}(S_{u_{\infty},\ubar})} + \intu \frac{a^2}{\upr^3}\scaletwoSuprime{a^{\frac{i}{2}}H_i}\duprime. \]We restrict attention to $0\leq i \leq N+4$. The first term 
    
    \[\intu \frac{a^2}{\upr^3} \scaletwoSuprime{\sumit \nablap \nabla^{i_3}H_0}\duprime \]can be bounded above by $\mathcal{R}[\rho]+\underline{\mathcal{R}}[\rho]+1$, precisely as in [AnAth]. The second, third and fourth of the terms can be bounded as in previous propositions by $1$. For the last term, we have 
    
    \begin{align}
        &\intu \frac{a^2}{\upr^3}\scaletwoSuprime{\sumiFi \nablap\hnablat \alphabarF \hnablaf (\rhoF,\sigmaF) \nabla^{i_5} \tildetr} \\ \lesssim &\intu \frac{a^2}{\upr^3} \frac{\upr}{a}\frac{\Gamma^3}{\upr^2}\duprime \lesssim 1.
    \end{align}Here, once again, we have made use of the bootstrap assumptions on the gauge-covariantly differentiated Yang-Mills curvature components, bounded above by $\Gamma$. The result follows.
\end{proof}We conclude this section with the corresponding estimate on $\etabar$ and its angular derivatives.

\begin{proposition}
\label{etabarestimate}
    Under the assumptions of Theorem \ref{main1} and the bootstrap assumptions \eqref{bootstrap}, there holds
    
    \[ \sum_{0\leq i \leq N+4} \scaletwoSu{\aln \etabar} \lesssim \mathcal{R}[\tbeta]+\underline{\mathcal{R}}[\tbetabar]+1.\]
\end{proposition}
    \begin{proof}
    The schematic equation for $\etabar$ is as follows:
    
    \begin{equation} \nabla_3 \etabar + \frac{1}{2}\tr\chibar \hsp \etabar = \tbetabar + \tr\chibar \eta +\psi_g \chibarhat + \alphabarF\cdot(\rhoF, \hsp \sigmaF). \end{equation}
    
    \begin{align}
       \nonumber \nabla_3 \nabla^i \etabar + \frac{i+1}{2}\tr\chibar \nabla^i \etabar = &\nabla^i \tbetabar + \sum_{i_1+i_2+i_3=i-1} \nablapp \nablat \tbetabar \\ \nonumber &+\sumif \nablap \nablat \psi_g \nablaf (\chibarhat, \tr\chibar) \\ \nonumber &+\sumif \nablap \hnablat \alphabarF \hnablaf (\rhoF, \sigmaF) \\ \nonumber  &+ \sumif \nablap \nablat (\chibarhat, \tildetr)\nablaf \etabar \\ \nonumber &+\sumifi \nablapp \nablat \tr\chibar \nablaf \etabar \\\nonumber  &+ \sumifim \nablapp \nablat(\chibarhat,\tr\chibar)\nablaf \etabar \\ &+ \sumiFi \nablap \hnablat \alphabarF \hnablaf (\rhoF, \sigmaF) \nabla^{i_5}\etabar := H_i.
    \end{align}Calculating in scale-invariant norms, we arrive at
    
    \be \frac{1}{\lvert u \rvert}\scaletwoSu{\aln \etabar} \lesssim \frac{1}{\lvert u_{\infty} \rvert}\lVert \aln \etabar \rVert_{L^2_{(sc)}(S_{u_{\infty},\ubar})} + \intu \frac{a}{\upr^3} \scaletwoSuprime{a^{\frac{i}{2}}H_i}\duprime. \ee We restrict attention to $0\leq i \leq N+4$. For the first term, there holds
    \be \frac{1}{\lvert u_{\infty} \rvert}\lVert \aln \etabar \rVert_{L^2_{(sc)}(S_{u_{\infty},\ubar})} \lesssim \frac{\mathcal{I}}{\lvert u_{\infty}\rvert} \lesssim \frac{\mathcal{I}}{\lvert u \rvert}. \ee There holds
    
    \begin{align}
       \nonumber  &\intu \frac{a}{\upr^3}\scaletwoSuprime{\aln \tbetabar} \duprime  \\ \nonumber \lesssim  &\left(\intu \frac{a}{\upr^2}\scaletwoSuprime{\aln \tbetabar}^2 \duprime\right)^{\frac{1}{2}}\left( \intu \frac{a}{\upr^4}\duprime\right)^{\frac{1}{2}} \\   \lesssim &\underline{\mathcal{R}}[\tbetabar]\cdot \frac{\al}{\lvert u \rvert^{\frac{3}{2}}} \lesssim \frac{\underline{\mathcal{R}}[\tbetabar]}{\lvert u \rvert}.
    \end{align}For the next term, there holds   \begin{align}
       \nonumber  &\intu \frac{a}{\upr^3}\scaletwoSuprime{\aln \sumitm \nablapp \nablat \tbetabar} \duprime  \\\lesssim &\intu \frac{a}{\upr^3} \cdot \frac{\al \hsp \Gamma^2}{\upr}\duprime \lesssim \frac{1}{\lvert u \rvert}.
    \end{align}For the third term, there holds
     \begin{align}
     &\intu \frac{a}{\upr^3}\scaletwoSuprime{a^{\frac{i}{2}} \sumif \nablap \nablat \psi_g \nablaf (\chibarhat,\tr\chibar)} \duprime  \lesssim \frac{\mathcal{R}[\tbeta]+1}{\lvert u \rvert}.
    \end{align}This is done by further taking into account that the schematic product appearing is actually $\psi_g (\chibarhat, \tr\chibar) = \chibarhat \cdot \psi_g + \tr\chibar \eta$. As such, we use the improvement obtained in Proposition \ref{trchibarbound}  for the $\tr\chibar$-term as well as the improvement obtained in Proposition \ref{etabound} for $\eta$. The term $\chibarhat$ is less anomalous than $\tr\chibar$ and hence the above bound is easier to obtain. Continuing the estimates, there holds
    
    \begin{align}
    \nonumber      &\intu \frac{a}{\upr^3}\scaletwoSuprime{a^{\frac{i}{2}} \sumif \nablap \hnablat \alphabarF \hnablaf(\rhoF, \sigmaF)} \duprime \\ \lesssim& \intu\frac{a}{\upr^3} \cdot \frac{\Gamma^2}{\upr}\duprime \lesssim \frac{1}{\lvert u \rvert}.
    \end{align}For the fifth term, there holds
    
      \begin{align}
        \nonumber  &\intu \frac{a}{\upr^3}\scaletwoSuprime{a^{\frac{i}{2}} \sumif \nablap \nablat (\chibarhat,\tildetr) \nablaf \etabar} \duprime \\ \lesssim& \intu\frac{a}{\upr^3} \cdot \frac{\upr}{\al}\cdot \frac{\Gamma^2}{\upr} \duprime \lesssim \frac{1}{\lvert u \rvert}.
    \end{align}For the sixth term, there holds
    
        \begin{align}
      \nonumber   &\intu \frac{a}{\upr^3}\scaletwoSuprime{a^{\frac{i}{2}} \sumifi \nablapp \nablat \tr\chibar \nablaf \etabar} \duprime  \\ \lesssim &\intu \frac{a}{\upr^3}\cdot \frac{\upr^2}{a}\cdot\frac{\Gamma^3}{\upr^2}\duprime \lesssim \frac{1}{\lvert u\rvert}.
    \end{align}For the seventh term, there holds
    
        \begin{align}
      \nonumber   &\intu \frac{a}{\upr^3}\scaletwoSuprime{a^{\frac{i}{2}} \sumifim \nablapp \nablat (\chibarhat, \tr\chibar) \nablaf \etabar} \duprime  \\ \lesssim &\intu \frac{a^{\frac{3}{2}}}{\upr^3}\cdot \frac{\upr^2}{a}\cdot\frac{\Gamma^3}{\upr^2}\duprime \lesssim \frac{\al \hsp \Gamma^3}{\lvert u \rvert^2} \lesssim \frac{1}{\lvert u\rvert}.
    \end{align}Finally, for the last term, there holds
    
    \begin{align}
     \nonumber    &\intu \frac{a}{\upr^3}\scaletwoSuprime{\sumiFi \nablap \hnablat\alphabarF\hnablaf(\rhoF,\sigmaF)\nabla^{i_5}\etabar} \duprime \\\lesssim   &\intu \frac{a^{\frac{3}{2}}}{\upr^3}\cdot \frac{\Gamma^3}{\upr^2}\duprime \lesssim \frac{1}{\lvert u \rvert}.
    \end{align}Putting everything together, we have
    
    \[ \frac{1}{\lvert u \rvert} \scaletwoSu{\aln\etabar}\lesssim \frac{\mathcal{R}[\tbeta]+\underline{\mathcal{R}}[\tbetabar]+1}{\lvert u \rvert}, \]whence the result follows.

    \end{proof}
    This concludes the estimates on Ricci coefficients.

\subsection{Estimates on the Yang-Mills fields}
\begin{proposition}\label{alphaFprop}
Under the assumptions of Theorem \ref{main1} and the bootstrap assumptions \eqref{bootstrap}, there holds 

\[ \frac{1}{\al}\sum_{0\leq i \leq N+4} \scaletwoSu{(\al\hnabla)^i \alpha^F} \lesssim \underline{\mathbb{YM}}[\rho^F, \sigma^F]+1.\]
\end{proposition}
\begin{proof}
Recall the null evolution equation for $\underline{\alpha}^{F}$
\begin{eqnarray}
\hnabla_3 \alpha^F + \frac{1}{2}\tr\chibar\alpha^F = \hnabla \rho^F + \Hodge{\hnabla}\sigma^F -2 \Hodge{\eta}\sigma^F+ 2 \eta \rho^F +2 \omegabar \alpha^F - \chihat \cdot \alphabar^F.
\end{eqnarray}
An application of the commutation identity \eqref{c2} yields
\begin{equation}
\begin{split}
&\hnabla_{3}\hnabla^{I}\alpha^{F}+\frac{I+1}{2}\tr\underline{\chi}\alpha^{F}\\=&\hnabla^{I+1}(\rho^{F},\sigma^{F})+\sum_{J_{1}+J_{2}=I-1}\nabla^{J_{1}+1}\tr\underline{\chi}\hnabla^{J_{2}}\alpha^{F}\\
&+\sum_{J_1+J_2+J_3+J_4=I-1}\nabla^{J_1} (\eta+\etabar)^{J_2+1}\nabla^{J_3}(\chibarhat,\tr\chibar)\hnabla^{J_4}\alpha^{F} \hspace{2mm}      \\&+ \sum_{J_{1}+J_{2}+J_{3}+J_{4}+J_{5}=I-1}\nabla^{J_{1}}(\eta+\underline{\eta})^{J_{2}}\hnabla^{J_{3}}\alphabar^{F}\hnabla^{J_{4}}(\rho^{F},\sigma^{F})\hnabla^{J_{5}}\alpha^{F}\\&+ \sum_{J_1+J_2+J_3+J_4=I-1}\nabla^{J_1}(\eta+\etabar)^{J_2}\hnabla^{J_3}\alphabar^F \hnabla^{J_4}\alpha^{F}
+\sum_{J_{1}+J_{2}+J_{3}=I-1}\nabla^{J_{1}}(\eta+\underline{\eta})^{J_{2}+1}\hnabla^{J_{3}+1}(\rho^{F},\sigma^{F})\\
&+\sum_{J_{1}+J_{2}+J_{3}+J_{4}=I}\nabla^{J_{1}}(\eta+\underline{\eta})^{J_{2}}\nabla^{J_{3}}\eta\hnabla^{J_{4}}(\rho^{F},\sigma^{F})+\sum_{J_{1}+J_{2}+J_{3}+J_{4}=I}\nabla^{J_{1}}(\eta+\underline{\eta})^{J_{2}}\nabla^{J_{3}}\underline{\omega}\hnabla^{J_{4}}\alpha^{F}\\& 
+\sum_{J_{1}+J_{2}+J_{3}+J_{4}=I}\nabla^{J_{1}}(\eta+\underline{\eta})^{J_{2}}\nabla^{J_{3}}\widehat{\chi}\hnabla^{J_{4}}\underline{\alpha}^{F}+\sum_{J_{1}+J_{2}+J_{3} +J_{4}=I}\nabla^{J_{1}}(\eta+\underline{\eta})^{J_{2}}\hat{\nabla}^{J_{3}}\widehat{\underline{\chi}}\hnabla^{J_{4}}\alpha^{F}\\
=&\mathfrak{G}_{1},
\end{split}
\end{equation}
where we denote the right hand side by $\mathfrak{G}_{1}$ for convenience. An application of the weighted transport inequality (\ref{36}) yields 
\begin{eqnarray}
\label{eq:transport_u1}
|u|^{I}||\hnabla^{I}\alpha^{F}||_{L^{2}(S_{u,\underline{u}})}\leq |u_{\infty}|^{I}||\hnabla^{I}\alpha^{F}||_{L^{2}(S_{u_{\infty},\underline{u}})}+\int_{u_{\infty}}^{u}|u^{'}|^{I}||\mathfrak{G}_{1}||_{L^{2}(S_{u^{'},\underline{u}})}du^{'}.
\end{eqnarray}
Now we recall the definitions of the scale-invariant norms \[ ||\phi||_{\mathcal{L}^{2}_{sc}(S_{u,\underline{u}})}:=a^{-s_{2}(\phi)}|u|^{2s_{2}(\phi)}||\phi||_{L^{2}(S_{u,\underline{u}})}\] to yield 
\begin{eqnarray}
||\hnabla^{I}\alpha^{F}||_{\mathcal{L}^{2}_{sc}(S_{u,\underline{u}})}=a^{-\frac{I}{2}}|u|^{I}||\hnabla^{I}\alpha^{F}||_{L^{2}(S_{u,\underline{u}})},
\end{eqnarray}
since $s_{2}(\hnabla^{I}\alpha^{F})=s_{2}(\alpha^{F})+I\times\frac{1}{2}=0+\frac{I}{2}=\frac{I}{2}$. Similarly $s_{2}(\mathfrak{G}_{1})=s_{2}(\hnabla_{3}\hnabla^{I}\alpha^{F})=0+1+I\times \frac{1}{2}=\frac{I}{2}+1$ and therefore 
\begin{eqnarray}
||\mathfrak{G}_{1}||_{\mathcal{L}^{2}_{sc}(S_{u,\underline{u}})}=a^{-\frac{I+2}{2}}|u|^{I+2}||\mathfrak{G}_{1}||_{L^{2}(S_{u,\underline{u}})}.
\end{eqnarray}
In terms of the scale invariant norms, the weighted transport inequality (\ref{eq:transport_u1}) for $\alpha^{F}$ 
reads 
\begin{equation}
\begin{split}
&a^{-\frac{1}{2}}||(a^{\frac{1}{2}}\hnabla)^{I}\alpha^{F}||_{\mathcal{L}^{2}_{sc}(S_{u,\underline{u}})}\\ \leq & a^{-\frac{1}{2}}||(a^{\frac{1}{2}}\hnabla)^{I}\alpha^{F}||_{\mathcal{L}^{2}_{sc}(S_{u_{\infty},\underline{u}})}+\int_{u_{\infty}}^{u}\frac{a^{\frac{1}{2}}}{|u^{'}|^{2}}||a^{\frac{I}{2}}\mathfrak{G}_{1}||_{\mathcal{L}^{2}_{sc}(S_{u^{'},\underline{u}})}du^{'}\\ 
\leq &a^{-\frac{1}{2}}||(a^{\frac{1}{2}}\hnabla)^{I}\alpha^{F}||_{\mathcal{L}^{2}_{sc}(S_{u_{\infty},\underline{u}})} +\int_{u_{\infty}}^{u}\frac{a^{\frac{1}{2}}}{|u^{'}|^{2}}||a^{\frac{I}{2}}\hnabla^{I+1}(\rho^{F},\sigma^{F})||_{\mathcal{L}^{2}(S_{u^{'},\underline{u}})}du^{'}\\+&\int_{u_{\infty}}^{u}\frac{a^{\frac{1}{2}}}{|u^{'}|^{2}}||a^{\frac{I}{2}}\sum_{J_{1}+J_{2}=I-1}\nabla^{J_{1}+1}\tr\underline{\chi}\hnabla^{J_{2}}\alpha^{F}||_{\mathcal{L}^{2}(S_{u^{'},\underline{u}})}\\
+&\int_{u_{\infty}}^{u}\frac{a^{\frac{1}{2}}}{|u^{'}|^{2}}||a^{\frac{I}{2}}\sum_{J_1+J_2+J_3+J_4=I-1}\nabla^{J_1} (\eta+\etabar)^{J_2+1}\nabla^{J_3}(\chibarhat,\tr\chibar)\hnabla^{J_4}\alpha^{F}||_{\mathcal{L}^{2}(S_{u^{'},\underline{u}})} \\+& \int_{u_{\infty}}^{u}\frac{a^{\frac{1}{2}}}{|u^{'}|^{2}}||a^{\frac{I}{2}}\sum_{J_{1}+J_{2}+J_{3}+J_{4}+J_{5}=I-1}\nabla^{J_{1}}(\eta+\underline{\eta})^{J_{2}}\hnabla^{J_{3}}\alphabar^{F}\hnabla^{J_{4}}(\rho^{F},\sigma^{F})\hnabla^{J_{5}}\alpha^{F}||_{\mathcal{L}^{2}(S_{u^{'},\underline{u}})}\\ +& \int_{u_{\infty}}^{u}\frac{a^{\frac{1}{2}}}{|u^{'}|^{2}}||a^{\frac{I}{2}}\sum_{J_1+J_2+J_3+J_4=I-1}\nabla^{J_1}(\eta+\etabar)^{J_2}\hnabla^{J_3}\alphabar^F \hnabla^{J_4}\alpha^{F}||_{\mathcal{L}^{2}(S_{u^{'},\underline{u}})}\\
+&\int_{u_{\infty}}^{u}\frac{a^{\frac{1}{2}}}{|u^{'}|^{2}}||a^{\frac{I}{2}}\sum_{J_{1}+J_{2}+J_{3}=I-1}\nabla^{J_{1}}(\eta+\underline{\eta})^{J_{2}+1}\hnabla^{J_{3}+1}(\rho^{F},\sigma^{F})||_{\mathcal{L}^{2}(S_{u^{'},\underline{u}})}\\ 
+&\int_{u_{\infty}}^{u}\frac{a^{\frac{1}{2}}}{|u^{'}|^{2}}||a^{\frac{I}{2}}\sum_{J_{1}+J_{2}+J_{3}+J_{4}=I}\nabla^{J_{1}}(\eta+\underline{\eta})^{J_{2}}\nabla^{J_{3}}\eta\hnabla^{J_{4}}(\rho^{F},\sigma^{F})||_{\mathcal{L}^{2}(S_{u^{'},\underline{u}})}\\+&\int_{u_{\infty}}^{u}\frac{a^{\frac{1}{2}}}{|u^{'}|^{2}}||a^{\frac{I}{2}}\sum_{J_{1}+J_{2}+J_{3}+J_{4}=I}\nabla^{J_{1}}(\eta+\underline{\eta})^{J_{2}}\nabla^{J_{3}}\underline{\omega}\hnabla^{J_{4}}\alpha^{F}||_{\mathcal{L}^{2}(S_{u^{'},\underline{u}})}\\ 
+&\int_{u_{\infty}}^{u}\frac{a^{\frac{1}{2}}}{|u^{'}|^{2}}||a^{\frac{I}{2}}\sum_{J_{1}+J_{2}+J_{3}+J_{4}=I}\nabla^{J_{1}}(\eta+\underline{\eta})^{J_{2}}\nabla^{J_{3}}\widehat{\chi}\hnabla^{J_{4}}\underline{\alpha}^{F}||_{\mathcal{L}^{2}(S_{u^{'},\underline{u}})}\\ 
+&\int_{u_{\infty}}^{u}\frac{a^{\frac{1}{2}}}{|u^{'}|^{2}}||a^{\frac{I}{2}}\sum_{J_{1}+J_{2}+J_{3} +J_{4}=I}\nabla^{J_{1}}(\eta+\underline{\eta})^{J_{2}}\hat{\nabla}^{J_{3}}\widehat{\underline{\chi}}\hnabla^{J_{4}}\alpha^{F}||_{\mathcal{L}^{2}(S_{u^{'},\underline{u}})}.
\end{split}
\end{equation}
We now control each term by means of the bootstrap assumption (\ref{bootstrap}). Notice that the first term $\int_{u_{\infty}}^{u}\frac{a^{\frac{1}{2}}}{|u^{'}|^{2}}||a^{\frac{I}{2}}\hnabla^{I+1}(\rho^{F},\sigma^{F})||_{\mathcal{L}^{2}(S_{u^{'},\underline{u}})}du^{'}$ has the top-most derivative (i.e., $N+5$) and therefore we bound it in $\mathcal{L}^{2}_{sc}(\underline{H}^{(u_{\infty},u)})$. An application of Cauchy-Schwartz yields 
\begin{equation}
\begin{split}
&\int_{u_{\infty}}^{u}\frac{a^{\frac{1}{2}}}{|u^{'}|^{2}}||a^{\frac{I}{2}}\hnabla^{I+1}(\rho^{F},\sigma^{F})||_{\mathcal{L}^{2}_{sc}(S_{u^{'},\underline{u}})}du^{'}\\ \leq &\left(\int_{u_{\infty}}^{u}\frac{a}{|u^{'}|^{2}}||a^{\frac{I}{2}}\hnabla^{I+1}(\rho^{F},\sigma^{F})||^{2}_{\mathcal{L}^{2}_{sc}(S_{u^{'},\underline{u}})}\right)^{\frac{1}{2}}\left(\int_{u_{\infty}}^{u}\frac{1}{|u^{'}|^{2}}du^{'}\right)^{\frac{1}{2}}\\  
\lesssim &||a^{\frac{I}{2}}\hnabla^{I+1}(\rho^{F},\sigma^{F})||_{\mathcal{L}^{2}_{sc}(\underline{H}^{(u_{\infty},u)})}\frac{1}{|u|^{\frac{1}{2}}}\lesssim \frac{1}{a^{\frac{1}{2}}}||a^{\frac{I}{2}}\hnabla^{I+1}(\rho^{F},\sigma^{F})||_{\mathcal{L}^{2}_{sc}(\underline{H}^{(u_{\infty},u)})}\frac{a^{\frac{1}{2}}}{|u|^{\frac{1}{2}}}\lesssim \underline{\mathbb{YM}}[\rho^{F},\sigma^{F}].
\end{split}
\end{equation}
We estimate the next term: 
\begin{equation}
\begin{split}
&\int_{u_{\infty}}^{u}\frac{a^{\frac{1}{2}}}{|u^{'}|^{2}}||a^{\frac{I}{2}}\sum_{J_{1}+J_{2}=I-1}\nabla^{J_{1}+1}\tr\underline{\chi}\hnabla^{J_{2}}\alpha^{F}||_{\mathcal{L}^{2}(S_{u^{'},\underline{u}})}\\
=&\int_{u_{\infty}}^{u}\frac{a^{\frac{1}{2}}}{|u^{'}|^{2}}||a^{\frac{I}{2}}\sum_{J_{1}+J_{2}=I-1}\nabla^{J_{1}+1}(\tr\underline{\chi}+\frac{2}{|u^{'}|})\hnabla^{J_{2}}\alpha^{F}||_{\mathcal{L}^{2}(S_{u^{'},\underline{u}})}\lesssim \int_{u_{\infty}}^{u}\frac{1}{|u^{'}|^{2}}\bbGamma^{2}\lesssim \frac{\Gamma^{2}}{|u|} \lesssim 1
\end{split}
\end{equation}
due to integrability of $\frac{1}{|u^{'}|^{2}}$ away from zero. Now consider the next term:
\begin{equation}
\begin{split}
&\int_{u_{\infty}}^{u}\frac{a^{\frac{1}{2}}}{|u^{'}|^{2}}||a^{\frac{I}{2}}\sum_{J_1+J_2+J_3+J_4=I-1}\nabla^{J_1} (\eta+\etabar)^{J_2+1}\nabla^{J_3}(\chibarhat,\tr\chibar)\hnabla^{J_4}\alpha^{F}||_{\mathcal{L}^{2}(S_{u^{'},\underline{u}})}\\ 
=&\int_{u_{\infty}}^{u}\frac{a^{\frac{1}{2}}}{|u^{'}|^{2}}||a^{\frac{I}{2}}\sum_{J_1+J_2+J_3+J_4=I-1}\nabla^{J_1} (\eta+\etabar)^{J_2+1}\nabla^{J_3}(\chibarhat,(\tr\chibar+\frac{2}{|u^{'}|}-\frac{2}{|u^{'}|}))\hnabla^{J_4}\alpha^{F}||_{\mathcal{L}^{2}(S_{u^{'},\underline{u}})}\\ 
\lesssim &\int_{u_{\infty}}^{u}\frac{a^{\frac{1}{2}}}{|u^{'}|^{2}}\frac{a^{\frac{1}{2}}}{|u^{'}|}\bbGamma^{3}+ \int_{u_{\infty}}^{u}\frac{a^{\frac{1}{2}}}{|u^{'}|^{2}}\bbGamma^{3}\lesssim \frac{a}{|u|^{2}}\Gamma^{3}+\frac{a^{\frac{1}{2}}}{|u|}\Gamma^{3}\lesssim 1.
\end{split}
\end{equation}
Similarly, for the next term we have: 
\begin{eqnarray}
\nonumber \int_{u_{\infty}}^{u}\frac{a^{\frac{1}{2}}}{|u^{'}|^{2}}||a^{\frac{I}{2}}\sum_{J_{1}+J_{2}+J_{3}+J_{4}+J_{5}=I-1}\nabla^{J_{1}}(\eta+\underline{\eta})^{J_{2}}\hnabla^{J_{3}}\alphabar^{F}\hnabla^{J_{4}}(\rho^{F},\sigma^{F})\hnabla^{J_{5}}\alpha^{F}||_{\mathcal{L}^{2}(S_{u^{'},\underline{u}})}\\
\lesssim \int_{u_{\infty}}^{u}\frac{a^{\frac{1}{2}}}{|u'|^{2}}\frac{a^{\frac{3}{2}}}{|u^{'}|^{3}}\bbGamma^{4}+\int_{u_{\infty}}^{u}\frac{a^{\frac{1}{2}}}{|u^{'}|^{2}}\frac{a}{|u^{'}|^{2}}\bbGamma^{3}\lesssim \frac{a^{2}}{|u|^{4}}\Gamma^{4}+\frac{a^{\frac{3}{2}}}{|u|^{3}}\Gamma^{3}\lesssim 1.
\end{eqnarray}
For the next term: 
\begin{eqnarray}
\nonumber \int_{u_{\infty}}^{u}\frac{a^{\frac{1}{2}}}{|u^{'}|^{2}}||a^{\frac{I}{2}}\sum_{J_1+J_2+J_3+J_4=I-1}\nabla^{J_1}(\eta+\etabar)^{J_2}\hnabla^{J_3}\alphabar^F \hnabla^{J_4}\alpha^{F}||_{\mathcal{L}^{2}(S_{u^{'},\underline{u}})}\\
\lesssim \int_{u_{\infty}}^{u}\frac{a^{\frac{1}{2}}}{|u^{'}|^{2}}\frac{a^{\frac{3}{2}}}{|u^{'}|^{2}}\bbGamma^{3}+\int_{u_{\infty}}^{u}\frac{a^{\frac{1}{2}}}{|u^{'}|^{2}}\frac{a}{|u^{'}|}\bbGamma^{2}\lesssim \frac{a^{2}}{|u|^{3}}\Gamma^{3}+\frac{a^{\frac{3}{2}}}{|u|^{2}}\Gamma^{2}\lesssim 1.
\end{eqnarray}
We proceed with the seventh term:
\begin{eqnarray}
\nonumber \int_{u_{\infty}}^{u}\frac{a^{\frac{1}{2}}}{|u^{'}|^{2}}||a^{\frac{I}{2}}\sum_{J_{1}+J_{2}+J_{3}=I-1}\nabla^{J_{1}}(\eta+\underline{\eta})^{J_{2}+1}\hnabla^{J_{3}+1}(\rho^{F},\sigma^{F})||_{\mathcal{L}^{2}(S_{u^{'},\underline{u}})}\\
\lesssim \int_{u_{\infty}}^{u}\frac{a^{\frac{1}{2}}}{|u^{'}|^{3}}\bbGamma^{2}\lesssim \frac{a^{\frac{1}{2}}}{|u|^{2}}\Gamma^{2}\lesssim 1.
\end{eqnarray}
For the eighth term:
\begin{eqnarray}
\nonumber\int_{u_{\infty}}^{u}\frac{a^{\frac{1}{2}}}{|u^{'}|^{2}}||a^{\frac{I}{2}}\sum_{J_{1}+J_{2}+J_{3}+J_{4}=I}\nabla^{J_{1}}(\eta+\underline{\eta})^{J_{2}}\nabla^{J_{3}}\eta\hnabla^{J_{4}}(\rho^{F},\sigma^{F})||_{\mathcal{L}^{2}(S_{u^{'},\underline{u}})}\\ \lesssim \int_{u_{\infty}}^{u}\frac{a^{\frac{1}{2}}}{|u^{'}|^{2}}\frac{1}{|u^{'}|}\bbGamma^{2}+\int_{u_{\infty}}^{u}\frac{a^{\frac{1}{2}}}{|u^{'}|^{2}}\frac{a^{\frac{1}{2}}}{|u^{'}|}\bbGamma^{3}\lesssim \frac{a^{\frac{1}{2}}}{|u|^{2}}\Gamma^{2}+\frac{a}{|u|^{3}}\Gamma^{3}\lesssim 1.
\end{eqnarray}
For the ninth term:
\begin{eqnarray}
\nonumber\int_{u_{\infty}}^{u}\frac{a^{\frac{1}{2}}}{|u^{'}|^{2}}||a^{\frac{I}{2}}\sum_{J_{1}+J_{2}+J_{3}+J_{4}=I}\nabla^{J_{1}}(\eta+\underline{\eta})^{J_{2}}\nabla^{J_{3}}\underline{\omega}\hnabla^{J_{4}}\alpha^{F}||_{\mathcal{L}^{2}(S_{u^{'},\underline{u}})}\\ \lesssim\int_{u_{\infty}}^{u}\frac{a^{\frac{1}{2}}}{|u^{'}|^{2}}\frac{a^{\frac{1}{2}}}{|u^{'}|}\bbGamma^{2}+\int_{u_{\infty}}^{u}\frac{a^{\frac{1}{2}}}{|u^{'}|^{2}}\frac{a}{|u^{'}|^{2}}\bbGamma^{3}\lesssim \frac{a}{|u|^{2}}\Gamma^{2}+\frac{a^{\frac{3}{2}}}{|u|^{3}}\Gamma^{3}\lesssim 1.
\end{eqnarray}
The penultimate term reads 
\begin{eqnarray}
\nonumber \int_{u_{\infty}}^{u}\frac{a^{\frac{1}{2}}}{|u^{'}|^{2}}||a^{\frac{I}{2}}\sum_{J_{1}+J_{2}+J_{3}+J_{4}=I}\nabla^{J_{1}}(\eta+\underline{\eta})^{J_{2}}\nabla^{J_{3}}\widehat{\chi}\hnabla^{J_{4}}\underline{\alpha}^{F}||_{\mathcal{L}^{2}(S_{u^{'},\underline{u}})}\\
\lesssim\int_{u_{\infty}}^{u}\frac{a^{\frac{1}{2}}}{|u^{'}|^{2}}\frac{a^{\frac{1}{2}}}{|u^{'}|}\bbGamma^{2}+\int_{u_{\infty}}^{u}\frac{a^{\frac{1}{2}}}{|u^{'}|^{2}}\frac{a}{|u^{'}|^{2}}\bbGamma^{3}\lesssim \frac{a}{|u|^{2}}\Gamma^{2}+\frac{a^{\frac{3}{2}}}{|u|^{3}}\Gamma^{3}\lesssim 1,
\end{eqnarray}
while the last term reads
\begin{eqnarray}\nonumber
\int_{u_{\infty}}^{u}\frac{a^{\frac{1}{2}}}{|u^{'}|^{2}}||a^{\frac{I}{2}}\sum_{J_{1}+J_{2}+J_{3} +J_{4}=I}\nabla^{J_{1}}(\eta+\underline{\eta})^{J_{2}}\hat{\nabla}^{J_{3}}\widehat{\underline{\chi}}\hnabla^{J_{4}}\alpha^{F}||_{\mathcal{L}^{2}(S_{u^{'},\underline{u}})}\\ 
\lesssim \int_{u_{\infty}}^{u}\frac{a^{\frac{1}{2}}}{|u^{'}|^{2}}\frac{|u^{'}|}{a^{\frac{1}{2}}}\frac{a}{|u^{'}|^{2}}\bbGamma^{3}+\int_{u_{\infty}}^{u}\frac{a^{\frac{1}{2}}}{|u^{'}|^{2}}\frac{|u^{'}|}{a^{\frac{1}{2}}}\frac{a^{\frac{1}{2}}}{|u^{'}|}\bbGamma^{2}\lesssim \frac{a}{|u|^{2}}\Gamma^{3}+\frac{a^{\frac{1}{2}}}{|u|}\Gamma^{2}\lesssim 1.
\end{eqnarray}
Collecting all the terms, we have
\begin{eqnarray}
a^{-\frac{1}{2}}||(a^{\frac{1}{2}}\hnabla)^{I}\alpha^{F}||_{\mathcal{L}^{2}_{sc}(S_{u,\underline{u}})}\lesssim \underline{\mathbb{YM}}[\rho^{F},\sigma^{F}]+1,
\end{eqnarray}
since $a^{-\frac{1}{2}}||(a^{\frac{1}{2}}\hnabla)^{I}\alpha^{F}||_{\mathcal{L}^{2}_{sc}(S_{u_{\infty},\underline{u}})}\lesssim 1$. This concludes the proof. 
\end{proof}
\begin{remark}\label{remarkF}
For the non-extremal cases $0\leq i \leq N+3$, we have the improved estimate since the topmost ($N+5$) derivatives of $(\rho^{F},\sigma^{F})$ do not appear in the commuted null evolution equation for $\alpha^{F}$
\[ \frac{1}{\al}\sum_{0\leq i \leq N+3} \scaletwoSu{(\al\hnabla)^i \alpha^F} \lesssim 1,\]which implies, by Sobolev embedding, the following $L^{\infty}$ estimate:

\[ \frac{1}{\al}\sum_{0\leq i \leq N+1} \scaleinfinitySu{(\al\hnabla)^i \alpha^F} \lesssim 1.\]This will be useful later on, for example in Proposition \ref{trchiprop}.
\end{remark}

\begin{proposition}
Under the assumptions of Theorem \ref{main1} and the bootstrap assumptions \eqref{bootstrap}, there holds 

\[\sum_{0\leq i \leq N+4} \scaletwoSu{(\al\hnabla)^i (\rho^F,\sigma^{F})} \lesssim \underline{\mathbb{YM}}[\underline{\alpha}^{F}]+1.\]
\end{proposition}
\begin{proof}
We use the null evolution equation 
\begin{equation}
    \hnabla_3 \rho^F + \tr\chibar \rho^F = - \widehat{\text{div}} \alphabar^F + \left( \eta-\etabar \right)\cdot \alphabar^F,
\end{equation}
\begin{equation}
    \hnabla_3 \sigma^F + \tr\chibar \sigma^F = -\widehat{\text{curl}}\alphabar^F+\left(\eta-\etabar\right)\cdot\Hodge{\alphabar}^F.
\end{equation}
and the commutation relation to get the following schematic null evolution equation for the pair $(\rho^{F},\sigma^{F})$:
\begin{equation}
\begin{split}
&\hnabla_{3}\hnabla^{I}(\rho^{F},\sigma^{F})+(\frac{I}{2}+1)\tr\underline{\chi}\hnabla^{I}(\rho^{F},\sigma^{F})\\=&\hnabla^{I+1}\underline{\alpha}^{F}+\sum_{J_{1}+J_{2}=I-1}\nabla^{J_{1}+1}\tr\underline{\chi}\hnabla^{J_{2}}(\rho^{F},\sigma^{F})\\ 
+&\sum_{J_{1}+J_{2}+J_{3}+J_{4}=I-1}\nabla^{J_{1}}(\eta+\underline{\eta})^{J_{2}+1}\nabla^{J_{3}}(\widehat{\underline{\chi}},\tr\underline{\chi})\hnabla^{J_{4}}(\rho^{F},\sigma^{F})\\+&\sum_{J_{1}+J_{2}+J_{3}+J_{4}=I}\nabla^{J_{1}}(\eta+\underline{\eta})^{J_{2}}\nabla^{J_{3}}(\eta+\underline{\eta})
\hnabla^{J_{4}}\underline{\alpha}^{F}\\ 
+&\sum_{J_1+J_2+J_3+J_4=I}\nabla^{J_1}(\eta+\etabar)^{J_2}\nabla^{J_3}(\chibarhat,\hsp \tildetr)\hnabla^{J_4}(\rho^{F},\sigma^{F})\\+&\sum_{J_{1}+J_{2}+J_{3} +J_{4}=I-1}\nabla^{J_{1}}(\eta+\underline{\eta})^{J_{2} +1}\hat{\nabla}^{J_{3}}\tr\underline{\chi}\hnabla^{J_{4}}(\rho^{F},\sigma^{F})\\+&\sum_{J_1+J_2+J_3+J_4=I-1}\nabla^{J_1} (\eta+\etabar)^{J_2+1}\nabla^{J_3}(\chibarhat,\tr\chibar)\hnabla^{J_4}(\rho^{F},\sigma^{F})\\+& \sum_{J_1+J_2+J_3+J_4=I-1}\nabla^{J_1}(\eta+\etabar)^{J_2}\hnabla^{J_3}\alphabar^F \hnabla^{J_4}(\rho^{F},\sigma^{F})\\+ &\sum_{J_{1}+J_{2}+J_{3}+J_{4}+J_{5}=I-1}\nabla^{J_{1}}(\eta+\underline{\eta})^{J_{2}}\hnabla^{J_{3}}\alphabar^{F}\hnabla^{J_{4}}(\rho^{F},\sigma^{F})\hnabla^{J_{5}}(\rho^{F},\sigma^{F}),
\end{split}
\end{equation}
which is compactly written as follows 
\begin{eqnarray}
\hnabla_{3}\hnabla^{I}(\rho^{F},\sigma^{F})+(\frac{I}{2}+1)\tr\underline{\chi}\hnabla^{I}(\rho^{F},\sigma^{F})=\mathfrak{G}_{2}.
\end{eqnarray}
An application of the weighted transport inequality from Proposition \ref{3.6} yields 
\begin{equation}
|u|^{I+1}||\hnabla^{I}(\rho^{F},\sigma^{F})||_{L^{2}(S_{u,\underline{u}})}\lesssim |u_{\infty}|^{I+1}||\hnabla^{I}(\rho^{F},\sigma^{F})||_{L^{2}(S_{u_{\infty},\underline{u}})}+\int_{u_{\infty}}^{u}|u^{'}|^{I+1}||\mathfrak{G}_{2}||_{L^{2}(S_{u^{'},\underline{u}})}du^{'}.
\end{equation}
Recalling the definitions of the scale-invariant norms \[||\phi||_{\mathcal{L}^{2}_{sc}(S_{u,\underline{u}})}:=a^{-s_{2}(\phi)}|u|^{2s_{2}(\phi)}||\phi||_{L^{2}(S_{u,\underline{u}})}\] we arrive at
\begin{eqnarray}
||\hnabla^{I}(\rho^{F},\sigma^{F})||_{\mathcal{L}^{2}_{sc}(S_{u,\underline{u}})}=a^{-\frac{I+1}{2}}|u|^{I+1}||\hnabla^{I}\alpha^{F}||_{L^{2}(S_{u,\underline{u}})}
\end{eqnarray}
since $s_{2}(\hnabla^{I}(\rho^{F},\sigma^{F}))=s_{2}((\rho^{F},\sigma^{F}))+I\times\frac{1}{2}=\frac{I}{2}+\frac{1}{2}$. Similarly $s_{2}(\mathfrak{G}_{2})=s_{2}(\hnabla_{3}\hnabla^{I}\rho^{F})=\frac{1}{2}+1+I\times \frac{1}{2}=\frac{I}{2}+\frac{3}{2}$ and therefore 
\begin{eqnarray}
||\mathfrak{G}_{2}||_{\mathcal{L}^{2}_{sc}(S_{u,\underline{u}})}=a^{-\frac{I+3}{2}}|u|^{I+3}||\mathfrak{G}_{2}||_{L^{2}(S_{u,\underline{u}})}.
\end{eqnarray}
Therefore, the weighted transport inequality from Proposition \ref{3.6} yields

\begin{align} \nonumber
&||(a^{\frac{1}{2}}\hnabla)^{I}(\rho^{F},\sigma^{F})||_{\mathcal{L}^{2}_{sc}(S_{u,\underline{u}})}\\ \nonumber\lesssim &||(a^{\frac{1}{2}}\hnabla)^{I}(\rho^{F},\sigma^{F})||_{\mathcal{L}^{2}_{sc}(S_{u_{\infty},\underline{u}})} +\int_{u_{\infty}}^{u}\frac{a}{|u^{'}|^{2}}||a^{\frac{I}{2}}\mathfrak{G}_{2}||_{\mathcal{L}^{2}_{sc}(S_{u^{'},\underline{u}})}du^{'}\\ \nonumber
\lesssim &||(a^{\frac{1}{2}}\hnabla)^{I}(\rho^{F},\sigma^{F})||_{\mathcal{L}^{2}_{sc}(S_{u_{\infty},\underline{u}})} 
+\int_{u_{\infty}}^{u}\frac{a}{|u^{'}|^{2}}||a^{\frac{I}{2}}\hnabla^{I+1}\underline{\alpha}^{F}||_{L^{2}_{Sc}(S_{u^{'},\underline{u}})}du^{'}\\ \nonumber +&\int_{u_{\infty}}^{u}\frac{a}{|u^{'}|^{2}}||a^{\frac{I}{2}}\sum_{J_{1}+J_{2}=I-1}\nabla^{J_{1}+1}\tr\underline{\chi}\hnabla^{J_{2}}(\rho^{F},\sigma^{F}) ||_{L^{2}_{Sc}(S_{u^{'},\underline{u}})}du^{'}\\  \nonumber
+&\int_{u_{\infty}}^{u}\frac{a}{|u^{'}|^{2}}||a^{\frac{I}{2}}\sum_{J_{1}+J_{2}+J_{3}+J_{4}=I-1}\nabla^{J_{1}}(\eta+\underline{\eta})^{J_{2}+1}\nabla^{J_{3}}(\widehat{\underline{\chi}},\tr\underline{\chi})\hnabla^{J_{4}}(\rho^{F},\sigma^{F})||_{L^{2}_{sc}(S_{u^{'},\underline{u}})}du^{'}\\  \nonumber
+&\int_{u_{\infty}}^{u}\frac{a}{|u^{'}|^{2}}||a^{\frac{I}{2}} \sum_{J_{1}+J_{2}+J_{3}+J_{4}=I}\nabla^{J_{1}}(\eta+\underline{\eta})^{J_{2}}\nabla^{J_{3}}(\eta+\underline{\eta})
\hnabla^{J_{4}}\underline{\alpha}^{F} ||_{L^{2}_{sc}(S_{u^{'},\underline{u}})}du^{'}\\  \nonumber
+&\int_{u_{\infty}}^{u}\frac{a}{|u^{'}|^{2}}||a^{\frac{I}{2}} \sum_{J_1+J_2+J_3+J_4=I}\nabla^{J_1}(\eta+\etabar)^{J_2}\nabla^{J_3}(\chibarhat,\hsp \tildetr)\hnabla^{J_4}(\rho^{F},\sigma^{F}) ||_{L^{2}_{sc}(S_{u^{'},\underline{u}})}du^{'}\\  \nonumber
+&\int_{u_{\infty}}^{u}\frac{a}{|u^{'}|^{2}}||a^{\frac{I}{2}} \sum_{J_{1}+J_{2}+J_{3} +J_{4}=I-1}\nabla^{J_{1}}(\eta+\underline{\eta})^{J_{2} +1}\hat{\nabla}^{J_{3}}\tr\underline{\chi}\hnabla^{J_{4}}(\rho^{F},\sigma^{F}) ||_{L^{2}_{sc}(S_{u^{'},\underline{u}})}du^{'}\\ \nonumber
+&\int_{u_{\infty}}^{u}\frac{a}{|u^{'}|^{2}}||a^{\frac{I}{2}} \sum_{J_1+J_2+J_3+J_4=I-1}\nabla^{J_1} (\eta+\etabar)^{J_2+1}\nabla^{J_3}(\chibarhat,\tr\chibar)\hnabla^{J_4}(\rho^{F},\sigma^{F}) ||_{L^{2}_{sc}(S_{u^{'},\underline{u}})}du^{'}\\ \nonumber 
+&\int_{u_{\infty}}^{u}\frac{a}{|u^{'}|^{2}}||a^{\frac{I}{2}} \sum_{J_1+J_2+J_3+J_4=I-1}\nabla^{J_1}(\eta+\etabar)^{J_2}\hnabla^{J_3}\alphabar^F \hnabla^{J_4}(\rho^{F},\sigma^{F})||_{L^{2}_{sc}(S_{u^{'},\underline{u}})}du^{'}\\ 
+&\int_{u_{\infty}}^{u}\frac{a}{|u^{'}|^{2}}||a^{\frac{I}{2}} \sum_{J_{1}+J_{2}+J_{3}+J_{4}+J_{5}=I-1}\nabla^{J_{1}}(\eta+\underline{\eta})^{J_{2}}\hnabla^{J_{3}}\alphabar^{F}\hnabla^{J_{4}}(\rho^{F},\sigma^{F})\hnabla^{J_{5}}(\rho^{F},\sigma^{F})||_{L^{2}_{sc}(S_{u^{'},\underline{u}})}du^{'}.
\end{align}

Now we estimate each term separately. 
\begin{equation}
\begin{split}
&\int_{u_{\infty}}^{u}\frac{a}{|u^{'}|^{2}}||a^{\frac{I}{2}}\hnabla^{I+1}\underline{\alpha}^{F}||_{L^{2}_{Sc}(S_{u^{'},\underline{u}})}du^{'}\\ \lesssim&\left(\int_{u_{\infty}}^{u}\frac{a}{|u^{'}|^{2}}||a^{\frac{I}{2}}\widehat{\nabla}^{I+1}\underline{\alpha}^{F}||^{2}_{\mathcal{L}^{2}_{sc}(S_{u^{'},\underline{u}})}du^{'}\right)^{\frac{1}{2}}\left(\int_{u_{\infty}}^{u}\frac{a}{|u^{'}|^{2}}du^{'}\right)^{\frac{1}{2}}
\lesssim \underline{\mathbb{YM}}[\underline{\alpha}^{F}]\cdot \frac{a^{\frac{1}{2}}}{|u|^{\frac{1}{2}}}\lesssim  \underline{\mathbb{YM}}[\underline{\alpha}^{F}],
\end{split}
\end{equation}
while
\begin{align}
\nonumber &\int_{u_{\infty}}^{u}\frac{a}{|u^{'}|^{2}}||a^{\frac{I}{2}}\sum_{J_{1}+J_{2}=I-1}\nabla^{J_{1}+1}\tr\underline{\chi}\hnabla^{J_{2}}(\rho^{F},\sigma^{F}) ||_{L^{2}_{Sc}(S_{u^{'},\underline{u}})}du^{'}\\\nonumber 
=&\int_{u_{\infty}}^{u}\frac{a}{|u^{'}|^{2}}||a^{\frac{I}{2}}\sum_{J_{1}+J_{2}=I-1}\nabla^{J_{1}+1}(\tr\underline{\chi}+\frac{2}{|u|}-\frac{2}{|u|})\hnabla^{J_{2}}(\rho^{F},\sigma^{F}) ||_{L^{2}_{Sc}(S_{u^{'},\underline{u}})}du^{'}\\\nonumber 
=&\int_{u_{\infty}}^{u}\frac{a}{|u^{'}|^{2}}||a^{\frac{I}{2}}\sum_{J_{1}+J_{2}=I-1}\nabla^{J_{1}+1}\widetilde{\tr\underline{\chi}}\hnabla^{J_{2}}(\rho^{F},\sigma^{F}) ||_{L^{2}_{Sc}(S_{u^{'},\underline{u}})}du^{'}\lesssim \int_{u_{\infty}}^{u} \frac{a}{\upr^2} \cdot \frac{\upr}{a}\cdot\frac{\Gamma^2}{\upr} \duprime \\ \lesssim &\frac{1}{|u|}\Gamma^{2}\lesssim 1.
\end{align}
Here we note a crucial point. $\tr\underline{\chi}$ appears in a differentiated (angular) manner and therefore, we can simply replace it by $\widetilde{\tr\underline{\chi}}$ allowing us to obtain an integrable factor $\frac{1}{|u^{'}|^{2}}$. The next term is estimated as follows: 
\begin{align}\nonumber &\int_{u_{\infty}}^{u}\frac{a}{|u^{'}|^{2}}||a^{\frac{I}{2}}\sum_{J_{1}+J_{2}+J_{3}+J_{4}=I-1}\nabla^{J_{1}}(\eta+\underline{\eta})^{J_{2}+1}\nabla^{J_{3}}(\widehat{\underline{\chi}},\tr\underline{\chi})\hnabla^{J_{4}}(\rho^{F},\sigma^{F})||_{L^{2}_{sc}(S_{u^{'},\underline{u}})}du^{'}\\ 
\lesssim &\int_{u_{\infty}}^{u}\frac{a}{|u^{'}|^{3}}\bbGamma^{3}+\int_{u_{\infty}}^{u}\frac{a^{\frac{1}{2}}}{|u^{'}|^{2}}\bbGamma^{3}\lesssim \frac{a}{|u|^{2}}\Gamma^{3}+\frac{a^{\frac{1}{2}}}{|u|}\Gamma^{3}\lesssim 1.
\end{align}
The next term:
\begin{align}\nonumber 
&\int_{u_{\infty}}^{u}\frac{a}{|u^{'}|^{2}}||a^{\frac{I}{2}} \sum_{J_{1}+J_{2}+J_{3}+J_{4}=I}\nabla^{J_{1}}(\eta+\underline{\eta})^{J_{2}}\nabla^{J_{3}}(\eta+\underline{\eta})
\hnabla^{J_{4}}\underline{\alpha}^{F} ||_{L^{2}_{sc}(S_{u^{'},\underline{u}})}du^{'}\\
\lesssim &\int_{U_{\infty}}^{u}\frac{a}{|u^{'}|^{3}}\bbGamma^{2}\lesssim \frac{a}{|u|^{2}}\Gamma^{2}\lesssim 1.
\end{align}
The next term is estimated as follows:
\begin{align}\nonumber 
&\int_{u_{\infty}}^{u}\frac{a}{|u^{'}|^{2}}||a^{\frac{I}{2}} \sum_{J_1+J_2+J_3+J_4=I}\nabla^{J_1}(\eta+\etabar)^{J_2}\nabla^{J_3}(\chibarhat,\hsp \tildetr)\hnabla^{J_4}(\rho^{F},\sigma^{F}) ||_{L^{2}_{sc}(S_{u^{'},\underline{u}})}du^{'}\\
\lesssim &\int_{u_{\infty}}^{u}\frac{a^{\frac{1}{2}}}{|u^{'}|^{2}}\bbGamma^{2}+\int_{u_{\infty}}^{u}\frac{a}{|u^{'}|^{3}}\bbGamma^{3}\lesssim \frac{a^{\frac{1}{2}}}{|u|}\Gamma^{2}+\frac{a}{|u|^{2}}\Gamma^{3}\lesssim 1. 
\end{align}
For the next terms we have:
\begin{align}\nonumber
&\int_{u_{\infty}}^{u}\frac{a}{|u^{'}|^{2}}||a^{\frac{I}{2}} \sum_{J_{1}+J_{2}+J_{3} +J_{4}=I-1}\nabla^{J_{1}}(\eta+\underline{\eta})^{J_{2} +1}\hat{\nabla}^{J_{3}}\tr\underline{\chi}\hnabla^{J_{4}}(\rho^{F},\sigma^{F}) ||_{L^{2}_{sc}(S_{u^{'},\underline{u}})}du^{'}\\ 
\lesssim  &\int_{u_{\infty}}^{u}\frac{a^{\frac{1}{2}}}{|u^{'}|^{2}}\bbGamma^{3}\lesssim \frac{a^{\frac{1}{2}}}{|u|}\Gamma^3 \lesssim 1,
\end{align}
\begin{align}
&\int_{u_{\infty}}^{u}\frac{a}{|u^{'}|^{2}}||a^{\frac{I}{2}} \sum_{J_1+J_2+J_3+J_4=I-1}\nabla^{J_1} (\eta+\etabar)^{J_2+1}\nonumber\nabla^{J_3}(\chibarhat,\tr\chibar)\hnabla^{J_4}(\rho^{F},\sigma^{F}) ||_{L^{2}_{sc}(S_{u^{'},\underline{u}})}du^{'}\\ 
\lesssim  &\int_{u_{\infty}}^{u}\frac{a}{|u^{'}|^{3}}\bbGamma^{3}+ \int_{u_{\infty}}^{u}\frac{a^{\frac{1}{2}}}{|u^{'}|^{2}}\bbGamma^{3}\lesssim \frac{a^{\frac{1}{2}}}{|u|}\Gamma^3\lesssim 1,
\end{align}

\begin{align} 
&\int_{u_{\infty}}^{u}\frac{a}{|u^{'}|^{2}}||a^{\frac{I}{2}} \sum_{J_1+J_2+J_3+J_4=I-1}\nabla^{J_1}(\eta+\etabar)^{J_2}\hnabla^{J_3}\alphabar^F \hnabla^{J_4}(\rho^{F},\sigma^{F})||_{L^{2}_{sc}(S_{u^{'},\underline{u}})}du^{'}\\
\lesssim &\int_{u_{\infty}}^{u}\frac{a^{\frac{3}{2}}}{|u^{'}|^{3}}\bbGamma^{2}+\int_{u_{\infty}}^{u}\frac{a^{2}}{|u^{'}|^{4}}\bbGamma^{3}\lesssim \frac{a^{\frac{3}{2}}}{|u|^{2}}\Gamma^{2}+\frac{a^{2}}{|u|^{3}}\Gamma^{3}\lesssim 1.
\end{align}
Note that in this term, the true non-linear characteristics of the Yang-Mills equations show up. Finally, we have

\begin{align}\nonumber 
&\int_{u_{\infty}}^{u}\frac{a}{|u^{'}|^{2}}||a^{\frac{I}{2}} \sum_{J_{1}+J_{2}+J_{3}+J_{4}+J_{5}\nonumber=I-1}\nabla^{J_{1}}(\eta+\underline{\eta})^{J_{2}}\hnabla^{J_{3}}\alphabar^{F}\hnabla^{J_{4}}(\rho^{F},\sigma^{F})\hnabla^{J_{5}}(\rho^{F},\sigma^{F})||_{L^{2}_{sc}(S_{u^{'},\underline{u}})}du^{'}\\
\lesssim &\int_{u_{\infty}}^{u}\frac{a^{\frac{3}{2}}}{|u^{'}|^{4}}\bbGamma^{3}+\int_{u_{\infty}}^{u}\frac{a^{2}}{|u^{'}|^{5}}\bbGamma^{4}\lesssim\frac{a^{\frac{3}{2}}}{|u|^{3}}\Gamma^{3}+\frac{a^{2}}{|u|^{4}}\Gamma^{5}\lesssim 1.
\end{align}
Putting everything together yields the desired estimate, namely

\[\sum_{0\leq i \leq N+4} \scaletwoSu{(\al\hnabla)^i (\rho^F,\sigma^{F})} \lesssim \underline{\mathbb{YM}}[\underline{\alpha}^{F}]+1.\]
\end{proof}
\begin{remark}\label{remarkF2}
For the non-extremal cases $0\leq i \leq N+3$, we have the improved estimate since the topmost ($N+5$) derivatives of $\underline{\alpha}^{F}$ do not appear in the commuted null evolution equation for $(\rho^{F},\sigma^{F})$
\[ \sum_{0\leq i \leq N+3} \scaletwoSu{(\al\hnabla)^i (\rho^F,\sigma^{F})} \lesssim 1+\frac{a^{\frac{1}{2}}}{|u|}\lesssim 1,\]which implies, by Sobolev embedding, the following $L^{\infty}$ estimate:

\[\sum_{0\leq i \leq N+1} \scaleinfinitySu{(\al\hnabla)^i (\rho^F,\sigma^{F})} \lesssim 1.\]
\end{remark}

\begin{proposition}
Under the assumptions of Theorem \ref{main1} and the bootstrap assumptions \eqref{bootstrap}, there holds 

\[\sum_{0\leq i \leq N+4} \scaletwoSu{(\al\hnabla)^i \underline{\alpha}^{F}} \lesssim \mathbb{YM}[\rho^{F},\sigma^{F}]+\underline{\mathbb{YM}}[\rho^{F},\sigma^{F}]+1.\]
\end{proposition}
\begin{proof}
Let us recall the null evolution equation for $\underline{\alpha}^{F}$:
\begin{eqnarray}
\hnabla_4 \alphabar^F +\frac{1}{2}\tr\chi\alphabar^F = - \hnabla \rho^F + \Hodge{\hnabla}\sigma^F -2 \Hodge{\etabar} \sigma^F - 2 \etabar \rho^F + 2 \omega \alphabar^F - \chibarhat \cdot \alpha^F.
\end{eqnarray}
Commuting with the gauge-covariant operator $\hnabla^{I}$ (we restrict to $0\leq I \leq N+4$ here), using Proposition \ref{3.6}, we get
\begin{align}\nonumber
\hnabla_{4}\hnabla^{I}\underline{\alpha}^{F}=&\hnabla^{I+1}(\rho^{F},\sigma^{F})\\\nonumber+&\sum_{J_{1}+J_{2}+J_{3}=I-1}\nabla^{J_{1}}(\eta+\underline{\eta})^{J_{2}+1}\hnabla^{J_{3}+1}(\rho^{F},\sigma^{F})\\ \nonumber
+&\sum_{J_{1}+J_{2}+J_{3}+J_{4}=I}\nabla^{J_{1}}(\eta+\underline{\eta})^{J_{2}}\nabla^{J_{3}}\tr\chi\hnabla^{J_{4}}\underline{\alpha}^{F}\\ \nonumber+&\sum_{J_{1}+J_{2}+J_{3}+J_{4}=I}\nabla^{J_{1}}(\eta+\underline{\eta})^{J_{2}}\nabla^{J_{3}}\underline{\eta}\hnabla^{J_{4}}(\sigma^{F},\rho^{F})\\ \nonumber
+&\sum_{J_{1}+J_{2}+J_{3}+J_{4}=I}\nabla^{J_{1}}(\eta+\underline{\eta})^{J_{2}}\nabla^{J_{3}}\omega\hnabla^{J_{4}}\underline{\alpha}^{F}\\ \nonumber +&\sum_{J_{1}+J_{2}+J_{3}+J_{4}=I}\nabla^{J_{1}}(\eta+\underline{\eta})^{J_{2}}\nabla^{J_{3}}\widehat{\underline{\chi}}\hnabla^{J_{4}}\alpha^{F}\\ \nonumber
+&\sum_{J_1+J_2+J_3 +J_4=I} \nabla^{J_1}(\eta+\etabar)^{J_2}\nabla^{J_3}(\chihat,\tr\chi)\hnabla^{J_4}\underline{\alpha}^{F} \\ \nonumber +&\sum_{J_1+J_2+J_3+J_4=I-1}\nabla^{J_1} (\eta+\etabar)^{J_2+1}\nabla^{J_3}(\chihat,\tr\chi)\hnabla^{J_4}\underline{\alpha}^{F}\\ \nonumber +& \sum_{J_{1}+J_{2}+J_{3}+J_{4}+J_{5}=I-1}\nabla^{J_{1}}(\eta+\underline{\eta})^{J_{2}}\hnabla^{J_{3}}\alpha^{F}\hnabla^{J_{4}}(\rho^{F},\sigma^{F})\hnabla^{J_{5}}\underline{\alpha}^{F} \\  + &\sum_{J_1+J_2+J_3+J_4=I-1}\nabla^{J_1}(\eta+\etabar)^{J_2}\hnabla^{J_3}\alpha^F \hnabla^{J_4}\underline{\alpha}^{F}.
\end{align}
Now recall that $s_{2}(\underline{\alpha}^{F})=1$ and therefore $s_{2}(\widehat{\nabla}^{I}\underline{\alpha}^{F})=\frac{I}{2}+1$. Similarly $\hnabla_{4}\hnabla^{I}\underline{\alpha}^{F}=\frac{I}{2}+1$. Definition of the scale-invariant norms yields
\begin{eqnarray}
||\hnabla^{I}\underline{\alpha}^{F}||_{\mathcal{L}^{2}_{sc}(S_{u,\underline{u}})}=a^{-\frac{I+2}{2}}|u|^{I+2}||\hnabla^{I}\underline{\alpha}^{F}||_{L^{2}(S_{u,\underline{u}})}.
\end{eqnarray}
Therefore an application of the scale-invariant transport inequality along $e_{4}$ yields
\begin{align}
\nonumber &||(a^{\frac{1}{2}}\hnabla)^{I}\underline{\alpha}^{F}||_{\mathcal{L}^{2}_{sc}(S_{u,\underline{u}})}\\ \nonumber \lesssim &||(a^{\frac{1}{2}}\hnabla)^{I}\underline{\alpha}^{F}||_{L^{2}_{sc}(S_{u,0})}+\int_{0}^{\underline{u}}||a^{\frac{I}{2}}\hnabla_{4}\hnabla^{I}\underline{\alpha}^{F}||_{L^{2}_{sc}(S_{u,\underline{u}^{'}})}d\underline{u}^{'}\\\nonumber 
\lesssim &||(a^{\frac{1}{2}}\hnabla)^{I}\underline{\alpha}^{F}||_{L^{2}_{sc}(S_{u,0})}+\int_{0}^{\underline{u}}||a^{\frac{I}{2}}\hnabla^{I+1}(\rho^{F},\sigma^{F})||_{L^{2}_{sc}(S_{u,\underline{u}^{'}})}d\underline{u}^{'}\\\nonumber
+&\int_{0}^{\underline{u}}||a^{\frac{I}{2}}\sum_{J_{1}+J_{2}+J_{3}=I-1}\nabla^{J_{1}}(\eta+\underline{\eta})^{J_{2}+1}\hnabla^{J_{3}+1}(\rho^{F},\sigma^{F})||_{L^{2}_{sc}(S_{u,\underline{u}^{'}})}d\underline{u}^{'}\\\nonumber 
+&\int_{0}^{\underline{u}}||a^{\frac{I}{2}}\sum_{J_{1}+J_{2}+J_{3}+J_{4}=I}\nabla^{J_{1}}(\eta+\underline{\eta})^{J_{2}}\nabla^{J_{3}}\tr\chi\hnabla^{J_{4}}\underline{\alpha}^{F}||_{L^{2}_{sc}(S_{u,\underline{u}^{'}})}d\underline{u}^{'}\\\nonumber 
+&\int_{0}^{\underline{u}}||a^{\frac{I}{2}}\sum_{J_{1}+J_{2}+J_{3}+J_{4}=I}\nabla^{J_{1}}(\eta+\underline{\eta})^{J_{2}}\nabla^{J_{3}}\underline{\eta}\hnabla^{J_{4}}(\sigma^{F},\rho^{F})||_{L^{2}_{sc}(S_{u,\underline{u}^{'}})}d\underline{u}^{'}\\\nonumber 
+&\int_{0}^{\underline{u}}||a^{\frac{I}{2}}\sum_{J_{1}+J_{2}+J_{3}+J_{4}=I}\nabla^{J_{1}}(\eta+\underline{\eta})^{J_{2}}\nabla^{J_{3}}\omega\hnabla^{J_{4}}\underline{\alpha}^{F}||_{L^{2}_{sc}(S_{u,\underline{u}^{'}})}d\underline{u}^{'}\\\
+&\int_{0}^{\underline{u}}||a^{\frac{I}{2}}\sum_{J_{1}+J_{2}+J_{3}+J_{4}=I}\nabla^{J_{1}}\nonumber(\eta+\underline{\eta})^{J_{2}}\nabla^{J_{3}}\widehat{\underline{\chi}}\hnabla^{J_{4}}\alpha^{F}||_{L^{2}_{sc}(S_{u,\underline{u}^{'}})}d\underline{u}^{'}\\\nonumber 
+&\int_{0}^{\underline{u}}||a^{\frac{I}{2}}\sum_{J_1+J_2+J_3 +J_4=I} \nabla^{J_1}(\eta+\etabar)^{J_2}\nabla^{J_3}(\chihat,\tr\chi)\hnabla^{J_4}\underline{\alpha}^{F}||_{L^{2}_{sc}(S_{u,\underline{u}^{'}})}d\underline{u}^{'}\\\nonumber 
+&\int_{0}^{\underline{u}}||a^{\frac{I}{2}}\sum_{J_1+J_2+J_3+J_4=I-1}\nabla^{J_1} (\eta+\etabar)^{J_2+1}\nabla^{J_3}(\chihat,\tr\chi)\hnabla^{J_4}\underline{\alpha}^{F}||_{L^{2}_{sc}(S_{u,\underline{u}^{'}})}d\underline{u}^{'}\\\nonumber 
+&\int_{0}^{\underline{u}}||a^{\frac{I}{2}}\sum_{J_{1}+J_{2}+J_{3}+J_{4}+J_{5}=I-1}\nabla^{J_{1}}(\eta+\underline{\eta})^{J_{2}}\hnabla^{J_{3}}\alpha^{F}\hnabla^{J_{4}}(\rho^{F},\sigma^{F})\hnabla^{J_{5}}\underline{\alpha}^{F}||_{L^{2}_{sc}(S_{u,\underline{u}^{'}})}d\underline{u}^{'}\\
+&\int_{0}^{\underline{u}}||a^{\frac{I}{2}}\sum_{J_1+J_2+J_3+J_4=I-1}\nabla^{J_1}(\eta+\etabar)^{J_2}\hnabla^{J_3}\alpha^F \hnabla^{J_4}\underline{\alpha}^{F}||_{L^{2}_{sc}(S_{u,\underline{u}^{'}})}d\underline{u}^{'}.
\end{align}
Now we estimate each term separately. For the first term there holds:
\begin{align}
\int_{0}^{\underline{u}}||a^{\frac{I}{2}}\hnabla^{I+1}(\rho^{F},\sigma^{F})||_{L^{2}_{sc}(S_{u,\underline{u}^{'}})}d\underline{u}^{'}\lesssim \mathbb{YM}[\rho^{F},\sigma^{F}].
\end{align}

\noindent The next term is estimated as follows:
\begin{align}\nonumber 
&\int_{0}^{\underline{u}}||a^{\frac{I}{2}}\sum_{J_{1}+J_{2}+J_{3}=I-1}\nabla^{J_{1}}(\eta+\underline{\eta})^{J_{2}+1}\hnabla^{J_{3}+1}(\rho^{F},\sigma^{F})||_{L^{2}_{sc}(S_{u,\underline{u}^{'}})}d\underline{u}^{'}\\
\lesssim &\int_{0}^{\underline{u}}\frac{1}{|u|}\bbGamma^{2}d\underline{u}^{'}\lesssim \frac{\Gamma^{2}}{|u|}|\underline{u}|\lesssim 1.
\end{align}
For the third term we have:
\begin{align} \nonumber
&\int_{0}^{\underline{u}}||a^{\frac{I}{2}}\sum_{J_{1}+J_{2}+J_{3}+J_{4}=I}\nabla^{J_{1}}(\eta+\underline{\eta})^{J_{2}}\nabla^{J_{3}}\tr\chi\hnabla^{J_{4}}\underline{\alpha}^{F}||_{L^{2}_{sc}(S_{u,\underline{u}^{'}})}d\underline{u}^{'}\\
\lesssim &\int_{0}^{\underline{u}}\frac{1}{|u|}\bbGamma^{2}d\underline{u}^{'}\lesssim  \frac{\Gamma^{2}}{|u|}|\underline{u}|\lesssim 1.
\end{align}
For the fourth term we have:
\begin{align}
&\int_{0}^{\underline{u}}||a^{\frac{I}{2}}\sum_{J_{1}+J_{2}+J_{3}+J_{4}=I}\nabla^{J_{1}}(\eta+\underline{\eta})^{J_{2}}\nabla^{J_{3}}\underline{\eta}\hnabla^{J_{4}}(\sigma^{F},\rho^{F})||_{L^{2}_{sc}(S_{u,\underline{u}^{'}})}d\underline{u}^{'}\nonumber\\ 
\lesssim &\int_{0}^{\underline{u}}\frac{1}{|u|}\bbGamma^{2}d\underline{u}^{'}\lesssim  \frac{\Gamma^{2}}{|u|}|\underline{u}|\lesssim 1.
\end{align}
For the fifth term there holds:
\begin{align}\nonumber 
&\int_{0}^{\underline{u}}||a^{\frac{I}{2}}\sum_{J_{1}+J_{2}+J_{3}+J_{4}=I}\nabla^{J_{1}}(\eta+\underline{\eta})^{J_{2}}\nabla^{J_{3}}\omega\hnabla^{J_{4}}\underline{\alpha}^{F}||_{L^{2}_{sc}(S_{u,\underline{u}^{'}})}d\underline{u}^{'}\\
\lesssim &\int_{0}^{\underline{u}}\frac{1}{|u|}\bbGamma^{2}d\underline{u}^{'}\lesssim  \frac{\Gamma^{2}}{|u|}|\underline{u}|\lesssim 1.
\end{align}
For the next term we have:
\begin{align}\nonumber
&\int_{0}^{\underline{u}}||a^{\frac{I}{2}}\sum_{J_{1}+J_{2}+J_{3}+J_{4}=I}\nabla^{J_{1}}\nonumber(\eta+\underline{\eta})^{J_{2}}\nabla^{J_{3}}\widehat{\underline{\chi}}\hnabla^{J_{4}}\alpha^{F}||_{L^{2}_{sc}(S_{u,\underline{u}^{'}})}d\underline{u}^{'}\\ 
\lesssim &\int_{0}^{\underline{u}}\frac{\al}{\lvert u \rvert}\scaletwoSu{(a^{\frac{1}{2}}\nabla)^{I} \chibarhat}a^{-\frac{1}{2}}||(a^{\frac{1}{2}}\hnabla)^{I}\alpha^{F}||_{\mathcal{L}^{2}_{sc}(S_{u,\underline{u}^{'}})}d\underline{u}^{'}\lesssim \underline{\mathbb{YM}}[\rho^{F},\sigma^{F}]+1,
\end{align}
where we have used Proposition \ref{chihats} to estimate the $\widehat{\underline{\chi}}$ term and Proposition \ref{alphaFprop} to estimate the $\alpha^{F}$ term.
For the seventh term:
\begin{align} \nonumber 
&\int_{0}^{\underline{u}}||a^{\frac{I}{2}}\sum_{J_1+J_2+J_3+J_4=I-1}\nabla^{J_1} (\eta+\etabar)^{J_2+1}\nabla^{J_3}(\chihat,\tr\chi)\hnabla^{J_4}\underline{\alpha}^{F}||_{L^{2}_{sc}(S_{u,\underline{u}^{'}})}d\underline{u}^{'}\\
\lesssim &\int_{0}^{\underline{u}}\frac{a^{\frac{1}{2}}}{|u|^{2}}\bbGamma^{3}d\underline{u}^{'}\lesssim \frac{a^{\frac{1}{2}}|\underline{u}|}{|u|^{2}}\Gamma^{3}\lesssim 1.
\end{align}
For the eighth term:
\begin{align}\nonumber
\int_{0}^{\underline{u}}||a^{\frac{I}{2}}\sum_{J_{1}+J_{2}+J_{3}+J_{4}+J_{5}=I-1}\nabla^{J_{1}}(\eta+\underline{\eta})^{J_{2}}\hnabla^{J_{3}}\alpha^{F}\hnabla^{J_{4}}(\rho^{F},\sigma^{F})\hnabla^{J_{5}}\underline{\alpha}^{F}||_{L^{2}_{sc}(S_{u,\underline{u}^{'}})}d\underline{u}^{'}\\ 
\lesssim \int_{0}^{\underline{u}}\frac{a}{|u|^{2}}\bbGamma^{3}d\underline{u}^{'}\lesssim \frac{a|\underline{u}|}{|u|^{2}}\Gamma^{3}\lesssim 1.
\end{align}
Finally, for the ninth term we have:
\begin{align}
\label{eq:impv}\nonumber 
&\int_{0}^{\underline{u}}||a^{\frac{I}{2}}\sum_{J_1+J_2+J_3+J_4=I-1}\nabla^{J_1}(\eta+\etabar)^{J_2}\hnabla^{J_3}\alpha^F \hnabla^{J_4}\underline{\alpha}^{F}||_{L^{2}_{sc}(S_{u,\underline{u}^{'}})}d\underline{u}^{'}\\\nonumber
\lesssim &\int_{0}^{\underline{u}}\frac{a}{|u|}a^{-\frac{1}{2}}||(a^{\frac{1}{2}}\hnabla)^{I-1}\alpha^{F}||_{\mathcal{L}^{2}_{sc}(S_{u,\underline{u}^{'}})}||(a^{\frac{1}{2}}\hnabla)^{I-1}\underline{\alpha}^{F}||_{\mathcal{L}^{2}_{sc}(S_{u,\underline{u}^{'}})}d\underline{u}^{'}+\frac{a^{\frac{3}{2}}}{|u|^{2}}\Gamma^{3}\\
\lesssim &\frac{a}{|u|}\int_{0}^{\underline{u}}||(a^{\frac{1}{2}}\hnabla)^{I-1}\underline{\alpha}^{F}||_{\mathcal{L}^{2}_{sc}(S_{u,\underline{u}^{'}})}d\underline{u}^{'}+1\lesssim\int_{0}^{\underline{u}}||(a^{\frac{1}{2}}\hnabla)^{I-1}\underline{\alpha}^{F}||_{\mathcal{L}^{2}_{sc}(S_{u,\underline{u}^{'}})}d\underline{u}^{'}+1
\end{align}
Here we have made use of the improved bounds from the remark on $\alpha^{F}$ from Proposition \ref{alphaFprop}, since \eqref{eq:impv} does not contain top-order derivatives. This is a crucial point, asserting that the  Yang-Mills non-linearities are under control. Collecting all the terms, we have:
\begin{align} &\sum_{0\leq i \leq N+4} \scaletwoSu{(\al\hnabla)^i \underline{\alpha}^{F}} \\ \lesssim &\mathbb{YM}[\rho^{F},\sigma^{F}]+\underline{\mathbb{YM}}[\rho^{F},\sigma^{F}]+1+\int_{0}^{\underline{u}}||\sum_{0\leq i \leq N+4} \scaletwoSu{(\al\hnabla)^i \underline{\alpha}^{F}}d\underline{u}^{'}.\end{align}An application of Gr\"onwall's inequality yields 
\[\sum_{0\leq i \leq N+4} \scaletwoSu{(\al\hnabla)^i \underline{\alpha}^{F}}\lesssim \mathbb{YM}[\rho^{F},\sigma^{F}]+\underline{\mathbb{YM}}[\rho^{F},\sigma^{F}]+1,\]
since $|\underline{u}|\lesssim 1$. This concludes the proof of the lemma. 
\end{proof}

\begin{proposition}\label{newprop}
Under the assumptions of Theorem \ref{main1} and the bootstrap assumptions \eqref{bootstrap}, there holds 

\[ \frac{1}{\al} \sum_{0\leq i \leq N+2} \scaletwoSu{(\al\hnabla)^i \hnabla_{4}\alpha^{F}} \lesssim 1.\]
\end{proposition}
\begin{proof}
Because of the fact that there is no null Yang-Mills equation for $\hnabla_4 \alphaF$, we commute with $\hnabla_3$. There holds

\begin{align}
    \nonumber \hnabla_3\hnabla_4 \alphaF+\frac{1}{2}\tr\chibar \hnabla_4 \alphaF &=\hsp  \psi_g\hsp \tr\chibar \hsp \alphaF + (\psi_g,\chihat)(\psi_g,\chibarhat)(\alphaF, 
    \mathcal{Y}_{\ubar})+ \nabla (\eta, \etabar) \alphaF +\psi_g \hnabla(\alphaF,\rhoF,\sigmaF)\\ \nonumber &+\hnabla^2 \alphaF + \nabla(\psi_g,\chihat) (\rhoF,\sigmaF) + (\chihat,\psi_g)\hnabla(\rhoF,\sigmaF) + \alphaF (\rhoF,\sigmaF)(\rhoF,\sigmaF) \\ \nonumber &+\alphaF(\rhoF,\sigmaF) + \alpha\hsp  \alphabarF + (\rho,\sigma)\alphaF + \tbeta (\rhoF,\sigmaF) + (\psi_g,\chihat)\hnabla_4 \alphaF \\:&= \sum_{j=1}^{13}T_j.
\end{align}We commute with $i$ angular derivatives, Using \eqref{c2}. We obtain:

\begin{align}
   \nonumber   \hnabla_3 \hnabla^i \hnabla_4 \alphaF + \frac{i+1}{2}\tr\chibar \hnabla^i \hnabla_4 \alphaF&= \sumiF \nablap \nablat \psi_g \nablaf \tr\chibar \hnablaF \alphaF \\ \nonumber &+ \sumiF \nablap \nablat(\psi_g,\chihat)\nablaf (\psi_g,\chibarhat)\hnablaF (\alphaF, \mathcal{Y}_{\ubar}) \\ \nonumber &+ \sumif \nablap \nabla^{i_3+1}(\eta,\etabar) \hnablaf \alphaF \\ \nonumber &+ \sumitm \nablapp \hnablat(\alphaF, \rhoF,\sigmaF) \\ \nonumber &+ \sumit \nablap \hnabla^{i_3+2}\alphaF \\ \nonumber &+ \sumif \nablap \nabla^{i_3+1}(\psi_g,\chihat)\hnablaf(\rhoF,\sigmaF) \\ \nonumber &+    \sumif \nablap \nabla^{i_3}(\psi_g,\chihat)\hnabla^{i_4+1}(\rhoF,\sigmaF) \\ \nonumber &+\sumiF \nablap \hnablat \alphaF \hnablaf (\rhoF,\sigmaF)\hnablaF (\rhoF,\sigmaF) \\ \nonumber &+\sumif \nablap \hnablat \alphaF \hnablaf(\rhoF,\sigmaF) \\ \nonumber &+ \sumif \nablap \nablat \alpha \hnablaf \alphabarF \\ \nonumber &+ \sumif \nablap \nablat (\rho,\sigma) \nablaf \alphaF \\ \nonumber &+ \sumif \nablap \nablat \tbeta \hnablaf (\rhoF, \sigmaF) \\ \nonumber &+ \sumif \nablap \nablat (\psi_g, \chihat)\hnablaf \hnabla_4 \alphaF \\ \nonumber &+ \sumif \nablap \nablat (\chibarhat,\tildetr) \hnablaf \hnabla_4 \alphaF \\ \nonumber &+ \sumifi \nablapp \nabla^{i_3+1}\tr\chibar \hnablaf \hnabla_4 \alphaF \\ \nonumber &+ \sumifim \nablapp \nablat (\chibarhat,\tr\chibar) \hnablaf \hnabla_4 \alphaF \\ \nonumber &+ \sumiFi \nablap \hnablat \alphabarF \hnablaf (\rhoF,\sigmaF) \hnablaF \hnabla_4 \alphaF \\&+ \sumif \nablap \hnablat \alphabarF \hnablaf \hnabla_4 \alphaF := G_i.
\end{align}An application of Proposition \ref{3.6}, expressed in scale-invariant norms, yields

\begin{equation}
    \begin{split}
        \frac{1}{\al}\scaletwoSu{\aln \hnabla_4 \alphaF} \lesssim \frac{1}{\al}\lVert \aln \hnabla_4 \alphaF \rVert_{L^2_{(sc)}(S_{u_{\infty},\ubar})}+\intu \frac{\al}{\upr^2}\scaletwoSuprime{\ali G_i} \duprime.
    \end{split}
\end{equation}We restrict attention to $0\leq i \leq N+2$ and treat each of the 18 terms separately. 

\par\noindent For the first term, we have

\begin{align}
    \nonumber  &\intu \frac{\al}{\upr^2} \scaletwoSuprime{\ali \sumifi \nablapp \nablat \tr\chibar \hnablaf \alphaF} \duprime \\ \lesssim &\intu \frac{\al}{\upr^2}\cdot \frac{\upr^2}{\al}\frac{\Gamma^3}{\upr^2} \duprime \lesssim \frac{\Gamma^3}{\lvert u \rvert} \lesssim 1.
\end{align}
\par\noindent For the second term, we have

\begin{align}
   \nonumber   &\intu \frac{\al}{\upr^2} \scaletwoSuprime{\ali \sumiF \nablap \nablat (\psi_g, \chihat) \nablaf (\psi_g,\chibarhat) \hnablaf(\alphaF, \Yub)} \duprime \\\lesssim &\intu \frac{\al}{\upr^2}\cdot \upr \cdot \al \cdot \frac{\Gamma^3}{\upr^2} \duprime \lesssim \frac{a \Gamma^3}{\lvert u \rvert^2} \lesssim 1.
\end{align}For the third term, there holds
\begin{align}
   \nonumber   &\intu \frac{\al}{\upr^2} \scaletwoSuprime{\ali \sumif \nablap \nabla^{i_3+1} (\eta,\etabar) \hnablaf \alphaF} \duprime \\\lesssim& \intu \frac{\al}{\upr^2}\frac{\Gamma^2}{\upr} \duprime \lesssim \frac{a \Gamma^2}{\lvert u \rvert^2}\lesssim 1.
\end{align}This is because, for $i_3$ up to $N+2$, we can control $\nabla^{i_3+1}(\eta,\etabar)$ in terms of the $\bbGamma$ norm. For the fourth term we have

\begin{align}
  \nonumber    &\intu \frac{\al}{\upr^2} \scaletwoSuprime{\ali \sumitm \nablapp \hnabla^{i_3} (\alphaF, \rhoF,\sigmaF)} \duprime \\ \lesssim& \intu \frac{a}{\upr^2}\cdot \frac{\Gamma^2}{\upr}\duprime \lesssim \frac{a\Gamma^2}{\lvert u \rvert^2}\lesssim 1.
\end{align}For the fifth term, there holds 

\begin{align}
      &\intu \frac{\al}{\upr^2} \scaletwoSuprime{\ali \sumit \nablap \hnabla^{i_3+2} \alphaF} \duprime \lesssim \intu \frac{a}{\upr^2} \frac{\Gamma}{a \upr} \duprime \lesssim 1.
\end{align}Note, crucially, that for  $0\leq i_3 \leq N+2$ we can bound $\hnabla^{i_3+2} \alphaF$ using the $\bbGamma$ norm. For the sixth and seventh terms, there holds

\begin{align}
  \nonumber     & \intu \frac{\al}{\upr^2} \scaletwoSuprime{\ali \sumif \nablap \hnabla^{i_3+1}(\psi_g,\chihat)\hnablaf ( \rhoF,\sigmaF)} \duprime\\+ \nonumber  &\intu \frac{\al}{\upr^2} \scaletwoSuprime{\ali \sumif \nablap \hnabla^{i_3}(\psi_g,\chihat)\hnabla^{i_4+1} ( \rhoF,\sigmaF)} \duprime \\ \lesssim& \intu \frac{a}{\upr^2}\frac{1}{\al}\frac{\Gamma^2}{\upr} \duprime \lesssim 1. 
\end{align}\noindent For the eighth term, we have 

\begin{align}
    &  \nonumber  \intu \frac{\al}{\upr^2} \scaletwoSuprime{\ali \sumiF \nablap \hnablat \alphaF \hnablaf ( \rhoF,\sigmaF)\hnablaF (\rhoF,\sigmaF)} \duprime \\ \lesssim & \intu \frac{\al}{\upr^2}\cdot \al \cdot \frac{\Gamma^3}{\upr^2}\duprime \lesssim \frac{a \Gamma^3}{\lvert u \rvert^3}\lesssim 1.
\end{align}Similarly, we have

\begin{align}
   &  \nonumber  \intu \frac{\al}{\upr^2} \scaletwoSuprime{\ali \sumif \nablap \hnablat \alphaF \hnablaf ( \rhoF,\sigmaF)} \duprime \\ \lesssim & \intu \frac{\al}{\upr^2}\cdot \al \cdot \frac{\Gamma^2}{\upr}\duprime \lesssim \frac{a \Gamma^3}{\lvert u \rvert^2}\lesssim 1.
\end{align}For the tenth and eleventh terms -and for $i$ always up to $N+2$- we have 
\begin{align}
    & \intu \frac{\al}{\upr^2} \scaletwoSuprime{\ali \sumif \nablap \nablat \alpha \hnablaf \alphabarF } \duprime \lesssim \intu \frac{\al}{\upr^2}\cdot \frac{\al \Gamma^2}{\upr} \duprime \lesssim \frac{a \Gamma^2}{\upr^2}\lesssim 1
\end{align}and
\begin{align}
    \intu &\frac{\al}{\upr^2} \scaletwoSuprime{\ali \sumif \nablap \nablat (\rho,\sigma) \hnablaf \alphaF } \duprime \lesssim \intu \frac{\al}{\upr^2}\cdot \frac{\al \Gamma^2}{\upr} \duprime \lesssim 1.
\end{align}For the twelfth term, we have 
\begin{align}
    &  \nonumber \intu \frac{\al}{\upr^2} \scaletwoSuprime{\ali \sumif \nablap \nablat \tbeta \hnablaf (\rhoF,\sigmaF) } \duprime \\ \lesssim & \intu \frac{\al}{\upr^2}\cdot \frac{\Gamma^2}{\upr} \duprime \lesssim \frac{\al \Gamma^2}{\lvert u \rvert^2}\lesssim 1.
\end{align} For the thirteenth and fourteenth terms, we have

\begin{align}
  \nonumber &   \intu \frac{\al}{\upr^2}\scaletwoSuprime{\ali \sumif \nablap \nablat(\psi_g,\chihat,\chibarhat,\tildetr)\hnablaf \hnabla_4 \alphaF} \duprime 
    \\ \lesssim &\intu \frac{\al}{\upr^2}a \frac{\Gamma^2}{\upr} \duprime + \intu \frac{a}{\upr^2} \frac{\upr}{\al}  \frac{\Gamma^2}{\upr} \duprime \lesssim \frac{a^{\frac{3}{2}} \Gamma^2}{\lvert u \rvert^2}+ \frac{\al\Gamma^2}{\lvert u \rvert}\lesssim 1.
\end{align}For the fifteenth and sixteenth terms, we have

\begin{align}
   \nonumber   &\intu \frac{\al}{\upr^2}\scaletwoSuprime{\sumifi\nablapp \nabla^{i_3+1}\tr\chibar \hnablaf \hnabla_4 \alphaF} \duprime \\  \nonumber +&\intu \frac{\al}{\upr^2}\scaletwoSuprime{\sumifim\nablapp \nabla^{i_3}(\chibarhat, \tr\chibar) \hnablaf \hnabla_4 \alphaF} \duprime \\  \lesssim &\intu \frac{a}{\upr^2} \cdot \al \cdot \frac{\upr^2}{a}\frac{\Gamma^3}{\upr^2} \duprime \lesssim \frac{\al\Gamma^3}{\lvert u \rvert}\lesssim 1.
\end{align}Finally, the seventeenth and eighteenth terms are estimated as follows:

\begin{align}
 \nonumber     &\intu \frac{\al}{\upr^2}\scaletwoSuprime{\sumiFi\nablap \hnablat \alphabarF  \hnablaf (\rhoF,\sigmaF) \hnablaF \hnabla_4 \alphaF} \duprime \\  \nonumber  +&\intu \frac{\al}{\upr^2}\scaletwoSuprime{\sumif\nablap \hnablat \alphabarF \hnablaf \hnabla_4 \alphaF} \duprime \\ \lesssim &\intu \frac{a}{\upr^2} \cdot \frac{\Gamma^3}{\upr^2}\duprime + \intu \frac{a}{\upr^2} \frac{\Gamma^2}{\upr}\duprime \lesssim 1.
\end{align}Adding the estimates together, we obtain the desired result.

\end{proof}


\begin{proposition}
Under the assumptions of Theorem \ref{main1} and the bootstrap assumptions \eqref{bootstrap}, there holds 

\[\sum_{ 0\leq i \leq N+2} \frac{a}{\lvert u \rvert}\scaletwoSu{(\al\hnabla)^i \hnabla_{3}\underline{\alpha}^{F}} \lesssim 1. \]
\end{proposition}
\begin{proof}
As with the preceding Proposition, there is no null Yang-Mills equation for $\hnabla_3 \alphabarF$ and so, to get an estimate, we commute with $\hnabla_4$. Schematically, we have 

\begin{align}
   \nonumber   \hnabla_4 \hnabla_3 \alphabarF = &(\psig,\chihat, \chibarhat)(\psig,\chibarhat,\tildetr,\tr\chibar)(\alphaF,\Yub) +\Yub \Yub\Yub +\Yub\Yub+(\psig,\chibarhat,\tildetr,\tr\chibar)\hnabla \Yub \\ +&\nabla(\psig,\chibarhat,\tildetr) \Yub  +(\alphabar, \tbetabar,\rho,\sigma)(\alphaF,\Yub) +(\psig,\chibarhat)\hnabla_3\alphabarF + \hnabla^2\alphabarF.
\end{align}Commuting with $i$ angular derivatives using \eqref{c1}, we have

\begin{align}
   \nonumber   \hnabla_4 \hnabla^i \hnabla_3\alphabarF &= \sumiF \nablap \nablat (\psig,\chihat,\chibarhat)\nablaf(\psig,\chibarhat,\tildetr,\tr\chibar)\hnablaF(\alphaF,\Yub) \\  \nonumber &+ \sumiF\nablap \hnablat \Yub \hnablaf\Yub \hnablaF\Yub + \sumif \nablap \hnablat \Yub \hnablaf \Yub \\  \nonumber &+ \sumif \nablap \nabla^{i_3}(\psig,\chibarhat,\tildetr,\tr\chibar)\hnabla^{i_4+1}\Yub \\  \nonumber &+ \sumif \nablap \nabla^{i_3+1}(\psig,\chibarhat,\tildetr)\hnablaf\Yub \\ \nonumber &+\sumif \nablap \nablat \Psi_{\ubar} \hnablaf (\alphaF,\Yub) \\  \nonumber &+ \sumif \nablap \nablat (\psig,\chibarhat)\hnablaf \hnabla_3\alphabarF + \sumit \nablap \hnabla^{i_3+2}\alphabarF\\  \nonumber  &+  \sumif \nablap \nablat (\chibarhat, \tildetr) \hnablaf \hnabla_3\alphabarF \\  \nonumber &+\sumifi \nablapp \nablat \tr\chibar \hnablaf \hnabla_3 \alphabarF \\  \nonumber  &+\sumifim \nablapp \nablat \tr\chibar \hnablaf \hnabla_3 \alphabarF\\  \nonumber &+\sumiFi \nablap \hnablat \alphabarF \hnablaf (\rhoF,\sigmaF) \hnablaF \hnabla_3\alphabarF \\ &+ \sumifi \nablap \hnablat \alphabarF \hnablaf \hnabla_3 \alphabarF := H_i.
\end{align}Passing to scale-invariant norms, there holds

\begin{align} \frac{a}{\lvert u \rvert} \scaletwoSu{\aln \hnabla_3 \alphabarF} \lesssim \frac{a}{\lvert u \rvert} \intubar \scaletwoSuubarprime{\ali H_i} \dubarprime.\end{align}

From now on, we restrict attention to $0\leq i \leq N+2$. The first term contains a triple anomaly:

\begin{align}
    \nonumber  &\frac{a}{\lvert u \rvert} \intubar \scaletwoSuubarprime{\ali \sumiF \nablap \nablat (\psig,\chihat,\chibarhat)\nablaf(\psig,\chibarhat,\tildetr,\tr\chibar)\hnablaF(\alphaF,\Yub)}\dubarprime \\  \lesssim &\frac{a}{\lvert u \rvert}   \intubar \frac{\lvert u \rvert}{\al}\frac{\lvert u \rvert^2}{a}\al \frac{\Gamma^3}{\lvert u \rvert^2}\dubarprime \lesssim \Gamma[\chibarhat]\Gamma[\tr\chibar]\Gamma[\alphaF].
\end{align}At this point, we can bound the product $\Gamma[\chibarhat]\Gamma[\tr\chibar]\Gamma[\alphaF]$ by $1$, using the improvements obtained  in Propositions \ref{chihats}, \ref{trchibarbound} and \ref{alphaFprop} respectively. 
\vspace{3mm}

For the fourth term, there holds 

\begin{align}
    \nonumber   &\frac{a}{\lvert u \rvert} \intubar \scaletwoSuubarprime{\ali \sumif \nablap \nablat (\psig,\chibarhat,\tildetr, \tr\chibar)\hnabla^{i_4+1} \Yub} \dubarprime \\  \nonumber   \lesssim &\frac{\al}{\lvert u \rvert}   \intubar \scaletwoSuubarprime{(\al)^{i+1} \sum_{i_1+i_2+i_3+i_4+1=i+1} \nablap \nablat (\psig,\chibarhat,\tildetr, \tr\chibar)\hnabla^{i_4+1} \Yub} \dubarprime \\  \lesssim &\frac{\al}{\lvert u \rvert} \intubar \frac{\lvert u \rvert^2}{a} \frac{\Gamma^2}{\lvert u \rvert}\dubarprime \lesssim \frac{\Gamma^2}{\al}\lesssim 1.
\end{align}The rest of the terms are not borderline and can be bounded above by 1 as before. The result follows.

\end{proof}

    \subsection{Estimates on the curvature components}\label{csec}
    \begin{proposition}
Under the assumptions of Theorem \ref{main1} and the bootstrap assumptions \eqref{bootstrap}, there holds 
\[ \sum_{0\leq i \leq N+2} \frac{1}{\al}\scaletwoSu{\aln \alpha} \lesssim 1. \]
\end{proposition}
\begin{proof}
    Recall the Bianchi equation for $\alpha$:
    
    \begin{align}
	      \nonumber   \nabla_3 \alpha + \frac{1}{2}\tr\chibar \alpha = &\nabla \hat{\otimes}\beta + 4 \omegabar \alpha - 3\left(\chihat \rho + \Hodge{\chihat} \sigma \right)+ (\zeta+4\eta)\hat{\otimes}\beta \\ & - D_4 R_{AB} + \frac{1}{2}\left(D_B R_{4A} + D_A R_{4B}\right) + \frac{1}{2}\left(D_4 R_{43} - D_3 R_{44}\right)\gslash_{AB}.
    \end{align}Schematically, the above rewrites as 
    
    \begin{align}
	\nonumber         \nabla_3 \alpha + \frac{1}{2}\tr\chibar \alpha = &\nabla \tbeta+ \psi_g \alpha + \chihat(\rho,\sigma) + \psi_g \tbeta +\Yub \hnabla \alphaF + (\alphaF, \Yub) \hnabla \Yub \\&+ (\psi_g,\chihat,\chibarhat, \tr\chibar) (\alphaF, \rhoF,\sigmaF)(\alphaF, \Yub) + \Yub \hnabla_4 \alphaF.
    \end{align}Commuting with $i$ angular derivatives using \eqref{c2}, we arrive at
    
    \begin{align}
      \nonumber  & \nabla_3 \nabla^i \alpha + \frac{i+1}{2}\tr\chibar \nabla^i\alpha\\  = &\nabla^{i+1}\tbeta  \nonumber + \sumitm \nablapp \nabla^{i_3+1}\tbeta +\sumitm \nablapp \nablat \alpha\\ \nonumber  +& \sumif \nablap \nablat (\psi_g,\chihat) \nablaf (\rho,\sigma, \tbeta) \\ \nonumber +& \sumif \nablap \hnablat \Yub \hnabla^{i_4+1}\alphaF \\ \nonumber+  &\sum_{i_1+i_2+i_3+i_4=i}\nablap \hnablat (\alphaF, \Yub) \hnabla^{i_4+1}\Yub \\ \nonumber +&\sumiF \nablap \nablat(\psi_g,\chihat,\chibarhat,\tr\chibar)\hnablaf(\alphaF, \rhoF,\sigmaF)\hnablaF(\alphaF,\Yub)\\ \nonumber +& \sumif \nablap \hnablat\Yub \hnablaf \hnabla_4\alphaF \ \nonumber\\  \nonumber  + &\sumif \nablap \nablat (\chibarhat,\tildetr) \nablaf \alpha \\ \nonumber +& \sumifi \nablapp \nablat \tr\chibar \nablaf \alpha \\+&\sumifim \nablapp \nablat(\chibarhat,\tr\chibar)\nablaf \alpha:=T_i^1 +\dots + T_i^{11}.
    \end{align}Passing to scale-invariant norms and using the weighted transport equality from Proposition \ref{3.6}, we can estimate as follows.
    
    \begin{align}
    \frac{1}{\al}\scaletwoSu{\aln \alpha} &\lesssim \frac{1}{\al}\lVert \aln \alpha \rVert_{L^2_{(sc)}(S_{u_{\infty}, u})} + \sum_{1\leq j \leq 12} \intu \frac{\al}{\upr^2}\scaletwoSuprime{\ali T_i^{j}}\duprime.
    \end{align}The expressions for $T_i^1, T_i^2, T_i^3, T_i^4, T_i^9, T_i^{10}, T_i^{11}$ do not contain Yang-Mills components and their treatment is exactly the same as in \cite{AnAth}, bounded by $1$. For the rest of the terms, there holds:
    
    \begin{align}
        \intu \frac{\al}{\upr^2}\left(\scaletwoSuprime{\ali T_i^5}+ \scaletwoSuprime{\ali T_i^6}\right) \duprime \lesssim \intu \frac{\al\Gamma^2}{\upr^3}\duprime \lesssim 1.
    \end{align}Moreover, there holds
     \begin{align}
        \intu \frac{\al}{\upr^2}\scaletwoSuprime{\ali T_i^7} \duprime \lesssim \intu \frac{\al\Gamma^3}{\upr^2}\duprime \lesssim 1.
    \end{align}\noindent For the eighth term, we can bound
    
    \begin{align}
        \intu \frac{\al}{\upr^2}\scaletwoSuprime{\ali T_i^8} \duprime \lesssim \intu \frac{a \Gamma^2}{\upr^3}\lesssim 1.
    \end{align}The result follows.
    
\end{proof}
\begin{proposition}\label{psiuprop}
For $\Psi_u$ defined as in Section \ref{Norms}, there holds

\[ \sum_{0\leq i \leq N+3} \scaletwoSu{\aln \Psi_u} \lesssim \mathcal{R}[\alpha] +1.\]
\end{proposition} 

\begin{proof}
Each of the $\Psi_u$ satisfies the following schematic equation:

\be \nabla_4 \Psi_u = \nabla\left(\Psi_u, \alpha\right) + (\psi, \chibarhat) \left(\Psi_u, \alpha\right) + (\alpha_F , \mathcal{Y}_{\ubar})\hnabla (\alpha_F, \mathcal{Y}_{\ubar}) + (\psi,\chihat,\chibarhat,\tr\chibar)(\alphaF, \mathcal{Y}_{\ubar})(\alphaF, \mathcal{Y}_{\ubar}). \ee Commuting with $i$ angular derivatives using \ref{c1}, we obtain

\begin{align}
 \nonumber   \nabla_4 \nabla^i \Psi_u =& \nabla^{i+1}(\Psi_u,\alpha) +\sum_{i_1+i_2+i_3=i-1} \nablapp \nabla^{i_3+1} (\Psi_u,\alpha)\\  \nonumber &+ \sumif \nablap\nablat(\psi,\chibarhat)\nablaf (\Psi_u,\alpha) \\  \nonumber &+ \sumif \nablap\hnablat(\alphaF, \mathcal{Y}_{\ubar}) \hnabla^{i_4+1}(\alphaF, \mathcal{Y}_{\ubar}) \\  \nonumber &+\sumiF \nablap \nablat (\psi,\chihat,\chibarhat,\tr\chibar)\hnablaf(\alphaF, \mathcal{Y}_{\ubar})\hnabla^{i_5}(\alphaF, \mathcal{Y}_{\ubar})\\  \nonumber &+\sumif \nablap \nablat(\chihat,\tr\chi)\nablaf \Psi_u \\  \nonumber &+\sumifim \nablapp\nablat(\chihat,\tr\chi)\nablaf \Psi_u\\&+ \sumiFi \nablap \hnablat \alphaF \hnablaf (\rhoF,\sigmaF)\nabla^{i_5}\Psi_u.
\end{align} Passing to scale-invariant norms and estimating, we have

\begin{align}
 \nonumber     &\hspace{7mm}\scaletwoSu{\aln \Psi_u} \\  \nonumber \lesssim &\intubar \scaletwoSuubarprime{a^{\frac{i}{2}} \nabla^{i+1}\alpha} \dubarprime +\intubar \scaletwoSuubarprime{a^{\frac{i}{2}} \nabla^{i+1}\Psi_u} \dubarprime \\  \nonumber + &\intubar \scaletwoSuubarprime{a^{\frac{i}{2}}\sum_{i_1+i_2+i_3=i-1}\nablapp \nabla^{i_3+1}(\Psi_u,\alpha)}\dubarprime \\  \nonumber +&\intubar \scaletwoSuubarprime{a^{\frac{i}{2}}\sumif\nablap \nablat(\psi_g, \chibarhat)\nablaf(\Psi_u,\alpha)} \dubarprime \\  \nonumber +&\intubar \scaletwoSuubarprime{a^{\frac{i}{2}}\sumif \nablap \hnablat(\alphaF, \mathcal{Y}_{\ubar})\hnabla^{i_4+1}(\alpha^F, \mathcal{Y}_{\ubar}) }\dubarprime \\  \nonumber +&\intubar \scaletwoSuubarprime{a^{\frac{i}{2}}\sumiF \nablap \nablat (\psi,\chihat,\chibarhat,\tr\chibar) \hnabla^{i_4}(\alphaF, \Yub) \hnablaF (\alphaF, \Yub)}\dubarprime\\  \nonumber +& \intubar \scaletwoSuubarprime{a^{\frac{i}{2}}\sumif \nablap \nablat(\chihat,\tr\chi)\nablaf\Psi_u }\dubarprime \\  \nonumber +& \intubar \scaletwoSuubarprime{a^{\frac{i}{2}}\sumifim \nablapp \nablat(\chihat,\tr\chi)\nablaf \Psi_u}\dubarprime \\ +&\intubar \scaletwoSuubarprime{a^{\frac{i}{2}}\sumiFi \nablap \hnablat \alphaF \hnabla^{i_4}(\rhoF, \sigmaF) \hnablaF \Psi_u}\dubarprime
\end{align}We restrict attention to $0\leq i \leq N+3$. For the first term, we have, 

\be \intubar \scaletwoSuubarprime{a^{\frac{i}{2}}\nabla^{i+1}\alpha} \dubarprime \lesssim \left(\intubar \scaletwoSuubarprime{a^{\frac{i}{2}}\nabla^{i+1}\alpha}^2 \dubarprime\right)^{1/2}= \mathcal{R}[\alpha],\ee by using H\"older's inequality. For the second term, since the $\Psi_u$ are regular with respect to scaling, we conclude that 
 \be \intubar \scaletwoSuubarprime{a^{\frac{i}{2}}\nabla^{i+1}\Psi_u} \dubarprime \lesssim \frac{1}{\al} \mathcal{R} \lesssim \frac{1}{\al}R\lesssim 1,\ee by the bootstrap assumptions \eqref{bootstrap}. For the third term, we have $i_3+1 \leq i$, hence everything can be closed using the $\bbGamma$ total norm. We have
 
 \be \intubar \scaletwoSuubarprime{\ali \sumitm \nablapp \nabla^{i_3+1}(\Psi_u, \alpha)} \lesssim \frac{\Gamma^2}{\lvert u \rvert} \lesssim 1.\ee The fourth term can be estimated by \be \intubar \scaletwoSuubarprime{a^{\frac{i}{2}}\sumif\nablap \nablat(\psi_g, \chibarhat)\nablaf(\Psi_u,\alpha)} \dubarprime\lesssim \intubar \frac{\lvert u \rvert}{a}\frac{\Gamma^2}{\lvert u \rvert} \dubarprime \lesssim 1.\ee For the fifth term, we have \begin{align} \nonumber &\intubar \scaletwoSuubarprime{a^{\frac{i}{2}}\sumif \nablap \hnablat(\alphaF, \mathcal{Y}_{\ubar})\hnabla^{i_4+1}(\alpha^F, \mathcal{Y}_{\ubar}) }\dubarprime \\= \nonumber  &\intubar \scaletwoSuubarprime{a^{\frac{i}{2}}\sum_{i_1+i_2+i_3+i_4+1=i+1} \nablap \hnablat(\alphaF, \mathcal{Y}_{\ubar})\hnabla^{i_4+1}(\alpha^F, \mathcal{Y}_{\ubar}) }\dubarprime\\\lesssim &\intubar \frac{1}{\al}\frac{a \Gamma^2}{\lvert u \rvert} \dubarprime \lesssim \frac{\al \Gamma^2}{\lvert u \rvert} \lesssim 1. \end{align}For the sixth term (which is the most borderline one), we have the following estimate (the worst possible scenario is when we have the triple anomaly $\tr\chibar\lvert \alphaF \rvert^2$, all other cases are bounded above by $1$ as before):
 
 \begin{align}
    \nonumber   &\intubar \scaletwoSuubarprime{a^{\frac{i}{2}}\sumiF \nablap \nablat (\psi,\chihat,\chibarhat,\tr\chibar) \hnabla^{i_4}(\alphaF, \Yub) \hnablaF (\alphaF, \Yub)}\dubarprime \\ \lesssim &\intubar \frac{\lvert u \rvert^2}{a}\cdot a \cdot \frac{\bbGamma(\tr\chibar)\bbGamma(\alphaF)\bbGamma(\alphaF)}{\lvert u \rvert^2}\dubarprime  \lesssim \bbGamma(\tr\chibar)\bbGamma(\alphaF)\bbGamma(\alphaF).
 \end{align}Using Remark \ref{remarkF} from Proposition \ref{alphaFprop}, since $0\leq i \leq N+3$, we see an improvement for $\bbGamma(\alphaF),$ whence we can bound $\bbGamma_{\infty}(\alphaF) + \bbGamma_{2}(\alphaF) \lesssim 1$. Moreover, for $\bbGamma(\tr\chi)$, we make use of Proposition \ref{trchibarbound} to bound $\bbGamma(\tr\chibar)\lesssim 1$. For the seventh term, there holds
 
 \begin{align}
     \intubar \scaletwoSuubarprime{\ali \sumif \nablap \nablat(\chihat, \tr\chi)\nablaf \Psi_u} \dubarprime \lesssim \intubar \frac{\al \Gamma^2}{\lvert u \rvert}\dubarprime \lesssim 1.
 \end{align}
 
\noindent For the eighth term, there holds
 \begin{align}
   \intubar \scaletwoSuubarprime{\ali \sumifim \nablapp \nablat(\chihat, \tr\chi)\nablaf \Psi_u} \dubarprime \lesssim \intubar \frac{a^{\frac{3}{2}} \Gamma^3}{\lvert u \rvert^2}\dubarprime \lesssim 1.
\end{align}Finally, for the ninth term, there holds

\begin{align}
    \intubar \scaletwoSuubarprime{\ali \sumiFi \nablap \hnablat\alphaF \hnablaf (\rhoF,\sigmaF)\hnablaF \Psi_u}\dubarprime \lesssim \intubar \frac{a \Gamma^3}{\lvert u \rvert^2}\dubarprime \lesssim 1.
\end{align}
Putting all of the above together, the result follows.
\end{proof}
\begin{proposition}
Under the assumptions of Theorem \ref{main1} and the bootstrap assumptions \eqref{bootstrap}, there holds

\[\sum_{0\leq i \leq N+2} \scaletwoSu{\aln \alphabar}\lesssim 1. \]
\end{proposition}
\begin{proof}
Recall the Bianchi equation for $\alphabar$:
	\begin{equation}
	    \begin{split}
	        \nabla_4 \alphabar + \frac{1}{2}\tr\chi \hsp \alphabar = &-\nabla \hat{\otimes}\betabar +4 \omega\hsp \alphabar -3\left(\chibarhat \rho - \Hodge{\chibarhat}\sigma \right) +\left(\zeta - 4\etabar\right)\hat{\otimes}\betabar \\ &- D_3 R_{AB} + \frac{1}{2}\left(D_A R_{3B} + D_B R_{3A}\right) + \frac{1}{2}\left(D_3 R_{34} - D_4 R_{43}\right) \slashed{g}_{AB}.
	    \end{split}
\end{equation}
Schematically, the above rewrites as

	    \begin{align}
	        \nabla_4\alphabar = \nabla\Psi_u + \psi_g (\alphabar, \tbetabar) +\chibarhat(\rho,\sigma)+ \Yub \hnabla \Yub +  (\psi_g,\chihat,\chibarhat,\tr\chibar)(\alphaF,\Yub)\Yub + \alphaF \hnabla_3 \alphabarF.
	    \end{align}Commuting with $i$ angular derivatives, we arrive at 


\begin{align}
 \nonumber     &\hspace{7mm}\scaletwoSu{\aln \alphabar} \\  \nonumber \lesssim&\intubar \scaletwoSuubarprime{a^{\frac{i}{2}} \nabla^{i+1}\alphabar} \dubarprime +\intubar \scaletwoSuubarprime{a^{\frac{i}{2}} \nabla^{i+1}\Psi_u} \dubarprime \\  \nonumber + &\intubar \scaletwoSuubarprime{a^{\frac{i}{2}}\sum_{i_1+i_2+i_3=i-1}\nablapp \nabla^{i_3+1}\Psi_u}\dubarprime \\ \nonumber +&\intubar \scaletwoSuubarprime{a^{\frac{i}{2}}\sumif\nablap \nablat(\psi_g, \chibarhat)\nablaf(\Psi_{\ubar})} \dubarprime \\ \nonumber +&\intubar \scaletwoSuubarprime{a^{\frac{i}{2}}\sumif \nablap \hnablat \mathcal{Y}_{\ubar}\hnabla^{i_4+1} \mathcal{Y}_{\ubar} }\dubarprime \\ \nonumber +&\intubar \scaletwoSuubarprime{a^{\frac{i}{2}}\sumiF \nablap \nablat (\psi,\chihat,\chibarhat,\tr\chibar) \hnabla^{i_4}(\alphaF, \Yub) \hnablaF  \Yub}\dubarprime\\ \nonumber +&  \intubar \scaletwoSuubarprime{\ali \sumif \nablap \hnablat \alphaF \hnablaf \hnabla_3 \alphabarF} \dubarprime         \\ \nonumber +& \intubar \scaletwoSuubarprime{a^{\frac{i}{2}}\sumif \nablap \nablat(\chihat,\tr\chi)\nablaf\alphabar }\dubarprime \\ \nonumber +& \intubar \scaletwoSuubarprime{a^{\frac{i}{2}}\sumifim \nablapp \nablat(\chihat,\tr\chi)\nablaf \alphabar}\dubarprime \\+&\intubar \scaletwoSuubarprime{a^{\frac{i}{2}}\sumiFi \nablap \hnablat \alphaF \hnabla^{i_4}(\rhoF, \sigmaF) \hnablaF \alphabar}\dubarprime
\end{align}In the above expression, all terms except the seventh one can be bounded above by $1$, in the same way as in the preceding Proposition. For the seventh term, we have

\begin{equation}
    \intubar \scaletwoSuubarprime{\ali \sumif \nablap \hnablat \alphaF \hnablaf \hnabla_3 \alphabarF} \dubarprime \lesssim \intubar \al \frac{\lvert u \rvert}{a} \frac{ \Gamma^2}{\lvert u \rvert}\dubarprime \lesssim 1,
\end{equation}as the $\bbGamma$ norm bootstraps up to $N+2$ angular derivatives of $\hnabla_3 \alphabarF$. The result follows.

\end{proof}
\section{Elliptic Estimates for the top-order derivatives of the Ricci coefficients}
While estimating the energy norms for curvature, we will need control of the $\mathcal{L}^{2}_{sc}(H,\underline{H})$ norms for  $N+5$ derivatives of the Ricci coefficients. We obtain such estimates by means of the elliptic equations. Let us define the following norm on which we must make a bootstrap assumption: 
\begin{align}
 \nonumber \mathcal{O}_{N+5,2}:=&\frac{1}{a^{\frac{1}{2}}}\|(a^{\frac{1}{2}})^{N+4}\nabla^{N+5}\widehat{\chi}\|_{\mathcal{L}^{2}_{sc}(H^{(0,\underline{u})}_{u})}+\|(a^{\frac{1}{2}})^{N+4}\nabla^{N+5}(\tr\chi,\omega)\|_{\mathcal{L}^{2}_{sc}(H^{(0,\underline{u})}_{u})}\\\nonumber 
+&\|(a^{\frac{1}{2}})^{N+4}\nabla^{N+5}\eta\|_{\mathcal{L}^{2}_{sc}(\underline{H}^{(u_{\infty},u)}_{\underline{u}})}+\frac{a}{|u|}\|(a^{\frac{1}{2}})^{N+4}\nabla^{N+5}(\eta,\underline{\eta})\|_{\mathcal{L}^{2}_{sc}(H^{(0,\underline{u})}_{u})}\\\nonumber 
+&\|(a^{\frac{1}{2}})^{N+4}\nabla^{N+5}\underline{\omega}\|_{\mathcal{L}^{2}_{sc}(\underline{H}^{(u_{\infty},u)}_{\underline{u}})}+ \int_{u_{\infty}}^{u}\frac{a^{\frac{3}{2}}}{|u^{'}|^{3}}\|(a^{\frac{1}{2}})^{N+4}\nabla^{N+5}\widehat{\underline{\chi}}\|_{\mathcal{L}^{2}_{sc}(S_{u^{'},\underline{u}})}du^{'}\\
+&\int_{u_{\infty}}^{u}\frac{a^{2}}{|u^{'}|^{3}}\|(a^{\frac{1}{2}})^{N+4}\nabla^{N+5}\tr\underline{\chi}\|_{\mathcal{L}^{2}_{sc}(S_{u^{'},\underline{u}})}du^{'}
\end{align}
The bootstrap assumption reads \begin{equation} \label{ellboot} \mathcal{O}_{N+5,2}\lesssim  \Gammatop \lesssim  a^{1/1000}.\end{equation} Throughout this section, we shall be repeatedly invoking a pivotal Proposition, which lies at the heart of the rationale behind the elliptic estimates. We include a statement and refer the reader to \cite{AnAth} for a proof.

\begin{proposition}
\label{divcurlprop} Under the assumptions of Theorem \ref{main1} and the bootstrap assumptions \eqref{bootstrap}, let $\phi$ be a totally symmetric, $(r+1)-$covariant tensorfield on a two-sphere $(\mathbb{S}^2, \gamma)$ satisfying the following:

\[ \div \phi = f, \hspace{2mm} \curl \phi = g, \hspace{2mm} \tr\hsp \phi = h.\]Then, for $0\leq i \leq N+5$, there holds:

\begin{align*}
\scaletwoSu{\aln \phi} \lesssim \al \sum_{j=0}^{i-1}\scaletwoSu{\aln(f,g)} + \sum_{j=0}^{i-1}\scaletwoSu{\aln(\phi,h)}.
\end{align*}
\end{proposition}
We begin with the estimates for $\chihat, \hsp \tr\chi$.
\begin{proposition}
Under the assumptions of Theorem \ref{main1} and the bootstrap assumptions \eqref{bootstrap} the following estimates hold for the top derivatives of the ricci coefficients $(\tr\chi,\chihat)$
\begin{align}
 \nonumber  a^{-\frac{1}{2}}\|(a^{\frac{1}{2}})^{N+4}\nabla^{N+5}\widehat{\chi}\|_{\mathcal{L}^{2}_{sc}(S_{u,\ubar})}\lesssim 1+\mathcal{R}[\tilde{\beta}]+\mathcal{R}[\alpha]+\mathbb{YM}[\alpha^{F}]+\underline{\mathbb{YM}}[\rho^F, \sigma^F],\\ \nonumber 
\|(a^{\frac{1}{2}})^{N+4}\nabla^{N+5}\tr\chi\|_{\mathcal{L}^{2}_{sc}(S_{u,\ubar})}\lesssim 1+\mathcal{R}[\tilde{\beta}]+\mathcal{R}[\alpha]+\mathbb{YM}[\alpha^{F}]+\underline{\mathbb{YM}}[\rho^F, \sigma^F].
\end{align}
\end{proposition}
\begin{proof}
Recall the null-transport equation for $\tr\chi$ 
\begin{eqnarray}
\nabla_{4}\tr\chi+\frac{1}{2}(\tr\chi)^{2}=-|\chihat|^{2}-|\alpha^{F}|^{2}-2\omega\tr\chi
\end{eqnarray}
and commute with $\hnabla^{I}$ to obtain 
\begin{align}
\nonumber \nabla_4\nabla^{I}\tr\chi=&\sum_{J_{1}+J_{2}+J_{3}+J_{4}=I}\nabla^{J_{1}}(\eta+\etabar)^{J_{2}}\nabla^{J_{3}}\tr\chi\nabla^{J_{4}}\tr\chi\\\nonumber+&\sum_{J_{1}+J_{2}+J_{3}+J_{4}=I}\nabla^{J_{1}}(\eta+\etabar)^{J_{2}}\nabla^{J_{3}}\chihat\nabla^{J_{4}}\chihat\\\nonumber +&\sum_{J_{1}+J_{2}+J_{3}+J_{4}=I}\nabla^{J_{1}}(\eta+\etabar)^{J_{2}}\hnabla^{J_{3}}\alpha^{F}\hnabla^{J_{4}}\alpha^{F}\\ \nonumber +&\sum_{J_{1}+J_{2}+J_{3}+J_{4}=I}\nabla^{J_{1}}(\eta+\etabar)^{J_{2}}\nabla^{J_{3}}\omega\nabla^{J_{4}}\tr\chi\\\nonumber 
+&\sum_{J_{1}+J_{2}+J_{3}+J_{4}=I}\nabla^{J_{1}}(\eta+\etabar)^{J_{2}}\nabla^{J_{3}}(\chihat,\tr\chi)\nabla^{J_{4}}\tr\chi\\ \nonumber +&\sum_{J_{1}+J_{2}+J_{3}+J_{4}=I-1}\nabla^{J_{1}}(\eta+\etabar)^{J_{2}+1}\nabla^{J_{3}}(\chihat,\tr\chi)\nabla^{J_{4}}\tr\chi\\ +&\sum_{J_{1}+J_{2}+J_{3}+J_{4}+J_{5}=I-2}\nabla^{J_{1}}(\eta+\etabar)^{J_{2}}\nabla^{J_{3}}\alpha^{F}\nabla^{J_{4}}(\rho^{F},\sigma^{F})\nabla^{J_{5}+1}\tr\chi.
\end{align}
Applying the transport inequality from Proposition \ref{3.5} with respect to the scale-invariant norm, we have:
\begin{align} \nonumber
&||a^{\frac{N+4}{2}}\nabla^{N+5}\tr\chi||_{L^{2}_{sc}(S_{u,\ubar})}\\ \lesssim  \nonumber&||a^{\frac{N+4}{2}}\nabla^{N+5}\tr\chi||_{L^{2}_{sc}(S_{u,0})}+\int_{0}^{\ubar}||a^{\frac{N+4}{2}}\sum_{J_{1}+J_{2}+J_{3}+J_{4}=N+5}\nabla^{J_{1}}(\eta+\etabar)^{J_{2}}\nabla^{J_{3}}\tr\chi\nabla^{J_{4}}\tr\chi||_{L^{2}_{sc}(S_{u,\ubar^{'}})}\\\nonumber+&\int_{0}^{\ubar}||a^{\frac{N+4}{2}}\sum_{J_{1}+J_{2}+J_{3}+J_{4}=N+5}\nabla^{J_{1}}(\eta+\etabar)^{J_{2}}\nabla^{J_{3}}\chihat\nabla^{J_{4}}\chihat||_{L^{2}_{sc}(S_{u,\ubar^{'}})}\\\nonumber 
+&\int_{0}^{\ubar}||a^{\frac{N+4}{2}}\sum_{J_{1}+J_{2}+J_{3}+J_{4}=N+5}\nabla^{J_{1}}(\eta+\etabar)^{J_{2}}\nabla^{J_{3}}\alpha^{F}\nabla^{J_{4}}\alpha^{F}||_{L^{2}_{sc}(S_{u,\ubar^{'}})}\\\nonumber 
+&\int_{0}^{\ubar}||a^{\frac{N+4}{2}}\sum_{J_{1}+J_{2}+J_{3}+J_{4}=N+5}\nabla^{J_{1}}(\eta+\etabar)^{J_{2}}\nabla^{J_{3}}\omega\nabla^{J_{4}}\tr\chi||_{L^{2}_{sc}(S_{u,\ubar^{'}})}\\\nonumber 
+&\int_{0}^{\ubar}||a^{\frac{N+4}{2}}\sum_{J_{1}+J_{2}+J_{3}+J_{4}=N+5}\nabla^{J_{1}}(\eta+\etabar)^{J_{2}}\nabla^{J_{3}}(\chihat,\tr\chi)\nabla^{J_{4}}\tr\chi||_{L^{2}_{sc}(S_{u,\ubar^{'}})}\\\nonumber 
+&\int_{0}^{\ubar}||a^{\frac{N+4}{2}}\sum_{J_{1}+J_{2}+J_{3}+J_{4}=N+4}\nabla^{J_{1}}(\eta+\etabar)^{J_{2}}\nabla^{J_{3}}(\chihat,\tr\chi)\nabla^{J_{4}}\tr\chi||_{L^{2}_{sc}(S_{u,\ubar^{'}})}\\\nonumber 
+&\int_{0}^{\ubar}||a^{\frac{N+4}{2}}\sum_{J_{1}+J_{2}+J_{3}+J_{4}+J_{5}=N+3}\nabla^{J_{1}}(\eta+\etabar)^{J_{2}}\nabla^{J_{3}}\alpha^{F}\nabla^{J_{4}}(\rho^{F},\sigma^{F})\nabla^{J_{5}+1}\tr\chi||_{L^{2}_{sc}(S_{u,\ubar^{'}})}\\\nonumber 
\lesssim & ||a^{\frac{N+4}{2}}\nabla^{N+5}\tr\chi||_{L^{2}_{sc}(S_{u,0})}
+\frac{a}{|u|}O(\alpha^{F})_{\infty}\int_{0}^{\ubar}a^{-\frac{1}{2}}||a^{\frac{N+4}{2}}\hnabla^{N+4}\alpha^{F}||_{L^{2}_{sc}(S_{u,\ubar^{'}})}d\ubar^{'}\\ +&\frac{a^{\frac{1}{2}}}{|u|}O(\chihat)_{\infty}\int_{0}^{\ubar}||a^{\frac{N+4}{2}}\nabla^{N+5}\chihat||_{L^{2}_{sc}(S_{u,\ubar^{'}})}d\ubar^{'}+1,
\end{align}
where we observe that the most dangerous terms are the ones containing top-order derivatives. The remaining terms are estimated by means of the $L^{2}$ connection and Yang-Mills estimates (Propositions \ref{etabarestimate}-\ref{alphaFprop}). Treating the initial data as $O(1)$, we write 
\begin{align}
\label{eq:chiestimate}\nonumber 
&||a^{\frac{N+4}{2}}\nabla^{N+5}\tr\chi||_{L^{2}_{sc}(S_{u,\ubar})}\\ \lesssim &1+\frac{a^{\frac{1}{2}}}{|u|}O(\chihat)_{\infty}\int_{0}^{\ubar}||a^{\frac{N+4}{2}}\nabla^{N+5}\chihat||_{L^{2}_{sc}(S_{u,\ubar^{'}})}d\ubar^{'}+\frac{a}{|u|}O(\alpha^{F})_{\infty}\mathbb{YM}[\alpha^{F}].
\end{align}
We now need the $L^{2}_{sc}$ estimate for the top derivative of $\chihat$. We utilize the elliptic equation for $\chihat$ to this end. Once we obtain an $L^{2}_{sc}$ estimate for $\chihat$, the idea is to utilize a Gr\"onwall inequality to complete the estimate for $\tr\chi$. Recall the elliptic equation for $\chihat$:
\begin{eqnarray}
\text{div}\hat{\chi}=\frac{1}{2}\nabla \tr\chi-\frac{1}{2}(\eta-\underline{\eta})\cdot(\hat{\chi}-\frac{1}{2}\tr\chi\gamma)-\tilde{\beta}+\frac{1}{2}\mathfrak{T}(e_{4},\cdot).
\end{eqnarray}
Apply the elliptic estimate from Proposition 
\ref{divcurlprop} to obtain
\begin{align}
\label{eq:chihatestimate}
||(a^{\frac{1}{2}})^{N+4}\nabla^{N+5}\chihat||_{L^{2}_{sc}(S_{u,\ubar})}\lesssim &\sum_{I\leq N+5}\frac{1}{a^{\frac{1}{2}}}\nonumber||(a^{\frac{1}{2}}\nabla)^{I}\tr\chi||_{L^{2}_{sc}(S_{u,\ubar})}+\sum_{I\leq N+4}||(a^{\frac{1}{2}}\nabla)^{I}\tilde{\beta}||_{L^{2}_{sc}(S_{u,\ubar})}\\\nonumber 
+&\sum_{I\leq N+4}\sum_{J_{1}+J_{2}= I}||a^{\frac{I}{2}}\nabla^{J_{1}}(\eta,\etabar)\nabla^{J_{2}}(\chihat,\tr\chi)||_{L^{2}_{sc}(S_{u,\ubar})}\\\nonumber 
+&\sum_{I\leq N+4}\sum_{J_{1}+J_{2}=I}||a^{\frac{I}{2}}\hnabla^{J_{1}}\alpha^{F}\cdot\hnabla^{J_{2}}(\rho^{F},\sigma^{F})||_{L^{2}_{sc}(S_{u,\ubar})}\\
+&\sum_{I\leq N+4}\frac{1}{a^{\frac{1}{2}}}||(a^{\frac{1}{2}}\nabla)^{I}\chihat||_{L^{2}_{sc}(S_{u,\ubar})}.
\end{align}
Now, after substituting (\ref{eq:chihatestimate}) into (\ref{eq:chiestimate}),  an application of Cauchy-Schwartz and Gr\"onwall's inequality together with the $L^{2}_{sc}(S_{u,\ubar})$ estimates for the connection and curvature components proved in Section 4,  we obtain 
\begin{align}\nonumber 
||a^{\frac{N+4}{2}}\nabla^{N+5}\tr\chi||_{L^{2}_{sc}(S_{u,\ubar})}\lesssim &(1+\mathcal{R}[\tilde{\beta}]+\mathcal{R}[\alpha]+\mathbb{YM}[\alpha^{F}]+\underline{\mathbb{YM}}[\rho^F, \sigma^F]) e^{|\ubar|\frac{a^{\frac{1}{2}}}{|u|}}\\
\lesssim& 1+\mathcal{R}[\tilde{\beta}]+\mathcal{R}[\alpha]+\mathbb{YM}[\alpha^{F}]+\underline{\mathbb{YM}}[\rho^F, \sigma^F],
\end{align}
since $|\ubar|\lesssim 1$.
Here we have made use of Proposition \ref{trchiprop} to estimate the lower order norm $\sum_{I\leq N+4}\frac{1}{a^{\frac{1}{2}}}\nonumber||(a^{\frac{1}{2}}\nabla)^{I}\tr\chi||_{L^{2}_{sc}(S_{u,\ubar})}$. Using this estimate, we obtain 
\begin{eqnarray}
a^{-\frac{1}{2}}\|(a^{\frac{1}{2}})^{N+4}\nabla^{N+5}\widehat{\chi}\|_{\mathcal{L}^{2}_{sc}(S_{u,\ubar})}\lesssim 1+\mathcal{R}[\tilde{\beta}]+\mathcal{R}[\alpha]+\mathbb{YM}[\alpha^{F}]+\underline{\mathbb{YM}}[\rho^F, \sigma^F].
\end{eqnarray}
This concludes the proof of the lemma.
\end{proof}

\begin{proposition}
\label{omegaestimate}
Under the assumption of Theorem \ref{main1} and the bootstrap assumption \eqref{bootstrap}-\eqref{ellboot}, the following estimates for $\omega$ hold:
\begin{eqnarray}
||a^{\frac{N+4}{2}}\nabla^{N+5}\omega||_{L^{2}_{sc}(H_{u}^{0,\ubar)})}\lesssim 1+\mathcal{R}[\tilde{\beta}].    
\end{eqnarray}
\end{proposition}
\begin{proof}
We first introduce $\omega^{\dag}$, defined as the solution of the following equation 
\begin{eqnarray}
\label{eq:omegadag}
\nabla_{3}\omega^{\dag}=\frac{1}{2}\sigma,
\end{eqnarray}
with $\omegad=0$ on $H_{u_{\infty}}$. We now introduce the new notation $\langle\omega\rangle$ for the pair $(-\omega,\omega^{\dag})$ and define $\kappa$ as follows: 
\begin{eqnarray}
\kappa:=^{*}\mathcal{D}\langle\omega\rangle-\frac{1}{2}\tilde{\beta}=\nabla\omega+^{*}\nabla\omega^{\dag}-\frac{1}{2}\tilde{\beta}.
\end{eqnarray}
We intend to obtain a transport equation for $^{*}\mathcal{D}\langle\omega\rangle$. We have the following commutation relation:
\begin{align} \nonumber
[\nabla_{3},^{*}\mathcal{D}](-\omega,\omega^{\dag})&=\frac{1}{2}\tr\chibar(\nabla\omega+^{*}\nabla\omega^{\dag})+\chibarhat\cdot(-\nabla\omega+^{*}\nabla\omega^{\dag})\\+&\frac{1}{2}(\eta+\etabar)\nabla_{3}\omega +\frac{1}{2}(^{*}\eta+^{*}\etabar)\nabla_{3}\omega^{\dag}.
\end{align}
A $\nabla_{3}$ equation for $\omega^{\dag}$ (\ref{eq:omegadag}) has been defined. The $\nabla_{3}$ equation for $\omega$ reads 
\begin{eqnarray}
\nabla_{3}\omega=2\omega\underline{\omega}+\frac{3}{4}|\eta-\underline{\eta}|^{2}+\frac{1}{4}(\eta-\underline{\eta})\cdot(\eta+\underline{\eta})-\frac{1}{8}|\eta+\underline{\eta}|^{2}\nonumber+\frac{1}{2}\rho+\frac{1}{4}\mathfrak{T}_{43}
\end{eqnarray}
which we can write in the following schematic notation 
\begin{align}
\nabla_{3}\omega=\frac{1}{2}\rho+\psi_g\psi_g+|(\rho^{F},\sigma^{F})|^{2}.
\end{align}
 We can further denote the term $\psi_g\psi_g+|(\rho^{F},\sigma^{F})|^{2}$ by $F$. An explicit computation yields 
\begin{align}\nonumber 
\nabla_{3}~^{*}\mathcal{D}\langle\omega\rangle=&\frac{1}{2}(\nabla\rho+^{*}\nabla\sigma)\nonumber+\nabla F-\frac{1}{2}\tr\chibar ^{*}\mathcal{D}\langle\omega\rangle+\chibarhat\cdot(-\nabla\omega+^{*}\nabla\omega^{\dag})\\ -&\frac{1}{2}(\eta+\etabar)(\frac{1}{2}\rho+F) +\frac{1}{4}(^{*}\eta+^{*}\etabar)\sigma.
\end{align}
This transport equation contains derivatives of the Weyl curvatures $\rho$ and $\sigma$, which do not allows us to close the estimates on account of regularity. This is precisely the reason why a modified variable $\kappa$ was constructed. Indeed, note the following: \begin{eqnarray}
\nabla_{3}\kappa=\nabla_{3}~^{*}\mathcal{D}\langle\omega\rangle-\frac{1}{2}\nabla_{3}\tilde{\beta},
\end{eqnarray}
and therefore $\frac{1}{2}(\nabla\rho+^{*}\nabla\sigma)$ in $\nabla_{3}~^{*}\mathcal{D}_{1}\langle\omega\rangle$ is cancelled point-wise by $-\frac{1}{2}(\nabla\rho+^{*}\nabla\sigma)$ in $\frac{1}{2}\nabla_{3}\tilde{\beta}$. Now we denote the Weyl curvature components $(\rho,\sigma,\tilde{\beta})$ by $\Psi$ and Yang-Mills components $(\rho^{F},\sigma^{F})$ by $\Upsilon$. Collecting all the terms, we obtain the following evolution equation for $\kappa$:
\begin{align} \nonumber
&\nabla_{3}\kappa+\frac{1}{2}\tr\chibar\kappa=(\psi_g,\chihat)\Psi+(\psi_g,\chibarhat)\nabla\psi_g+\Upsilon\cdot\hnabla(\Upsilon,\alpha^{F})+(\alpha_{F},\Upsilon)\cdot\hnabla\Upsilon\\
+&(\psi_g,\chibarhat,\tr\chibar,\chihat)\cdot(\alpha_{F},\Upsilon)\cdot\Upsilon+\psi_g\psi_g\psi_g.
\end{align}
Using the commutation lemma (\ref{commutation}), we obtain for $I\leq N+4$ (since $\kappa\sim \nabla\omega+\tilde{\beta}$, we only need to estimate up to $N+4$ derivatives of $\kappa$):
\begin{align}\nonumber 
\nabla_{3}\nabla^{I}\kappa+\frac{I+1}{2}\tr\chibar\nabla^{I}\kappa &=\sum_{J_{1}+J_{2}+J_{3}+J_{4}=I}\nabla^{J_{1}}\psi_g^{J_{2}}\nabla^{J_{3}}(\psi_g,\chihat)\nabla^{J_{4}}\Psi\\\nonumber 
+&\sum_{J_{1}+J_{2}+J_{3}+J_{4}=I}\nabla^{J_{1}}\psi_g^{J_{2}}\nabla^{J_{3}}(\psi_g,\chibarhat)\nabla^{J_{4}+1}\psi_g\\\nonumber+&\sum_{J_{1}+J_{2}+J_{3}+J_{4}=I}\nabla^{J_{1}}\psi_g^{J_{2}}\nabla^{J_{3}}\Upsilon\nabla^{J_{4}+1}(\Upsilon,\alpha_{F})\\\nonumber 
+&\sum_{J_{1}+J_{2}+J_{3}+J_{4}=I}\nabla^{J_{1}}\psi_g^{J_{2}}\hnabla^{J_{3}}(\alpha_{F},\Upsilon)\cdot\hnabla^{J_{4}+1}\Upsilon\\\nonumber+&\sum_{J_{1}+J_{2}+J_{3}+J_{4}+J_{5}=I}\nabla^{J_{1}}\psi_g^{J_{2}}\nabla^{J_{3}}(\psi_g,\chibarhat,\tr\chibar,\chihat)\hnabla^{J_{4}}(\alpha_{F},\Upsilon)\cdot\hnabla^{J_{5}}\Upsilon\\\nonumber 
+&\sum_{J_{1}+J_{2}+J_{3}+J_{4}+J_{5}=I}\nabla^{J_{1}}\psi_g^{J_{2}}\nabla^{J_{3}}\psi_g\nabla^{J_{4}}\psi_g\nabla^{J_{5}}\psi_g\\\nonumber+&\sum_{J_{1}+J_{2}+J_{3}+J_{4}=I}\nabla^{J_{1}}\psi_g^{J_{2}}\nabla^{J_{3}}(\psi_g,\chibarhat,\widetilde{\tr\chibar})\nabla^{J_{4}}\kappa\\
+&\sum_{J_{1}+J_{2}+J_{3}+J_{4}+1=I}\nabla^{J_{1}}\psi_g^{J_{2}+1}\nabla^{J_{3}}\tr\chibar\nabla^{J_{4}}\kappa:=\mathcal{F}_{1}
\end{align}
Applying the weighted transort inequality from Proposition \ref{3.6} on $\kappa$ with $\lambda_{0}=\frac{I+1}{2}$, we get
\begin{eqnarray}
|u|^{I}||\nabla^{I}\kappa||_{L^{2}(S_{u,\ubar})}\lesssim |u_{\infty}|^{I}||\nabla^{I}\kappa||_{L^{2}(S_{u_{\infty},\ubar})}+\int_{u_{\infty}}^{u}|u^{'}|^{I}||\mathcal{F}_{1}||_{L^{2}(S_{u^{'},\ubar})}\hsp du^{'}.
\end{eqnarray}
We want to convert this to its scale-invariant form. Noting $s_{2}(\kappa)=s_{2}(\nabla\omega,\tilde{\beta})=\frac{1}{2}$, we have $s_{2}(\nabla^{I}\kappa)=\frac{I+1}{2}$. Therefore 
\begin{eqnarray}
||(a^{\frac{1}{2}}\nabla)^{I}\kappa||_{L^{2}_{sc}(S_{u,\ubar})}=a^{\frac{I-1}{2}}|u|^{I+1}||\nabla^{I}\kappa||_{L^{2}(u,\ubar)}
\end{eqnarray}
and hence
\begin{align} \nonumber
||(a^{\frac{1}{2}}\nabla)^{I}\kappa||_{L^{2}_{sc}(S_{u,\ubar})}\lesssim &a^{\frac{I-1}{2}}|u|\left(|u_{\infty}|^{I}||\nabla^{I}\kappa||_{L^{2}(S_{u_{\infty},\ubar})}+\int_{u_{\infty}}^{u}|u^{'}|^{I}||\mathcal{F}_{1}||_{L^{2}(S_{u^{'},\ubar})}du^{'}\right)\\\nonumber 
\lesssim & a^{\frac{I-1}{2}}|u_{\infty}|^{I+1}||\nabla^{I}\kappa||_{L^{2}(S_{u_{\infty},\ubar})}+a^{\frac{I-1}{2}}\int_{u_{\infty}}^{u}|u^{'}|^{I+1}||\mathcal{F}_{1}||_{L^{2}(S_{u^{'},\ubar})}du^{'}\\
\lesssim &||(a^{\frac{1}{2}}\nabla)^{I}\kappa||_{L^{2}_{sc}(S_{u_{\infty},\ubar})}+\int_{u_{\infty}}^{u}\frac{a}{|u^{'}|^{2}}||a^{\frac{I}{2}}\mathcal{F}_{1}||_{L^{2}_{sc}(S_{u^{'},u})}du^{'},
\end{align}
where we have used the fact that $|u^{'}|\in [u_{\infty},u]$ and $|u|\leq |u|_{\infty}$. We now use the expression for $\mathcal{F}_{1}$ to estimate term by term:  
\begin{align}
\nonumber &||(a^{\frac{1}{2}}\nabla)^{I}\kappa||_{L^{2}_{sc}(S_{u,\ubar})}\\ \lesssim & \nonumber ||(a^{\frac{1}{2}}\nabla)^{I}\kappa||_{L^{2}_{sc}(S_{u_{\infty},\ubar})}\ \\ \nonumber 
+&\int_{u_{\infty}}^{u}\frac{a}{|u^{'}|^{2}}||a^{\frac{I}{2}}\sum_{J_{1}+J_{2}+J_{3}+J_{4}+J_{5}+J_{6}=I}\nabla^{J_{1}}\psi_g^{J_{2}}\nabla^{J_{3}}\psi_g^{J_{4}}\nabla^{J_{5}}(\psi_g,\chihat)\nabla^{J_{6}}\Psi||_{L^{2}_{sc}(S_{u^{'},\ubar})}\\\nonumber 
+&\int_{u_{\infty}}^{u}\frac{a}{|u^{'}|^{2}}||a^{\frac{I}{2}}\sum_{J_{1}+J_{2}+J_{3}+J_{4}=I}\nabla^{J_{1}}\psi_g^{J_{2}}\nabla^{J_{3}}(\psi_g,\chibarhat)\nabla^{J_{4}+1}\psi_g||_{L^{2}_{sc}(S_{u^{'},\ubar})}\\\nonumber 
+&\int_{u_{\infty}}^{u}\frac{a}{|u^{'}|^{2}}||a^{\frac{I}{2}}\sum_{J_{1}+J_{2}+J_{3}+J_{4}=I}\nabla^{J_{1}}\psi_g^{J_{2}}\nabla^{J_{3}}\Upsilon\nabla^{J_{4}+1}(\Upsilon,\alpha_{F})||_{L^{2}_{sc}(S_{u^{'},\ubar})}\\\nonumber 
+&\int_{u_{\infty}}^{u}\frac{a}{|u^{'}|^{2}}||a^{\frac{I}{2}}\sum_{J_{1}+J_{2}+J_{3}+J_{4}=I}\nabla^{J_{1}}\psi_g^{J_{2}}\hnabla^{J_{3}}(\alpha_{F},\Upsilon)\cdot\hnabla^{J_{4}+1}\Upsilon||_{L^{2}_{sc}(S_{u^{'},\ubar})}\\\nonumber 
+&\int_{u_{\infty}}^{u}\frac{a}{|u^{'}|^{2}}||a^{\frac{I}{2}}\sum_{J_{1}+J_{2}+J_{3}+J_{4}+J_{5}=I}\nabla^{J_{1}}\psi_g^{J_{2}}\nabla^{J_{3}}(\psi_g,\chibarhat,\tr\chibar,\chihat)\hnabla^{J_{4}}(\alpha_{F},\Upsilon)\cdot\hnabla^{J_{5}}\Upsilon||_{L^{2}_{sc}(S_{u^{'},\ubar})}\\\nonumber 
+&\int_{u_{\infty}}^{u}\frac{a}{|u^{'}|^{2}}||a^{\frac{I}{2}}\sum_{J_{1}+J_{2}+J_{3}+J_{4}+J_{5}=I}\nabla^{J_{1}}\psi_g^{J_{2}}\nabla^{J_{3}}\psi_g\nabla^{J_{4}}\psi_g\nabla^{J_{5}}\psi_g||_{L^{2}_{sc}(S_{u^{'},\ubar})}\\\nonumber 
+&\int_{u_{\infty}}^{u}\frac{a}{|u^{'}|^{2}}||a^{\frac{I}{2}}\sum_{J_{1}+J_{2}+J_{3}+J_{4}=I}\nabla^{J_{1}}\psi_g^{J_{2}}\nabla^{J_{3}}(\psi_g,\chibarhat,\widetilde{\tr\chibar})\nabla^{J_{4}}\kappa||_{L^{2}_{sc}(S_{u^{'},\ubar})}\\
+&\int_{u_{\infty}}^{u}\frac{a}{|u^{'}|^{2}}||a^{\frac{I}{2}}\sum_{J_{1}+J_{2}+J_{3}+J_{4}+1=I}\nabla^{J_{1}}\psi_g^{J_{2}+1}\nabla^{J_{3}}\tr\chibar\nabla^{J_{4}}\kappa:=\mathcal{F}_{1}||_{L^{2}_{sc}(S_{u^{'},\ubar})}.
\end{align}
Integrating along $\ubar$, we have:
\begin{align} \nonumber 
&\int_{0}^{\ubar}||(a^{\frac{1}{2}}\nabla)^{I}\kappa||_{L^{2}_{sc}(S_{u,\ubar})}\\ \lesssim \nonumber  &\int_{0}^{\ubar}||(a^{\frac{1}{2}}\nabla)^{I}\kappa||_{L^{2}_{sc}(S_{u_{\infty},\ubar})}\\\nonumber 
+&\int_{0}^{\ubar}\int_{u_{\infty}}^{u}\frac{a}{|u^{'}|^{2}}||a^{\frac{I}{2}}\sum_{J_{1}+J_{2}+J_{3}+J_{4}+J_{5}+J_{6}=I}\nabla^{J_{1}}\psi_g^{J_{2}}\nabla^{J_{3}}\psi_g^{J_{4}}\nabla^{J_{5}}(\psi_g,\chihat)\nabla^{J_{6}}\Psi||_{L^{2}_{sc}(S_{u^{'},\ubar})}\\\nonumber 
+&\int_{0}^{\ubar}\int_{u_{\infty}}^{u}\frac{a}{|u^{'}|^{2}}||a^{\frac{I}{2}}\sum_{J_{1}+J_{2}+J_{3}+J_{4}=I}\nabla^{J_{1}}\psi_g^{J_{2}}\nabla^{J_{3}}(\psi_g,\chibarhat)\nabla^{J_{4}+1}\psi_g||_{L^{2}_{sc}(S_{u^{'},\ubar})}\\\nonumber 
+&\int_{0}^{\ubar}\int_{u_{\infty}}^{u}\frac{a}{|u^{'}|^{2}}||a^{\frac{I}{2}}\sum_{J_{1}+J_{2}+J_{3}+J_{4}=I}\nabla^{J_{1}}\psi_g^{J_{2}}\nabla^{J_{3}}\Upsilon\nabla^{J_{4}+1}(\Upsilon,\alpha_{F})||_{L^{2}_{sc}(S_{u^{'},\ubar})}\\\nonumber 
+&\int_{0}^{\ubar}\int_{u_{\infty}}^{u}\frac{a}{|u^{'}|^{2}}||a^{\frac{I}{2}}\sum_{J_{1}+J_{2}+J_{3}+J_{4}=I}\nabla^{J_{1}}\psi_g^{J_{2}}\hnabla^{J_{3}}(\alpha_{F},\Upsilon)\cdot\hnabla^{J_{4}+1}\Upsilon||_{L^{2}_{sc}(S_{u^{'},\ubar})}\\\nonumber 
+&\int_{0}^{\ubar}\int_{u_{\infty}}^{u}\frac{a}{|u^{'}|^{2}}||a^{\frac{I}{2}}\sum_{J_{1}+J_{2}+J_{3}+J_{4}+J_{5}=I}\nabla^{J_{1}}\psi_g^{J_{2}}\nabla^{J_{3}}(\psi_g,\chibarhat,\tr\chibar,\chihat)\hnabla^{J_{4}}(\alpha_{F},\Upsilon)\cdot\hnabla^{J_{5}}\Upsilon||_{L^{2}_{sc}(S_{u^{'},\ubar})}\\\nonumber 
+&\int_{0}^{\ubar}\int_{u_{\infty}}^{u}\frac{a}{|u^{'}|^{2}}||a^{\frac{I}{2}}\sum_{J_{1}+J_{2}+J_{3}+J_{4}+J_{5}=I}\nabla^{J_{1}}\psi_g^{J_{2}}\nabla^{J_{3}}\psi_g\nabla^{J_{4}}\psi_g\nabla^{J_{5}}\psi_g||_{L^{2}_{sc}(S_{u^{'},\ubar})}\\\nonumber 
+&\int_{0}^{\ubar}\int_{u_{\infty}}^{u}\frac{a}{|u^{'}|^{2}}||a^{\frac{I}{2}}\sum_{J_{1}+J_{2}+J_{3}+J_{4}=I}\nabla^{J_{1}}\psi_g^{J_{2}}\nabla^{J_{3}}(\psi_g,\chibarhat,\widetilde{\tr\chibar})\nabla^{J_{4}}\kappa||_{L^{2}_{sc}(S_{u^{'},\ubar})}\\ 
+&\int_{0}^{\ubar}\int_{u_{\infty}}^{u}\frac{a}{|u^{'}|^{2}}||a^{\frac{I}{2}}\sum_{J_{1}+J_{2}+J_{3}+J_{4}+1=I}\nabla^{J_{1}}\psi_g^{J_{2}+1}\nabla^{J_{3}}\tr\chibar\nabla^{J_{4}}\kappa:=\mathcal{F}_{1}||_{L^{2}_{sc}(S_{u^{'},\ubar})}.
\end{align}

Now we estimate term by term, as follows (we shall often omit, when there is no confusion, the $\duprime, \dubarprime$ at the end of integrals, for brevity): 
\begin{align}  \nonumber
&\int_{0}^{\ubar}||(a^{\frac{1}{2}}\nabla)^{I}\kappa||_{L^{2}_{sc}(S_{u_{\infty},\ubar^{'}})}d\ubar^{'}\\ \lesssim &\int_{0}^{\ubar}||(a^{\frac{1}{2}}\nabla)^{I}\widetilde{\beta}||_{L^{2}_{sc}(S_{u_{\infty},\ubar^{'}})}d\ubar^{'}\nonumber+\int_{0}^{\ubar}\frac{1}{a^{\frac{1}{2}}}||(a^{\frac{1}{2}}\nabla)^{I+1}\omega||_{L^{2}_{sc}(S_{u_{\infty},\ubar^{'}})}d\ubar^{'}\\
\lesssim& \mathcal{R}[\tilde{\beta}]+\mathcal{I}_{0}\lesssim 1+\mathcal{R}[\widetilde{\beta}],
\end{align}
by Cauchy-Schwartz, where $I\leq N+5$. The next terms read
\begin{align}
\int_{0}^{\ubar}\int_{u_{\infty}}^{u}\frac{a}{|u^{'}|^{2}}||a^{\frac{I}{2}}\sum_{J_{1}+J_{2}+J_{3}+J_{4}+J_{5}+J_{6}=I}\nabla^{J_{1}}\psi_g^{J_{2}}\nabla^{J_{3}}\psi_g^{J_{4}}\nabla^{J_{5}}(\psi_g,\chihat)\nabla^{J_{6}}\Psi||_{L^{2}_{sc}(S_{u^{'},\ubar})}
\lesssim \frac{a^{\frac{3}{2}}}{|u|^{2}}\mathcal{R},
\end{align}
\begin{align}
\int_{0}^{\ubar}\int_{u_{\infty}}^{u}\frac{a}{|u^{'}|^{2}}||a^{\frac{I}{2}}\sum_{J_{1}+J_{2}+J_{3}+J_{4}=I}\nabla^{J_{1}}\psi_g^{J_{2}}\nabla^{J_{3}}(\psi_g,\chibarhat)\nabla^{J_{4}+1}\psi_g||_{L^{2}_{sc}(S_{u^{'},\ubar})}\lesssim &\nonumber \frac{a^{\frac{1}{2}}}{|u|}\Gamma_{top}+\frac{a^{\frac{3}{2}}}{|u|^{2}}\Gamma_{top}\\  \lesssim &\frac{a^{\frac{1}{2}+\frac{1}{1000}}}{|u|}+\frac{a^{\frac{3}{2}+\frac{1}{1000}}}{|u|^{2}}\lesssim 1,
\end{align}
since $|u|<a$. Continuing, we have:
\begin{align} \nonumber 
&\int_{0}^{\ubar}\int_{u_{\infty}}^{u}\frac{a}{|u^{'}|^{2}}||a^{\frac{I}{2}}\sum_{J_{1}+J_{2}+J_{3}+J_{4}=I}\nabla^{J_{1}}\psi_g^{J_{2}}\nabla^{J_{3}}\Upsilon\nabla^{J_{4}+1}(\Upsilon,\alpha_{F})||_{L^{2}_{sc}(S_{u^{'},\ubar})}\\
\lesssim &\frac{a}{|u|^{2}}\mathbb{YM}[\Upsilon]+\frac{a^{\frac{3}{2}}}{|u|^{2}}\mathbb{YM}[\alpha^{F}],
\end{align}
\begin{align}
\int_{0}^{\ubar}\int_{u_{\infty}}^{u}\frac{a}{|u^{'}|^{2}}||a^{\frac{I}{2}}\sum_{J_{1}+J_{2}+J_{3}+J_{4}=I}\nabla^{J_{1}}\psi_g^{J_{2}}\hnabla^{J_{3}}(\alpha_{F},\Upsilon)\cdot\hnabla^{J_{4}+1}\Upsilon||_{L^{2}_{sc}(S_{u^{'},\ubar})}
\lesssim \frac{a^{\frac{3}{2}}}{|u|^{2}}\mathbb{YM}[\Upsilon],
\end{align}
\begin{align}\nonumber 
&\int_{0}^{\ubar}\int_{u_{\infty}}^{u}\frac{a}{|u^{'}|^{2}}||a^{\frac{I}{2}}\sum_{J_{1}+J_{2}+J_{3}+J_{4}+J_{5}=I}\nabla^{J_{1}}\psi_g^{J_{2}}\nabla^{J_{3}}(\psi_g,\chibarhat,\tr\chibar,\chihat)\hnabla^{J_{4}}(\alpha_{F},\Upsilon)\cdot\hnabla^{J_{5}}\Upsilon||_{L^{2}_{sc}(S_{u^{'},\ubar})}\\
\lesssim &\frac{a^{\frac{1}{2}}}{|u|}\Gamma\lesssim 1,
\end{align}
\begin{align} \int_{0}^{\ubar}\int_{u_{\infty}}^{u}\frac{a}{|u^{'}|^{2}}||a^{\frac{I}{2}}\sum_{J_{1}+J_{2}+J_{3}+J_{4}+J_{5}=I}\nabla^{J_{1}}\psi_g^{J_{2}}\nabla^{J_{3}}\psi_g\nabla^{J_{4}}\psi_g\nabla^{J_{5}}\psi_g||_{L^{2}_{sc}(S_{u^{'},\ubar})}\lesssim \frac{a}{|u|^{3}}\Gamma\lesssim 1,
\end{align}
\begin{align}
& \nonumber \int_{0}^{\ubar}\int_{u_{\infty}}^{u}\frac{a}{|u^{'}|^{2}}||a^{\frac{I}{2}}\sum_{J_{1}+J_{2}+J_{3}+J_{4}=I}\nabla^{J_{1}}\psi_g^{J_{2}}\nabla^{J_{3}}(\psi_g,\chibarhat,\widetilde{\tr\chibar})\nabla^{J_{4}}\kappa||_{L^{2}_{sc}(S_{u^{'},\ubar})}\\ \lesssim &\int_{u_{\infty}}^{u}\frac{a^{\frac{1}{2}}}{|u^{'}|^{2}}\left(\int_{0}^{\ubar}||a^{\frac{I}{2}}\nabla^{I}\kappa||_{L^{2}_{sc}(S_{u^{'},\ubar^{'}})}d\ubar^{'}\right)du^{'},
\end{align}while for the final term we have 
\begin{align} \nonumber 
&\int_{0}^{\ubar}\int_{u_{\infty}}^{u}\frac{a}{|u^{'}|^{2}}||a^{\frac{I}{2}}\sum_{J_{1}+J_{2}+J_{3}+J_{4}+1=I}\nabla^{J_{1}}\psi_g^{J_{2}+1}\nabla^{J_{3}}\tr\chibar\nabla^{J_{4}}\kappa||_{L^{2}_{sc}(S_{u^{'},\ubar})}\\
\lesssim  &\int_{u_{\infty}}^{u}\frac{1}{|u^{'}|^{2}}\left(\int_{0}^{\ubar}||a^{\frac{I}{2}}\nabla^{I}\kappa||_{L^{2}_{sc}(S_{u^{'},\ubar^{'}})}d\ubar^{'}\right)du^{'}.
\end{align}
A collection of all the terms yields 
\begin{align} \nonumber 
\int_{0}^{\ubar}||(a^{\frac{1}{2}}\nabla)^{I}\kappa||_{L^{2}_{sc}(S_{u,\ubar})}d\ubar^{'}\lesssim &1+\mathcal{R}[\widetilde{\beta}]+\int_{u_{\infty}}^{u}\frac{a^{\frac{1}{2}}}{|u^{'}|^{2}}\left(\int_{0}^{\ubar}||a^{\frac{I}{2}}\nabla^{I}\kappa||_{L^{2}_{sc}(S_{u^{'},\ubar^{'}})}d\ubar^{'}\right)du^{'}\\
+&\int_{u_{\infty}}^{u}\frac{1}{|u^{'}|^{2}}\left(\int_{0}^{\ubar}||a^{\frac{I}{2}}\nabla^{I}\kappa||_{L^{2}_{sc}(S_{u^{'},\ubar^{'}})}d\ubar^{'}\right)du^{'},
\end{align}
which, after an application of Gr\"onwall's inequality, yields 
\begin{eqnarray}
\int_{0}^{\ubar}||(a^{\frac{1}{2}}\nabla)^{I}\kappa||_{L^{2}_{sc}(S_{u,\ubar})}d\ubar^{'}\lesssim 1+\mathcal{R}[\widetilde{\beta}].
\end{eqnarray}
Now recall the elliptic system 
\begin{eqnarray} \nonumber
\text{div}\nabla\omega=\text{div}\kappa+\frac{1}{2}\nabla\tilde{\beta},\\
\text{div}\nabla\omega^{\dag}=\text{curl}\kappa+\frac{1}{2}\text{curl}\tilde{\beta}.
\end{eqnarray}
We can apply the $\text{div}-\text{curl}$ estimates from Proposition \ref{divcurlprop} to obtain 
\begin{eqnarray} \nonumber
||a^{\frac{N+4}{2}}\nabla^{N+5}\omega||_{L^{2}_{sc}(S_{u,\ubar})}\lesssim \sum_{J=0}^{N+4}||(a^{\frac{1}{2}}\nabla)^{J}\kappa||_{L^{2}_{sc}(S_{u,\ubar})}+\sum_{J=0}^{N+4}||(a^{\frac{1}{2}}\nabla)^{J}||_{L^{2}_{sc}(S_{u,\ubar})}\\ +\frac{1}{a^{\frac{1}{2}}}\sum_{J=0}^{N+4}||(a^{\frac{1}{2}}\nabla)^{J}(\omega,\omega^{\dag})||_{L^{2}_{sc}(S_{u,\ubar})},
\end{eqnarray}
which, after integrating along the $\ubar$ direction and applying Gr\"onwall's inequality, yields
\begin{eqnarray}
||a^{\frac{N+4}{2}}\nabla^{N+5}\omega||_{L^{2}_{sc}(H_{u}^{0,\ubar)})}\lesssim 1+\mathcal{R}[\tilde{\beta}]. 
\end{eqnarray}
This concludes the proof of the lemma.
\end{proof}

\begin{proposition}
Under the assumptions of the main theorem (\ref{main1}) and the bootstrap assumptions (\ref{ellboot}), the following estimates for $\eta$ and $\etabar$ hold:
\begin{eqnarray} \nonumber
\frac{a}{|u|}||a^{\frac{N+4}{2}}\nabla^{N+5}\eta||_{L^{2}_{sc}(H^{(0,\ubar)}_{u})}\lesssim 1+\mathcal{R},\\ \nonumber
||a^{\frac{N+4}{2}}\nabla^{N+5}\eta||_{L^{2}_{sc}(\Hbar^{(u_{\infty},u)}_{\ubar})}\lesssim 1+\mathcal{R},\\ \nonumber
\frac{a}{|u|}||a^{\frac{N+4}{2}}\nabla^{N+5}\etabar||_{L^{2}_{sc}(H^{(0,\ubar)}_{u})}\lesssim 1+\mathcal{R}.
\end{eqnarray}
\end{proposition}
\begin{proof}
In order to prove these estimates, we introduce the mass aspect functions 
\begin{eqnarray}
\mu=-\text{div}~\eta-\rho,\hspace{2mm} \mubar:=-\text{div}~\etabar-\rho.
\end{eqnarray}
We want to obtain a $\nabla_{4}$ evolution equation for $\mu$ and a $\nabla_{3}$ evolution equation for $\mubar$. We explicitly compute 
\begin{align}  \nonumber
\nabla_{4}\mu&=-\nabla_{4}\text{div}~\eta-\nabla_{4}\rho\\\nonumber 
&=-\text{div}\nabla_{4}\eta-\nabla_{4}\rho-[\nabla_{4},\text{div}]\eta\\\nonumber 
&=\text{div}(\chi\cdot(\eta-\etabar))+\text{div}\widetilde{\beta}-\text{div}\widetilde{\beta}-\frac{3}{2}\tr\chi \rho+\mathcal{E}_{1}\\ 
&= \nabla(\tr\chi,\chihat)(\eta-\etabar)+(\tr\chi,\chihat)\nabla(\eta,\etabar)-\tr\chi\rho+\mathcal{E}_{1},
\end{align}
where $\mathcal{E}_{1}$ reads 
\begin{align} \nonumber
\mathcal{E}_{1}\sim \left(\widetilde{\beta}+\alpha^{F}\cdot(\rho^{F}+\sigma^{F})\right)\eta+(\tr\chi,\chihat)(\eta+\etabar)(\eta+\etabar)+\alpha^{F}\cdot(\rho^{F}\nonumber+\sigma^{F})(\tr\chi,\chihat)\\+(\tr\chi,\chihat)\nabla\eta 
+\hnabla\alpha^{F}\cdot(\rho^{F}+\sigma^{F})+\alpha^{F}\cdot(\hnabla\rho^{F}+\hnabla\sigma^{F}).
\end{align}
Notice the reason behind constructing the function $\mu$ instead of directly working with $\eta$. Since the $\nabla_{4}$ equation for $\eta$ contains $\widetilde{\beta}$, one needs the same level of regularity of $\widetilde{\beta}$ as that of the Ricci coefficients in order to apply the transport inequality therefore the estimates do not close. However, as we have just observed, the spatial derivative of $\widetilde{\beta}$ cancels pointwise in the $\nabla_{4}$ evolution equation for $\mu$ i.e., $\nabla_{4}$ equation for $\mu$ only contains algebraic terms in Weyl curvatures. We now commute the $\nabla_{4}$ equation for $\mu$ with $\nabla^{I},~I=N+4$ and get:
\begin{align} \nonumber
\nabla_{4}\nabla^{I}\mu= & \sum_{J_1+J_2+J_3+J_{4}=I} \nabla^{J_1}(\eta+\etabar)^{J_2}\hnabla^{J_3+1}(\tr\chi,\chihat)\nabla^{J_{4}}(\eta,\etabar)\\\nonumber
+&\sum_{J_1+J_2+J_3+J_{4}=I} \nabla^{J_1}(\eta+\etabar)^{J_2}\nabla^{J_{3}}(\tr\chi,\chihat)\nabla^{J_{4}+1}(\eta,\etabar)\\\nonumber
+&\sum_{J_1+J_2+J_3+J_{4}=I}\nabla^{J_1}(\eta+\etabar)^{J_2}\nabla^{J_{3}}\tr\chi\nabla^{J_{4}}\rho
\\\nonumber 
+&\sum_{J_1+J_2+J_3+J_{4}=I}\nabla^{J_1}(\eta+\etabar)^{J_2}\nabla^{J_{3}}\widetilde{\beta}\nabla^{J_{4}}\eta \\\nonumber
+&\sum_{J_1+J_2+J_3+J_{4}+J_{5}=I}\nabla^{J_1}(\eta+\etabar)^{J_2}\hnabla^{J_{3}}\alpha^{F}\hnabla^{J_{4}}(\rho^{F},\sigma^{F})\nabla^{J_{5}}\eta\\\nonumber
+&\sum_{J_1+J_2+J_3+J_{4}+J_{5}=I}\nabla^{J_1}(\eta+\etabar)^{J_2}\hnabla^{J_{3}}(\tr\chi,\chihat)\hnabla^{J_{4}}(\eta,\etabar)\nabla^{J_{5}}(\eta,\etabar)\\\nonumber
+&\sum_{J_1+J_2+J_3+J_{4}+J_{5}=I}\nabla^{J_1}(\eta+\etabar)^{J_2}\hnabla^{J_{3}}\alpha^{F}\hnabla^{J_{4}}(\rho^{F},\sigma^{F})\nabla^{J_{5}}(\tr\chi,\chihat)\\\nonumber
+&\sum_{J_1+J_2+J_3+J_{4}=I}\nabla^{J_1}(\eta+\etabar)^{J_2}\hnabla^{J_{3}+1}\alpha^{F}\hnabla^{J_{4}}(\rho^{F},\sigma^{F})\\\nonumber
+&\sum_{J_1+J_2+J_3+J_{4}=I}\nabla^{J_1}(\eta+\etabar)^{J_2}\hnabla^{J_{3}}\alpha^{F}\hnabla^{J_{4}+1}(\rho^{F},\sigma^{F})\\\nonumber
+&\sum_{J_1+J_2+J_3 +J_4=I} \nabla^{J_1}(\eta+\etabar)^{J_2}\nabla^{J_3}(\chihat,\tr\chi)\nabla^{J_4} \mathcal{\eta} \\\nonumber +&\sum_{J_1+J_2+J_3+J_4=I-1 }\nabla^{J_1} (\eta+\etabar)^{J_2+1}\nabla^{J_3}(\chihat,\tr\chi)\nabla^{J_4}\mathcal{\eta}\\\nonumber+& \sum_{J_{1}+J_{2}+J_{3}+J_{4}+J_{5}=I-1}\nabla^{J_{1}}(\eta+\underline{\eta})^{J_{2}}\hnabla^{J_{3}}\alpha^{F}\hnabla^{J_{4}}(\rho^{F},\sigma^{F})\nabla^{J_{5}}\eta
\\+& \sum_{J_1+J_2+J_3+J_4=I-1}\nabla^{J_1}(\eta+\etabar)^{J_2}\hnabla^{J_3}\alpha^F \nabla^{J_4}\mathcal{\eta}.
\end{align}
We now apply the $\nabla_{4}$ transport inequality to obtain 
\begin{align} \nonumber
||a^{\frac{I}{2}}\nabla^{I}\mu||_{L^{2}_{sc}(S_{u,\ubar})}&\lesssim ||a^{\frac{I}{2}}\nabla^{I}\mu||_{L^{2}_{sc}(S_{u,0})}+\int_{0}^{\ubar}||a^{\frac{I}{2}}\nabla_{4}\nabla^{I}\mu||_{L^{2}_{sc}(S_{u,\ubar^{'}})}d\ubar^{'}\\ 
&=\int_{0}^{\ubar}||a^{\frac{I}{2}}\nabla_{4}\nabla^{I}\mu||_{L^{2}_{sc}(S_{u,\ubar^{'}})}d\ubar^{'},
\end{align}
since $\nabla^{I}\mu|_{\Hbar_{\ubar=0}}=0$ due to the data being Minkowski on $\Hbar_{\ubar=0}$ (notice that $\nabla_{4}\mu=0$ on $\Hbar_{\ubar=0}$). We therefore obtain 
\begin{align} \nonumber
&||a^{\frac{I}{2}}\nabla^{I}\mu||_{L^{2}_{sc}(S_{u,\ubar})}\\ \lesssim &\int_{0}^{\ubar}||a^{\frac{I}{2}} \sum_{J_1+J_2+J_3+J_{4}=I} \nonumber\nabla^{J_1}(\eta+\etabar)^{J_2}\hnabla^{J_3+1}(\tr\chi,\chihat)\nabla^{J_{4}}(\eta,\etabar)||_{L^{2}_{sc}(S_{u,\ubar^{'}})}d\ubar^{'}\\ \nonumber 
+&\int_{0}^{\ubar}||a^{\frac{I}{2}} \sum_{J_1+J_2+J_3+J_{4}=I} \nabla^{J_1}(\eta+\etabar)^{J_2}\nabla^{J_{3}}(\tr\chi,\chihat)\nabla^{J_{4}+1}(\eta,\etabar)||_{L^{2}_{sc}(S_{u,\ubar^{'}})}d\ubar^{'}\\\nonumber 
+&\int_{0}^{\ubar}||a^{\frac{I}{2}} \sum_{J_1+J_2+J_3+J_{4}=I}\nabla^{J_1}(\eta+\etabar)^{J_2}\nabla^{J_{3}}\tr\chi\nabla^{J_{4}}\rho||_{L^{2}_{sc}(S_{u,\ubar^{'}})}d\ubar^{'}\\\nonumber 
+&\int_{0}^{\ubar}||a^{\frac{I}{2}} \sum_{J_1+J_2+J_3+J_{4}=I}\nabla^{J_1}(\eta+\etabar)^{J_2}\nabla^{J_{3}}\widetilde{\beta}\nabla^{J_{4}}\eta||_{L^{2}_{sc}(S_{u,\ubar^{'}})}d\ubar^{'}\\\nonumber 
+&\int_{0}^{\ubar}||a^{\frac{I}{2}} \sum_{J_1+J_2+J_3+J_{4}+J_{5}=I}\nabla^{J_1}(\eta+\etabar)^{J_2}\hnabla^{J_{3}}\alpha^{F}\hnabla^{J_{4}}(\rho^{F},\sigma^{F})\nabla^{J_{5}}\eta||_{L^{2}_{sc}(S_{u,\ubar^{'}})}d\ubar^{'}\\\nonumber 
+&\int_{0}^{\ubar}||a^{\frac{I}{2}} \sum_{J_1+J_2+J_3+J_{4}+J_{5}=I}\nabla^{J_1}(\eta+\etabar)^{J_2}\hnabla^{J_{3}}(\tr\chi,\chihat)\hnabla^{J_{4}}(\eta,\etabar)\nabla^{J_{5}}(\eta,\etabar)||_{L^{2}_{sc}(S_{u,\ubar^{'}})}d\ubar^{'}\\\nonumber 
+&\int_{0}^{\ubar}||a^{\frac{I}{2}} \sum_{J_1+J_2+J_3+J_{4}+J_{5}=I}\nabla^{J_1}(\eta+\etabar)^{J_2}\hnabla^{J_{3}}\alpha^{F}\hnabla^{J_{4}}(\rho^{F},\sigma^{F})\nabla^{J_{5}}(\tr\chi,\chihat)||_{L^{2}_{sc}(S_{u,\ubar^{'}})}d\ubar^{'}\\\nonumber 
+&\int_{0}^{\ubar}||a^{\frac{I}{2}} \sum_{J_1+J_2+J_3+J_{4}=I}\nabla^{J_1}(\eta+\etabar)^{J_2}\hnabla^{J_{3}+1}\alpha^{F}\hnabla^{J_{4}}(\rho^{F},\sigma^{F})||_{L^{2}_{sc}(S_{u,\ubar^{'}})}d\ubar^{'}\\\nonumber 
+&\int_{0}^{\ubar}||a^{\frac{I}{2}} \sum_{J_1+J_2+J_3+J_{4}=I}\nabla^{J_1}(\eta+\etabar)^{J_2}\hnabla^{J_{3}}\alpha^{F}\hnabla^{J_{4}+1}(\rho^{F},\sigma^{F})||_{L^{2}_{sc}(S_{u,\ubar^{'}})}d\ubar^{'}\\\nonumber 
+&\int_{0}^{\ubar}||a^{\frac{I}{2}} \sum_{J_1+J_2+J_3 +J_4=I} \nabla^{J_1}(\eta+\etabar)^{J_2}\nabla^{J_3}(\chihat,\tr\chi)\nabla^{J_4} \mu ||_{L^{2}_{sc}(S_{u,\ubar^{'}})}d\ubar^{'}\\\nonumber 
+&\int_{0}^{\ubar}||a^{\frac{I}{2}} \sum_{J_1+J_2+J_3+J_4=I-1 }\nabla^{J_1} (\eta+\etabar)^{J_2+1}\nabla^{J_3}(\chihat,\tr\chi)\nabla^{J_4}\mu ||_{L^{2}_{sc}(S_{u,\ubar^{'}})}d\ubar^{'}\\\nonumber 
+&\int_{0}^{\ubar}||a^{\frac{I}{2}} \sum_{J_{1}+J_{2}+J_{3}+J_{4}+J_{5}=I-1}\nabla^{J_{1}}(\eta+\underline{\eta})^{J_{2}}\hnabla^{J_{3}}\alpha^{F}\hnabla^{J_{4}}(\rho^{F},\sigma^{F})\nabla^{J_{5}}\mu ||_{L^{2}_{sc}(S_{u,\ubar^{'}})}d\ubar^{'}\\ 
+&\int_{0}^{\ubar}||a^{\frac{I}{2}} \sum_{J_1+J_2+J_3+J_4=I-1}\nabla^{J_1}(\eta+\etabar)^{J_2}\hnabla^{J_3}\alpha^F \nabla^{J_4}\mu||_{L^{2}_{sc}(S_{u,\ubar^{'}})}d\ubar^{'}.
\end{align}
We estimate each term separately as follows: 
\begin{align}
&\int_{0}^{\ubar}||a^{\frac{I}{2}} \sum_{J_1+J_2+J_3+J_{4}=I} \nonumber\nabla^{J_1}(\eta+\etabar)^{J_2}\hnabla^{J_3+1}(\tr\chi,\chihat)\nabla^{J_{4}}(\eta,\etabar)||_{L^{2}_{sc}(S_{u,\ubar^{'}})}d\ubar^{'}\\ 
\lesssim &\frac{\Gamma^{2}}{|u|a^{\frac{1}{2}}}+\frac{\Gamma^{2}}{|u|}\lesssim 1,
\end{align}
\begin{align} \nonumber
&\int_{0}^{\ubar}||a^{\frac{I}{2}} \sum_{J_1+J_2+J_3+J_{4}=I} \nabla^{J_1}(\eta+\etabar)^{J_2}\nabla^{J_{3}}(\tr\chi,\chihat)\nabla^{J_{4}+1}(\eta,\etabar)||_{L^{2}_{sc}(S_{u,\ubar^{'}})}d\ubar^{'}\\ 
\lesssim &\frac{\Gamma\Gamma_{top}}{a}+\frac{\Gamma\Gamma_{top}}{a^{\frac{1}{2}}}\lesssim a^{-1+\frac{1}{1000}}+a^{-\frac{1}{2}+\frac{1}{1000}}\lesssim 1, 
\end{align}
\begin{align}
\int_{0}^{\ubar}||a^{\frac{I}{2}} \sum_{J_1+J_2+J_3+J_{4}=I}\nabla^{J_1}(\eta+\etabar)^{J_2}\nabla^{J_{3}}\tr\chi\nabla^{J_{4}}\rho||_{L^{2}_{sc}(S_{u,\ubar^{'}})}d\ubar^{'}\lesssim \frac{1}{|u|}\Gamma^{2}\lesssim 1,
\end{align}
\begin{eqnarray}
\int_{0}^{\ubar}||a^{\frac{I}{2}} \sum_{J_1+J_2+J_3+J_{4}=I}\nabla^{J_1}(\eta+\etabar)^{J_2}\nabla^{J_{3}}\widetilde{\beta}\nabla^{J_{4}}\eta||_{L^{2}_{sc}(S_{u,\ubar^{'}})}d\ubar^{'}\lesssim \frac{\Gamma^{2}}{|u|}\lesssim 1,
\end{eqnarray}
\begin{align}\nonumber
&\int_{0}^{\ubar}||a^{\frac{I}{2}} \sum_{J_1+J_2+J_3+J_{4}+J_{5}=I}\nabla^{J_1}(\eta+\etabar)^{J_2}\hnabla^{J_{3}}\alpha^{F}\hnabla^{J_{4}}(\rho^{F},\sigma^{F})\nabla^{J_{5}}\eta||_{L^{2}_{sc}(S_{u,\ubar^{'}})}d\ubar^{'}\\
\lesssim & \frac{a^{\frac{1}{2}}}{|u|^{2}}\Gamma^{3}\lesssim 1,
\end{align}
\begin{align}
&\int_{0}^{\ubar}||a^{\frac{I}{2}} \sum_{J_1+J_2+J_3+J_{4}+J_{5}=I}\nabla^{J_1}(\eta+\etabar)^{J_2}\hnabla^{J_{3}}(\tr\chi,\chihat)\hnabla^{J_{4}}(\eta,\etabar)\nabla^{J_{5}}(\eta,\etabar)||_{L^{2}_{sc}(S_{u,\ubar^{'}})}d\ubar^{'}\\
\lesssim &\frac{\Gamma^{3}}{|u|^{2}}+\frac{a^{\frac{1}{2}}\Gamma^{3}}{|u|^{2}}\lesssim 1,
\end{align}
\begin{align}
&\int_{0}^{\ubar}||a^{\frac{I}{2}} \sum_{J_1+J_2+J_3+J_{4}+J_{5}=I}\nabla^{J_1}(\eta+\etabar)^{J_2}\hnabla^{J_{3}}\alpha^{F}\hnabla^{J_{4}}(\rho^{F},\sigma^{F})\nabla^{J_{5}}(\tr\chi,\chihat)||_{L^{2}_{sc}(S_{u,\ubar^{'}})}d\ubar^{'}\\
\lesssim &\frac{a^{\frac{1}{2}}\Gamma^{3}}{|u|^{2}}+\frac{a\Gamma^{3}}{|u|^{2}}\lesssim 1,
\end{align}
\begin{align}\nonumber
&\int_{0}^{\ubar}||a^{\frac{I}{2}} \sum_{J_1+J_2+J_3+J_{4}=I}\nabla^{J_1}(\eta+\etabar)^{J_2}\hnabla^{J_{3}+1}\alpha^{F}\hnabla^{J_{4}}(\rho^{F},\sigma^{F})||_{L^{2}_{sc}(S_{u,\ubar^{'}})}d\ubar^{'}\\
\lesssim &\frac{a^{\frac{1}{2}}}{|u|}\mathbb{YM}~\Gamma\lesssim 1,
\end{align}
\begin{align}\nonumber
&\int_{0}^{\ubar}||a^{\frac{I}{2}} \sum_{J_1+J_2+J_3+J_{4}=I}\nabla^{J_1}(\eta+\etabar)^{J_2}\hnabla^{J_{3}}\alpha^{F}\hnabla^{J_{4}+1}(\rho^{F},\sigma^{F})||_{L^{2}_{sc}(S_{u,\ubar^{'}})}d\ubar^{'}\\ \lesssim &
\frac{a^{\frac{1}{2}}}{|u|}\mathbb{YM}~\Gamma\lesssim 1,
\end{align}
\begin{align}\nonumber&
\int_{0}^{\ubar}||a^{\frac{I}{2}} \sum_{J_1+J_2+J_3 +J_4=I} \nabla^{J_1}(\eta+\etabar)^{J_2}\nabla^{J_3}(\chihat,\tr\chi)\nabla^{J_4} \mu ||_{L^{2}_{sc}(S_{u,\ubar^{'}})}d\ubar^{'}\\ \lesssim & (\frac{a^{\frac{1}{2}}}{|u|}+\frac{1}{|u|})\int_{0}^{\ubar}||\nabla^{I}\mu||_{L^{2}_{sc}(S_{u,\ubar^{'}})}\lesssim \int_{0}^{\ubar}||\nabla^{I}\mu||_{L^{2}_{sc}(S_{u,\ubar^{'}})},
\end{align}
\begin{align} \nonumber
&\int_{0}^{\ubar}||a^{\frac{I}{2}} \sum_{J_1+J_2+J_3+J_4=I-1 }\nabla^{J_1} (\eta+\etabar)^{J_2+1}\nabla^{J_3}(\chihat,\tr\chi)\nabla^{J_4}\mu ||_{L^{2}_{sc}(S_{u,\ubar^{'}})}d\ubar^{'}\\
\lesssim & \int_{0}^{\ubar}||a^{\frac{I}{2}} \sum_{J_1+J_2+J_3+J_4=I-1 }\nabla^{J_1} (\eta+\etabar)^{J_2+1}\nabla^{J_3}(\chihat,\tr\chi)(\nabla^{J_4+1}\eta+\nabla^{J_{4}}\rho)||_{L^{2}_{sc}(S_{u,\ubar^{'}})}d\ubar^{'}\\\nonumber 
\lesssim &\frac{1+a^{\frac{1}{2}}}{|u|^{2}}\Gamma^{3}+\frac{a}{|u|^{2}}\Gamma^{3}\lesssim 1,
\end{align}
\begin{align} \nonumber
&\int_{0}^{\ubar}||a^{\frac{I}{2}} \sum_{J_{1}+J_{2}+J_{3}+J_{4}+J_{5}=I-1}\nabla^{J_{1}}(\eta+\underline{\eta})^{J_{2}}\hnabla^{J_{3}}\alpha^{F}\hnabla^{J_{4}}(\rho^{F},\sigma^{F})\nabla^{J_{5}}\mu ||_{L^{2}_{sc}(S_{u,\ubar^{'}})}d\ubar^{'}\\\nonumber 
\lesssim &\int_{0}^{\ubar}||a^{\frac{I}{2}} \sum_{J_{1}+J_{2}+J_{3}+J_{4}+J_{5}=I-1}\nabla^{J_{1}}(\eta+\underline{\eta})^{J_{2}}\hnabla^{J_{3}}\alpha^{F}\hnabla^{J_{4}}(\rho^{F},\sigma^{F})(\nabla^{J_5+1}\eta+\nabla^{J_{5}}\rho) ||_{L^{2}_{sc}(S_{u,\ubar^{'}})}d\ubar^{'}\\ 
\lesssim &\frac{a^{\frac{1}{2}}}{|u|^{2}}\Gamma^{3}+\frac{a}{|u|^{2}}\Gamma^{3}\lesssim 1
\end{align}and finally
\begin{align} \nonumber
&\int_{0}^{\ubar}||a^{\frac{I}{2}} \sum_{J_1+J_2+J_3+J_4=I-1}\nabla^{J_1}(\eta+\etabar)^{J_2}\hnabla^{J_3}\alpha^F \nabla^{J_4}\mu||_{L^{2}_{sc}(S_{u,\ubar^{'}})}d\ubar^{'}\\ \nonumber
\lesssim &\int_{0}^{\ubar}||a^{\frac{I}{2}} \sum_{J_1+J_2+J_3+J_4=I-1}\nabla^{J_1}(\eta+\etabar)^{J_2}\hnabla^{J_3}\alpha^F (\nabla^{J_4+1}\eta+\nabla^{J_{4}}\rho)||_{L^{2}_{sc}(S_{u,\ubar^{'}})}d\ubar^{'}\\\nonumber 
\lesssim &\frac{a^{\frac{1}{2}}}{|u|}\Gamma^{2}+\frac{a}{|u|}\mathcal{R}[\alpha]\lesssim 1+\mathcal{R}[\alpha].
\end{align}
Collecting all the terms together and an application of the Gr\"onwall's inequality yields 
\begin{eqnarray}
\label{eq:muestimate}
||a^{\frac{I}{2}}\nabla^{I}\mu||_{L^{2}_{sc}(S_{u,\ubar})}\lesssim 1+\mathcal{R}[\alpha].
\end{eqnarray}
Now the elliptic equation
\begin{eqnarray}
\text{div}~\eta=-\mu-\rho, \hspace{2mm}\text{curl}~\eta=\sigma \epsilon+\frac{1}{2}\chibarhat\wedge\chihat
\end{eqnarray} along with Proposition \ref{divcurlprop}
yield: 
\begin{align} \nonumber
\label{eq:elliptic} \nonumber
&||a^{\frac{N+5}{2}}\nabla^{N+5}\eta||_{L^{2}_{sc}(S_{u,\ubar})}\lesssim a^{\frac{1}{2}}\sum_{I=0}^{N+4}||a^{\frac{I}{2}}\nabla^{I}\mu||_{L^{2}_{sc}(S_{u,\ubar})}+a^{\frac{1}{2}}\sum_{I=0}^{N+4}||a^{\frac{I}{2}}\nabla^{I}(\rho,\sigma)||_{L^{2}_{sc}(S_{u,\ubar})}\\&
+a^{\frac{1}{2}}\sum_{I=0}^{N+4}||a^{\frac{I}{2}}\nabla^{I}(\chibarhat \chihat)||_{L^{2}_{sc}(S_{u,\ubar})}+\sum_{I=0}^{N+4}||a^{\frac{I}{2}}\nabla^{I}\eta||_{L^{2}_{sc}(S_{u,\ubar})}
\end{align}
which, upon integrating along $\ubar$, applying Cauchy-Schwartz and taking into account the estimates from Section 4 for $I\leq N+4$, yield the result
\begin{eqnarray}
\frac{a}{|u|}||a^{\frac{N+4}{2}}\nabla^{N+5}\eta||_{L^{2}_{sc}(H^{(0,\ubar)}_{u})}\lesssim 1+\mathcal{R}.
\end{eqnarray}
Similarly, the elliptic estimate (\ref{eq:elliptic}) together with the estimate on the mass aspect function and integration along $\Hbar$ yield 
\begin{eqnarray}
||a^{\frac{N+4}{2}}\nabla^{N+5}\eta||_{L^{2}_{sc}(\Hbar^{(u_{\infty},u)}_{\ubar})}\lesssim 1+\mathcal{R}.
\end{eqnarray}

\noindent We now consider the $\nabla_{3}$ evolution equation for $\mubar$. Explicit calculations yield 
\begin{align}\nonumber
\nabla_{3}\mubar&=-\text{div}\nabla_{3}\etabar-\nabla_{3}\rho+[\text{div},\nabla_{3}]\etabar\\\nonumber 
=&\text{div}(\chibar\cdot(\etabar-\eta))-\text{div}\widetilde{\betabar}+\text{div}\widetilde{\betabar}+\frac{3}{2}\tr\chibar\rho+\mathcal{E}_{2}\\\nonumber 
=&-\tr\chibar(-\text{div}\etabar-\rho)+\frac{1}{2}\tr\chibar\rho+\mathcal{E}_{2}\\
=&-\tr\chibar \mubar+\frac{1}{2}\tr\chibar\rho+\mathcal{E}_{2},
\end{align}
where $\mathcal{E}_{2}$ reads 
\begin{eqnarray}
\mathcal{E}_{2}\sim (\eta+\etabar)\widetilde{\betabar}+\hnabla\alphabar^{F}\cdot(\rho^{F},\sigma^{F})+\alphabar^{F}\cdot\hnabla(\rho^{F},\sigma^{F})+\tr\chibar (\rho^{F}\cdot\rho^{F}\nonumber+\sigma^{F}\cdot\sigma^{F})\\
+\widetilde{\betabar}\etabar+\alphabar^{F}\cdot(\rho^{F},\sigma^{F})(\etabar,\eta)+(\chibarhat,\tr\chibar)(\eta-\etabar)(\eta-\etabar)+(\eta,\etabar)\widetilde{\betabar}.
\end{eqnarray}
Commuting the operator $\nabla^{I}$ for $I=N+4$ with the $\nabla_{3}$ evolution equation for $\mu$ yields
\begin{align}\nonumber
\nabla_{3}\nabla^{I}\mubar+\frac{I+2}{2}\tr\chibar\nabla^{I}\mu= &\sum_{J_1+J_2+J_3+J_{4}=I}\nabla^{J_1}(\eta+\etabar)^{J_2}\nabla^{J_3}(\eta,\etabar)\nabla^{J_{4}}\widetilde{\betabar}\\\nonumber 
+&\sum_{J_1+J_2+J_3+J_{4}=I}\nabla^{J_1}(\eta+\etabar)^{J_2}\nabla^{J_3}\tr\chibar\nabla^{J_{4}}\rho\\\nonumber
+&\sum_{J_1+J_2+J_3+J_{4}=I}\nabla^{J_1}(\eta+\etabar)^{J_2}\hnabla^{J_3+1}\alphabar^{F}\nabla^{J_{4}}(\rho^{F},\sigma^{F})\\\nonumber 
+&\sum_{J_1+J_2+J_3+J_{4}=I}\nabla^{J_1}(\eta+\etabar)^{J_2}\hnabla^{J_3}\alphabar^{F}\nabla^{J_{4}+1}(\rho^{F},\sigma^{F})\\\nonumber 
+&\sum_{J_1+J_2+J_3+J_{4}+J_{5}=I}\nabla^{J_1}(\eta+\etabar)^{J_2}\nabla^{J_3}\tr\chibar\nabla^{J_{4}}(\rho^{F},\sigma^{F})\nabla^{J_{5}}(\rho^{F},\sigma^{F})\\\nonumber 
+&\sum_{J_1+J_2+J_3+J_{4}+J_{5}=I}\nabla^{J_1}(\eta+\etabar)^{J_2}\nabla^{J_3}(\eta,\etabar)\nabla^{J_{4}}(\rho^{F},\sigma^{F})\nabla^{J_{5}}\alphabar^{F}\\\nonumber 
+&\sum_{J_1+J_2+J_3+J_{4}+J_{5}=I}\nabla^{J_1}(\eta+\etabar)^{J_2}\nabla^{J_3}(\tr\chibar,\chibarhat)\nabla^{J_{4}}(\eta,\etabar)\nabla^{J_{5}}(\eta,\etabar)\\\nonumber 
+&\sum_{J_1+J_2+J_3+J_{4}=I}\nabla^{J_1}(\eta+\etabar)^{J_2}\nabla^{J_3}(\eta,\etabar)\nabla^{J_{4}}\widetilde{\betabar}\\\nonumber 
+&\sum_{J_1+J_2+J_3+J_4=I}\nabla^{J_1}(\eta+\etabar)^{J_2}\nabla^{J_3}(\chibarhat,\hsp \tildetr)\nabla^{J_4}\mubar\\\nonumber+&\sum_{J_{1}+J_{2}+J_{3} +J_{4}=I-1}\nabla^{J_{1}}(\eta+\underline{\eta})^{J_{2} +1}\hat{\nabla}^{J_{3}}\tr\underline{\chi}\nabla^{J_{4}}\mubar\\\nonumber+& \sum_{J_1+J_2+J_3+J_4=I-1}\nabla^{J_1} (\eta+\etabar)^{J_2+1}\nabla^{J_3}(\chibarhat,\tr\chibar)\hnabla^{J_4}\mubar\\\nonumber+& \sum_{J_{1}+J_{2}+J_{3}+J_{4}+J_{5}=I-1}\nabla^{J_{1}}(\eta+\underline{\eta})^{J_{2}}\hnabla^{J_{3}}\alphabar^{F}\hnabla^{J_{4}}(\rho^{F},\sigma^{F})\hnabla^{J_{5}}\mubar\\+ &\sum_{J_1+J_2+J_3+J_4=I-1}\nabla^{J_1}(\eta+\etabar)^{J_2}\hnabla^{J_3}\alphabar^F \hnabla^{J_4}\mubar.
\end{align}
We now apply the transport inequality in the $e_{3}$ direction with $\lambda_{0}=\frac{I+2}{2}$ (equivalently $\lambda_{1}=I+1$) which gives
\begin{eqnarray}
|u|^{I+1}||\nabla^{I}\mubar||_{L^{2}(S_{u,\ubar})}\lesssim |u_{\infty}|^{I+1}||\nabla^{I}\mubar||_{L^{2}(S_{u_{\infty},\ubar})}+\int_{u_{\infty}}^{u}|u^{'}|^{I+1}||\nabla_{3}\nabla^{I}\mubar||_{L^{2}(S_{u^{'},\ubar})}du^{'}.
\end{eqnarray}
Now recall that $s_{2}(\mubar)=\frac{1}{2}+\frac{1}{2}=1,$ $s_{2}(\nabla^{I}\mubar)=\frac{I}{2}+1$, and $s_{2}(\nabla_{3}\nabla^{I}\mubar)=\frac{I}{2}+2$ and therefore 
\begin{eqnarray}
||\nabla^{I}\mubar||_{L^{2}_{sc}(S_{u,\ubar})}=a^{-\frac{I}{2}-1}|u|^{I+2}||\nabla^{I}\mubar||_{L^{2}(S_{u,\ubar})}.
\end{eqnarray}
The scale invariant transport inequality reads 
\begin{align}\nonumber &
\frac{a}{|u|}||a^{\frac{I}{2}}\nabla^{I}\mubar||_{L^{2}_{sc}(S_{u,\ubar})}\\ \nonumber\lesssim &\frac{a}{|u_{\infty}|}||a^{\frac{I}{2}}\nabla^{I}\mubar||_{L^{2}_{sc}(S_{u,\ubar})}+\int_{u_{\infty}}^{u}\frac{a^{2}}{|u^{'}|^{3}}||a^{\frac{I}{2}}\nabla_{3}\nabla^{I}\mubar||_{L^{2}_{sc}(S_{u^{'},\ubar})}\\\nonumber 
\lesssim& \frac{a}{|u_{\infty}|}||a^{\frac{I}{2}}\nabla^{I}\mubar||_{L^{2}_{sc}(S_{u,\ubar})}+\int_{u_{\infty}}^{u}\frac{a^{2}}{|u^{'}|^{3}}||a^{\frac{I}{2}}\sum_{J_1+J_2+J_3+J_{4}=I}\nabla^{J_1}(\eta+\etabar)^{J_2}\nabla^{J_3}(\eta,\etabar)\nabla^{J_{4}}\widetilde{\betabar}||_{L^{2}_{sc}(S_{u^{'},\ubar})}\\\nonumber 
+&\int_{u_{\infty}}^{u}\frac{a^{2}}{|u^{'}|^{3}}||a^{\frac{I}{2}}\sum_{J_1+J_2+J_3+J_{4}=I}\nabla^{J_1}(\eta+\etabar)^{J_2}\nabla^{J_3}\tr\chibar\nabla^{J_{4}}\rho||_{L^{2}_{sc}(S_{u^{'},\ubar})}\\\nonumber 
+&\int_{u_{\infty}}^{u}\frac{a^{2}}{|u^{'}|^{3}}||a^{\frac{I}{2}}\sum_{J_1+J_2+J_3+J_{4}=I}\nabla^{J_1}(\eta+\etabar)^{J_2}\hnabla^{J_3+1}\alphabar^{F}\nabla^{J_{4}}(\rho^{F},\sigma^{F})||_{L^{2}_{sc}(S_{u^{'},\ubar})}\\\nonumber +&\int_{u_{\infty}}^{u}\frac{a^{2}}{|u^{'}|^{3}}||a^{\frac{I}{2}}\sum_{J_1+J_2+J_3+J_{4}=I}\nabla^{J_1}(\eta+\etabar)^{J_2}\hnabla^{J_3}\alphabar^{F}\nabla^{J_{4}+1}(\rho^{F},\sigma^{F})||_{L^{2}_{sc}(S_{u^{'},\ubar})}\\\nonumber +&\int_{u_{\infty}}^{u}\frac{a^{2}}{|u^{'}|^{3}}||a^{\frac{I}{2}}\sum_{J_1+J_2+J_3+J_{4}+J_{5}=I}\nabla^{J_1}(\eta+\etabar)^{J_2}\nabla^{J_3}\tr\chibar\nabla^{J_{4}}(\rho^{F},\sigma^{F})\nabla^{J_{5}}(\rho^{F},\sigma^{F})||_{L^{2}_{sc}(S_{u^{'},\ubar})}\\\nonumber 
+&\int_{u_{\infty}}^{u}\frac{a^{2}}{|u^{'}|^{3}}||a^{\frac{I}{2}}\sum_{J_1+J_2+J_3+J_{4}+J_{5}=I}\nabla^{J_1}(\eta+\etabar)^{J_2}\nabla^{J_3}(\eta,\etabar)\nabla^{J_{4}}(\rho^{F},\sigma^{F})\nabla^{J_{5}}\alphabar^{F}||_{L^{2}_{sc}(S_{u^{'},\ubar})}\\\nonumber 
+&\int_{u_{\infty}}^{u}\frac{a^{2}}{|u^{'}|^{3}}||a^{\frac{I}{2}}\sum_{J_1+J_2+J_3+J_{4}+J_{5}=I}\nabla^{J_1}(\eta+\etabar)^{J_2}\nabla^{J_3}(\tr\chibar,\chibarhat)\nabla^{J_{4}}(\eta,\etabar)\nabla^{J_{5}}(\eta,\etabar)||_{L^{2}_{sc}(S_{u^{'},\ubar})}\\\nonumber 
+&\int_{u_{\infty}}^{u}\frac{a^{2}}{|u^{'}|^{3}}||a^{\frac{I}{2}}\sum_{J_1+J_2+J_3+J_{4}=I}\nabla^{J_1}(\eta+\etabar)^{J_2}\nabla^{J_3}(\eta,\etabar)\nabla^{J_{4}}\widetilde{\betabar}||_{L^{2}_{sc}(S_{u^{'},\ubar})}\\\nonumber 
+&\int_{u_{\infty}}^{u}\frac{a^{2}}{|u^{'}|^{3}}||a^{\frac{I}{2}}\sum_{J_1+J_2+J_3+J_4=I}\nabla^{J_1}(\eta+\etabar)^{J_2}\nabla^{J_3}(\chibarhat,\hsp \tildetr)\nabla^{J_4}\mubar||_{L^{2}_{sc}(S_{u^{'},\ubar})}\\\nonumber 
+&\int_{u_{\infty}}^{u}\frac{a^{2}}{|u^{'}|^{3}}||a^{\frac{I}{2}}\sum_{J_{1}+J_{2}+J_{3} +J_{4}=I-1}\nabla^{J_{1}}(\eta+\underline{\eta})^{J_{2} +1}\hat{\nabla}^{J_{3}}\tr\underline{\chi}\nabla^{J_{4}}\mubar||_{L^{2}_{sc}(S_{u^{'},\ubar})}\\\nonumber 
+&\int_{u_{\infty}}^{u}\frac{a^{2}}{|u^{'}|^{3}}||a^{\frac{I}{2}}\sum_{J_1+J_2+J_3+J_4=I-1}\nabla^{J_1} (\eta+\etabar)^{J_2+1}\nabla^{J_3}(\chibarhat,\tr\chibar)\hnabla^{J_4}\mubar||_{L^{2}_{sc}(S_{u^{'},\ubar})}\\\nonumber 
+&\int_{u_{\infty}}^{u}\frac{a^{2}}{|u^{'}|^{3}}||a^{\frac{I}{2}}\sum_{J_{1}+J_{2}+J_{3}+J_{4}+J_{5}=I-1}\nabla^{J_{1}}(\eta+\underline{\eta})^{J_{2}}\hnabla^{J_{3}}\alphabar^{F}\hnabla^{J_{4}}(\rho^{F},\sigma^{F})\hnabla^{J_{5}}\mubar||_{L^{2}_{sc}(S_{u^{'},\ubar})}\\
+&\int_{u_{\infty}}^{u}\frac{a^{2}}{|u^{'}|^{3}}||a^{\frac{I}{2}}\sum_{J_1+J_2+J_3+J_4=I-1}\nabla^{J_1}(\eta+\etabar)^{J_2}\hnabla^{J_3}\alphabar^F \hnabla^{J_4}\mubar||_{L^{2}_{sc}(S_{u^{'},\ubar})}.
\end{align}
We now estimate each term separately as follows:
\begin{eqnarray} \nonumber
\int_{u_{\infty}}^{u}\frac{a^{2}}{|u^{'}|^{3}}||a^{\frac{I}{2}}\sum_{J_1+J_2+J_3+J_{4}=I}\nabla^{J_1}(\eta+\etabar)^{J_2}\nabla^{J_3}(\eta,\etabar)\nabla^{J_{4}}\widetilde{\betabar}||_{L^{2}_{sc}(S_{u^{'},\ubar})}\\
\lesssim 1+\left(\int_{u_{\infty}}^{u}\frac{a^{3}}{|u^{'}|^{6}}du^{'}\right)^{\frac{1}{2}}\left(\int_{u_{\infty}}^{u}\frac{a}{|u^{'}|^{2}}||a^{\frac{I}{2}}\nabla^{I}\widetilde{\betabar}||^{2}_{L^{2}_{sc}(S_{u^{'}\ubar})}du^{'}\right)^{\frac{1}{2}}\lesssim 1,
\end{eqnarray}
\begin{align} \nonumber
&\int_{u_{\infty}}^{u}\frac{a^{2}}{|u^{'}|^{3}}||a^{\frac{I}{2}}\sum_{J_1+J_2+J_3+J_{4}=I}\nabla^{J_1}(\eta+\etabar)^{J_2}\nabla^{J_3}\tr\chibar\nabla^{J_{4}}\rho||_{L^{2}_{sc}(S_{u^{'},\ubar})}\\\nonumber 
\lesssim  &1+\int_{u_{\infty}}^{u}\frac{a}{|u^{'}|^{2}}||a^{\frac{I}{2}}\nabla^{I}\rho||_{L^{2}_{sc}(S_{u^{'},\ubar})}du^{'}\lesssim 1+\frac{a}{|u|}\left(\int_{u_{\infty}}^{u}\frac{a}{|u^{'}|^{2}}||a^{\frac{I}{2}}\nabla^{I}\rho||^{2}_{L^{2}_{sc}(S_{u^{'},\ubar})}du^{'}\right)^{\frac{1}{2}}\\ 
\lesssim &1+\mathcal{R}[\rho],
\end{align}
\begin{align}\nonumber
&\int_{u_{\infty}}^{u}\frac{a^{2}}{|u^{'}|^{3}}||a^{\frac{I}{2}}\sum_{J_1+J_2+J_3+J_{4}=I}\nabla^{J_1}(\eta+\etabar)^{J_2}\hnabla^{J_3+1}\alphabar^{F}\nabla^{J_{4}}(\rho^{F},\sigma^{F})||_{L^{2}_{sc}(S_{u^{'},\ubar})}\\ 
\lesssim & 1+\left(\int_{u_{\infty}}^{u}\frac{a^{3}}{|u^{'}|^{6}}du^{'}\right)^{\frac{1}{2}}\left(\int_{u_{\infty}}^{u}\frac{a}{|u^{'}|^{2}}||a^{\frac{I}{2}}\nabla^{I+1}\alpha^{F}||^{2}_{L^{2}_{sc}(S_{u^{'},\ubar})}\right)^{\frac{1}{2}}\lesssim 1,
\end{align}
\begin{align} \nonumber
&\int_{u_{\infty}}^{u}\frac{a^{2}}{|u^{'}|^{3}}||a^{\frac{I}{2}}\sum_{J_1+J_2+J_3+J_{4}=I}\nabla^{J_1}(\eta+\etabar)^{J_2}\hnabla^{J_3}\alphabar^{F}\nabla^{J_{4}+1}(\rho^{F},\sigma^{F})||_{L^{2}_{sc}(S_{u^{'},\ubar})}\\
\lesssim &\frac{a^{\frac{3}{2}}}{|u|^{\frac{5}{2}}}\mathbb{YM}[\rho^{F},\sigma^{F}]\lesssim 1,
\end{align}
\begin{align} \nonumber
\int_{u_{\infty}}^{u}\frac{a^{2}}{|u^{'}|^{3}}||a^{\frac{I}{2}}\sum_{J_1+J_2+J_3+J_{4}+J_{5}=I}\nabla^{J_1}(\eta+\etabar)^{J_2}\nabla^{J_3}\tr\chibar\nabla^{J_{4}}(\rho^{F},\sigma^{F})\nabla^{J_{5}}(\rho^{F},\sigma^{F})||_{L^{2}_{sc}(S_{u^{'},\ubar})}\\ 
\lesssim \frac{a}{|u|^{2}}\Gamma\lesssim 1,
\end{align}
\begin{align} \nonumber
&\int_{u_{\infty}}^{u}\frac{a^{2}}{|u^{'}|^{3}}||a^{\frac{I}{2}}\sum_{J_1+J_2+J_3+J_{4}+J_{5}=I}\nabla^{J_1}(\eta+\etabar)^{J_2}\nabla^{J_3}(\eta,\etabar)\nabla^{J_{4}}(\rho^{F},\sigma^{F})\nabla^{J_{5}}\alphabar^{F}||_{L^{2}_{sc}(S_{u^{'},\ubar})}\\
\lesssim & \frac{a^{2}}{|u|^{4}}\Gamma\lesssim 1,
\end{align}
\begin{align} \nonumber
&\int_{u_{\infty}}^{u}\frac{a^{2}}{|u^{'}|^{3}}||a^{\frac{I}{2}}\sum_{J_1+J_2+J_3+J_{4}+J_{5}=I}\nabla^{J_1}(\eta+\etabar)^{J_2}\nabla^{J_3}(\tr\chibar,\chibarhat)\nabla^{J_{4}}(\eta,\etabar)\nabla^{J_{5}}(\eta,\etabar)||_{L^{2}_{sc}(S_{u^{'},\ubar})}\\
\lesssim &\frac{a}{|u|^{2}}\Gamma\lesssim 1,
\end{align}
\begin{align}  \nonumber
&\int_{u_{\infty}}^{u}\frac{a^{2}}{|u^{'}|^{3}}||a^{\frac{I}{2}}\sum_{J_1+J_2+J_3+J_{4}=I}\nabla^{J_1}(\eta+\etabar)^{J_2}\nabla^{J_3}(\eta,\etabar)\nabla^{J_{4}}\widetilde{\betabar}||_{L^{2}_{sc}(S_{u^{'},\ubar})}\\ \lesssim &\frac{a^{\frac{3}{2}}}{|u|^{\frac{5}{2}}}\mathcal{R}[\widetilde{\betabar}]\lesssim 1,
\end{align}
\begin{align}  \nonumber&
\int_{u_{\infty}}^{u}\frac{a^{2}}{|u^{'}|^{3}}||a^{\frac{I}{2}}\sum_{J_1+J_2+J_3+J_4=I}\nabla^{J_1}(\eta+\etabar)^{J_2}\nabla^{J_3}(\chibarhat,\hsp \tildetr)\nabla^{J_4}\mubar||_{L^{2}_{sc}(S_{u^{'},\ubar})}\\
\lesssim &\int_{u_{\infty}}^{u}\frac{a^{\frac{1}{2}}}{|u^{'}|^{2}}\frac{a}{|u^{'}|}||a^{\frac{I}{2}}\nabla^{I}\mubar||||_{L^{2}_{sc}(S_{u^{'},\ubar})}du^{'},
\end{align}
\begin{align} \nonumber &
\int_{u_{\infty}}^{u}\frac{a^{2}}{|u^{'}|^{3}}||a^{\frac{I}{2}}\sum_{J_{1}+J_{2}+J_{3} +J_{4}=I-1}\nabla^{J_{1}}(\eta+\underline{\eta})^{J_{2} +1}\hat{\nabla}^{J_{3}}\tr\underline{\chi}\nabla^{J_{4}}\mubar||_{L^{2}_{sc}(S_{u^{'},\ubar})}\\
\lesssim  \nonumber &\int_{u_{\infty}}^{u}\frac{a}{|u^{'}|^{3}}\left(||a^{\frac{I}{2}}\nabla^{I}\etabar||_{L^{2}_{sc}(S_{u^{'},\ubar})}+||a^{\frac{I}{2}}\nabla^{I-1}\rho||_{L^{2}_{sc}(S_{u^{'},\ubar})}\right)du^{'}\\ 
\lesssim& \frac{a}{|u|^{2}}\Gamma+\frac{a^{\frac{3}{2}}}{|u|^{2}}\Gamma\lesssim 1,
\end{align}
\begin{align}&  \nonumber
\int_{u_{\infty}}^{u}\frac{a^{2}}{|u^{'}|^{3}}||a^{\frac{I}{2}}\sum_{J_1+J_2+J_3+J_4=I-1}\nabla^{J_1} (\eta+\etabar)^{J_2+1}\nabla^{J_3}(\chibarhat,\tr\chibar)\hnabla^{J_4}\mubar||_{L^{2}_{sc}(S_{u^{'},\ubar})}\\
\lesssim \frac{a}{|u|^{2}}\Gamma+\frac{a^{\frac{3}{2}}}{|u|^{2}}\Gamma\lesssim 1,
\end{align}
\begin{align}  \nonumber
&\int_{u_{\infty}}^{u}\frac{a^{2}}{|u^{'}|^{3}}||a^{\frac{I}{2}}\sum_{J_{1}+J_{2}+J_{3}+J_{4}+J_{5}=I-1}\nabla^{J_{1}}(\eta+\underline{\eta})^{J_{2}}\hnabla^{J_{3}}\alphabar^{F}\hnabla^{J_{4}}(\rho^{F},\sigma^{F})\hnabla^{J_{5}}\mubar||_{L^{2}_{sc}(S_{u^{'},\ubar})}\\
\lesssim &\frac{a^{2}}{|u|^{4}}+\frac{a^{\frac{5}{2}}}{|u|^{4}}\Gamma\lesssim 1.
\end{align}
Collecting all the terms and an application of Gr\"onwall yields 
\begin{align}\nonumber 
\frac{a}{|u|}||a^{\frac{I}{2}}\nabla^{I}\mubar||_{L^{2}_{sc}(S_{u,\ubar})}\lesssim &(1+\frac{a}{|u_{\infty}|}||a^{\frac{I}{2}}\nabla^{I}\mubar||_{L^{2}_{sc}(S_{u,\ubar})}+\mathcal{R}[\rho])e^{\frac{a^{\frac{1}{2}}}{|u|}}\\\nonumber 
\lesssim& 1+\frac{a}{|u_{\infty}|}||a^{\frac{I}{2}}\nabla^{I}\mubar||_{L^{2}_{sc}(S_{u,\ubar})}+\mathcal{R}[\rho]\\
\lesssim& 1+\mathcal{R}[\rho].
\end{align}
Now we utilize the elliptic equations
\begin{eqnarray}
\text{div}~\etabar=-\mubar-\rho,~\text{curl}~\etabar=-\sigma\epsilon-\chibarhat\wedge\chihat
\end{eqnarray}
to obtain, using Proposition \ref{divcurlprop},
\begin{align} \nonumber 
\frac{a}{|u|}||a^{\frac{N+4}{2}}\nabla^{N+5}\etabar||_{L^{2}_{sc}(S_{u,\ubar})}\lesssim &\sum_{I=0}^{N+4}\left(\frac{a}{|u|}||a^{\frac{I}{2}}\nabla^{I}\mubar||_{L^{2}_{sc}(S_{u,\ubar})}+||a^{\frac{I}{2}}\nabla^{I}(\rho,\sigma)||_{L^{2}_{sc}(S_{u,\ubar})}\right.\\ 
&\left.+||a^{\frac{I}{2}}\nabla^{I}(\chibarhat\chibar)||_{L^{2}_{sc}(S_{u,\ubar})}+a^{-\frac{1}{2}}||a^{\frac{I}{2}}\nabla^{I}\etabar||_{L^{2}_{sc}(S_{u,\ubar})}\right),
\end{align}
which, after integrating along the $e_{4}$ direction, yields 
\begin{eqnarray}
\\\nonumber 
\frac{a}{|u|}||a^{\frac{N+4}{2}}\nabla^{N+5}\etabar||_{L^{2}_{sc}(H^{(0,\ubar)}_{u})}\lesssim 1+\mathcal{R}.
\end{eqnarray}
The claim follows.
\end{proof}

\begin{proposition}
Under the assumptions of Theorem \ref{main1} and the bootstrap assumptions \eqref{bootstrap}-\eqref{ellboot}, there hold:
\begin{eqnarray}
\|(a^{\frac{1}{2}})^{N+4}\nabla^{N+5}\underline{\omega}\|_{\mathcal{L}^{2}_{sc}(\underline{H}^{(u_{\infty},u)}_{\underline{u}})}\lesssim 1+\mathcal{R}+\mathbb{YM}.
\end{eqnarray}
\end{proposition}
\begin{proof}
Similarly to the $\omega$ estimates, we define the auxiliary entity $\omegabar^{\dag}$ to be zero on $\underline{H}_{0}$ whilst verifying the following $\nabla_{4}$ equation:
\begin{eqnarray}
\label{eq:omegabardag}
\nabla_{4}\omegabar^{\dag}=\frac{1}{2}\sigma.
\end{eqnarray}
Now we define $\kappabar$ as follows 
\begin{eqnarray}
\kappabar=-\nabla\omegabar+~^{*}\nabla\omegabar^{\dag}-\frac{1}{2}\widetilde{\betabar}.
\end{eqnarray}
Recall the transport equation for $\omegabar$
\begin{eqnarray}
\nabla_{4}\underline{\omega}=2\omega\underline{\omega}+\frac{3}{4}|\eta-\underline{\eta}|^{2}-\frac{1}{4}(\eta-\underline{\eta})\cdot(\eta+\underline{\eta})-\frac{1}{8}|\eta+\underline{\eta}|^{2}+\frac{1}{2}\rho+\frac{1}{4}\mathfrak{T}_{43}.
\end{eqnarray}
An explicit computation yields 
\begin{align} \nonumber 
\nabla_{4}\kappabar&=-\nabla\nabla_{4}\omegabar+^{*}\nabla\nabla_{4}\omegabar^{\dag}-\frac{1}{2}\nabla_{4}\widetilde{\betabar}+[\nabla,\nabla_{4}]\omegabar+[\nabla_{4},~^{*}\nabla]\omegabar^{\dag}\\\nonumber 
&=-\frac{1}{2}\nabla\rho+\frac{1}{2}\nabla\sigma-\frac{1}{2}( -\nabla \rho + \Hodge{\nabla} \sigma)+\mathcal{E}\\
&=\mathcal{E},
\end{align}
where $\mathcal{E}$ reads 
\begin{align}
\mathcal{E}\sim &\omegabar\nabla\omega+\omega\nabla\omegabar+(\eta,\etabar)\nabla(\eta,\etabar)+\rho^{F}\cdot\hnabla\rho^{F}+\sigma^{F}\cdot\hnabla\sigma^{F}+\widetilde{\beta}(\omegabar,\omegabar^{\dag})\nonumber+\alpha^{F}\cdot(\rho^{F},\sigma^{F})(\omegabar,\omegabar^{\dag})\\\nonumber+&(\eta,\etabar)\omega\omegabar(\eta\eta+\eta\etabar+\etabar\etabar)+(\eta,\etabar)(\rho,\sigma)+(\eta,\etabar)|\rho^{F},\sigma^{F}|^{2}+(\tr\chi,\chihat)\nabla(\omegabar,\omegabar^{\dag})\\ +&(\tr\chi,\chihat)\etabar(\omegabar,\omegabar^{\dag})+(\tr\chibar,\chibarhat)\alpha^{F}\cdot(\rho^{F},\sigma^{F})+(\tr\chi,\chihat)\alphabar^{F}\cdot(\rho^{F},\sigma^{F})+\tr\chi\widetilde{\betabar}+\omega\widetilde{\betabar}+\chibarhat\widetilde{\beta}.
\end{align}
We now commute $\nabla^{I},I=N+4$ with the $\nabla_{4}$ evolution equation for $\kappabar$
to obtain 
\begin{align}
\nabla_{4}\nabla^{I}\kappabar=& \sum_{J_1+J_2+J_3+J_{4}=I} \nabla^{J_1}(\eta+\etabar)^{J_2}\nabla^{J_3}\omegabar\nabla^{J_{4}+1}\omega
\\\nonumber +&
\sum_{J_1+J_2+J_3+J_{4}=I} \nabla^{J_1}(\eta+\etabar)^{J_2}\nabla^{J_3+1}\omegabar\nabla^{J_{4}}\omega\\\nonumber 
+&\sum_{J_1+J_2+J_3+J_{4}=I} \nabla^{J_1}(\eta+\etabar)^{J_2}\nabla^{J_3}(\eta,\etabar)\nabla^{J_{4}+1}(\eta,\etabar)\\\nonumber 
+&\sum_{J_1+J_2+J_3+J_{4}=I} \nabla^{J_1}(\eta+\etabar)^{J_2}\hnabla^{J_3}\rho^{F}\hnabla^{J_{4}+1}\rho^{F}\\\nonumber 
+&\sum_{J_1+J_2+J_3+J_{4}=I} \nabla^{J_1}(\eta+\etabar)^{J_2}\hnabla^{J_3}\sigma^{F}\hnabla^{J_{4}+1}\sigma^{F}\\\nonumber 
+&\sum_{J_1+J_2+J_3+J_{4}=I} \nabla^{J_1}(\eta+\etabar)^{J_2}\nabla^{J_3}\widetilde{\beta}\nabla^{J_{4}}(\omegabar,\omegabar^{\dag})\\\nonumber 
+&\sum_{J_1+J_2+J_3+J_{4}+J_{5}=I} \nabla^{J_1}(\eta+\etabar)^{J_2}\hnabla^{J_3}\alpha^{F}\hnabla^{J_{4}}(\rho^{F},\sigma^{F})\nabla^{J_{5}}(\omegabar,\omegabar^{\dag})\\\nonumber +&\sum_{J_1+J_2+J_3+J_{4}+J_{5}=I} \nabla^{J_1}(\eta+\etabar)^{J_2}\nabla^{J_3}(\eta,\etabar)\nabla^{J_{4}}(\omega,\omegabar)\nabla^{J_{5}}(\eta,\etabar)\\\nonumber 
+&\sum_{J_1+J_2+J_3+J_{4}=I} \nabla^{J_1}(\eta+\etabar)^{J_2}\hnabla^{J_3}(\eta,\etabar)\nabla^{J_{4}}(\rho,\sigma)\\\nonumber 
+&\sum_{J_1+J_2+J_3+J_{4}+J_{5}=I} \nabla^{J_1}(\eta+\etabar)^{J_2}\hnabla^{J_3}(\eta,\etabar)\hnabla^{J_{4}}(\rho^{F},\sigma^{F})\hnabla^{J_{5}}(\rho^{F},\sigma^{F})\\\nonumber 
+&\sum_{J_1+J_2+J_3+J_{4}=I} \nabla^{J_1}(\eta+\etabar)^{J_2}\nabla^{J_3}(\tr\chi,\chihat)\hnabla^{J_{4}+1}(\omegabar,\omegabar^{\dag})\\\nonumber 
+&\sum_{J_1+J_2+J_3+J_{4}+J_{5}=I} \nabla^{J_1}(\eta+\etabar)^{J_2}\nabla^{J_3}(\tr\chi,\chihat)\nabla^{J_{4}}\etabar\nabla^{J_{5}}(\omegabar,\omegabar^{\dag})\\\nonumber 
+&\sum_{J_1+J_2+J_3+J_{4}+J_{5}=I} \nabla^{J_1}(\eta+\etabar)^{J_2}\nabla^{J_3}(\tr\chibar,\chibarhat)\hnabla^{J_{4}}\alpha^{F}\hnabla^{J_{5}}(\rho^{F},\sigma^{F})\\\nonumber 
+&\sum_{J_1+J_2+J_3+J_{4}+J_{5}=I} \nabla^{J_1}(\eta+\etabar)^{J_2}\nabla^{J_3}(\tr\chi,\chihat)\hnabla^{J_{4}}\alphabar^{F}\hnabla^{J_{5}}(\rho^{F},\sigma^{F})\\\nonumber 
+&\sum_{J_1+J_2+J_3+J_{4}=I} \nabla^{J_1}(\eta+\etabar)^{J_2}\nabla^{J_3}\tr\chi\hnabla^{J_{4}}\widetilde{\betabar}\\\nonumber 
+&\sum_{J_1+J_2+J_3+J_{4}=I} \nabla^{J_1}(\eta+\etabar)^{J_2}\nabla^{J_3}\omega\hnabla^{J_{4}}\widetilde{\betabar}\\\nonumber +&\sum_{J_1+J_2+J_3+J_{4}=I} \nabla^{J_1}(\eta+\etabar)^{J_2}\nabla^{J_3}\chibarhat\hnabla^{J_{4}}\widetilde{\betabar}
\\\nonumber 
+&\sum_{J_1+J_2+J_3 +J_4=I} \nabla^{J_1}(\eta+\etabar)^{J_2}\nabla^{J_3}(\chihat,\tr\chi)\nabla^{J_4} \kappabar \\ \nonumber+&\sum_{J_1+J_2+J_3+J_4=I-1 }\nabla^{J_1} (\eta+\etabar)^{J_2+1}\nabla^{J_3}(\chihat,\tr\chi)\nabla^{J_4}\kappabar\\+ &\sum_{J_{1}+J_{2}+J_{3}+J_{4}+J_{5}=I-1}\nabla^{J_{1}}(\eta+\underline{\eta})^{J_{2}}\hnabla^{J_{3}}\alpha^{F}\hnabla^{J_{4}}(\rho^{F},\sigma^{F})\nabla^{J_{5}}\kappabar.
\end{align}
Now we apply the scale invariant $\nabla_{4}$ transport inequality 
\begin{align}
&||a^{\frac{I}{2}}\nabla^{I}\kappa||_{L^{2}_{sc}(S_{u,\ubar})}\\ \nonumber \lesssim & ||a^{\frac{I}{2}}\nabla^{I}\kappa||_{L^{2}_{sc}(S_{u,0})}\\\nonumber +&\int_{0}^{\ubar}||a^{\frac{I}{2}}\sum_{J_1+J_2+J_3+J_{4}=I} \nabla^{J_1}(\eta+\etabar)^{J_2}\nabla^{J_3}\omegabar\nabla^{J_{4}+1}\omega||_{L^{2}_{sc}(S_{u,\ubar^{'}})}d\ubar^{'}\\\nonumber
+&\int_{0}^{\ubar}||a^{\frac{I}{2}}\sum_{J_1+J_2+J_3+J_{4}=I} \nabla^{J_1}(\eta+\etabar)^{J_2}\nabla^{J_3+1}\omegabar\nabla^{J_{4}}\omega||_{L^{2}_{sc}(S_{u,\ubar^{'}})}d\ubar^{'}\\\nonumber 
+&\int_{0}^{\ubar}||a^{\frac{I}{2}}\sum_{J_1+J_2+J_3+J_{4}=I} \nabla^{J_1}(\eta+\etabar)^{J_2}\nabla^{J_3}(\eta,\etabar)\nabla^{J_{4}+1}(\eta,\etabar)||_{L^{2}_{sc}(S_{u,\ubar^{'}})}d\ubar^{'}\\\nonumber 
+&\int_{0}^{\ubar}||a^{\frac{I}{2}}\sum_{J_1+J_2+J_3+J_{4}=I} \nabla^{J_1}(\eta+\etabar)^{J_2}\hnabla^{J_3}\rho^{F}\hnabla^{J_{4}+1}\rho^{F}||_{L^{2}_{sc}(S_{u,\ubar^{'}})}d\ubar^{'}\\\nonumber +&\int_{0}^{\ubar}||a^{\frac{I}{2}}\sum_{J_1+J_2+J_3+J_{4}=I} \nabla^{J_1}(\eta+\etabar)^{J_2}\hnabla^{J_3}\sigma^{F}\hnabla^{J_{4}+1}\sigma^{F}||_{L^{2}_{sc}(S_{u,\ubar^{'}})}d\ubar^{'}\\\nonumber 
+&\int_{0}^{\ubar}||a^{\frac{I}{2}}\sum_{J_1+J_2+J_3+J_{4}=I} \nabla^{J_1}(\eta+\etabar)^{J_2}\nabla^{J_3}\widetilde{\beta}\nabla^{J_{4}}(\omegabar,\omegabar^{\dag})||_{L^{2}_{sc}(S_{u,\ubar^{'}})}d\ubar^{'}\\\nonumber 
+&\int_{0}^{\ubar}||a^{\frac{I}{2}}\sum_{J_1+J_2+J_3+J_{4}+J_{5}=I} \nabla^{J_1}(\eta+\etabar)^{J_2}\hnabla^{J_3}\alpha^{F}\hnabla^{J_{4}}(\rho^{F},\sigma^{F})\nabla^{J_{5}}(\omegabar,\omegabar^{\dag})||_{L^{2}_{sc}(S_{u,\ubar^{'}})}d\ubar^{'}\\\nonumber 
+&\int_{0}^{\ubar}||a^{\frac{I}{2}}\sum_{J_1+J_2+J_3+J_{4}+J_{5}=I} \nabla^{J_1}(\eta+\etabar)^{J_2}\nabla^{J_3}(\eta,\etabar)\nabla^{J_{4}}(\omega,\omegabar)\nabla^{J_{5}}(\eta,\etabar)||_{L^{2}_{sc}(S_{u,\ubar^{'}})}d\ubar^{'}\\\nonumber 
+&\int_{0}^{\ubar}||a^{\frac{I}{2}}\sum_{J_1+J_2+J_3+J_{4}=I} \nabla^{J_1}(\eta+\etabar)^{J_2}\hnabla^{J_3}(\eta,\etabar)\nabla^{J_{4}}(\rho,\sigma)||_{L^{2}_{sc}(S_{u,\ubar^{'}})}d\ubar^{'}\\\nonumber 
+&\int_{0}^{\ubar}||a^{\frac{I}{2}}\sum_{J_1+J_2+J_3+J_{4}+J_{5}=I} \nabla^{J_1}(\eta+\etabar)^{J_2}\hnabla^{J_3}(\eta,\etabar)\hnabla^{J_{4}}(\rho^{F},\sigma^{F})\hnabla^{J_{5}}(\rho^{F},\sigma^{F})||_{L^{2}_{sc}(S_{u,\ubar^{'}})}d\ubar^{'}\\\nonumber 
+&\int_{0}^{\ubar}||a^{\frac{I}{2}}\sum_{J_1+J_2+J_3+J_{4}=I} \nabla^{J_1}(\eta+\etabar)^{J_2}\nabla^{J_3}(\tr\chi,\chihat)\hnabla^{J_{4}+1}(\omegabar,\omegabar^{\dag})||_{L^{2}_{sc}(S_{u,\ubar^{'}})}d\ubar^{'}\\\nonumber 
+&\int_{0}^{\ubar}||a^{\frac{I}{2}}\sum_{J_1+J_2+J_3+J_{4}+J_{5}=I} \nabla^{J_1}(\eta+\etabar)^{J_2}\nabla^{J_3}(\tr\chi,\chihat)\nabla^{J_{4}}\etabar\nabla^{J_{5}}(\omegabar,\omegabar^{\dag})||_{L^{2}_{sc}(S_{u,\ubar^{'}})}d\ubar^{'}\\\nonumber 
+&\int_{0}^{\ubar}||a^{\frac{I}{2}}\sum_{J_1+J_2+J_3+J_{4}+J_{5}=I} \nabla^{J_1}(\eta+\etabar)^{J_2}\nabla^{J_3}(\tr\chibar,\chibarhat)\hnabla^{J_{4}}\alpha^{F}\hnabla^{J_{5}}(\rho^{F},\sigma^{F})||_{L^{2}_{sc}(S_{u,\ubar^{'}})}d\ubar^{'}\\\nonumber 
+&\int_{0}^{\ubar}||a^{\frac{I}{2}}\sum_{J_1+J_2+J_3+J_{4}+J_{5}=I} \nabla^{J_1}(\eta+\etabar)^{J_2}\nabla^{J_3}(\tr\chi,\chihat)\hnabla^{J_{4}}\alphabar^{F}\hnabla^{J_{5}}(\rho^{F},\sigma^{F})||_{L^{2}_{sc}(S_{u,\ubar^{'}})}d\ubar^{'}\\\nonumber 
+&\int_{0}^{\ubar}||a^{\frac{I}{2}}\sum_{J_1+J_2+J_3+J_{4}=I} \nabla^{J_1}(\eta+\etabar)^{J_2}\nabla^{J_3}\tr\chi\hnabla^{J_{4}}\widetilde{\betabar})||_{L^{2}_{sc}(S_{u,\ubar^{'}})}d\ubar^{'}\\\nonumber 
+&\int_{0}^{\ubar}||a^{\frac{I}{2}}\sum_{J_1+J_2+J_3+J_{4}=I} \nabla^{J_1}(\eta+\etabar)^{J_2}\nabla^{J_3}\omega\hnabla^{J_{4}}\widetilde{\betabar}||_{L^{2}_{sc}(S_{u,\ubar^{'}})}d\ubar^{'}\\\nonumber 
+&\int_{0}^{\ubar}||a^{\frac{I}{2}}\sum_{J_1+J_2+J_3+J_{4}=I} \nabla^{J_1}(\eta+\etabar)^{J_2}\nabla^{J_3}\chibarhat\hnabla^{J_{4}}\widetilde{\betabar}||_{L^{2}_{sc}(S_{u,\ubar^{'}})}d\ubar^{'}\\\nonumber 
+&\int_{0}^{\ubar}||a^{\frac{I}{2}}\sum_{J_1+J_2+J_3 +J_4=I} \nabla^{J_1}(\eta+\etabar)^{J_2}\nabla^{J_3}(\chihat,\tr\chi)\nabla^{J_4} \kappabar ||_{L^{2}_{sc}(S_{u,\ubar^{'}})}d\ubar^{'}\\\nonumber 
+&\int_{0}^{\ubar}||a^{\frac{I}{2}}\sum_{J_1+J_2+J_3+J_4=I-1 }\nabla^{J_1} (\eta+\etabar)^{J_2+1}\nabla^{J_3}(\chihat,\tr\chi)\nabla^{J_4}\kappabar||_{L^{2}_{sc}(S_{u,\ubar^{'}})}d\ubar^{'}\\
+&\int_{0}^{\ubar}||a^{\frac{I}{2}}\sum_{J_{1}+J_{2}+J_{3}+J_{4}+J_{5}=I-1}\nabla^{J_{1}}(\eta+\underline{\eta})^{J_{2}}\hnabla^{J_{3}}\alpha^{F}\hnabla^{J_{4}}(\rho^{F},\sigma^{F})\nabla^{J_{5}}\kappabar||_{L^{2}_{sc}(S_{u,\ubar^{'}})}d\ubar^{'}.
\end{align}
Now we estimate each term separately. We have 
\begin{eqnarray}
||a^{\frac{I}{2}}\nabla^{I}\kappa||_{L^{2}_{sc}(S_{u,0})}\lesssim \mathcal{I}^{0}\lesssim 1,
\end{eqnarray}
\begin{align} \nonumber &
\int_{0}^{\ubar}||a^{\frac{I}{2}}\sum_{J_1+J_2+J_3+J_{4}=I} \nabla^{J_1}(\eta+\etabar)^{J_2}\nabla^{J_3}\omegabar\nabla^{J_{4}+1}\omega||_{L^{2}_{sc}(S_{u,\ubar^{'}})}d\ubar^{'}\\
\lesssim &1+\frac{1}{|u|}||a^{\frac{I}{2}}\nabla^{I+1}\omega||_{L^{2}_{sc}(S_{u,\ubar^{'}})}d\ubar^{'}\lesssim 1,
\end{align}
\begin{align}\nonumber & 
\int_{0}^{\ubar}||a^{\frac{I}{2}}\sum_{J_1+J_2+J_3+J_{4}=I} \nabla^{J_1}(\eta+\etabar)^{J_2}\nabla^{J_3+1}\omegabar\nabla^{J_{4}}\omega||_{L^{2}_{sc}(S_{u,\ubar^{'}})}d\ubar^{'}\\
\lesssim &1+\frac{1}{|u|}\int_{0}^{\ubar}||a^{\frac{I}{2}}\nabla^{I+1}\omegabar||_{L^{2}_{sc}(S_{u,\ubar^{'}})}d\ubar^{'},
\end{align}
\begin{align}\nonumber
&\int_{0}^{\ubar}||a^{\frac{I}{2}}\sum_{J_1+J_2+J_3+J_{4}=I} \nabla^{J_1}(\eta+\etabar)^{J_2}\nabla^{J_3}(\eta,\etabar)\nabla^{J_{4}+1}\eta||_{L^{2}_{sc}(S_{u,\ubar^{'}})}d\ubar^{'}\\
\lesssim &1+\frac{1}{a}\frac{a}{|u|}||a^{\frac{N+4}{2}}\nabla^{N+5}\eta||_{L^{2}_{sc}(H^{(0,\ubar)}_{u})}\lesssim 1+\frac{1}{a}\mathcal{R},
\end{align}
\begin{align}\nonumber &
\int_{0}^{\ubar}||a^{\frac{I}{2}}\sum_{J_1+J_2+J_3+J_{4}=I} \nabla^{J_1}(\eta+\etabar)^{J_2}\nabla^{J_3}(\eta,\etabar)\nabla^{J_{4}+1}\etabar||_{L^{2}_{sc}(S_{u,\ubar^{'}})}d\ubar^{'}\\
\lesssim &1+\frac{1}{a}\frac{a}{|u|}||a^{\frac{N+4}{2}}\nabla^{N+5}\etabar||_{L^{2}_{sc}(H^{(0,\ubar)}_{u})}\lesssim 1+\frac{1}{a}\mathcal{R},
\end{align}

\begin{align}\nonumber &
\int_{0}^{\ubar}||a^{\frac{I}{2}}\sum_{J_1+J_2+J_3+J_{4}=I} \nabla^{J_1}(\eta+\etabar)^{J_2}\hnabla^{J_3}\rho^{F}\hnabla^{J_{4}+1}\rho^{F}||_{L^{2}_{sc}(S_{u,\ubar^{'}})}d\ubar^{'}\\
\lesssim & 1+\frac{1}{|u|}||a^{\frac{I}{2}}\hnabla^{I+1}\rho^{F}||_{L^{2}_{sc}(H^{(0,\ubar)}_{u})}\lesssim 1+\frac{\mathbb{YM}}{|u|}, 
\end{align}
\begin{align}\nonumber & 
\int_{0}^{\ubar}||a^{\frac{I}{2}}\sum_{J_1+J_2+J_3+J_{4}=I} \nabla^{J_1}(\eta+\etabar)^{J_2}\hnabla^{J_3}\sigma^{F}\hnabla^{J_{4}+1}\sigma^{F}||_{L^{2}_{sc}(S_{u,\ubar^{'}})}d\ubar^{'}\\ \lesssim & 1+\frac{1}{|u|}||a^{\frac{I}{2}}\hnabla^{I+1}\sigma^{F}||_{L^{2}_{sc}(H^{(0,\ubar)}_{u})}\lesssim 1+\frac{\mathbb{YM}}{|u|},
\end{align}
\begin{align}\nonumber& 
\int_{0}^{\ubar}||a^{\frac{I}{2}}\sum_{J_1+J_2+J_3+J_{4}=I} \nabla^{J_1}(\eta+\etabar)^{J_2}\nabla^{J_3}\widetilde{\beta}\nabla^{J_{4}}(\omegabar,\omegabar^{\dag})||_{L^{2}_{sc}(S_{u,\ubar^{'}})}d\ubar^{'}\\ 
\lesssim &1+\frac{1}{|u|}(\mathcal{R}+\Gamma)\lesssim 1+\frac{\mathcal{R}}{|u|},
\end{align}
\begin{align} \nonumber
&\int_{0}^{\ubar}||a^{\frac{I}{2}}\sum_{J_1+J_2+J_3+J_{4}+J_{5}=I} \nabla^{J_1}(\eta+\etabar)^{J_2}\hnabla^{J_3}\alpha^{F}\hnabla^{J_{4}}(\rho^{F},\sigma^{F})\nabla^{J_{5}}(\omegabar,\omegabar^{\dag})||_{L^{2}_{sc}(S_{u,\ubar^{'}})}d\ubar^{'}\\
\lesssim & \frac{a^{\frac{1}{2}}}{|u|^{2}}\Gamma\lesssim 1,
\end{align}
\begin{align}\nonumber &
\int_{0}^{\ubar}||a^{\frac{I}{2}}\sum_{J_1+J_2+J_3+J_{4}+J_{5}=I} \nabla^{J_1}(\eta+\etabar)^{J_2}\nabla^{J_3}(\eta,\etabar)\nabla^{J_{4}}(\omega,\omegabar)\nabla^{J_{5}}(\eta,\etabar)||_{L^{2}_{sc}(S_{u,\ubar^{'}})}d\ubar^{'}\\ \lesssim &\frac{1}{|u|^{2}}\Gamma\lesssim 1,
\end{align}
\begin{align} \nonumber
&\int_{0}^{\ubar}||a^{\frac{I}{2}}\sum_{J_1+J_2+J_3+J_{4}=I} \nabla^{J_1}(\eta+\etabar)^{J_2}\hnabla^{J_3}(\eta,\etabar)\nabla^{J_{4}}(\rho,\sigma)||_{L^{2}_{sc}(S_{u,\ubar^{'}})}d\ubar^{'}\\
\lesssim &1+\frac{1}{|u|}||a^{\frac{I}{2}}\nabla^{I}(\rho,\sigma)||_{L^{2}_{sc}(H^{(0,\ubar)}_{u})}\lesssim 1+\frac{1}{|u|}\mathcal{R}[\rho,\sigma],
\end{align}
\begin{align} \nonumber
&\int_{0}^{\ubar}||a^{\frac{I}{2}}\sum_{J_1+J_2+J_3+J_{4}+J_{5}=I} \nabla^{J_1}(\eta+\etabar)^{J_2}\hnabla^{J_3}(\eta,\etabar)\hnabla^{J_{4}}(\rho^{F},\sigma^{F})\hnabla^{J_{5}}(\rho^{F},\sigma^{F})||_{L^{2}_{sc}(S_{u,\ubar^{'}})}d\ubar^{'}\\
\lesssim &\frac{1}{|u|^{2}}\Gamma\lesssim 1,
\end{align}
\begin{align}\nonumber &
\int_{0}^{\ubar}||a^{\frac{I}{2}}\sum_{J_1+J_2+J_3+J_{4}=I} \nabla^{J_1}(\eta+\etabar)^{J_2}\nabla^{J_3}(\tr\chi,\chihat)\hnabla^{J_{4}+1}(\omegabar,\omegabar^{\dag})||_{L^{2}_{sc}(S_{u,\ubar^{'}})}d\ubar^{'} \\ 
\lesssim & 1+\frac{a^{\frac{1}{2}}}{|u|}\int_{0}^{\ubar}||a^{\frac{I}{2}}\nabla^{I+1}(\omegabar,\omegabar^{\dag})||_{L^{2}_{sc}(S_{u,\ubar^{'}})}d\ubar^{'},
\end{align}
\begin{align} \nonumber &
\int_{0}^{\ubar}||a^{\frac{I}{2}}\sum_{J_1+J_2+J_3+J_{4}+J_{5}=I} \nabla^{J_1}(\eta+\etabar)^{J_2}\nabla^{J_3}(\tr\chi,\chihat)\nabla^{J_{4}}\etabar\nabla^{J_{5}}(\omegabar,\omegabar^{\dag})||_{L^{2}_{sc}(S_{u,\ubar^{'}})}d\ubar^{'}\\
\lesssim &\frac{a^{\frac{1}{2}}}{|u|^{2}}\Gamma\lesssim 1,
\end{align}
\begin{align}
\nonumber &\int_{0}^{\ubar}||a^{\frac{I}{2}}\sum_{J_1+J_2+J_3+J_{4}+J_{5}=I} \nabla^{J_1}(\eta+\etabar)^{J_2}\nabla^{J_3}(\tr\chibar,\chibarhat)\hnabla^{J_{4}}\alpha^{F}\hnabla^{J_{5}}(\rho^{F},\sigma^{F})||_{L^{2}_{sc}(S_{u,\ubar^{'}})}d\ubar^{'}\\
\lesssim & \frac{1}{a^{\frac{1}{2}}}\Gamma\lesssim 1,
\end{align}
\begin{align} \nonumber
&\int_{0}^{\ubar}||a^{\frac{I}{2}}\sum_{J_1+J_2+J_3+J_{4}+J_{5}=I} \nabla^{J_1}(\eta+\etabar)^{J_2}\nabla^{J_3}(\tr\chi,\chihat)\hnabla^{J_{4}}\alphabar^{F}\hnabla^{J_{5}}(\rho^{F},\sigma^{F})||_{L^{2}_{sc}(S_{u,\ubar^{'}})}d\ubar^{'}\\
\lesssim &
\frac{a^{\frac{1}{2}}}{|u|}\Gamma\lesssim 1,
\end{align}
\begin{align}
\int_{0}^{\ubar}||a^{\frac{I}{2}}\sum_{J_1+J_2+J_3+J_{4}=I} \nabla^{J_1}(\eta+\etabar)^{J_2}\nabla^{J_3}\tr\chi\hnabla^{J_{4}}\widetilde{\betabar})||_{L^{2}_{sc}(S_{u,\ubar^{'}})}d\ubar^{'}
\lesssim 1+\frac{1}{|u|}\mathcal{R}[\widetilde{\betabar}]
\end{align}
\begin{align} 
&\int_{0}^{\ubar}||a^{\frac{I}{2}}\sum_{J_1+J_2+J_3+J_{4}=I} \nabla^{J_1}(\eta+\etabar)^{J_2}\nabla^{J_3}\omega\hnabla^{J_{4}}\widetilde{\betabar}||_{L^{2}_{sc}(S_{u,\ubar^{'}})}d\ubar^{'}
\lesssim 1+\frac{1}{|u|}\mathcal{R}[\widetilde{\betabar}],
\end{align}
\begin{eqnarray}
\int_{0}^{\ubar}||a^{\frac{I}{2}}\sum_{J_1+J_2+J_3+J_{4}=I} \nabla^{J_1}(\eta+\etabar)^{J_2}\nabla^{J_3}\chibarhat\hnabla^{J_{4}}\widetilde{\betabar}||_{L^{2}_{sc}(S_{u,\ubar^{'}})}d\ubar^{'}
\lesssim 1+\frac{1}{a^{\frac{1}{2}}}\mathcal{R}[\widetilde{\betabar}],
\end{eqnarray}
\begin{align}\nonumber&
\int_{0}^{\ubar}||a^{\frac{I}{2}}\sum_{J_1+J_2+J_3 +J_4=I} \nabla^{J_1}(\eta+\etabar)^{J_2}\nabla^{J_3}(\chihat,\tr\chi)\nabla^{J_4} \kappabar ||_{L^{2}_{sc}(S_{u,\ubar^{'}})}d\ubar^{'}\\
\lesssim & \frac{a^{\frac{1}{2}}}{|u|}\int_{0}^{\ubar}||a^{\frac{I}{2}}\nabla^{I}\kappabar||||_{L^{2}_{sc}(S_{u,\ubar^{'}})}d\ubar^{'},
\end{align}
\begin{align} \nonumber
&\int_{0}^{\ubar}||a^{\frac{I}{2}}\sum_{J_1+J_2+J_3+J_4=I-1 }\nabla^{J_1} (\eta+\etabar)^{J_2+1}\nabla^{J_3}(\chihat,\tr\chi)\nabla^{J_4}\kappabar||_{L^{2}_{sc}(S_{u,\ubar^{'}})}d\ubar^{'}\\
\lesssim &\frac{a}{|u|^{2}}\int_{0}^{\ubar}||a^{\frac{I}{2}}\nabla^{I}\kappabar||||_{L^{2}_{sc}(S_{u,\ubar^{'}})}d\ubar^{'}
\end{align}and finally 
\begin{align}\nonumber &\int_{0}^{\ubar}||a^{\frac{I}{2}}\sum_{J_{1}+J_{2}+J_{3}+J_{4}+J_{5}=I-1}\nabla^{J_{1}}(\eta+\underline{\eta})^{J_{2}}\hnabla^{J_{3}}\alpha^{F}\hnabla^{J_{4}}(\rho^{F},\sigma^{F})\nabla^{J_{5}}\kappabar||_{L^{2}_{sc}(S_{u,\ubar^{'}})}d\ubar^{'}\\
\lesssim &\frac{a}{|u|^{2}}\int_{0}^{\ubar}||a^{\frac{I}{2}}\nabla^{I}\kappabar||||_{L^{2}_{sc}(S_{u,\ubar^{'}})}d\ubar^{'}.
\end{align}
Here we have utilized the elliptic estimate for $\omega$ from proposition (\ref{omegaestimate}). In addition, the $L^{2}_{sc}(S_{u,\ubar})$ norm of $\nabla^{J}\omegabar^{\dag}$ for $J\leq N+4$ may be estimated directly from the evolution equation (\ref{eq:omegabardag}) along with the trivial boundary condition for $\omegabar^{\dag}$ along $\Hbar_{0}$. Collecting all the terms together, we have 
\begin{align}\nonumber
||a^{\frac{I}{2}}\nabla^{I}\kappabar||_{L^{2}_{sc}(S_{u,\ubar})}\lesssim& 1+\frac{a^{\frac{1}{2}}}{|u|}\int_{0}^{\ubar}||a^{\frac{I}{2}}\nabla^{I+1}(\omegabar,\omegabar^{\dag})||_{L^{2}_{sc}(S_{u,\ubar^{'}})}d\ubar^{'}\nonumber+\frac{a^{\frac{1}{2}}}{|u|}\int_{0}^{\ubar}||a^{\frac{I}{2}}\nabla^{I}\kappabar||_{L^{2}_{sc}(S_{u,\ubar^{'}})}d\ubar^{'}\\
+&\frac{a}{|u|^{2}}\int_{0}^{\ubar}||a^{\frac{I}{2}}\nabla^{I}\kappabar||||_{L^{2}_{sc}(S_{u,\ubar^{'}})}d\ubar^{'}+\mathbb{YM}+\mathcal{R},
\end{align}
which, upon using Gr\"onwall's inequality, yields 
\begin{align}
||a^{\frac{I}{2}}\nabla^{I}\kappabar||_{L^{2}_{sc}(S_{u,\ubar})}\lesssim 1+\mathbb{YM}+\mathcal{R}+\frac{a^{\frac{1}{2}}}{|u|}\int_{0}^{\ubar}||a^{\frac{I}{2}}\nabla^{I+1}(\omegabar,\omegabar^{\dag})||_{L^{2}_{sc}(S_{u,\ubar^{'}})}d\ubar^{'},
\end{align}
due to the fact that $|\ubar|\lesssim 1$. We now notice that the elliptic equations
\begin{align}
-\nabla\omegabar+~^{*}\nabla\omegabar^{\dag}=-\kappabar-\frac{1}{2}\widetilde{\betabar}
\end{align}
yield the $\text{div}-\text{curl}$ system 
\begin{align}
\text{div}\nabla\omegabar=-\text{div}\kappabar-\frac{1}{2}\text{div}\widetilde{\betabar},~
\text{curl}\nabla\omegabar=0,\\
\text{curl}\omegabar^{\dag}=\text{curl}\kappabar+\frac{1}{2}\text{curl}\widetilde{\betabar},~
\text{div}\nabla\omegabar^{\dag}=0.
\end{align}
An application of the elliptic estimates and transport estimates for $\kappabar$ yield
\begin{align}\nonumber
||a^{\frac{I+1}{2}}\nabla^{I+1}(\omegabar,\omegabar^{\dag})||_{L^{2}_{sc}(S_{u,\ubar})}\lesssim& \sum_{J\leq N+4}||a^{\frac{J}{2}}\nabla^{J}\kappabar||_{L^{2}_{sc}(S_{u,\ubar})}+\sum_{J\leq N+4}||a^{\frac{J}{2}}\nabla^{J}\widetilde{\betabar}||_{L^{2}_{sc}(S_{u,\ubar})}\\\nonumber 
+&\sum_{J\leq N+4}||a^{\frac{J}{2}}\nabla^{J}(\omegabar,\omegabar^{\dag})||_{L^{2}_{sc}(S_{u,\ubar})}\\
\lesssim &1+\mathbb{YM}+\mathcal{R}+\frac{a^{\frac{1}{2}}}{|u|}\int_{0}^{\ubar}||a^{\frac{I}{2}}\nabla^{I+1}(\omegabar,\omegabar^{\dag})||_{L^{2}_{sc}(S_{u,\ubar^{'}})}d\ubar^{'}
\end{align}
and therefore 
\begin{align}
\label{eq:omegabar}
||a^{\frac{I+1}{2}}\nabla^{I+1}(\omegabar,\omegabar^{\dag})||_{L^{2}_{sc}(S_{u,\ubar})}
\lesssim 1+\mathbb{YM}+\mathcal{R}.
\end{align}
Now an integration of the square of the previous inequality (\ref{eq:omegabar}) with respect to the measure $\frac{a}{|u^{'}|^{2}}d\ubar^{'}$ yields 
\begin{eqnarray}
\|(a^{\frac{1}{2}})^{N+4}\nabla^{N+5}\underline{\omega}\|_{\mathcal{L}^{2}_{sc}(\underline{H}^{(u_{\infty},u)}_{\underline{u}})}\lesssim 1+\mathcal{R}+\mathbb{YM}.
\end{eqnarray}
This concludes the proof of the lemma.
\end{proof}

\begin{proposition}
Under the assumptions of Theorem \ref{main1} and the bootstrap assumptions \eqref{bootstrap}-\eqref{ellboot}, there hold:

\[ \intu \frac{a^{\frac{3}{2}}}{\upr^3}\scaletwoSuprime{(\al)^{N+4}\nabla^{N+5}\chibarhat} \duprime \lesssim 1,\]
\[ \intu \frac{a^2}{\upr^3} \scaletwoSuprime{(\al)^{N+4}\nabla^{N+5}\tr\chibar}\duprime \lesssim \mathcal{R}+\underline{\mathcal{R}}+\mathbb{YM}+1.\]
\end{proposition}

\begin{proof}
We start the proof by recalling the $\nabla_{3}$ transport equation for $\widetilde{\tr\chibar}$
\begin{eqnarray}
\nabla_{3}\widetilde{\tr\chibar}+\tr\chibar\widetilde{\tr\chibar}=\frac{2}{|u|^{2}}(\Omega^{-1}-1)+|\widetilde{\tr\chibar}|^{2}+2\omegabar\tr\chibar-|\chibarhat|^{2}-|\alphabar^{F}|^{2}.
\end{eqnarray}
We now commute $\nabla^{I},I=N+5$ with $\nabla_3$: 
\begin{align}\nonumber &
\nabla_{3}\nabla^{I}\widetilde{\tr\chibar}+(\frac{I}{2}+1)\tr\chibar\widetilde{\tr\chibar}\\ =&\frac{1}{|u|^{2}}\sum_{J_1+J_2+J_3=I}\nabla^{J_1}(\eta+\etabar)^{J_2}\nabla^{J_3}(\Omega^{-1}-1)\\\nonumber
+&\sum_{J_1+J_2+J_3+J_4=I}\nabla^{J_1}(\eta+\etabar)^{J_2}\nabla^{J_3}\widetilde{\tr\chibar}\nabla^{J_4}\widetilde{\tr\chibar}\\\nonumber 
+&\sum_{J_1+J_2+J_3+J_4=I}\nabla^{J_1}(\eta+\etabar)^{J_2}\nabla^{J_3}\omegabar\nabla^{J_4}\tr\chibar\\\nonumber +&\sum_{J_1+J_2+J_3+J_4=I}\nabla^{J_1}(\eta+\etabar)^{J_2}\nabla^{J_3}\chibarhat\nabla^{J_4}\chibarhat\\\nonumber 
+&\sum_{J_1+J_2+J_3+J_4=I}\nabla^{J_1}(\eta+\etabar)^{J_2}\hnabla^{J_3}\alphabar^{F}\hnabla^{J_4}\alphabar^{F}
\\\nonumber 
+&\sum_{J_1+J_2+J_3+J_4=I}\nabla^{J_1}(\eta+\etabar)^{J_2}\nabla^{J_3}(\chibarhat,\hsp \tildetr)\nabla^{J_4}\widetilde{\tr\chibar}\\\nonumber+&\sum_{J_{1}+J_{2}+J_{3} +J_{4}=I-1}\nabla^{J_{1}}(\eta+\underline{\eta})^{J_{2} +1}\hnabla^{J_{3}}\tr\underline{\chi}\hnabla^{J_{4}}\widetilde{\tr\chibar}\\\nonumber+&\sum_{J_1+J_2+J_3+J_4=I-1}\nabla^{J_1} (\eta+\etabar)^{J_2+1}\nabla^{J_3}(\chibarhat,\tr\chibar)\nabla^{J_4}\widetilde{\tr\chibar}\\\nonumber+ &\sum_{J_{1}+J_{2}+J_{3}+J_{4}+J_{5}=I-1}\nabla^{J_{1}}(\eta+\underline{\eta})^{J_{2}}\hnabla^{J_{3}}\alphabar^{F}\hnabla^{J_{4}}(\rho^{F},\sigma^{F})\nabla^{J_{5}}\widetilde{\tr\chibar}\\+ &\sum_{J_1+J_2+J_3+J_4=I-1}\nabla^{J_1}(\eta+\etabar)^{J_2}\hnabla^{J_3}\alphabar^F \nabla^{J_4}\widetilde{\tr\chibar}:=\mathcal{E}.
\end{align}
An application of the weighted transport inequality from Proposition \ref{3.6} yields
\begin{eqnarray}
|u|^{I+1}||\nabla^{I}\widetilde{\tr\chibar}||_{L^{2}(S_{u,\ubar})}\lesssim |u_{\infty}|^{I+1}||\nabla^{I}\widetilde{\tr\chibar}||_{L^{2}(S_{u_{\infty},\ubar})}+\int_{u_{\infty}}^{u}|u^{'}|^{I+1}||\mathcal{E}||_{L^{2}(S_{u^{'},\ubar})}du^{'}.
\end{eqnarray}
We now calculate the relevant signatures.  We have $s_{2}(\widetilde{\tr\chibar})=1$ and therefore $s_{2}(\nabla^{I}\widetilde{\tr\chibar})=\frac{I}{2}+1$ and $s_{2}(\mathcal{E})=s_{2}(\nabla_{3}\nabla^{I}\widetilde{\tr\chibar})=\frac{I}{2}+2$. Therefore 
\begin{eqnarray}
|u|^{I+1}||\nabla^{I}\widetilde{\tr\chibar}||_{L^{2}(S_{u,\ubar})}&=&a^{\frac{I}{2}+1}|u|^{-(I+2)}||\nabla^{I}\widetilde{\tr\chibar}||_{L^{2}_{sc}(S_{u,\ubar})},\\
|u^{'}|^{I+1}||\mathcal{E}||_{L^{2}(S_{u,\ubar})}&=&a^{\frac{I}{2}+2}|u^{'}|^{-(I+4)}||\mathcal{E}||_{L^{2}_{sc}(S_{u^{'},\ubar})}.
\end{eqnarray}
The transport inequality in terms of the scale  invariant norm reads \footnote{The factor of $a/|u|$ is kept with $||\nabla^{I}\widetilde{\tr\chibar}||_{L^{2}_{sc}(S_{u,\ubar})}$ since we need to integrate in the $u$ direction with respect to the measure $\frac{a}{|u|^{2}}du$}
\begin{align} \nonumber
&\frac{a}{|u|}a^{\frac{I-1}{2}}||\nabla^{I}\widetilde{\tr\chibar}||_{L^{2}_{sc}(S_{u,\ubar})} \\ \lesssim & \frac{a}{|u_{\infty}|}a^{\frac{I-1}{2}}||\nabla^{I}\widetilde{\tr\chibar}||_{L^{2}_{sc}(S_{u_{\infty},\ubar})}\\\nonumber 
+&\int_{u_{\infty}}^{u}\frac{a^{2}}{|u^{'}|^{3}}||a^{\frac{I-1}{2}}\frac{1}{|u^{'}|^{2}}\sum_{J_1+J_2+J_3=I}\nabla^{J_1}(\eta+\etabar)^{J_2}\nabla^{J_3}(\Omega^{-1}-1)||_{L^{2}_{sc}(S_{u^{'},\ubar})}\\\nonumber 
+&\int_{u_{\infty}}^{u}\frac{a^{2}}{|u^{'}|^{3}}||a^{\frac{I-1}{2}}\sum_{J_1+J_2+J_3+J_4=I}\nabla^{J_1}(\eta+\etabar)^{J_2}\nabla^{J_3}\widetilde{\tr\chibar}\nabla^{J_4}\widetilde{\tr\chibar}||_{L^{2}_{sc}(S_{u^{'},\ubar})}\\\nonumber 
+&\int_{u_{\infty}}^{u}\frac{a^{2}}{|u^{'}|^{3}}||a^{\frac{I-1}{2}}\sum_{J_1+J_2+J_3+J_4=I}\nabla^{J_1}(\eta+\etabar)^{J_2}\nabla^{J_3}\omegabar\nabla^{J_4}\tr\chibar||_{L^{2}_{sc}(S_{u^{'},\ubar})}\\\nonumber 
+&\int_{u_{\infty}}^{u}\frac{a^{2}}{|u^{'}|^{3}}||a^{\frac{I-1}{2}}\sum_{J_1+J_2+J_3+J_4=I}\nabla^{J_1}(\eta+\etabar)^{J_2}\nabla^{J_3}\chibarhat\nabla^{J_4}\chibarhat||_{L^{2}_{sc}(S_{u^{'},\ubar})}\\\nonumber 
+&\int_{u_{\infty}}^{u}\frac{a^{2}}{|u^{'}|^{3}}||a^{\frac{I-1}{2}}\sum_{J_1+J_2+J_3+J_4=I}\nabla^{J_1}(\eta+\etabar)^{J_2}\hnabla^{J_3}\alphabar^{F}\hnabla^{J_4}\alphabar^{F}||_{L^{2}_{sc}(S_{u^{'},\ubar})}\\\nonumber 
+&\int_{u_{\infty}}^{u}\frac{a^{2}}{|u^{'}|^{3}}||a^{\frac{I-1}{2}}\sum_{J_1+J_2+J_3+J_4=I}\nabla^{J_1}(\eta+\etabar)^{J_2}\nabla^{J_3}(\chibarhat,\hsp \tildetr)\nabla^{J_4}\widetilde{\tr\chibar}||_{L^{2}_{sc}(S_{u^{'},\ubar})}\\\nonumber 
+&\int_{u_{\infty}}^{u}\frac{a^{2}}{|u^{'}|^{3}}||a^{\frac{I-1}{2}}\sum_{J_{1}+J_{2}+J_{3} +J_{4}=I-1}\nabla^{J_{1}}(\eta+\underline{\eta})^{J_{2} +1}\hnabla^{J_{3}}\tr\underline{\chi}\hnabla^{J_{4}}\widetilde{\tr\chibar}||_{L^{2}_{sc}(S_{u^{'},\ubar})}\\\nonumber 
+&\int_{u_{\infty}}^{u}\frac{a^{2}}{|u^{'}|^{3}}||a^{\frac{I-1}{2}}\sum_{J_1+J_2+J_3+J_4=I-1}\nabla^{J_1} (\eta+\etabar)^{J_2+1}\nabla^{J_3}(\chibarhat,\tr\chibar)\nabla^{J_4}\widetilde{\tr\chibar}||_{L^{2}_{sc}(S_{u^{'},\ubar})}\\ 
+&\int_{u_{\infty}}^{u}\frac{a^{2}}{|u^{'}|^{3}}||a^{\frac{I-1}{2}}\sum_{J_{1}+J_{2}+J_{3}+J_{4}+J_{5}=I-1}\nabla^{J_{1}}(\eta+\underline{\eta})^{J_{2}}\hnabla^{J_{3}}\alphabar^{F}\hnabla^{J_{4}}(\rho^{F},\sigma^{F})\nabla^{J_{5}}\widetilde{\tr\chibar}||_{L^{2}_{sc}(S_{u^{'},\ubar})}.
\end{align}
\noindent Now we estimate each term separately as follows:
\begin{eqnarray}
\frac{a}{|u_{\infty}|}a^{\frac{I-1}{2}}||\nabla^{I}\widetilde{\tr\chibar}||_{L^{2}_{sc}(S_{u_{\infty},\ubar})}\lesssim 1,
\end{eqnarray}
\begin{align}\nonumber &
\int_{u_{\infty}}^{u}\frac{a^{2}}{|u^{'}|^{3}}||a^{\frac{I-1}{2}}\frac{1}{|u^{'}|^{2}}\sum_{J_1+J_2+J_3=I}\nabla^{J_1}(\eta+\etabar)^{J_2}\nabla^{J_3}(\Omega^{-1}-1)||_{L^{2}_{sc}(S_{u^{'},\ubar})}\\\nonumber 
=& \int_{u_{\infty}}^{u}\frac{a^{2}}{|u^{'}|^{3}}||a^{\frac{I-1}{2}}\frac{1}{|u^{'}|^{2}}\sum_{J_1+J_2+J_3=I}\nabla^{J_1}(\eta+\etabar)^{J_2}\nabla^{J_3}(\int_{0}^{\ubar}\omega \hsp \text{d}\ubar^{'})||_{L^{2}_{sc}(S_{u^{'},\ubar})}\\ 
\lesssim& 1+\mathcal{\underline{\mathcal{R}}},
\end{align}
\begin{align} \nonumber
&\int_{u_{\infty}}^{u}\frac{a^{2}}{|u^{'}|^{3}}||a^{\frac{I-1}{2}}\sum_{J_1+J_2+J_3+J_4=I}\nabla^{J_1}(\eta+\etabar)^{J_2}\nabla^{J_3}\widetilde{\tr\chibar}\nabla^{J_4}\widetilde{\tr\chibar}||_{L^{2}_{sc}(S_{u^{'},\ubar})}\\
\lesssim & \int_{u_{\infty}}^{u}\frac{1}{|u^{'}|^{2}}\frac{a}{|u^{'}|}||a^{\frac{I-1}{2}}\nabla^{I}\widetilde{\tr\chibar}||_{L^{2}_{sc}(S_{u,\ubar})}du^{'},
\end{align}
\begin{align}\nonumber &
\int_{u_{\infty}}^{u}\frac{a^{2}}{|u^{'}|^{3}}||a^{\frac{I-1}{2}}\sum_{J_1+J_2+J_3+J_4=I}\nabla^{J_1}(\eta+\etabar)^{J_2}\nabla^{J_3}\omegabar\nabla^{J_4}\tr\chibar||_{L^{2}_{sc}(S_{u^{'},\ubar})}\\\nonumber 
\lesssim& \int_{u_{\infty}}^{u}\frac{a}{|u^{'}|^{3}}\frac{a}{|u^{'}|}||a^{\frac{I-1}{2}}\nabla^{I}\widetilde{\tr\chibar}||_{L^{2}_{sc}(S_{u,\ubar})}du^{'}+\int_{u_{\infty}}^{u}\frac{a}{|u^{'}|^{2}}||\nabla^{I}\omegabar||_{L^{2}_{sc}(S_{u^{'},\ubar})}\\\nonumber 
\lesssim & \int_{u_{\infty}}^{u}\frac{a}{|u^{'}|^{3}}\frac{a}{|u^{'}|}||a^{\frac{I-1}{2}}\nabla^{I}\widetilde{\tr\chibar}||_{L^{2}_{sc}(S_{u,\ubar})}du^{'}\\\nonumber +&\left(\int_{u_{\infty}}^{u}\frac{a}{|u^{'}|^{2}}du^{'}\right)^{\frac{1}{2}}\left(\int_{u_{\infty}}^{u}\frac{a}{|u^{'}|^{2}}||a^{\frac{I-1}{2}}\nabla^{I}\omegabar||^{2}_{L^{2}_{sc}(S_{u^{'},\ubar)}}du^{'}\right)^{\frac{1}{2}}\\ 
\lesssim & \int_{u_{\infty}}^{u}\frac{a}{|u^{'}|^{3}}\frac{a}{|u^{'}|}||a^{\frac{I-1}{2}}\nabla^{I}\widetilde{\tr\chibar}||_{L^{2}_{sc}(S_{u,\ubar})}du^{'}+\mathcal{R}+\mathbb{YM},
\end{align}
\begin{align} \nonumber
&\int_{u_{\infty}}^{u}\frac{a^{2}}{|u^{'}|^{3}}||a^{\frac{I-1}{2}}\sum_{J_1+J_2+J_3+J_4=I}\nabla^{J_1}(\eta+\etabar)^{J_2}\nabla^{J_3}\chibarhat\nabla^{J_4}\chibarhat||_{L^{2}_{sc}(S_{u^{'},\ubar})}\\ 
\lesssim  &\int_{u_{\infty}}^{u}\frac{a}{|u^{'}|^{2}}\frac{a^{\frac{1}{2}}}{|u^{'}|}||a^{\frac{I-1}{2}}\nabla^{I}\chibarhat||_{L^{2}_{sc}(S_{u^{'},\ubar})},
\end{align}
\begin{align} \nonumber
&\int_{u_{\infty}}^{u}\frac{a^{2}}{|u^{'}|^{3}}||a^{\frac{I-1}{2}}\sum_{J_1+J_2+J_3+J_4=I}\nabla^{J_1}(\eta+\etabar)^{J_2}\hnabla^{J_3}\alphabar^{F}\hnabla^{J_4}\alphabar^{F}||_{L^{2}_{sc}(S_{u^{'},\ubar})}\\ 
\lesssim & \left(\int_{u_{\infty}}^{u}\frac{a^{3}}{|u^{'}|^{6}}du^{'}\right)^{\frac{1}{2}}\left(\int_{u_{\infty}}^{u}\frac{a}{|u^{'}|^{2}}||a^{\frac{I-1}{2}}\hnabla^{I}\alphabar^{F}||^{2}_{L^{2}_{sc}(S_{u^{'}\ubar)}}du^{'}\right)^{\frac{1}{2}}
\lesssim  \frac{a^{\frac{3}{2}}}{|u|^{\frac{5}{2}}}\mathbb{YM}[\alphabar^{F}]\lesssim 1,
\end{align}
\begin{align} \nonumber
&\int_{u_{\infty}}^{u}\frac{a^{2}}{|u^{'}|^{3}}||a^{\frac{I-1}{2}}\sum_{J_1+J_2+J_3+J_4=I}\nabla^{J_1}(\eta+\etabar)^{J_2}\nabla^{J_3}(\chibarhat,\hsp \tildetr)\nabla^{J_4}\widetilde{\tr\chibar}||_{L^{2}_{sc}(S_{u^{'},\ubar})}\\ \nonumber
\lesssim & \int_{u_{\infty}}^{u}\frac{a^\frac{1}{2}}{|u^{'}|^{2}}\frac{a}{|u^{'}|}||a^{\frac{I-1}{2}}\nabla^{I}\widetilde{\tr\chibar}||_{L^{2}_{sc}(S_{u,\ubar})}du^{'}+\int_{u_{\infty}}^{u}\frac{a^\frac{1}{2}}{|u^{'}|^{2}}\frac{a^{\frac{1}{2}}}{|u^{'}|}||a^{\frac{I-1}{2}}\nabla^{I}\chibarhat||_{L^{2}_{sc}(S_{u^{'},\ubar})}du^{'}\\ 
+&\int_{u_{\infty}}^{u}\frac{1}{|u^{'}|^{2}}\frac{a}{|u^{'}|}||a^{\frac{I-1}{2}}\nabla^{I}\widetilde{\tr\chibar}||_{L^{2}_{sc}(S_{u,\ubar})}du^{'},
\end{align}
\begin{align} \nonumber &
\int_{u_{\infty}}^{u}\frac{a^{2}}{|u^{'}|^{3}}||a^{\frac{I-1}{2}}\sum_{J_1+J_2+J_3+J_4=I-1}\nabla^{J_1} (\eta+\etabar)^{J_2+1}\nabla^{J_3}(\chibarhat,\tr\chibar)\nabla^{J_4}\widetilde{\tr\chibar}||_{L^{2}_{sc}(S_{u^{'},\ubar})}\\ 
\lesssim &\int_{u_{\infty}}^{u}|u^{'}|^{-2}\Gamma \duprime\lesssim 1,
\end{align}and for the last term we have
\begin{align} \nonumber &
\int_{u_{\infty}}^{u}\frac{a^{2}}{|u^{'}|^{3}}||a^{\frac{I-1}{2}}\sum_{J_{1}+J_{2}+J_{3}+J_{4}+J_{5}=I-1}\nabla^{J_{1}}(\eta+\underline{\eta})^{J_{2}}\hnabla^{J_{3}}\alphabar^{F}\hnabla^{J_{4}}(\rho^{F},\sigma^{F})\nabla^{J_{5}}\widetilde{\tr\chibar}||_{L^{2}_{sc}(S_{u^{'},\ubar})}\\ 
\lesssim &\int_{u_{\infty}}^{u}\frac{a}{|u^{'}|^{4}}\Gamma\lesssim\frac{a}{|u|^{3}}\Gamma\lesssim 1.
\end{align}
A collection of all the terms yields 
\begin{align}\nonumber &
\frac{a}{|u|}a^{\frac{I-1}{2}}||\nabla^{I}\widetilde{\tr\chibar}||_{L^{2}_{sc}(S_{u,\ubar})}\\ \lesssim &1+\mathcal{R}+\int_{u_{\infty}}^{u}\frac{a}{|u^{'}|^{3}}\frac{a}{|u^{'}|}||a^{\frac{I-1}{2}}\nabla^{I}\widetilde{\tr\chibar}||_{L^{2}_{sc}(S_{u,\ubar})}du^{'}\\\nonumber 
+&\int_{u_{\infty}}^{u}\frac{a}{|u^{'}|^{2}}\frac{a^{\frac{1}{2}}}{|u^{'}|}||a^{\frac{I-1}{2}}\nabla^{I}\chibarhat||_{L^{2}_{sc}(S_{u^{'},\ubar})}+\int_{u_{\infty}}^{u}\frac{a^\frac{1}{2}}{|u^{'}|^{2}}\frac{a}{|u^{'}|}||a^{\frac{I-1}{2}}\nabla^{I}\widetilde{\tr\chibar}||_{L^{2}_{sc}(S_{u,\ubar})}du^{'}\\ +&\int_{u_{\infty}}^{u}\frac{a^\frac{1}{2}}{|u^{'}|^{2}}\frac{a^{\frac{1}{2}}}{|u^{'}|}||a^{\frac{I-1}{2}}\nabla^{I}\chibarhat||_{L^{2}_{sc}(S_{u^{'},\ubar})}du^{'}
+\int_{u_{\infty}}^{u}\frac{1}{|u^{'}|^{2}}\frac{a}{|u^{'}|}||a^{\frac{I-1}{2}}\nabla^{I}\widetilde{\tr\chibar}||_{L^{2}_{sc}(S_{u,\ubar})}du^{'}.
\end{align}
An application of Gr\"onwall's inequality leads to 
\begin{align}
\label{eq:widechibar} \nonumber
& \frac{a}{|u|}a^{\frac{I-1}{2}}||\nabla^{I}\widetilde{\tr\chibar}||_{L^{2}_{sc}(S_{u,\ubar})}\\ \nonumber \lesssim &\left(1+\mathcal{R}+\underline{\mathcal{R}}+\mathbb{YM}+\int_{u_{\infty}}^{u}\frac{a}{|u^{'}|^{2}}\frac{a^{\frac{1}{2}}}{|u^{'}|}||a^{\frac{I-1}{2}}\nabla^{I}\chibarhat||_{L^{2}_{sc}(S_{u^{'},\ubar})}\right.\\\nonumber & 
\left.+\int_{u_{\infty}}^{u}\frac{a^\frac{1}{2}}{|u^{'}|^{2}}\frac{a^{\frac{1}{2}}}{|u^{'}|}||a^{\frac{I-1}{2}}\nabla^{I}\chibarhat||_{L^{2}_{sc}(S_{u^{'},\ubar})}du^{'}\right)e^{\frac{a}{|u|^{2}}+\frac{a^{\frac{1}{2}}}{|u|}+\frac{1}{|u|}}\\ \nonumber
\lesssim &1+\mathcal{R}+\underline{\mathcal{R}}+\mathbb{YM}+\int_{u_{\infty}}^{u}\frac{a}{|u^{'}|^{2}}\frac{a^{\frac{1}{2}}}{|u^{'}|}||a^{\frac{I-1}{2}}\nabla^{I}\chibarhat||_{L^{2}_{sc}(S_{u^{'},\ubar})}\\ +&\int_{u_{\infty}}^{u}\frac{a^\frac{1}{2}}{|u^{'}|^{2}}\frac{a^{\frac{1}{2}}}{|u^{'}|}||a^{\frac{I-1}{2}}\nabla^{I}\chibarhat||_{L^{2}_{sc}(S_{u^{'},\ubar})}du^{'}.
\end{align}
Now we return to the elliptic  constraint equation 
\begin{eqnarray}
\text{div}\hat{\underline{\chi}}\sim \frac{1}{2}\nabla \tr\underline{\chi}-\frac{1}{2}(\underline{\eta}-\eta)\cdot(\hat{\underline{\chi}}-\frac{1}{2}\tr\underline{\chi}\gamma)-\widetilde{\underline{\beta}}+\alphabar^{F}\cdot(\rho^{F},\sigma^{F})
\end{eqnarray}
which yields 
\begin{align}&
\label{eq:chibarhatest}
||a^{\frac{I-1}{2}}\nabla^{I}\chibarhat||_{L^{2}_{sc}(S_{u,\ubar})}\lesssim \sum_{J\leq I-1}||a^{\frac{J}{2}}\nabla^{I+1}\widetilde{\tr\chibar}||_{L^{2}_{sc}(S_{u,\ubar})}+\sum_{J\leq I-1}||a^{\frac{I}{2}}\nabla^{I}\widetilde{\betabar}||_{L^{2}_{sc}(S_{u,\ubar})}\\\nonumber 
+&\frac{1}{a^{\frac{1}{2}}}\sum_{J\leq I-1}||a^{\frac{J}{2}}\nabla^{J}\chibarhat||_{L^{2}_{sc}(S_{u,\ubar})}+\sum_{J\leq I-1}||a^{\frac{J}{2}}\sum_{J_{1}+J_{2}=J}\nabla^{J_{1}}(\eta,\etabar)\nabla^{J_{2}}(\tr\chibar,\chibarhat)||_{L^{2}_{sc}(S_{u,\ubar})}.
\end{align}
Now substitute the estimate for $\widetilde{\tr\chibar}$ from (\ref{eq:widechibar}) into (\ref{eq:chibarhatest}) and apply Gr\"onwall to obtain 
\begin{eqnarray}
\frac{a}{|u|}a^{\frac{I-1}{2}}||\nabla^{I}\widetilde{\tr\chibar}||_{L^{2}_{sc}(S_{u,\ubar})}\lesssim (1+\mathcal{R}+\underline{\mathcal{R}}+\mathbb{YM})e^{\frac{a^{\frac{1}{2}}}{|u|}+\frac{1}{|u|}}\lesssim 1+\mathcal{R}+\underline{\mathcal{R}}+\mathbb{YM}.
\end{eqnarray}
Now (\ref{eq:chibarhatest}) yields 
\begin{align}\nonumber &
\int_{u_{\infty}}^{u}\frac{a}{|u^{'}|^{2}}\frac{a^{\frac{1}{2}}}{|u^{'}|}||a^{\frac{I-1}{2}}\nabla^{I}\chibarhat||_{L^{2}_{sc}(S_{u^{'},\ubar})}\\ \lesssim \nonumber &\int_{u_{\infty}}^{u}\frac{a}{|u^{'}|^{2}}\frac{a^{\frac{1}{2}}}{|u^{'}|}\sum_{J\leq I-1}||a^{\frac{J}{2}}\nabla^{I+1}\widetilde{\tr\chibar}||_{L^{2}_{sc}(S_{u^{'},\ubar})}du^{'}
+\int_{u_{\infty}}^{u}\frac{a}{|u^{'}|^{2}}\frac{a^{\frac{1}{2}}}{|u^{'}|}\sum_{J\leq I-1}||a^{\frac{I}{2}}\nabla^{I}\widetilde{\betabar}||_{L^{2}_{sc}(S_{u^{'},\ubar})}du^{'}\\\nonumber 
\lesssim &(1+\mathcal{R}+\underline{\mathcal{R}}+\mathbb{YM})\int_{u_{\infty}}^{u}\frac{a^{\frac{1}{2}}}{|u^{'}|^{2}}du^{'}\\\nonumber +&\left(\int_{u_{\infty}}^{u}\frac{a^{2}}{|u^{'}|^{4}}du^{'}\right)^{\frac{1}{2}}\left(\int_{u_{\infty}}^{u}\frac{a}{|u^{'}|^{2}}\sum_{J\leq I-1}||a^{\frac{I}{2}}\nabla^{I}\widetilde{\betabar}||^{2}_{L^{2}_{sc}(S_{u^{'},\ubar})}du^{'}\right)^{\frac{1}{2}}+1
\lesssim 1.
\end{align}
This concludes the proof of the lemma. 
\end{proof}

\begin{remark}\label{usefulremark}
In the same way as above, one can obtain the following estimates:

\begin{equation}
\begin{split}
    \label{usefulhere} &\left(\intu \frac{a^2}{\upr^4} \scaletwoSuprime{(\al)^{N+4}\nabla^{N+5}\chibarhat}^2\duprime \right)^{\frac{1}{2}} \\+& \left(\intu \frac{a^3}{\upr^4} \scaletwoSuprime{(\al)^{N+4}\nabla^{N+5}\tr\chibar}^2\duprime \right)^{\frac{1}{2}} \lesssim \mathcal{R} + \underline{\mathcal{R}}+1.
\end{split}
\end{equation}
\end{remark}
\section{Energy Estimates Weyl and Yang-Mills curvature components}
\noindent In this section we finish the final estimates, namely those on the energy associated with the Weyl and Yang-Mills curvature. There are several ways to execute this step. Since both the Einstein and the Yang-Mills equations are derivable from a Lagrangian, one would expect the existence of stress-energy tensors. For pure gravity, however, one cannot simply find a canonical stress-energy tensor due to the equivalence principle. Nevertheless, one can utilize the Bel-Robinson tensor to construct suitable energy currents or use the quasi-local type energy constructed by \cite{Wang1,Wang2}. In the case of Yang-Mills theory, there exists a canonical stress-energy tensor that can be used directly to construct necessary energies. However, in the current context where we impose a double-null foliation, the use of Bel-Robionson or Yang-Mills stress energy tensors can be avoided altogether, since, in this setting, both the Bianchi and the Yang-Mills equations can be cast into a coupled symmetric hyperbolic system, as we have already shown in Proposition \ref{hyperbolic}. Therefore, we can use the direct integration by parts approach. To this end, we require the following pair of propositions which we prove using Lemma \ref{integration} and Proposition \ref{hyperbolic}. 
\begin{proposition}\label{integration2}
If $\varphi\in \text{section}(^{k}\otimes T^{*}S_{u,\ubar}\otimes \mathfrak{P}_{Ad,\mathfrak{g}})$ and  $\lambda_{1}=2\lambda_{0}-1$, then the following estimates hold:
\begin{align} \nonumber
2\int_{D_{u,\ubar}}|u^{'}|^{2\lambda_{1}}\langle\varphi,(\hnabla_{3}+\lambda_{0}\tr\chibar)\varphi\rangle &=\int_{H_{u}(0,\ubar)}|u^{'}|^{2\lambda_{1}}|\varphi|^{2}-\int_{H_{u_{\infty}}(0,\ubar)}|u_{\infty}|^{2\lambda_{1}}|\varphi|^{2}\\\nonumber 
&+\int_{D_{u,\ubar}}|u^{'}|^{2\lambda_{1}}f|\varphi|^{2},
\end{align}
where $f$ verifies the estimate $f\lesssim \Gamma|u|^{-2}$.
\end{proposition}
\begin{proof}
The proof follows in an exact same way as that of proposition (7.3) of \cite{AnAth}. The only difference is that $\varphi$ is a section of the fibered vector bundle $^{k}\otimes T^{*}S_{u,\ubar}\otimes \mathfrak{P}_{Ad,\mathfrak{g}}$. However, note that $|\varphi|^{2}$ is a gauge-invariant object and therefore 
\begin{eqnarray}
\nabla_{3}|\varphi|^{2}=2\langle\varphi,\hnabla_{3}\varphi\rangle,
\end{eqnarray}
due to the fact that the gauge covariant connection $\hnabla_{3}$ is metrics compatible.
\end{proof}
\noindent We will need Proposition $6.1$ for proving the following proposition.

\begin{proposition}\label{42}
For a Bianchi pair $(\Psi_1, \Psi_2)\in \left\{(\alpha,\widetilde{\beta}),(\widetilde{\beta},(\rho,\sigma)),((\rho,\sigma),\widetilde{\betabar}),(\widetilde{\betabar},\alphabar)\right\}\cup\left\{(\alpha_{F},(\rho^{F},\sigma^{F})),((\rho^{F},\sigma^{F}),\alphabar^{F})\right\}$ satisfying
\[ \hnabla_3 \hnabla^i \Psi_1 + \left( \frac{i+1}{2} + s_2(\Psi_1) \right) \tr\chibar \hnabla^i \Psi_1 - \widehat{\mathcal{D}}\hnabla^i \Psi_2 = P_i,\]\[\hnabla_4 \hnabla^i \Psi_2 - \Hodge{\widehat{\mathcal{D}}} \hnabla^i \Psi_1 = Q_i, \]the following holds true:
\end{proposition}
\begin{equation}
\begin{split}
&\int_{\Hu}\scaletwoSuubarprime{\hnabla^i \Psi_1}^2 \dubarprime + \int_{\Hbu}\frac{a}{\upr^2} \scaletwoSuprime{\hnabla^i \Psi_2}^2 \duprime \\ \lesssim &\int_{H_{u_{\infty}}^{(0,\ubar)}}\scaletwoSuzubarprime{\hnabla^i \Psi_1}^2 \dubarprime + \int_{\Hbar_0^{(u_{\infty},u)}} \frac{a}{\upr^2} \scaletwoSuzprime{\hnabla^i \Psi_2}^2 \duprime \\&+  \iint_{\mathcal{D}_{u,\ubar}}\frac{a}{\upr} \scaleoneSuprimeubarprime{\langle\hnabla^i \Psi_1 , P_i\rangle}\duprime \dubarprime + \iint_{\mathcal{D}_{u,\ubar}}\frac{a}{\upr} \scaleoneSuprimeubarprime{\langle\hnabla^i \Psi_2 ,Q_i\rangle}\duprime \dubarprime.
\end{split}
\end{equation}
\begin{proof}
This property essentially indicates the symmetric hyperbolic character of the null Yang-Mills equations. First note the following relation:
\begin{align}
\nonumber &\int_{S_{u,\ubar}}\langle\Psi_{1},\widehat{\mathcal{D}}\Psi_{2}\rangle+\int_{S_{u,\ubar}}\langle\Psi_{2},~^{*}\widehat{\mathcal{D}}\Psi_{2}\rangle\\ \nonumber =& \int_{S_{u,\ubar}}\nabla\langle\Psi_{1},\Psi_{2}\rangle-\int_{S_{u,\ubar}}\langle^{*}\widehat{\mathcal{D}}\Psi_{1},\Psi_{2}\rangle\nonumber+\int_{S_{u,\ubar}}\langle\Psi_{2},^{*}\widehat{\mathcal{D}}\Psi_{1}\rangle
-\int_{S_{u,\ubar}}(\eta+\etabar)\langle\Psi_{1},\Psi_{2}\rangle\\\nonumber 
=&-\int_{S_{u,\ubar}}(\eta+\etabar)\langle\Psi_{1},\Psi_{2}\rangle,
\end{align}
due to the fact that $\langle\Psi_{1},\Psi_{2}\rangle$ is gauge-invariant (see Lemma \ref{integration}). Similar identities hold for the higher derivatives as well:
\begin{eqnarray}
\label{eq:bypartshodge}
\int_{S_{u,\ubar}}\langle\hnabla^{I}\Psi_{1},\widehat{\mathcal{D}}\hnabla^{I}\Psi_{2}\rangle+\int_{S_{u,\ubar}}\langle\hnabla^{I}\Psi_{2},~^{*}\widehat{\mathcal{D}}\hnabla^{I}\Psi_{2}\rangle=-\int_{S_{u,\ubar}}(\eta+\etabar)\langle\hnabla^{I}\Psi_{1},\hnabla^{I}\Psi_{2}\rangle.
\end{eqnarray}
With these identities in mind, we apply Proposition \ref{integration2} with $\lambda_{0}=\frac{I+1}{2}+s_{2}(\Psi_{1})$ and $\lambda_{1}=2\lambda_{0}-1=I+2s_{2}(\Psi_{2})$ to yield 
\begin{align} \nonumber
&2\int_{D_{u,\ubar^{'}}}|u^{'}|^{2I+4s_{2}(\Psi_{1})}\langle\hnabla^{I}\Psi_{1},(\hnabla_{3}+(\frac{I+1}{2}+s_{2}(\Psi_{1}))\tr\chibar)\hnabla^{I}\Psi_{1}\rangle=\\ &\int_{H_{u}(0,\ubar)}|u|^{2I+4s_{2}(\Psi_{1})}|\hnabla^{I}\Psi_{1}|^{2}
-\int_{H_{u_{\infty}}(0,\ubar)}|u_{\infty}|^{2I+4s_{2}(\Psi_{1})}|\nabla^{I}\Psi_{1}|^{2}+\int_{D_{u,\ubar}}|u^{'}|^{2I+4s_{2}(\Psi_{1})}f|\hnabla^{I}\psi_{1}|^{2},
\end{align}
where $f$ verifies $|f|\lesssim \Gamma|u|^{-2}$. For $\Psi_{2}$ we have a simpler expression using the integration lemma:
\begin{align} \nonumber
&2\int_{D_{u,\ubar}}|u^{'}|^{2I+4s_{2}(\Psi_{1})}\langle\hnabla^{I}\Psi_{2},\hnabla_{4}\hnabla^{I}\Psi_{2}\rangle= \nonumber \int_{\Hbar_{\ubar}(u_{\infty},u)}|u^{'}|^{2I+4s_{2}}|\hnabla^{I}\Psi_{2}|^{2}\nonumber\\ -& \int_{\Hbar_{0}(u_{\infty},u)}|u^{'}|^{2I+4s_{2}}|\hnabla^{I}\Psi_{2}|^{2}
+\int_{D_{u,\ubar}}|u^{'}|^{2I+4s_{2}(\Psi_{1})}(2\omega-\tr\chi)|\hnabla^{I}\Psi_{2}|^{2}.
\end{align}
Adding the previous two equations together we get
\begin{align}  \nonumber
&\int_{H_{u}(0,\ubar)}|u|^{2I+4s_{2}(\Psi_{1})}|\hnabla^{I}\Psi_{1}|^{2}+\int_{\Hbar_{\ubar}(u_{\infty},u)}|u^{'}|^{2I+4s_{2}}|\hnabla^{I}\Psi_{2}|^{2}\\   \nonumber =& \int_{H_{u_{\infty}}(0,\ubar)}|u_{\infty}|^{2I+4s_{2}(\Psi_{1})}|\nabla^{I}\Psi_{1}|^{2}
+ \int_{\Hbar_{0}(u_{\infty},u)}|u^{'}|^{2I+4s_{2}}|\hnabla^{I}\Psi_{2}|^{2}\\  \nonumber +& 2\int_{D_{u,\ubar^{'}}}|u^{'}|^{2I+4s_{2}(\Psi_{1})}\langle\hnabla^{I}\Psi_{1},(\hnabla_{3}+(\frac{I+1}{2}+s_{2}(\Psi_{1}))\tr\chibar)\hnabla^{I}\Psi_{1}\rangle\\\nonumber 
+&2\int_{D_{u,\ubar}}|u^{'}|^{2I+4s_{2}(\Psi_{1})}\langle\hnabla^{I}\Psi_{2},\hnabla_{4}\hnabla^{I}\Psi_{2}\rangle-\int_{D_{u,\ubar}}|u^{'}|^{2I+4s_{2}(\Psi_{1})}f|\hnabla^{I}\Psi_{1}|^{2}\\ 
-&\int_{D_{u,\ubar}}|u^{'}|^{2I+4s_{2}(\Psi_{1})}(2\omega-\tr\chi)|\hnabla^{I}\Psi_{2}|^{2}.
\end{align}
Now we utilize the equations of motion to yield 
\begin{align} \nonumber &
\int_{H_{u}(0,\ubar)}|u|^{2I+4s_{2}(\Psi_{1})}|\hnabla^{I}\Psi_{1}|^{2}+\int_{\Hbar_{\ubar}(u_{\infty},u)}|u^{'}|^{2I+4s_{2}}|\hnabla^{I}\Psi_{2}|^{2}\\  \nonumber =&\int_{H_{u_{\infty}}(0,\ubar)}|u_{\infty}|^{2I+4s_{2}(\Psi_{1})}|\nabla^{I}\Psi_{1}|^{2}\\\nonumber 
+&\int_{\Hbar_{0}(u_{\infty},u)}|u^{'}|^{2I+4s_{2}}|\hnabla^{I}\Psi_{2}|^{2}+2\int_{D_{u,\ubar^{'}}}|u^{'}|^{2I+4s_{2}(\Psi_{1})}\langle\hnabla^{I}\Psi_{1},\widehat{\mathcal{D}}\hnabla^{I} \Psi_{2}+P\rangle\\\nonumber 
+&2\int_{D_{u,\ubar}}|u^{'}|^{2I+4s_{2}(\Psi_{1})}\langle\hnabla^{I}\Psi_{2},~^{*}\widehat{\mathcal{D}}\hnabla^{I}\Psi_{1}+Q\rangle-\int_{D_{u,\ubar}}|u^{'}|^{2I+4s_{2}(\Psi_{1})}f|\hnabla^{I}\Psi_{1}|^{2}\\
-&\int_{D_{u,\ubar}}|u^{'}|^{2I+4s_{2}(\Psi_{1})}(2\omega-\tr\chi)|\hnabla^{I}\Psi_{2}|^{2}.
\end{align}
An application of \eqref{eq:bypartshodge} yields
\begin{align}
&\int_{H_{u}(0,\ubar)}|u|^{2I+4s_{2}(\Psi_{1})}|\hnabla^{I}\Psi_{1}|^{2}+\int_{\Hbar_{\ubar}(u_{\infty},u)}|u^{'}|^{2I+4s_{2}}|\hnabla^{I}\Psi_{2}|^{2}\nonumber\\  \nonumber =&\int_{H_{u_{\infty}}(0,\ubar)}|u_{\infty}|^{2I+4s_{2}(\Psi_{1})}|\nabla^{I}\Psi_{1}|^{2} 
+\int_{\Hbar_{0}(u_{\infty},u)}|u^{'}|^{2I+4s_{2}}|\hnabla^{I}\Psi_{2}|^{2}\\ \nonumber +& 2\int_{D_{u,\ubar^{'}}}|u^{'}|^{2I+4s_{2}(\Psi_{1})}\langle\hnabla^{I}\Psi_{1},P\rangle
\\  \nonumber+ & 2\int_{D_{u,\ubar^{'}}}|u^{'}|^{2I+4s_{2}(\Psi_{1})}\langle(\eta+\etabar)\hnabla^{I}\Psi_{1},\hnabla^{I}\Psi_{2}\rangle+2\int_{D_{u,\ubar}}|u^{'}|^{2I+4s_{2}(\Psi_{1})}\langle\hnabla^{I}\Psi_{2},Q\rangle\\ -&\int_{D_{u,\ubar}}|u^{'}|^{2I+4s_{2}(\Psi_{1})}f|\hnabla^{I}\Psi_{1}|^{2}
-\int_{D_{u,\ubar}}|u^{'}|^{2I+4s_{2}(\Psi_{1})}(2\omega-\tr\chi)|\hnabla^{I}\Psi_{2}|^{2}.
\end{align}
Now notice that $P$ and $Q$ are lower order terms. In this integration by parts procedure, all the principal terms are cancelled. Using the definition of the scale invariant norms and the boot-strap assumptions \eqref{bootstrap}, we observe that $|\omega|\lesssim |u|^{-1}, |\eta,\etabar|\lesssim |u|^{-2},~|f|\lesssim |u|^{-2}$, and $|\tr\chi|\lesssim |u|^{-1}$. We need to apply Gr\"onwall's inequality twice. Notice the following important points 
\begin{align}
|\int_{D_{u,\ubar}}|u^{'}|^{2I+4s_{2}(\Psi_{1})}(2\omega-\tr\chi)|\hnabla^{I}\Psi_{2}|^{2}|\lesssim &\int_{0}^{\ubar}|u|^{-1}\int_{\Hbar(u^{\infty},u)}|u^{'}|^{2I+4s_{2}(\Psi_{1})}|\hnabla^{I}\Psi_{2}|^{2}|,\\ \nonumber
|\int_{D_{u,\ubar^{'}}}|u^{'}|^{2I+4s_{2}(\Psi_{1})}\langle(\eta+\etabar)\hnabla^{I}\Psi_{1},\hnabla^{I}\Psi_{2}\rangle|\lesssim& \int_{u_{\infty}}^{u}|u^{'}|^{-2}\int_{H(0,\ubar)}|u^{'}|^{2I+4s_{2}(\Psi_{1})}|\hnabla^{I}\Psi_{1}|^{2}\\ 
+&\int_{0}^{\ubar}|u|^{-2}\int_{\Hbar(u_{\infty},u)}|u^{'}|^{2I+4s_{2}(\Psi_{1})}|\hnabla^{I}\Psi_{2}|^{2},\\
|\int_{D_{u,\ubar}}|u^{'}|^{2I+4s_{2}(\Psi_{1})}f|\hnabla^{I}\Psi_{1}|^{2}|\lesssim& \int_{u_{\infty}}^{u}|u^{'}|^{-2}\int_{H(0,\ubar)}|u^{'}|^{2I+4s_{2}(\Psi_{1})}|\hnabla^{I}\Psi_{1}|^{2}
\end{align}
since $0>u^{'}\in [u_{\infty},u]$. Therefore applying Gr\"onwall in $u^{'}$, we obtain 
\begin{align}\nonumber
& \int_{H_{u}(0,\ubar)}|u|^{2I+4s_{2}(\Psi_{1})}|\hnabla^{I}\Psi_{1}|^{2}\\ \lesssim &\bigg(\int_{H_{u_{\infty}}(0,\ubar)}|u_{\infty}|^{2I+4s_{2}(\Psi_{1})}|\nabla^{I}\Psi_{1}|^{2} 
+\int_{\Hbar_{0}(u_{\infty},u)}|u^{'}|^{2I+4s_{2}}|\hnabla^{I}\Psi_{2}|^{2}\\\nonumber +& 2|\int_{D_{u,\ubar^{'}}}|u^{'}|^{2I+4s_{2}(\Psi_{1})}\langle\hnabla^{I}\Psi_{1},P\rangle|
+2|\int_{D_{u,\ubar}}|u^{'}|^{2I+4s_{2}(\Psi_{1})}\langle\hnabla^{I}\Psi_{2},Q\rangle|\\\nonumber +& \int_{0}^{\ubar}|u|^{-1}\int_{\Hbar(u^{\infty},u)}|u^{'}|^{2I+4s_{2}(\Psi_{1})}|\hnabla^{I}\Psi_{2}|^{2}|\bigg)e^{\int_{u_{\infty}}^{u}|u^{'}|^{-2}du^{'}}\\\nonumber 
\lesssim &\bigg(\int_{H_{u_{\infty}}(0,\ubar)}|u_{\infty}|^{2I+4s_{2}(\Psi_{1})}|\nabla^{I}\Psi_{1}|^{2} 
+\int_{\Hbar_{0}(u_{\infty},u)}|u^{'}|^{2I+4s_{2}}|\hnabla^{I}\Psi_{2}|^{2}\\\nonumber +& 2|\int_{D_{u,\ubar^{'}}}|u^{'}|^{2I+4s_{2}(\Psi_{1})}\langle\hnabla^{I}\Psi_{1},P\rangle|
+2|\int_{D_{u,\ubar}}|u^{'}|^{2I+4s_{2}(\Psi_{1})}\langle\hnabla^{I}\Psi_{2},Q\rangle|\\\nonumber +&\int_{0}^{\ubar}|u|^{-1}\int_{\Hbar(u^{\infty},u)}|u^{'}|^{2I+4s_{2}(\Psi_{1})}|\hnabla^{I}\Psi_{2}|^{2}|\bigg),
\end{align}
due to integrability of $|u^{'}|^{-2}$ at infinity. Now we need to apply the Gr\"onwall in the $\ubar$ direction. We obtain 
\begin{align} \nonumber
&\int_{\Hbar_{\ubar}(u_{\infty},u)}|u^{'}|^{2I+4s_{2}}|\hnabla^{I}\Psi_{2}|^{2}\\ \nonumber  \lesssim & \bigg(\int_{H_{u_{\infty}}(0,\ubar)}|u_{\infty}|^{2I+4s_{2}(\Psi_{1})}|\nabla^{I}\Psi_{1}|^{2} 
+\int_{\Hbar_{0}(u_{\infty},u)}|u^{'}|^{2I+4s_{2}(\Psi_{1})}|\hnabla^{I}\Psi_{2}|^{2}\\\nonumber +& 2|\int_{D_{u,\ubar^{'}}}|u^{'}|^{2I+4s_{2}(\Psi_{1})}\langle\hnabla^{I}\Psi_{1},P\rangle|
+2|\int_{D_{u,\ubar}}|u^{'}|^{2I+4s_{2}(\Psi_{1})}\langle\hnabla^{I}\Psi_{2},Q\rangle|\bigg)e^{\frac{\ubar}{|u|}}\\\nonumber 
\lesssim &\int_{H_{u_{\infty}}(0,\ubar)}|u_{\infty}|^{2I+4s_{2}(\Psi_{1})}|\nabla^{I}\Psi_{1}|^{2} 
+\int_{\Hbar_{0}(u_{\infty},u)}|u^{'}|^{2I+4s_{2}(\Psi_{1})}|\hnabla^{I}\Psi_{2}|^{2}\\\nonumber +&2|\int_{D_{u,\ubar^{'}}}|u^{'}|^{2I+4s_{2}(\Psi_{1})}\langle\hnabla^{I}\Psi_{1},P\rangle|
+2|\int_{D_{u,\ubar}}|u^{'}|^{2I+4s_{2}(\Psi_{1})}\langle\hnabla^{I}\Psi_{2},Q\rangle|.
\end{align}
Putting everything together, we obtain 
\begin{align}
    \nonumber &\int_{H_{u}(0,\ubar)}|u|^{2I+4s_{2}(\Psi_{1})}|\hnabla^{I}\Psi_{1}|^{2}+\int_{\Hbar_{\ubar}(u_{\infty},u)}|u^{'}|^{2I+4s_{2}(\Psi_{1})}|\hnabla^{I}\Psi_{2}|^{2}\\\nonumber 
\lesssim &\int_{H_{u_{\infty}}(0,\ubar)}|u_{\infty}|^{2I\nonumber+4s_{2}(\Psi_{1})}|\nabla^{I}\Psi_{1}|^{2} 
+\int_{\Hbar_{0}(u_{\infty},u)}|u^{'}|^{2I+4s_{2}(\Psi_{1})}|\hnabla^{I}\Psi_{2}|^{2}\\ + &|\int_{D_{u,\ubar^{'}}}|u^{'}|^{2I+4s_{2}(\Psi_{1})}\langle\hnabla^{I}\Psi_{1},P\rangle|
+|\int_{D_{u,\ubar}}|u^{'}|^{2I+4s_{2}(\Psi_{1})}\langle\hnabla^{I}\Psi_{2},Q\rangle| .
\end{align}
In order to convert this inequality in terms of scale invariant norms, we need to recall the signature of each term involved and the signature difference between $\Psi_{1}$ and $\Psi_{2}$. Firstly note the following
\begin{align}
 \nonumber &s_{2}(\Psi_{2})=s_{2}(\Psi_{1})+\frac{1}{2},~s_{2}(\nabla^{I}\Psi_{1})\nonumber=\frac{I}{2}+s_{2}(\Psi_{1}),~s_{2}(P)=s_{2}(\hnabla_{3}\hnabla^{I}\Psi_{1})=\frac{I}{2}+1+s_{2}(\Psi_{1})\\ 
&s_{2}(Q)=s_{2}(~^{*}\widehat{\mathcal{D}}\hnabla^{I}\Psi_{1})=\frac{I+1}{2}+s_{2}(\Psi_{1}),
\end{align}
which imply 
\begin{eqnarray}
s_{2}(\nabla^{I}\Psi_{2})=\frac{I}{2}+s_{2}(\Psi_{2})=\frac{I+1}{2}+s_{2}(\Psi_{1}).
\end{eqnarray}
We consider each term separately: 
\begin{align} \nonumber
&\int_{H_{u}(0,\ubar)}|u|^{2I+4s_{2}(\Psi_{1})}|\hnabla^{I}\Psi_{1}|^{2}=\int_{0}^{\ubar}|u|^{2I+4s_{2}(\Psi_{1})}||\hnabla^{I}\Psi_{1}||^{2}_{L^{2}(S_{u,\ubar^{'}})}d\ubar^{'}\\\nonumber =&\int_{0}^{\ubar}|u|^{2I+4s_{2}(\Psi_{1})}a^{I+2s_{2}(\Psi_{1})}|u|^{-2I-4s_{2}(\Psi_{1})}||\hnabla^{I}\Psi_{1}||^{2}_{L^{2}_{sc}(S_{u,\ubar^{'}})}d\ubar^{'}\\
=&\int_{0}^{\ubar}a^{I+2s_{2}(\Psi_{1})}||\hnabla^{I}\Psi_{1}||^{2}_{L^{2}_{sc}(S_{u,\ubar^{'}})}d\ubar^{'},
\end{align}
\begin{align} \nonumber &
\int_{\Hbar_{\ubar}(u_{\infty},u)}|u^{'}|^{2I+4s_{2}(\Psi_{1})}|\hnabla^{I}\Psi_{2}|^{2}=\int_{u_{\infty}}^{u}|u^{'}|^{2I+4s_{2}(\Psi_{1})}||\hnabla^{I}\Psi_{2}||^{2}_{L^{2}(S_{u^{'},\ubar})}du^{'}\\\nonumber
=&\int_{u_{\infty}}^{u}|u^{'}|^{2I+4s_{2}(\Psi_{1})}a^{I+1+2s_{2}(\Psi_{1})}|u|^{-2(I+1)-4s_{2}(\Psi_{1})}||\hnabla^{I}\Psi_{2}||^{2}_{L^{2}_{sc}(S_{u^{'},\ubar})}du^{'}\\ 
=&\int_{u_{\infty}}^{u}a^{I+1+2s_{2}(\Psi_{1})}|u^{'}|^{-2}||\hnabla^{I}\Psi_{2}||^{2}_{L^{2}_{sc}(S_{u^{'},\ubar})}du^{'},
\end{align}
\begin{align} \nonumber
&\int_{D_{u,\ubar^{'}}}|u^{'}|^{2I+4s_{2}(\Psi_{1})}|\langle\hnabla^{I}\Psi_{1},P\rangle|=\int_{u,\ubar}|u^{'}|^{2I+4s_{2}(\Psi_{1})}||\langle\hnabla^{I}\Psi_{1},P\rangle||_{L^{1}(S_{u^{'},\ubar^{'}})}du^{'}d\ubar^{'}\\\nonumber 
=&\int_{u,\ubar}|u^{'}|^{2I+4s_{2}(\Psi_{1})}a^{I+1+2s_{2}(\Psi_{1})}|u^{'}|^{-2I-1-4s_{2}(\Psi_{1})}||\langle\hnabla^{I}\Psi_{1},P\rangle||_{L^{1}_{sc}(S_{u^{'},\ubar^{'}})}du^{'}d\ubar^{'}\\ 
=&\int_{u,\ubar}a^{I+1+2s_{2}(\Psi_{1})}|u^{'}|^{-1}||\langle\hnabla^{I}\Psi_{1},P\rangle||_{L^{1}_{sc}(S_{u^{'},\ubar^{'}})}du^{'}d\ubar^{'},
\end{align}
and 
\begin{align}  \nonumber
&\int_{D_{u,\ubar}}|u^{'}|^{2I+4s_{2}(\Psi_{1})}\langle\hnabla^{I}\Psi_{2},Q\rangle|=\int_{u,\ubar}|u^{'}|^{2I+4s_{2}(\Psi_{1})}||\langle\hnabla^{I}\Psi_{2},Q\rangle||_{L^{1}(S_{u^{'},\ubar^{'}})}du^{'}d\ubar^{'}\\\nonumber 
=&\int_{u,\ubar}|u^{'}|^{2I+4s_{2}(\Psi_{1})}a^{I+1+2s_{2}(\Psi_{1})}|u^{'}|^{-2I-1-4s_{2}(\Psi_{1})}||\langle\hnabla^{I}\Psi_{2},Q\rangle||_{L^{1}_{sc}(S_{u^{'},\ubar^{'}})}du^{'}d\ubar^{'}\\ 
=&\int_{u,\ubar}a^{I+1+2s_{2}(\Psi_{1})}|u^{'}|^{-1}||\langle\hnabla^{I}\Psi_{2},Q\rangle||_{L^{1}_{sc}(S_{u^{'},\ubar^{'}})}du^{'}d\ubar^{'}.
\end{align}
Collecting all the terms yield 
\begin{equation}
\begin{split}
&\int_{\Hu}\scaletwoSuubarprime{\hnabla^I \Psi_1}^2 \dubarprime + \int_{\Hbu}\frac{a}{\upr^2} \scaletwoSuprime{\hnabla^I \Psi_2}^2 \duprime \\ \lesssim &\int_{H_{u_{\infty}}^{(0,\ubar)}}\scaletwoSuzubarprime{\hnabla^I \Psi_1}^2 \dubarprime + \int_{\Hbar_0^{(u_{\infty},u)}} \frac{a}{\upr^2} \scaletwoSuzprime{\hnabla^I \Psi_2}^2 \duprime \\+&  \iint_{\mathcal{D}_{u,\ubar}}\frac{a}{\upr} \scaleoneSuprimeubarprime{\langle\hnabla^I \Psi_1 , P\rangle}\duprime \dubarprime + \iint_{\mathcal{D}_{u,\ubar}}\frac{a}{\upr} \scaleoneSuprimeubarprime{\langle\hnabla^I \Psi_2 ,Q\rangle}\duprime \dubarprime.
\end{split}
\end{equation}
\end{proof}

\noindent Before embarking on the energy estimates, we provide a final helpful proposition, which can be found for example in \cite{Kl-Rod}.

\begin{proposition}\label{prop54}
Let $f(x,y), \hsp g(x,y)$ be positive functions defined on a rectangle $U := \begin{Bmatrix} 0\leq x \leq x_0, \hsp 0\leq y \leq y_0 \end{Bmatrix}.       $ Suppose there exist constants $J ,c_1, c_2$ such that

\[ f(x,y)+g(x,y) \lesssim J + c_1 \int_0^{x} f(x^{\prime},y) \hsp \text{d}x^{\prime} +\int_0^{y} g(x,y^{\prime}) \hsp \text{d}y^{\prime},\]for all $(x,y) \in U$. Then there holds

\[ \forall (x,y) \in U: \hspace{5mm} f(x,y)+g(x,y) \lesssim J \mathrm{e}^{c_1 x + c_2y}. \]

\end{proposition}

\subsection{Energy Estimates for the Yang-Mills components}

\noindent We begin with the pair $(\alphaF,\ (\rhoF,\sigmaF))$.

\begin{proposition}
Under the assumptions of Theorem \ref{main1} and the bootstrap assumptions, there holds, for $0 \leq i \leq N+5$:

\begin{equation}
    \begin{split}
     &\frac{1}{\al} \scaletwoHu{(\al)^{i-1}\hnabla^i \alphaF} + \frac{1}{\al}\scaletwoHbaru{(\al)^{i-1} \hnabla^i (\rhoF,\sigmaF)} \\ \lesssim & \frac{1}{\al} \scaletwoHzero{(\al)^{i-1}\hnabla^i \alphaF} + \frac{1}{\al}\scaletwoHbarzero{(\al)^{i-1} \hnabla^i (\rhoF,\sigmaF)}  +1.
    \end{split}
\end{equation}

\end{proposition}
\begin{proof}
We consider the schematic equations for $\alphaF, (-\rhoF,\sigmaF)$:

\begin{equation}
    \hnabla_3 \alphaF +\frac{1}{2}\tr\chibar \alphaF - \mathcal{D}(-\rhoF,\sigmaF) = \eta (\rhoF,\sigmaF) + \omegabar \alphaF + \chihat \alphabarF, 
\end{equation}

\begin{equation}
    \hnabla_4 (-\rhoF,\sigmaF) -\Hodge{\mathcal{D}}\alphaF = \tr\chi(\rhoF,\sigmaF) +(\eta,\etabar)\alphaF.
\end{equation}Commuting with $i$ angular derivatives, we arrive at the following two:

\begin{equation}\label{alphafenexpression}
    \begin{split}
        &\hnabla_3 \hnabla^i \alphaF + \frac{i+1}{2}\tr\chibar \hnabla^i \alphaF - \mathcal{D} \hnabla^i (-\rhoF,\sigmaF) \\= &\sumitm \nablapp \hnabla^{i_3+1} (-\rhoF,\sigmaF) + \sumif \nablap \nablat \eta \hnablaf (\rhoF,\sigmaF) \\ &+ \sumif \nablap \nablat \omegabar \hnablaf \alphaF + \sumif \nablap \nablat \chihat \hnablat \alphabarF  \\&+\sumif \nablap \nablat (\chibarhat,\tildetr) \hnablaf \alphaF +\sumifi \nablapp \nablat \tr\chibar \hnablaf \alphaF \\ &+\sumifim \nablapp \nablat(\chibarhat,\tr\chibar) \hnablaf \alphaF \\& +\sumiFi \nablap \hnablat \alphaF \hnablaf (\rhoF,\sigmaF) \hnablaF \alphaF \\&+ \sumifi \nablap \hnablat \alphabarF \hnablaf \alphaF := A_i,
    \end{split}
\end{equation}

\begin{equation}\label{rhofenexpression}
    \begin{split}
        &\hnabla_4\hnabla^i (-\rhoF,\sigmaF)- \Hodge{\mathcal{D}}\hnabla^i \alphaF \\=&\sumitm \nablapp \hnabla^{i_3+1}\alphaF + \sumif \nablap \nablat \tr\chi \hnablaf(\rhoF,\sigmaF) \\&+ \sumif \nablap \nablat (\eta,\etabar)\hnablaf  \alphaF +\sumif \nablap \nablat (\chihat,\tr\chi) \hnablaF (\rhoF,\sigmaF) \\&+ \sumifim \nablapp \nablat (\chihat,\tr\chi)\hnablaf (\rhoF,\sigmaF) \\ &+ \sumifi \nablap \hnablat (\rhoF,\sigmaF) \hnablaf \alphaF \hnablaF (\rhoF,\sigmaF) \\&+\sumifi \nablap \hnablat \alphaF \hnablaf (\rhoF,\sigmaF) :=B_i .
    \end{split}
\end{equation}Making use of Proposition \ref{42}, we arrive at

\begin{equation}
    \begin{split}
          & \scaletwoHu{(\al)^{i-1}\hnabla^i \alphaF}^2 + \scaletwoHbaru{(\al)^{i-1} \hnabla^i (\rhoF,\sigmaF)}^2 \\ \lesssim &  \scaletwoHzero{(\al)^{i-1}\hnabla^i \alphaF}^2 + \scaletwoHbarzero{(\al)^{i-1} \hnabla^i (\rhoF,\sigmaF)}^2  \\ +& \iint_{\mathcal{D}_{u,\ubar}}\frac{a}{\upr} \scaleoneSuprimeubarprime{\langle(\al)^{i-1}\hnabla^i \alphaF, (\al)^{i-1}A_i\rangle}\duprime \dubarprime \\+& \iint_{\mathcal{D}_{u,\ubar}}\frac{a}{\upr} \scaleoneSuprimeubarprime{\langle(\al)^{i-1}\hnabla^i (\rhoF,\sigmaF) ,(\al)^{i-1}B_i\rangle}\duprime \dubarprime.
    \end{split}
\end{equation}Call the last two spacetime integrals above $\widetilde{N}$ and $\widetilde{M}$ respectively. For the first one, we have, using H\"older's inequality, 

\begin{equation}\label{48}
\begin{split}
    \widetilde{N} =& \iint_{\mathcal{D}_{u,\ubar}}\frac{a}{\upr} \scaleoneSuprimeubarprime{\langle(\al)^{i-1}\hnabla^i \alphaF, (\al)^{i-1}A_i\rangle}\duprime \dubarprime \\ \lesssim & \intu \frac{a}{\upr^2}\left( \intubar \scaletwoSuprimeubarprime{(\al)^{i-1}A_i}^2 \right)^{\frac{1}{2}} \duprime \cdot \sup_{u^{\prime}} \scaletwoHprime{(\al)^{i-1}\hnabla^i \alphaF}.
\end{split}
\end{equation}We shall work on the term
\[ H_i := \left( \intubar \scaletwoSuprimeubarprime{(\al)^{i-1}A_i}^2 \dubarprime \right)^{\frac{1}{2}}. \]The term $A_i$ comprises nine summands, as seen in \eqref{alphafenexpression}. We bound each of the corresponding terms in in $H_i$ individually.
\vspace{3mm}
For the first term:

\begin{equation}\label{herewego}
\begin{split}
    \intubar \scaletwoSuprimeubarprime{(\al)^{i-1}\sumitm \nablapp \hnabla^{i_3+1}(\rhoF,\sigmaF)}^2 \dubarprime, 
\end{split}
\end{equation}we observe the following. If $i_3+1= N+5$, we estimate $\psi_g$ in $L^{\infty}_{(sc)}$ and $(\al)^{N+4} \hnabla^{N+5}(\rhoF,\sigmaF)$ in $L^2_{(sc)}(H_{u^{\prime}}^{(0,\ubar)})$. Otherwise, we can bound the entire term using estimates on the $\bbGamma$ norm. Therefore we estimate, using the bootstrap assumptions:

\begin{equation}
    \intubar \scaletwoSuprimeubarprime{(\al)^{i-1}\sumitm \nablapp \hnabla^{i_3+1}(\rhoF,\sigmaF)}^2 \dubarprime \lesssim \frac{\Gamma^2 M^2}{\upr^2} + \frac{\Gamma^4}{a^2\upr^2}.
\end{equation}Similarly, for the second term, we have \begin{equation}
\begin{split}
    &\intubar \scaletwoSuprimeubarprime{(\al)^{i-1}\sumif \nablap \nablat \eta \hnabla^{i_4}(\rhoF,\sigmaF)}^2 \dubarprime \\ \lesssim& \intubar \scaletwoSuprimeubarprime{(\al)^{N+4} \eta \hnabla^{N+5}(\rhoF,\sigmaF)}^2 \dubarprime \\+ &\intubar \scaletwoSuprimeubarprime{(\al)^{N+4}(\rhoF,\sigmaF) \nabla^{N+5} \eta   }^2 \dubarprime+ \frac{\Gamma^4}{a \upr^2} \\ \lesssim &\frac{\Gamma^2 M^2}{\upr^2} + \frac{\Gamma^2}{\upr^2} \frac{\upr^2}{a^2} \cdot \Gammatop^2 +\frac{\Gamma^4}{a \upr^2}, \end{split}
\end{equation}where in the last line we have utilised the bootstrap assumption \[ \frac{a}{\upr}\scaletwoHprime{(\al)^{N+4}\nabla^{N+5}\eta} \lesssim \Gamma_{\text{top}}.\]For the third term, we have 
\begin{equation}
\begin{split}
    &\intubar \scaletwoSuprimeubarprime{(\al)^{i-1}\sumif \nablap \nablat \omegabar \hnabla^{i_4}\alphaF}^2 \dubarprime \\ \lesssim& \intubar \scaletwoSuprimeubarprime{(\al)^{N+4} \omegabar \hnabla^{N+5}\alphaF}^2 \dubarprime \\+ &\intubar \scaletwoSuprimeubarprime{(\al)^{N+4}\alphaF \nabla^{N+5} \omegabar   }^2 \dubarprime+ \frac{\Gamma^4}{ \upr^2} \\ \lesssim &\frac{a \Gamma^2 M^2}{\upr^2}+\frac{\Gamma^4}{\upr^2} + \frac{a\Gamma^2}{\upr^2}\intubar \scaletwoSuprimeubarprime{(\al)^{N+4}\nabla^{N+5}\omegabar}^2 \dubarprime. \end{split}
\end{equation}For the fourth term, we have 
\begin{equation}
\begin{split}
    &\intubar \scaletwoSuprimeubarprime{(\al)^{i-1}\sumif \nablap \nablat \chihat \hnabla^{i_4}\alphabarF}^2 \dubarprime \\ \lesssim& \intubar \scaletwoSuprimeubarprime{(\al)^{N+4} \chihat \hnabla^{N+5}\alphabarF}^2 \dubarprime \\+ &\intubar \scaletwoSuprimeubarprime{(\al)^{N+4}\alphabarF \nabla^{N+5} \chihat   }^2 \dubarprime+ \frac{\Gamma^4}{ \upr^2} \\ \lesssim &\frac{a \Gamma^2 \Gamma_{\text{top}}^2}{\upr^2}+\frac{\Gamma^4}{\upr^2} + \frac{a\Gamma^2}{\upr^2}\intubar \scaletwoSuprimeubarprime{(\al)^{N+4}\hnabla^{N+5}\alphabarF}^2 \dubarprime. \end{split}
\end{equation}In the fifth term, the worst term that is encountered is

\begin{equation}
\begin{split}
    &\intubar \scaletwoSuprimeubarprime{(\al)^{i-1}\sumif \nablap \nablat \chibarhat \hnabla^{i_4}\alphaF}^2 \dubarprime \\ \lesssim& \intubar \scaletwoSuprimeubarprime{(\al)^{N+4} \chibarhat \hnabla^{N+5}\alphaF}^2 \dubarprime \\+ &\intubar \scaletwoSuprimeubarprime{(\al)^{N+4}\alphaF \nabla^{N+5} \chibarhat   }^2 \dubarprime+ \frac{\Gamma^4}{ a} \\ \lesssim &\Gamma^2[\chibarhat]M^2[\alphaF]+\frac{\Gamma^4}{a} + \frac{a\Gamma^2}{\upr^2}\intubar \scaletwoSuprimeubarprime{(\al)^{N+4}\nabla^{N+5}\chibarhat}^2 \dubarprime. \end{split}
\end{equation}For the sixth term, we have

\begin{equation}
\begin{split}
    &\intubar \scaletwoSuprimeubarprime{(\al)^{i-1}\sumifi \nablapp \nablat \tr\chibar \hnabla^{i_4}\alphaF}^2 \dubarprime \\ \lesssim& \frac{\upr^4}{a^2}\cdot \frac{\Gamma^6}{\upr^4}\lesssim \frac{\Gamma^6}{a^2},
\end{split}
\end{equation}since there are at most $N+4$ derivatives in the expression, thus each term is controllable under the $\bbGamma-$norm. For the seventh term, we have

\begin{equation}
\begin{split}
    &\intubar \scaletwoSuprimeubarprime{(\al)^{i-1}\sumifim \nablapp \nablat (\chibarhat,\tr\chibar) \hnabla^{i_4}\alphaF}^2 \dubarprime \\ \lesssim& \intubar  \frac{\upr^4}{a^2}\cdot a \cdot \frac{\Gamma^6}{\upr^4} \lesssim \frac{\Gamma^6}{a} \lesssim 1.\end{split}
\end{equation}The eighth and ninth terms are similarly controlled as follows:

\begin{equation}\label{herewearrive}
\begin{split}
    &\intubar \scaletwoSuprimeubarprime{(\al)^{i-1}\sumiFi \nablap \hnablat \alphaF \hnabla^{i_4}(\rhoF,\sigmaF) \hnablaF \alphaF}^2 \dubarprime \\ +& \intubar \scaletwoSuprimeubarprime{(\al)^{i-1}\sumifi \nablap \hnablat \alphabarF \hnabla^{i_4}\alphaF }^2 \dubarprime \lesssim \frac{a^2 \hsp  \Gamma^6}{\upr^4} + \frac{a \hsp \Gamma^4}{\upr^2} . \end{split}
\end{equation}Putting equations \eqref{herewego}-\eqref{herewearrive} together, we arrive at

\begin{equation}
\begin{split}
    H_i \lesssim 1 + \Gamma[\chibarhat]M[\alphaF] &+ \frac{\al \hsp \Gamma}{\upr} \left(\intubar \scaletwoSuprimeubarprime{(\al)^{N+4} \nabla^{N+5}(\omegabar, \chibarhat)}^2 \dubarprime \right)^{\frac{1}{2}}\\&+ \frac{\al \hsp \Gamma}{\upr} \left(\intubar \scaletwoSuprimeubarprime{(\al)^{N+4} \hnabla^{N+5}\alphabarF}^2 \dubarprime\right)^{\frac{1}{2}}.
\end{split}
\end{equation}The terms above involving integrals cannot be controlled along outgoing null hypersurfaces, but nevertheless, going back to \eqref{48}, we see that 

\begin{equation}
    \begin{split}
  &\intu \frac{a}{\upr^2}\cdot \frac{\al\Gamma}{\upr}\left(\intubar  \scaletwoSuprimeubarprime{(\al)^{N+4} \nabla^{N+5}(\omegabar, \chibarhat)}^2\right)^{\frac{1}{2}}\duprime \\\lesssim & \left( \intu \frac{a^2}{\upr^4} \intubar \scaletwoSuprimeubarprime{(\al)^{N+4} \nabla^{N+5}(\omegabar, \chibarhat)}^2 \dubarprime \duprime \right)^{\frac{1}{2}} \left( \intu \frac{a\Gamma^2}{\upr^2} \duprime\right)^{\frac{1}{2}} \\ \lesssim & \frac{\al \Gamma}{\lvert u \rvert^{\frac{1}{2}}} \left( \mathcal{R} +\underline{\mathcal{R}}+1 \right) \lesssim \frac{\al \Gamma R}{\lvert u \rvert^{\frac{1}{2}}}, 
    \end{split}
\end{equation}where we have made use of H\"older's inequality, the bootstrap assumptions as well as \eqref{usefulhere}. Similarly, we have

\begin{equation}
    \begin{split}
  &\intu \frac{a}{\upr^2}\cdot \frac{\al\Gamma}{\upr}\left(\intubar  \scaletwoSuprimeubarprime{(\al)^{N+4} \hnabla^{N+5}\alphabarF}^2\right)^{\frac{1}{2}}\duprime \\\lesssim & \left( \intubar \intu \frac{a}{\upr^2}   \scaletwoSuprimeubarprime{(\al)^{N+4} \hnabla^{N+5}\alphabarF}^2 \duprime \dubarprime  \right)^{\frac{1}{2}} \left( \intu \frac{a^2 \Gamma^2}{\upr^4}\duprime \right)^{\frac{1}{2}} \\ \lesssim &\frac{a \hsp \Gamma}{\lvert u \rvert^{\frac{3}{2}}} \hsp  \sup_{0\leq\ubar \leq 1}\scaletwoHbaru{(\al)^{N+4} \hnabla^{N+5}\alphabarF} \lesssim \frac{a \hsp \Gamma \hsp M}{\lvert u \rvert^{\frac{3}{2}}} \lesssim 1.
  \end{split}
  \end{equation}Putting everything together, we get
  
  \begin{equation}\label{421}
      \widetilde{N} \lesssim ( \Gamma[\chibarhat] M[\alphaF] + \Gamma R +1 ) \cdot \sup_{u^{\prime}} \scaletwoHuprime{(\al)^{i-1}\hnabla^i \alphaF}.
  \end{equation}Moving on to $\widetilde{M}$, 
we  similarly obtain

\begin{equation}\label{422}
\begin{split}
    \widetilde{M} =& \iint_{\mathcal{D}_{u,\ubar}}\frac{a}{\upr} \scaleoneSuprimeubarprime{\langle(\al)^{i-1}\hnabla^i \alphaF, (\al)^{i-1}B_i\rangle}\duprime \dubarprime \\ \lesssim &\left( \intubar \intu \frac{a}{\upr^2} \scaletwoSuprimeubarprime{(\al)^{i-1}B_i}^2 \right)^{\frac{1}{2}} \duprime \cdot \sup_{u^{\prime}} \scaletwoHbarprime{(\al)^{i-1}\hnabla^i (\rhoF,\sigmaF)}.
\end{split}
\end{equation}We now work on the spacetime integral in the inequality above. Here, there holds

\begin{equation}
\begin{split}
    &\intubar \intu \frac{a}{\upr^2}\scaletwoSuprimeubarprime{(\al)^{i-1} \sumitm \nablapp \hnabla^{i_3+1}\alphaF}^2 \duprime \dubarprime \\ \lesssim &\intu \frac{a}{\upr^2} \intubar \scaletwoSuprimeubarprime{(\al)^{i-1} \sumifi \nablap \nablat (\eta,\etabar) \hnabla^{i_4+1}\alphaF}^2 \dubarprime \duprime\\ \lesssim &\intu \frac{a}{\upr^2} \cdot \frac{\Gamma^2}{\upr^2} \cdot \intubar \scaletwoSuprimeubarprime{(\al)^{N+4} \hnabla^{N+5}\alphaF}^2 \dubarprime \duprime  + \intubar\intu \frac{a}{\upr^2} \cdot \frac{\Gamma^4}{\upr^2} \duprime \dubarprime \\ \lesssim & \intu \frac{a^2 \hsp \Gamma^2 }{\upr^4} \cdot \left( \frac{1}{a} \scaletwoHuprime{(\al)^{N+4}\hnabla^{N+5}\alphaF}^2\right) +1 \lesssim \frac{a^2\hsp \Gamma^2\hsp M[\alphaF]^2}{\lvert u \rvert^3}+1 \lesssim 1.
\end{split}
\end{equation}The second term is contained in the fourth, so we skip it for now. For the third term, there holds 

\begin{equation}
    \begin{split}
          &\intubar \intu \frac{a}{\upr^2}\scaletwoSuprimeubarprime{(\al)^{i-1} \sumif \nablap \nabla^{i_3}(\eta,\etabar)\hnabla^{i_4}\alphaF}^2 \duprime \dubarprime \\ \lesssim &\intu \frac{a}{\upr^2} \cdot\frac{a \hsp \Gamma^2}{\upr^2}\intubar \scaletwoSuprimeubarprime{(\al)^{N+4}\hnabla^{N+5}(\eta,\etabar)}^2 \dubarprime \\ +&\intu \frac{a}{\upr^2} \cdot \frac{\Gamma^2}{\upr^2} \intubar \scaletwoSuprimeubarprime{(\al)^{N+4}\hnabla^{N+5} \alphaF}^2\duprime \dubarprime +\intubar\intu \frac{a}{\upr^2} \cdot \frac{\Gamma^4}{\upr^2}  \dubarprime \duprime \\ \lesssim& \intu \frac{a^2 \Gamma^2}{\upr^4} \cdot \frac{\upr^2}{a^2} \cdot \Gammatop^2 \duprime + \intu \frac{a \hsp \Gamma^2}{\upr^4} \cdot a \cdot M[\alphaF]^2 \duprime +1 \lesssim \frac{\Gamma^2 \Gammatop^2}{\lvert u \rvert} + \frac{a^2\Gamma^2 M[\alphaF]^2}{\lvert u \rvert^3} +1\\\lesssim & 1.
    \end{split}
\end{equation}For the fourth term, we have 

\begin{equation}
    \begin{split}
                  &\intubar \intu \frac{a}{\upr^2}\scaletwoSuprimeubarprime{(\al)^{i-1} \sumif \nablap \nabla^{i_3}(\chihat,\tr\chi)\hnabla^{i_4}(\rhoF,\sigmaF)}^2 \duprime \dubarprime \\ \lesssim &\intu \frac{a}{\upr^2} \cdot\frac{a \hsp \Gamma^2}{\upr^2}\left( \frac{1}{a}\intubar \scaletwoSuprimeubarprime{(\al)^{N+4}\hnabla^{N+5}\chihat}^2 \dubarprime \right) \duprime \\ +&\intu \frac{a}{\upr^2} \cdot\frac{ \hsp \Gamma^2}{\upr^2}\intubar \scaletwoSuprimeubarprime{(\al)^{N+4}\hnabla^{N+5}\tr\chi}^2 \dubarprime \duprime \\ +& \intu \frac{a^2 \hsp \Gamma^2}{\upr^4}\intubar \scaletwoSuprimeubarprime{(\al)^{N+4}\hnabla^{N+5}(\rhoF,\sigmaF)}^2 \dubarprime \duprime + \intubar \intu \frac{a}{\upr^2}\frac{\Gamma^4}{\upr^2} \duprime \dubarprime\\ \lesssim &\frac{a^2 \hsp \Gamma^2\hsp \Gammatop^2}{\lvert u \rvert^3} + \frac{a \Gamma^2 \Gammatop^2}{\lvert u \rvert^3} + \frac{a^2 \hsp \Gamma^2 \hsp M[\rhoF,\sigmaF]^2}{\lvert u \rvert^3} + \frac{a \hsp \Gamma^4}{\lvert u \rvert^3} \lesssim 1.
    \end{split}
\end{equation}For the fifth term, we have 

\begin{equation}
    \begin{split}
            &\intubar \intu \frac{a}{\upr^2}\scaletwoSuprimeubarprime{(\al)^{i-1} \sumifim \nablapp \nabla^{i_3}(\chihat,\tr\chi)\hnabla^{i_4}(\rhoF,\sigmaF)}^2 \duprime \dubarprime  \\ \lesssim& \intubar \intu \frac{a}{\upr^2}\cdot \frac{a \hsp \Gamma^6}{\upr^4}\duprime \dubarprime \lesssim 1.
    \end{split}
\end{equation}The last two terms are, similarly, bounded above by $1$. Consequently, one has

\begin{equation}
    \label{429}
    \widetilde{M} \lesssim \sup_{u^{\prime}} \scaletwoHbarprime{(\al)^{i-1}\hnabla^i (\rhoF,\sigmaF)} .
\end{equation} Taking \eqref{421} and \eqref{429} into account and multiplying throughout by $a^{-1}$, we get 

    \begin{equation}
    \begin{split}
          & a^{-1}\scaletwoHu{(\al)^{i-1}\hnabla^i \alphaF}^2 +a^{-1} \scaletwoHbaru{(\al)^{i-1} \hnabla^i (\rhoF,\sigmaF)}^2 \\ \lesssim &  a^{-1}\scaletwoHzero{(\al)^{i-1}\hnabla^i \alphaF}^2 + a^{-1}\scaletwoHbarzero{(\al)^{i-1} \hnabla^i (\rhoF,\sigmaF)}^2  + a^{-1}(\widetilde{N} + \widetilde{M})\\ \lesssim &  a^{-1}\scaletwoHzero{(\al)^{i-1}\hnabla^i \alphaF}^2 + a^{-1}\scaletwoHbarzero{(\al)^{i-1} \hnabla^i (\rhoF,\sigmaF)}^2  \\ +& a^{- \frac{1}{2}} (\Gamma \hsp M + \Gamma \hsp R + 1)\cdot M + a^{-\frac{1}{2}} \cdot M \\ \lesssim &a^{-1}\scaletwoHzero{(\al)^{i-1}\hnabla^i \alphaF}^2 + a^{-1}\scaletwoHbarzero{(\al)^{i-1} \hnabla^i (\rhoF,\sigmaF)}^2 + 1.
    \end{split}
\end{equation}By taking the roots, we translate the above into the desired conclusion in a straightforward fashion. The result follows.

\end{proof}
\begin{proposition}
Under the assumptions of Theorem \ref{main1} and the bootstrap assumptions \eqref{bootstrap}, \eqref{ellboot}, we have, for $0\leq i \leq N+5$:

\begin{equation}
    \begin{split}
      &\scaletwoHu{(\al)^{i-1} \hnabla^i (\rhoF,\sigmaF)} + \scaletwoHbaru{(\al)^{i-1}\hnabla^i \alphabarF} \\ \lesssim & \scaletwoHzero{ (\al)^{i-1} \hnabla^i (\rhoF,\sigmaF)} + \scaletwoHbarzero{(\al)^{i-1}\hnabla^i \alphabarF} +1.
    \end{split}
\end{equation}

\end{proposition}
\begin{proof}
The proof follows along the same lines as the preceding Proposition. We begin with the schematic equations for $((-\rhoF,\sigmaF), \alphabarF):$

\begin{equation}
    \begin{split}
        \hnabla_3 (-\rhoF,\sigmaF) + \tr\chibar (-\rhoF,\sigmaF) - \mathcal{D} \alphabarF = (\eta,\etabar) \cdot \alphabarF ,
    \end{split}
\end{equation}
\begin{equation}
    \hnabla_4 \alphabarF - \Hodge{\mathcal{D}}(-\rhoF,\sigmaF)= (\tr\chi, \omega) \alphabarF + (\eta,\etabar)(\rhoF,\sigmaF) +\chibarhat\hsp \alphaF.
\end{equation}
Commuting these two equations with $i$ gauge-covariant angular derivatives we arrive at

\begin{equation}
    \begin{split}
        &\hnabla_3 \hnabla^i (-\rhoF,\sigmaF) + \frac{i+1}{2}\tr\chibar \hnabla^i (-\rhoF,\sigmaF) - \mathcal{D} \hnabla^i \alphabarF =\\ &\sumitm \nablapp \hnabla^{i_3+1}\alphabarF + \sumif \nablap \nablat (\eta,\etabar) \hnablaf \alphabarF \\ &+ \sumif \nablap \nablat(\chibarhat,\tildetr) \hnablaf (-\rhoF,\sigmaF) + \sumifi \nablapp \nablat \tr\chibar \hnablaf (-\rhoF,\sigmaF) \\ &+ \sumifim \nablapp \nablat (\chibarhat, \tr\chibar) \hnablaf (-\rhoF,\sigmaF)  \\&+ \sumiFi \nablap \hnablat \alphabarF \hnablaf (\rhoF,\sigmaF) \hnablaF (-\rhoF,\sigmaF) \\ &+ \sumifi \nablap \hnablat \alphabarF \hnablaf (-\rhoF,\sigmaF)  := \Delta_i
    \end{split}
\end{equation}

\begin{equation}
    \begin{split}
        &\hnabla_4 \hnabla^i \alphabarF - \Hodge{\mathcal{D}}\hnabla^i(-\rhoF,\sigmaF) \\= &\sumitm \nablapp \hnabla^{i_3+1}(-\rhoF,\sigmaF) + \sumif \nablap \nablat (\tr\chi,\omega) \hnablaf \alphabarF \\ &+ \sumif \nablap \nabla(\eta,\etabar) \hnablaf (\rhoF,\sigmaF) + \sumif \nablap \nablat \chibarhat \nablaf \alphaF \\ &+ \sumif \nablap \nablat (\chihat,\tr\chi) \hnablaf \alphabarF \\ &+ \sumifim \nablapp \nablat (\chihat,\tr\chi) \hnablaf \alphabarF \\ &+ \sumiFi \nablap \hnablat \alphaF \hnablaf (\rhoF,\sigmaF) \hnablaF \alphabarF \\ &+ \sumifi \nablap \hnablat \alphaF \hnablaf \alphabarF := E_i .
    \end{split}
\end{equation}We thus arrive, using Proposition \ref{42}, at

\begin{equation}\label{435}
\begin{split}
&\scaletwoHu{ (\al)^{i-1} \hnabla^i (\rhoF,\sigmaF)}^2 + \scaletwoHbaru{(\al)^{i-1} \hnabla^i \alphabarF}^2 \\  \lesssim & \scaletwoHzero{ (\al)^{i-1} \hnabla^i (\rhoF,\sigmaF)}^2 + \scaletwoHbarzero{(\al)^{i-1} \hnabla^i \alphabarF}^2 + \Xi_i + O_i, 
\end{split}
\end{equation}where
\begin{equation}
    \Xi_i := \intubar \intu \frac{a}{\upr} \scaleoneSuprimeubarprime{ (\al)^{i-1} \Delta_i (\al)^{i-1} \hnabla^i (\rhoF,\sigmaF)} \duprime \dubarprime
\end{equation}and 

\begin{equation}
 O_i := \intubar \intu \frac{a}{\upr} \scaleoneSuprimeubarprime{ (\al)^{i-1} E_i (\al)^{i-1} \hnabla^i \alphabarF} \duprime \dubarprime. 
\end{equation}

Focusing on $\Xi_i$, we have

\begin{equation}
    \begin{split}
        & \hspace{3mm}\Xi_i  \\ \lesssim &\intubar \intu \frac{a}{\upr^2}\scaletwoSuprime{(\al)^{i-1} \Delta_i}\scaletwoSuprime{(\al)^{i-1} \hnabla^i (\rhoF,\sigmaF)} \duprime \dubarprime  \\ \lesssim &\intu \frac{a}{\upr^2} \left( \intubar \scaletwoSuprime{(\al)^{i-1} \Delta_i}^2 \dubarprime \right)^{\frac{1}{2}} \left(\intubar \scaletwoSuprime{(\al)^{i-1} \hnabla^i (\rhoF,\sigmaF)}^2 \dubarprime   \right)^{\frac{1}{2}} \duprime \\ \lesssim & \left( \intubar \intu \frac{a}{\upr^2} \scaletwoSuprimeubarprime{(\al)^{i-1}\Delta_i}^2 \duprime \dubarprime \right)^{\frac{1}{2}} \cdot \sup_{u_{\infty} \leq u^{\prime} \leq u} \scaletwoHprime{(\al)^{i-1} \hnabla^i (\rhoF,\sigmaF)}.
    \end{split}
\end{equation}For the first term of the spacetime integral in the line above, as in previous Propositions, we have

\begin{equation}
    \begin{split}
        &\intubar \intu \frac{a}{\upr^2} \scaletwoSuprimeubarprime{(\al)^{i-1}\sumitm \nablapp \hnabla^{i_3+1}\alphabarF}^2 \dubarprime \dubarprime \\ \lesssim &\intubar \intu \frac{a}{\upr^2} \cdot \frac{\Gamma^2}{\upr^2}\scaletwoSuprimeubarprime{(\al)^{i-1} \hnabla^{i}\alphabarF}^2 \duprime \dubarprime \\ +& \intubar \intu \frac{a}{\upr^2}\cdot \frac{\Gamma^2}{\upr^2} \scaletwoSuprimeubarprime{(\al)^{i-1}\nabla^{i-1}(\eta,\etabar)}^2 \duprime\dubarprime + \frac{a \Gamma^2}{\upr^3} \lesssim \frac{1}{\al}. 
    \end{split}
\end{equation}Secondly, we have 

\begin{equation}
\begin{split}
  &\intubar \intu \frac{a}{\upr^2} \scaletwoSuprimeubarprime{(\al)^{i-1}\sumif \nablap \hnabla^{i_3}(\eta,\etabar) \hnablaf \alphabarF}^2 \dubarprime \dubarprime \\ \lesssim &\intubar \intu \frac{a}{\upr^2} \cdot \frac{\Gamma^2}{\upr^2} \scaletwoSuprimeubarprime{(\al)^{i-1}\hnabla^i \alphabarF}^2 \duprime \dubarprime \\&+ \intubar \intu \frac{a}{\upr^2} \cdot \frac{\Gamma^2}{\upr^2} \scaletwoSuprimeubarprime{(\al)^{i-1}\hnabla^i (\eta,\etabar)}^2 \duprime \dubarprime  + \intubar \intu \frac{a}{\upr^2}\cdot \frac{\Gamma^4}{\upr^2} \duprime \dubarprime\\ &\lesssim  \frac{a \Gamma^2 (\Gamma^2 + \Gammatop^2)}{\lvert u \rvert^3} \lesssim \frac{1}{\al}.
\end{split}
\end{equation}For the third term, the most borderline case is when we have 

\begin{equation}
\begin{split}
&\intubar \intu \frac{a}{\upr^2}\scaletwoSuprimeubarprime{(\al)^{i-1} \sumif \nablap \nablat \chibarhat \hnablaf (-\rhoF,\sigmaF)}^2 \duprime \dubarprime \\ \lesssim& \intubar\intu \frac{a}{\upr^2}\cdot \frac{\upr^2}{a}\cdot \frac{\Gamma^2}{\upr^2} \scaletwoSuprimeubarprime{(\al)^{i-1} \hnabla^i (-\rhoF,\sigmaF)}^2 \duprime \dubarprime \\ +& \intubar \intu \frac{a \Gamma^2}{\upr^4} \scaletwoSuprimeubarprime{(\al)^{i-1}\nabla^i \chibarhat}^2 \duprime \dubarprime + 1 \lesssim \frac{1}{\al},
\end{split}
\end{equation}where we have used the bootstrap assumptions \eqref{bootstrap} as well as equation \eqref{usefulhere} from Remark \ref{usefulremark}.

Out of the fourth and fifth terms, the most borderline one is the fifth. We estimate it as follows:

\begin{equation}
\begin{split}
 &\intubar \intu \frac{a}{\upr^2} \scaletwoSuprimeubarprime{(\al)^{i-1}\sumifim \nablapp \hnabla^{i_3}(\chibarhat,\tr\chibar) \hnablaf (\rhoF,\sigmaF)}^2 \dubarprime \dubarprime \\ \lesssim & \intubar \intu \frac{a}{\upr^2}\cdot \frac{\Gamma^6}{a} \duprime \dubarprime \lesssim \frac{1}{\al}.
\end{split}
\end{equation}
The sixth and seventh terms are similarly bounded above by $a^{-\frac{1}{2}}$, noting that since they involve up to $i-1$ derivatives, no elliptic estimates are involved in their estimation. Finally, we have

\begin{equation} \label{43}
\Xi_i \lesssim \frac{1}{a^{\frac{1}{4}}} \cdot \sup_{u_{\infty} \leq u^{\prime} \leq u} \scaletwoHprime{(\al)^{i-1} \hnabla^i (\rhoF,\sigmaF)} \lesssim 1,
\end{equation}by the bootstrap assumptions \eqref{bootstrap}. Finally, for $O_i$, we have:

\begin{equation}
\left(\intubar \intu \frac{a}{\upr^2}\scaletwoSuprimeubarprime{(\al)^{i-1}E_i}^2 \duprime \dubarprime \right)^{\frac{1}{2}} \cdot \sup_{0\leq \ubar \leq 1} \scaletwoHbarprime{(\al)^{i-1} \hnabla^i \alphabarF}.
\end{equation}
For the first term in the spacetime integral above, we estimate

\begin{equation}
\begin{split}
&\intubar \intu \frac{a}{\upr^2}\scaletwoSuprimeubarprime{(\al)^{i-1}\sumitm \nablapp \hnabla^{i_3+1}(-\rhoF,\sigmaF)}^2 \duprime \dubarprime \\ \lesssim & \intubar\intu \frac{a}{\upr^2} \cdot \frac{\Gamma^2}{\upr^2} \scaletwoSuprimeubarprime{(\al)^{i-1}\hnabla^i(-\rhoF,\sigmaF)}^2 \duprime\dubarprime +1 \\ \lesssim & \frac{a M^2 \Gamma^2}{\lvert u \rvert^3}+1 \lesssim 1.
\end{split}
\end{equation}The second and third terms can be similarly bounded above by $1$ using just the bootstrap assumptions \eqref{bootstrap} and \eqref{ellboot}. The most dangerous term is the fourth one. Indeed, distinguishing between those terms which are top-order and those who are not, we have

\begin{equation}
\begin{split}
&\intubar \intu \frac{a}{\upr^2}\scaletwoSuprimeubarprime{(\al)^{i-1}\sumif \nablap \hnablat \chibarhat \hnablaf \alphaF}^2 \duprime \dubarprime \\ \lesssim & \intubar\intu \frac{a}{\upr^2} \cdot \frac{\upr^2}{a}\frac{\Gamma[\chibarhat]^2}{\upr^2} \scaletwoSuprimeubarprime{(\al)^{i-1}\hnabla^i \alphaF}^2 \duprime \dubarprime \\ +& \intubar \intu \frac{a^2 \Gamma[\alphaF]^2}{\upr^4}\scaletwoSuprimeubarprime{(\al)^{i-1}\hnabla^i \chibarhat}^2 \duprime \dubarprime  \\ + & \intubar\intubar \frac{a}{\upr^2} \cdot \frac{\upr^2}{a}\cdot \frac{\Gamma^2[\chibarhat]\Gamma^2[\alphaF]}{\upr^2}\duprime \dubarprime  \lesssim  \Gamma[\chibarhat]^2\underline{\mathbb{YM}}[\alphaF]^2 +  \Gamma[\alphaF]^2\left(\mathcal{R} + \underline{\mathcal{R}}+1\right) + 1 \\\lesssim &\underline{\mathbb{YM}}[\alphaF]^2 +  \left(\mathcal{R} + \underline{\mathcal{R}}+1\right)^2 + 1 \lesssim \left(\mathcal{R} + \underline{\mathcal{R}}+1\right)^2.
\end{split}
\end{equation}The remaining terms are not of top order and hence can be bounded above by 1 in a by now standard way. Ultimately, we have

\begin{equation}\label{447}
\begin{split}
O_i \lesssim& \left(\intubar \intu \frac{a}{\upr^2}\scaletwoSuprimeubarprime{(\al)^{i-1}E_i}^2 \duprime \dubarprime \right)^{\frac{1}{2}}  \left(\intubar \scaletwoHbarprime{(\al)^{i-1}\hnabla^i \alphabarF}^2 \dubarprime \right)^{\frac{1}{2}}\\ \lesssim & \intubar \intu \frac{a}{\upr^2}\scaletwoSuprimeubarprime{(\al)^{i-1}E_i}^2 \duprime \dubarprime + \frac{1}{4} \intubar \scaletwoHbarprime{(\al)^{i-1}\hnabla^i \alphabarF}^2 \dubarprime \\ \lesssim &(\mathcal{R}+\underline{\mathcal{R}}+1)^2 + \frac{1}{4} \intubar \scaletwoHbarprime{(\al)^{i-1}\hnabla^i \alphabarF}^2 \dubarprime.
\end{split}
\end{equation}Combining equations \eqref{435}, \eqref{43} and \eqref{447}, we arrive at 

\begin{equation}\label{4345}
\begin{split}
&\scaletwoHu{ (\al)^{i-1} \hnabla^i (\rhoF,\sigmaF)}^2 + \scaletwoHbaru{(\al)^{i-1} \hnabla^i \alphabarF}^2 \\  \lesssim & \scaletwoHzero{ (\al)^{i-1} \hnabla^i (\rhoF,\sigmaF)}^2 + \scaletwoHbarzero{(\al)^{i-1} \hnabla^i \alphabarF}^2 \\+& (\mathcal{R}+\underline{\mathcal{R}}+1)^2 +   \frac{1}{4} \intubar \scaletwoHbarprime{(\al)^{i-1}\hnabla^i \alphabarF}^2 \dubarprime,
\end{split}
\end{equation}from where an application of Gr\"onwall's inequality yields the desired result.

\end{proof}

\subsection{Energy estimates for $\hnabla_4\alphaF, \hnabla_3 \alphabarF$}
 In this section we obtain energy estimates for up to $N+4$ derivatives of $\hnabla_4 \alphaF$ and $\hnabla_3 \alphabarF$. These will be useful in the energy estimates for $\alpha$ and $\alphabar$ respectively.

\begin{proposition}
Under the assumptions of Theorem \ref{main1} and the bootstrap assumptions \eqref{bootstrap} and \eqref{ellboot}, there holds, for all $0\leq i \leq N+4$:

\begin{equation}
\begin{split}
&\frac{1}{\al}\scaletwoHu{(\al \hnabla)^i \hnabla_4 \alphaF} + \frac{1}{\al}\scaletwoHbaru{(\al \hnabla)^i (-\hnabla_4\rhoF, \hnabla_4 \sigmaF)} \\ & \frac{1}{\al}\scaletwoHzero{(\al \hnabla)^i \hnabla_4 \alphaF} + \frac{1}{\al}\scaletwoHbarzero{(\al \hnabla)^i (-\hnabla_4\rhoF, \hnabla_4 \sigmaF)}+ \frac{1}{a^{ \frac{1}{3} }}.  
\end{split}
\end{equation}
\end{proposition}
\begin{proof}
We begin by noticing the equations
\begin{equation}
\begin{split}
&\hnabla_3 \hnabla_4 \alphaF + \frac{1}{2}\tr\chibar \hnabla_4 \alphaF -\widehat{\mathcal{D}}(\hnabla_4 \rhoF,\hnabla_4 \sigmaF) \\= &\psi_g \tr\chibar \alphaF + (\psi_g,\chihat)(\psi_g,\chibarhat)(\alphaF,\Yub) + \nabla(\eta,\etabar)\alphaF + \psi_g \hnabla(\alphaF,\rhoF,\sigmaF) + (\alphaF,\rhoF,\sigmaF)\nabla \psi_g \\ +& (\nabla \psi_g, \nabla \chihat)(\rhoF,\sigmaF) + (\chihat,\psi_g) \hnabla(\rhoF,\sigmaF) +\alphaF (\rhoF,\sigmaF)(\rhoF,\sigmaF) +\alphaF(\rhoF,\sigmaF) +\alpha \hsp  \alphabarF \\&+(\psi_g,\chihat)\hnabla_4\alphaF +(\rho,\sigma)\alphaF +\tbeta (\rhoF,\sigmaF),
\end{split}
\end{equation}
\begin{equation}
\begin{split}
&\hnabla_4 (\hnabla_4 \rhoF, \hnabla_4\sigmaF) - \Hodge{\widehat{\mathcal{D}}}\hnabla_4 \alphaF  =  \\&\tbeta \alphaF + \alphaF (\rhoF,\sigmaF) \alphaF + \alphaF\cdot \alphaF + \psi_g \hnabla_4 \alphaF \\&+ (\chihat,\psi_g)\hnabla \alphaF + (\psi_g,\chihat)\psi_g \alphaF + (\psi_g,\chihat)(\psi_g,\chihat)(\rhoF,\sigmaF)\\ &+ \psi_g \psi_g (\alphaF, \rhoF,\sigmaF) +(\nabla \omega) \alphaF.
\end{split}
\end{equation}Commuting with $i$ gauge-covariant angular derivatives, we arrive at

\begin{align} \nonumber
    &\hnabla_3 \hnabla^i \hnabla_4 \alphaF + \frac{i+1}{2}\tr\chibar \hnabla^i \hnabla_4 \alphaF - \widehat{\mathcal{D}}(\hnabla^i \hnabla_4 \rhoF, \hnabla^i \hnabla_4 \sigmaF)\\  \nonumber &= \sumitm \nablapp \hnabla^{i_3+1}(\hnabla_4 \rhoF,\hnabla_4 \sigmaF) \\  \nonumber &+ \sumiF \nablap \nablat \psi_g \nablaf \tr\chibar \hnablaF \alphaF \\  \nonumber &+ \sumiF \nablap \nablat(\psi_g,\chihat)\nablaf (\psi_g,\chibarhat)\hnablaF (\alphaF, \mathcal{Y}_{\ubar}) \\  \nonumber&+ \sumif \nablap \nabla^{i_3+1}(\eta,\etabar) \hnablaf \alphaF \\ \nonumber &+ \sumitm \nablapp \hnablat(\alphaF, \rhoF,\sigmaF)  \\ \nonumber &+ \sumif \nablap \nabla^{i_3+1}(\tr\chi,\chihat)\hnablaf(\rhoF,\sigmaF) \\ \nonumber &+    \sumif \nablap \nabla^{i_3}(\psi_g,\chihat)\hnabla^{i_4+1}(\rhoF,\sigmaF) \\ \nonumber &+\sumiF \nablap \hnablat \alphaF \hnablaf (\rhoF,\sigmaF)\hnablaF (\rhoF,\sigmaF) \\ \nonumber &+\sumif \nablap \hnablat \alphaF \hnablaf(\rhoF,\sigmaF) \\ \nonumber &+ \sumif \nablap \nablat \alpha \hnablaf \alphabarF \\ \nonumber &+ \sumif \nablap \nablat (\rho,\sigma) \nablaf \alphaF \\ \nonumber &+ \sumif \nablap \nablat \tbeta \hnablaf (\rhoF, \sigmaF) \\ \nonumber &+ \sumif \nablap \nablat (\psi_g, \chihat)\hnablaf \hnabla_4 \alphaF \\ \nonumber &+ \sumif \nablap \nablat (\chibarhat,\tildetr) \hnablaf \hnabla_4 \alphaF \\ \nonumber &+ \sumifi \nablapp \nabla^{i_3+1}\tr\chibar \hnablaf \hnabla_4 \alphaF \\ \nonumber &+ \sumifim \nablapp \nablat (\chibarhat,\tr\chibar) \hnablaf \hnabla_4 \alphaF \\ \nonumber &+ \sumiFi \nablap \hnablat \alphabarF \hnablaf (\rhoF,\sigmaF) \hnablaF \hnabla_4 \alphaF \\  &+ \sumif \nablap \hnablat \alphabarF \hnablaf \hnabla_4 \alphaF := G_i,
\end{align}

\begin{align} \nonumber
&\hnabla_4(\hnabla^i \hnabla_4 \rhoF, \hnabla^i \hnabla_4 \sigmaF) - \Hodge{\mathcal{D}}\hnabla^i \hnabla_4 \alphaF = \\ \nonumber &\sumitm \nablapp \hnabla^{i_3+1}\hnabla_4\alphaF \\  \nonumber&+ \sumif \nablap \nablat \tbeta \hnablaf \alphaF \\ \nonumber &+ \sumiF \nablap \hnablat \alphaF \hnablaf (\rhoF,\sigmaF) \hnablaF \alphaF \\ \nonumber &+ \sumif \nablap \hnablat \alphaF \hnablaf \alphaF \\  \nonumber &+ \sumif \nablap \nablat \psi_g \hnablaf \hnabla_4 \alphaF \\ \nonumber &+ \sumif \nablap \nablat (\chihat,\psi_g) \hnablaf \hnabla \alphaF \\ \nonumber &+ \sumiF \nablap \nablat (\psi_g,\chihat) \nablaf \psi_g \hnablaF \alphaF \\  \nonumber &+ \sumiF \nablap \nablat (\psi_g,\chihat)\nablaf (\psi_g,\chihat) \hnablaF (\rhoF,\sigmaF) \\  \nonumber &+ \sumiF \nablap \nablat \psi_g \nablaf \psi_g \hnablaF (\alphaF,\rhoF,\sigmaF) \\ \nonumber &+ \sumif \nablap \nabla^{i_3+1}\omega \hnablaf \alphaF \\ \nonumber &+ \sumit \nablap \hnablat (\hnabla_4 \rhoF,\hnabla_4 \sigmaF)\\ \nonumber &+ \sumif \nablap \nablat(\chibarhat,\tildetr)\hnablaf (\hnabla_4 \rhoF,\hnabla_4 \sigmaF) \\ \nonumber &+\sumifi \nablapp \nablat (\chibarhat, \tr\chibar) \hnablaf (\hnabla_4 \rhoF,\hnabla_4 \sigmaF)  \\ \nonumber &+ \sumiFi \nablap \hnablat \alphabarF \hnablaf (\rhoF,\sigmaF)\hnablaF (\hnabla_4 \rhoF,\hnabla_4 \sigmaF) \\ &+ \sumifi \nablap \hnablat \alphabarF \hnablaf (\hnabla_4 \rhoF,\hnabla_4 \sigmaF) := H_i.
\end{align}
We have 

\begin{equation}\label{05}
\begin{split}
&\scaletwoHu{(\al \hnabla)^i \hnabla_4 \alphaF}^2 + \scaletwoHbaru{(\al \hnabla)^i (-\hnabla_4\rhoF, \hnabla_4 \sigmaF)}^2 \\ &\scaletwoHzero{(\al \hnabla)^i \hnabla_4 \alphaF}^2 + \scaletwoHbarzero{(\al \hnabla)^i (-\hnabla_4\rhoF, \hnabla_4 \sigmaF)}^2+ X +Y,
\end{split}
\end{equation}where \begin{equation}
X:= \iint_{\mathcal{D}(u,\ubar)}\frac{a}{\lvert u^{\prime} \rvert} \scaleoneSuprimeubarprime{\langle (\al\hnabla)^i \hnabla_4 \alphaF , \ali G_i \rangle} \duprime \dubarprime,
\end{equation}
\begin{equation}
Y:= \iint_{\mathcal{D}(u,\ubar)}\frac{a}{\lvert u^{\prime} \rvert} \scaleoneSuprimeubarprime{\langle (\al\hnabla)^i (\hnabla_4 \rhoF, \hnabla_4 \sigmaF) , \ali H_i \rangle} \duprime \dubarprime,
\end{equation}

As before, we have the estimate

\begin{equation}\label{07}
X \leq \left(\intubar \intu \frac{a}{\upr^2} \scaletwoSuprimeubarprime{\ali G_i}^2 \duprime \dubarprime\right)^{\frac{1}{2}}\cdot \sup_{u_{\infty} \leq u^{\prime} \leq -\frac{a}{4}} \scaletwoHprime{ (\al \hnabla)^i \hnabla_4 \alphaF}.
\end{equation}
\end{proof}We begin with the first term. There holds

\begin{equation}
\begin{split}
&\intubar \intu \frac{a}{\upr^2} \scaletwoSuprimeubarprime{\ali \sumitm \nablapp \hnabla^{i_3+1}(\hnabla_4 \rhoF,\hnabla_4 \sigmaF)}^2 \duprime \dubarprime
\\ \lesssim & \intubar \intu \frac{a}{\upr^2} \scaletwoSuprimeubarprime{ \psi_g (\al \hnabla)^i (\hnabla_4 \rhoF,\ \hnabla_4 \sigmaF)}^2 \duprime \dubarprime \\+ & \intubar \intu \frac{a}{\upr^2} \scaletwoSuprimeubarprime{a^{\frac{i}{2}}  \sum_{\substack{i_1+i_2+i_3=i-1\\ i_3<i-1}}\nablapp \hnabla^{i_+1} (\hnabla_4 \rhoF,\ \hnabla_4 \sigmaF)}^2 \duprime \dubarprime \\ \lesssim & \intubar\intu \frac{a}{\upr^2} \frac{\Gamma^2}{\upr^2} \scaletwoSuprimeubarprime{\haln (\hnabla_4 \rhoF,\hnabla_4 \sigmaF)}^2 + \frac{1}{a^{\frac{1}{3}}} \lesssim \frac{1}{a^{\frac{1}{3}}},
\end{split}
\end{equation}where we have used the bootstrap assumptions \eqref{bootstrap}. For the second term, we have

\begin{equation}
\begin{split}
&\intubar \intu \frac{a}{\upr^2} \scaletwoSuprimeubarprime{\ali \sumif \nablap \nabla^{i_3}\psi_g \nablaf \tr\chibar \hnablaF \alphaF}^2 \duprime \dubarprime
\\ \lesssim & \intubar \intu \frac{a}{\upr^2} \cdot \left( \frac{\upr^2}{a}\cdot \al \cdot \frac{\Gamma^3}{\upr^2} \right)^2 \duprime \dubarprime \lesssim \frac{1}{a^{\frac{1}{3}}}.
\end{split}
\end{equation}For the third term, we have 

\begin{equation}
\begin{split}
&\intubar \intu \frac{a}{\upr^2} \scaletwoSuprimeubarprime{\ali \sumiF \nablap \nabla^{i_3}(\psi_g,\chihat) \nablaf (\psi_g,\chibarhat) \hnablaF (\alphaF, \Yub)}^2 \duprime \dubarprime
\\ \lesssim & \intubar \intu \frac{a}{\upr^2} \cdot \left( \al \cdot \frac{\upr}{\al}\cdot{\al}\cdot{\Gamma^3{\upr^2}} \right)^2 \duprime \dubarprime \lesssim  \frac{1}{a^{\frac{1}{3}}}. 
\end{split}
\end{equation}For the fourth term, we have

\begin{equation}
\begin{split}
&\intubar \intu \frac{a}{\upr^2} \scaletwoSuprimeubarprime{\ali \sumif \nablap \nabla^{i_3+1}(\eta, \etabar) \hnablaf \alphaF }^2 \duprime \dubarprime
\\ \lesssim & \intubar \intu \frac{a}{\upr^2} \frac{a \Gamma^2}{\upr^2}\scaletwoSuprimeubarprime{\ali \nabla^{i+1}(\eta, \etabar)}^2 \duprime \dubarprime \\ +& \intubar \intu \aupr \frac{\Gamma^4}{\upr^2} \duprime \dubarprime \lesssim \frac{1}{\al} \scaletwoHu{(\al)^{N+4}\nabla^{N+5} (\eta,\etabar)}^2 + \frac{1}{a^{\frac{1}{3}}} 
\\ \lesssim &\frac{1}{\al}\cdot \Gammatop^2 + \frac{1}{a^{\frac{1}{3}}}  \lesssim \frac{1}{a^{\frac{1}{3}}}.
\end{split}
\end{equation} Most of the other terms are bounded above by $a^{-\frac{1}{3}}$ in a similar fashion. We will give the details for the sixth and fifteenth terms, which present the most difficulties. For the sixth term, the most borderline term appears when $(\tr\chi, \chihat) = \chihat$. We then  have 

\begin{equation}
\begin{split}
&\intubar \intu \aupr \scaletwoSuprimeubarprime{\ali \sumif \nablap \nabla^{i_3+1}\chihat \hnabla^{i_4}(\rhoF,\sigmaF) }^2 \duprime \dubarprime \\ \lesssim & \intubar \intu \frac{a \Gamma^2}{\upr^2}\cdot \frac{a}{\upr^2}\scaletwoSuprimeubarprime{\ali \nabla^{i+1}\chihat}^2 \duprime \dubarprime  + \intubar \intu \frac{a}{\upr^2}\frac{\Gamma^4}{\upr^2} \duprime \dubarprime \\ \lesssim &\frac{1}{\al} \Gammatop[\chihat]^2 + \frac{a \Gamma^4}{\upr^3}\lesssim \frac{1}{a^{\frac{1}{3}}}.
\end{split}
\end{equation}Finally, the fifteenth term can be estimated as follows:

\begin{equation}
\begin{split}
&\intubar \intu \aupr \scaletwoSuprimeubarprime{\ali \sumifi \nablapp \nablat(\chibarhat,\tr\chibar)\hnablaf \hnabla_4 \alphaF}^2 \duprime \dubarprime \\ \lesssim &
\intubar \intu \aupr \frac{\upr^4}{a^2}\cdot a \cdot \frac{\Gamma^6}{\upr^4}\duprime \dubarprime \lesssim \frac{\Gamma^6}{\lvert u \rvert} \lesssim \frac{1}{a^{\frac{1}{3}}}.
\end{split}
\end{equation}As a consequence, using \eqref{07}, one can conclude that 

\begin{equation}\label{13}
X \lesssim \frac{1}{a^{\frac{1}{3}}} \cdot \sup_{u_{\infty} \leq u^{\prime} \leq -\frac{a}{4}} \scaletwoHprime{ (\al \hnabla)^i \hnabla_4 \alphaF} \lesssim 1.
\end{equation} For $Y$, we once again have 

\begin{equation}\label{015}
Y \leq \left(\intubar \intu \frac{a}{\upr^2} \scaletwoSuprimeubarprime{\ali H_i}^2 \duprime \dubarprime\right)^{\frac{1}{2}}\cdot \sup_{0 \leq \ubar \leq 1} \scaletwoHprime{ (\al \hnabla)^i (\hnabla_4 \rhoF,\hnabla_4 \sigmaF)}.
\end{equation}
We begin with the first term in $H_i$. We have 
\begin{equation}
\begin{split}
&\intubar \intu \aupr \scaletwoSuprimeubarprime{\ali \sumitm \nablapp \hnabla^{i_3+1}\hnabla_4 \alphaF} \duprime \dubarprime \\ \lesssim &\intubar \intu \aupr\cdot \frac{\Gamma^2}{\upr^2} \scaletwoSuprimeubarprime{\haln  \hnabla_4 \alphaF}^2 \duprime \dubarprime \\ &+ \intubar \intu \aupr \cdot \frac{a \Gamma^4}{\upr^2}\duprime \dubarprime \lesssim \frac{1}{a^{\frac{1}{3}}},
\end{split}
\end{equation}where we have used the bootstrap assumptions. Most of the rest of the terms can be bounded above by $a^{-\frac{1}{3}}$ similarly. The terms which are most borderline are the fourth, tenth and thirteenth. We treat them one by one. 
\begin{equation}
\begin{split}
&\intubar \intu \aupr \scaletwoSuprimeubarprime{\ali \sumif \nablap \hnablat \alphaF \hnablaf \alphaF}^2 \duprime \dubarprime \\ \lesssim & \intubar \intu \aupr \cdot \frac{a^2 \Gamma_2[\alphaF]^2\Gamma_{\infty}[\alphaF]^2}{\upr^2} \duprime \dubarprime \lesssim (\mathcal{R} + \underline{\mathcal{R}}+1)^2 \lesssim R^2,
\end{split}
\end{equation}where again we have used the bootstrap assumptions \eqref{bootstrap}.  For the tenth term, we have 

\begin{equation}
\begin{split}
&\intubar\intu \aupr \scaletwoSuprimeubarprime{\ali \sumif \nablap \nabla^{i_3+1}\omega \hnablaf \alphaF}^2 \duprime \dubarprime \\ \lesssim &\intubar\intu \aupr \cdot \frac{a\Gamma^2}{\upr^2}\scaletwoSuprimeubarprime{\ali \nabla^{i+1}\omega}^2 \duprime \dubarprime \\+& \intubar\intu \aupr \frac{\Gamma^4}{\upr^2}\duprime \dubarprime \lesssim \frac{1}{a^{\frac{1}{3}}},
\end{split}
\end{equation}where we have used the bootstrap assumptions \eqref{ellboot}. For the thirteenth term, there holds

\begin{equation}\label{20}
\begin{split}
&\intubar \intu \aupr \scaletwoSuprimeubarprime{\ali \sumifi \nablap \nablat \psi_g \nablaf (\chibarhat,\tr\chibar) \hnablaF(\hnabla_4 \rhoF,\hnablaF \sigmaF)}^2 \duprime \dubarprime \\ \lesssim & \intubar \intu \aupr \cdot a \cdot \frac{\upr^4}{a^2}\cdot \frac{\Gamma^6}{\upr^4} \lesssim \frac{1}{\lvert u \rvert}\lesssim \frac{1}{a^{\frac{1}{3}}}.
\end{split}
\end{equation}Combining everything, we arrive at

\begin{equation}
Y \leq (R+1) \cdot \sup_{0\leq \ubar \leq 1} \scaletwoHbarprime{\haln (\hnabla_4 \rhoF, \hnabla_4 \sigmaF)}.
\end{equation}Taking into account \eqref{05}, \eqref{13} and \eqref{20}, we 

\begin{equation}
\begin{split}
&\scaletwoHu{(\al \hnabla)^i \hnabla_4 \alphaF}^2 + \scaletwoHbaru{(\al \hnabla)^i (-\hnabla_4\rhoF, \hnabla_4 \sigmaF)}^2 \\ \lesssim &\scaletwoHzero{(\al \hnabla)^i \hnabla_4 \alphaF}^2 + \scaletwoHbarzero{(\al \hnabla)^i (-\hnabla_4\rhoF, \hnabla_4 \sigmaF)}^2\\+&  \frac{1}{a^{\frac{1}{3}}}\cdot \sup_{u_{\infty}\leq u \leq - \frac{a}{4}}\scaletwoHprime{\haln \nabla_4 \alphaF}  +(R+1) \cdot \sup_{0\leq \ubar \leq 1} \scaletwoHbarprime{\haln (\hnabla_4 \rhoF, \hnabla_4 \sigmaF)}.
\end{split}
\end{equation}Multiplying everything by $\frac{1}{a}$ and using the bootstrap assumptions \eqref{bootstrap},\eqref{ellboot} we arrive at 
\begin{equation}
\begin{split}
&\frac{1}{a}\scaletwoHu{(\al \hnabla)^i \hnabla_4 \alphaF}^2 + \frac{1}{a}\scaletwoHbaru{(\al \hnabla)^i (-\hnabla_4\rhoF, \hnabla_4 \sigmaF)}^2 \\\lesssim  &\frac{1}{a}\scaletwoHzero{(\al \hnabla)^i \hnabla_4 \alphaF}^2 +\frac{1}{a} \scaletwoHbarzero{(\al \hnabla)^i (-\hnabla_4\rhoF, \hnabla_4 \sigmaF)}^2\\+&  \frac{1}{a^{\frac{5}{6}}}\cdot \left(\frac{1}{\al} \sup_{u_{\infty}\leq u \leq - \frac{a}{4}}\scaletwoHprime{\haln \nabla_4 \alphaF} \right) \\ +&\frac{R+1}{\al} \cdot \left( \frac{1}{\al} \sup_{0\leq \ubar \leq 1} \scaletwoHbarprime{\haln (\hnabla_4 \rhoF, \hnabla_4 \sigmaF)}\right) \\ \lesssim& \frac{1}{a}\scaletwoHzero{(\al \hnabla)^i \hnabla_4 \alphaF}^2 +\frac{1}{a} \scaletwoHbarzero{(\al \hnabla)^i (-\hnabla_4\rhoF, \hnabla_4 \sigmaF)}^2 +1,
\end{split}
\end{equation}whence we obtain the result by taking the square root.\\

\noindent Now we need to estimate the energy associated with the mixed derivative of $\alphabar^{F}$ i.e. $\hnabla^{I}\hnabla_{3}\alphabar^{F}$. Before proceeding, we need the following integration lemma. 
\begin{proposition}
Let us define $\Psi_{1}:=\frac{a}{|u|}\hnabla_{3}(\rho^{F},\sigma^{F})$ and $\Psi_{2}:=\frac{a}{|u|}\hnabla_{3}\alphabar^{F}$. The following Hodge system 
\begin{align}  \nonumber
\label{eq:nabla31}
&\hnabla_{3}\hnabla^{I-1}\Psi_{1}+(\frac{I-1}{2}+1)\tr\chibar\hnabla^{I-1}\Psi_{1}=\widehat{\mathcal{D}}\hnabla^{I-1}\Psi_{2}+\frac{a}{|u|}\mathcal{E}^{I}_{1},\\ \nonumber
&\hnabla_{4}\hnabla^{I-1}\Psi_{2}=~^{*}\widehat{\mathcal{D}}\hnabla^{I-1}\Psi_{1}+\frac{a}{|u|}\mathcal{E}^{I}_{2}
\end{align}
 verifies the energy estimate 
\begin{align}  \nonumber
&\int_{\Hu}\scaletwoSuubarprime{\hnabla^{I-1} (\frac{a}{|u|}\hnabla_{3}\rho^{F},\frac{a}{|u|}\hnabla_{3}\sigma^{F})}^2 \dubarprime + \int_{u_{\infty}}^{u}\frac{a}{\upr^2} \scaletwoSuprime{\hnabla^{I-1}(\frac{a}{|u^{'}|} \hnabla_{3}\alphabar^{F})}^2 \duprime \\  \nonumber \lesssim  &\int_{H_{u_{\infty}}^{(0,\ubar)}}\scaletwoSuzubarprime{\hnabla^{I-1} (\frac{a}{|u_{\infty}|}\hnabla_{3}\rho^{F},\frac{a}{|u_{\infty}|}\hnabla_{3}\sigma^{F})}^2 \dubarprime \\  \nonumber +& \int_{u_{\infty}}^{u} \frac{a}{\upr^2} \scaletwoSuzprime{\hnabla^{I-1} (\frac{a}{|u|}\hnabla_{3}\alphabar^{F})}^2 \duprime \\+&   \nonumber \iint_{\mathcal{D}_{u,\ubar}}\frac{a}{\upr} \scaleoneSuprimeubarprime{\langle\hnabla^{I-1} (\frac{a}{|u^{'}|}\hnabla_{3}\rho^{F},\frac{a}{|u^{'}|}\hnabla_{3}\sigma^{F}),\frac{a}{|u^{'}|} \hnabla^{I-1}\mathcal{E}_{1}\rangle}\duprime \dubarprime \\\nonumber +& \iint_{\mathcal{D}_{u,\ubar}}\frac{a}{\upr} \scaleoneSuprimeubarprime{\langle\hnabla^{I-1} (\frac{a}{|u^{'}|}\hnabla_{3}\alphabar^{F}),\frac{a}{|u^{'}|}\hnabla^{I-1}\mathcal{E}_{2}\rangle}\duprime \dubarprime.
\end{align}
\end{proposition}
\begin{proof}
First we denote $\frac{a}{|u|}\hnabla_{3}(\rho^{F},\sigma^{F})$ by $\Psi_{1}$ while $\frac{a}{|u|}\hnabla_{3}\alphabar^{F}$ by $\Psi_{2}$.
Now we commute the evolution equations in the proposition with $\hnabla^{I-1}$ to obtain  
\begin{align}
&\hnabla_{3}\hnabla^{I-1}\Psi_{1}+(\frac{I-1}{2}+1)\tr\chibar\hnabla^{I-1}\Psi_{1}=\widehat{\mathcal{D}}\hnabla^{I-1}\Psi_{2}+\frac{a}{|u|}\mathcal{E}^{I}_{1},\\&
\hnabla_{4}\hnabla^{I-1}\Psi_{2}=~^{*}\widehat{\mathcal{D}}\hnabla^{I-1}\Psi_{1}+\frac{a}{|u|}\mathcal{E}^{I}_{2}
\end{align}
Now apply Proposition \ref{integration2} with $J=I-2$ $\lambda_{0}=\frac{J}{2}+s_{2}(\Psi_{1})$, $\lambda_{1}=2\lambda_{0}-1$ to obtain the following 
\begin{align} \nonumber
&2\int_{D_{u,\ubar}}|u^{'}|^{2J-1+4s_{2}(\Psi_{1})}\langle\hnabla^{J}\Psi_{1},(\hnabla_{3}+(\frac{J}{2}+s_{2}(\Psi_{1}))\tr\chibar)\hnabla\Psi_{1}\rangle\nonumber=\int_{H_{u}(0,\ubar)}|u|^{2J-1+4s_{2}(\Psi_{1})}|\hnabla^{J}\Psi_{1}|^{2}\\ 
&-\int_{H_{u_{\infty}}(0,\ubar)}|u_{\infty}|^{2J-1+4s_{2}(\Psi_{1})}|\hnabla^{J}\Psi_{1}|^{2}+\int_{D_{u,\ubar}}|u^{'}|^{2J-1+4s_{2}(\Psi_{1})}f|\hnabla^{J}\Psi_{1}|^{2}
\end{align}
and for $\Psi_{2}$
\begin{eqnarray}
2\int_{D_{u,\ubar}}|u^{'}|^{2J-1+4s_{2}(\Psi_{1})}\langle\hnabla^{J}\Psi_{2},\hnabla_{4}\hnabla^{J}\Psi_{2}\rangle=\int_{\Hbar_{\ubar}(u_{\infty},u)}|u^{'}|^{2J-1+4s_{2}}|\hnabla^{J}\Psi_{2}|^{2}\\\nonumber-\int_{\Hbar_{0}(u_{\infty},u)}|u^{'}|^{2J-1+4s_{2}}|\hnabla^{J}\Psi_{2}|^{2}
+\int_{D_{u,\ubar}}|u^{'}|^{2J-1+4s_{2}(\Psi_{1})}(2\omega-\tr\chi)|\hnabla^{J}\Psi_{2}|^{2}.
\end{eqnarray}
Adding the previous two expressions we obtain 
\begin{align}
 \nonumber& 2\int_{D_{u,\ubar}}|u^{'}|^{2J-1+4s_{2}(\Psi_{1})}\langle\hnabla^{J}\Psi_{1},(\hnabla_{3} +(\frac{J}{2}+s_{2}(\Psi_{1}))\tr\chibar)\hnabla\Psi_{1}\rangle\nonumber\\ \nonumber +&2\int_{D_{u,\ubar}}|u^{'}|^{2J-1+4s_{2}(\Psi_{1})}\langle\hnabla^{J}\Psi_{2},\hnabla_{4}\hnabla^{J}\Psi_{2}\rangle\\\nonumber 
=&\int_{H_{u}(0,\ubar)}|u|^{2J-1+4s_{2}(\Psi_{1})}|\hnabla^{J}\Psi_{1}|^{2}+\int_{\Hbar_{\ubar}(u_{\infty},u)}|u^{'}|^{2J-1+4s_{2}}|\hnabla^{J}\Psi_{2}|^{2}\\  \nonumber -&\int_{H_{u_{\infty}}(0,\ubar)}|u_{\infty}|^{2J-1+4s_{2}(\Psi_{1})}|\hnabla^{J}\Psi_{1}|^{2}\\\nonumber 
-& \int_{\Hbar_{0}(u_{\infty},u)}|u^{'}|^{2J-1+4s_{2}}|\hnabla^{J}\Psi_{2}|^{2}+\int_{D_{u,\ubar}}|u^{'}|^{2J-1+4s_{2}(\Psi_{1})}f|\hnabla^{J}\Psi_{1}|^{2}\\\nonumber +&\int_{D_{u,\ubar}}|u^{'}|^{2J-1+4s_{2}(\Psi_{1})}(2\omega-\tr\chi)|\hnabla^{J}\Psi_{2}|^{2}.
\end{align}
Now use the equations of motion to yield 
\begin{align} \nonumber &
\int_{H_{u}(0,\ubar)}|u|^{2J-1+4s_{2}(\Psi_{1})}|\hnabla^{J}\Psi_{1}|^{2}+\int_{\Hbar_{\ubar}(u_{\infty},u)}|u^{'}|^{2J-1+4s_{2}(\Psi_{1})}|\hnabla^{J}\Psi_{2}|^{2}\\\nonumber 
=&\int_{H_{u_{\infty}}(0,\ubar)}|u_{\infty}|^{2J-1+4s_{2}(\Psi_{1})}|\hnabla^{J}\Psi_{1}|^{2} 
+\int_{\Hbar_{0}(u_{\infty},u)}|u^{'}|^{2J-1+4s_{2}(\Psi_{1})}|\hnabla^{J}\Psi_{2}|^{2}\\\nonumber +&2\int_{D_{u,\ubar}}|u^{'}|^{2J-1+4s_{2}(\Psi_{1})}\langle\hnabla^{J}\Psi_{1},\frac{a}{|u^{'}|}\mathcal{E}^{I}_{1}\rangle+2\int_{D_{u,\ubar}}|u^{'}|^{2J-1+4s_{2}(\Psi_{1})}\langle\hnabla^{J}\Psi_{2},\frac{a}{|u^{'}|}\mathcal{E}^{I}_{2}\rangle\\\nonumber 
+&\int_{D_{u,\ubar}}|u^{'}|^{2J-1+4s_{2}(\Psi_{1})}(\eta+\etabar)\langle\hnabla^{J}\Psi_{1},\hnabla^{J}\Psi_{2}\rangle-\int_{D_{u,\ubar}}|u^{'}|^{2J-1+4s_{2}(\Psi_{1})}f|\hnabla^{J}\Psi_{1}|^{2}\\ -&\int_{D_{u,\ubar}}|u^{'}|^{2J-1+4s_{2}(\Psi_{1})}(2\omega-\tr\chi)|\hnabla^{J}\Psi_{2}|^{2}.
\end{align}
After applying Gr\"onwall's inequality, we obtain 
\begin{align} \nonumber &
\int_{H_{u}(0,\ubar)}|u|^{2J-1+4s_{2}(\Psi_{1})}|\hnabla^{J}\Psi_{1}|^{2}+\int_{\Hbar_{\ubar}(u_{\infty},u)}|u^{'}|^{2J-1+4s_{2}(\Psi_{1})}|\hnabla^{J}\Psi_{2}|^{2}\\\nonumber 
\lesssim &\bigg(\int_{H_{u_{\infty}}(0,\ubar)}|u_{\infty}|^{2J-1\nonumber+4s_{2}(\Psi_{1})}|\nabla^{J}\Psi_{1}|^{2} 
+\int_{\Hbar_{0}(u_{\infty},u)}|u^{'}|^{2J-1+4s_{2}(\Psi_{1})}|\hnabla^{I}\Psi_{2}|^{2}.\\ + &|\int_{D_{u,\ubar^{'}}}|u^{'}|^{2J-1+4s_{2}(\Psi_{1})}\langle\hnabla^{J}\Psi_{1},\frac{a}{|u^{'}|}\mathcal{E}^{I}_{1}\rangle|
+|\int_{D_{u,\ubar}}|u^{'}|^{2J-1+4s_{2}(\Psi_{1})}\langle\hnabla^{J}\Psi_{2},\frac{a}{|u^{'}|}\mathcal{E}^{I}_{2}\rangle|\bigg).
\end{align}
Similarly,
\begin{align}  \nonumber &
\int_{D_{u,\ubar}}|u^{'}|^{2J-1+4s_{2}(\Psi_{1})}|\langle\hnabla^{J}\Psi_{1},\frac{a}{|u^{'}|}\mathcal{E}^{I}_{1}\rangle|\\\nonumber=& \int_{D_{u,\ubar}}|u^{'}|^{2J-1+4s_{2}(\Psi_{1})}a^{J+1+2s_{2}(\Psi_{1})}|u^{'}|^{-2J-1-4s_{2}(\Psi_{1})}|\langle\hnabla^{J}\Psi_{1},\frac{a}{|u^{'}|}\mathcal{E}^{I}_{1}\rangle|_{L^{1}_{sc}(S_{u^{'},\ubar^{'}})}du^{'}d\ubar^{'}\\
=&\int_{D_{u,\ubar}}a^{J+1+2s_{2}(\Psi_{1})}|u^{'}|^{-2}|\langle\hnabla^{J}\Psi_{1},\frac{a}{|u^{'}|}\mathcal{E}^{I}_{1}\rangle|_{L^{1}_{sc}(S_{u^{'},\ubar^{'}})}du^{'}d\ubar^{'}
\end{align}
and 
\begin{align}& \nonumber
\int_{D_{u,\ubar}}|u^{'}|^{2J-1+4s_{2}(\Psi_{1})}|\langle\hnabla^{J}\Psi_{2},\frac{a}{|u^{'}|}\hnabla^{J}\mathcal{E}_{2}\rangle|\\ =&\int_{u,\ubar}a^{J+1+2s_{2}(\Psi_{1})}|u^{'}|^{-2}|\langle\hnabla^{J}\Psi_{2},\frac{a}{|u^{'}|}\nonumber\hnabla^{J}\mathcal{E}_{2}\rangle|_{L^{1}_{sc}(S_{u^{'},\ubar^{'}})}du^{'}d\ubar^{'}.
\end{align}
Collecting all the terms together and noting that $|u|\leq |u^{'}|$ yields the result.
\end{proof}
\noindent The following proposition allows us to obtain  the necessary estimates. 
\begin{proposition}
Under the assumption of the main theorem \ref{main1} and the bootstrap assumptions (\ref{bootstrap}), the following energy estimate is satisfied for all $0\leq I\leq N+5$:
\begin{equation*}
\begin{split}
&\int_{\Hu}\scaletwoSuubarprime{\hnabla^{I-1} (\frac{a}{|u|}\hnabla_{3}\rho^{F},\frac{a}{|u|}\hnabla_{3}\sigma^{F})}^2 \dubarprime + \int_{u_{\infty}}^{u}\frac{a}{\upr^2} \scaletwoSuprime{\hnabla^{I-1}(\frac{a}{|u^{'}|} \hnabla_{3}\alphabar^{F})}^2 \duprime \\ \lesssim &\int_{H_{u_{\infty}}^{(0,\ubar)}}\scaletwoSuzubarprime{\hnabla^{I-1} (\frac{a}{|u_{\infty}|}\hnabla_{3}\rho^{F},\frac{a}{|u_{\infty}|}\hnabla_{3}\sigma^{F})}^2 \dubarprime \\\nonumber  +& \int_{u_{\infty}}^{u} \frac{a}{\upr^2} \scaletwoSuzprime{\hnabla^{I-1} (\frac{a}{|u^{'}|}\hnabla_{3}\alphabar^{F})}^2 \duprime +1.
\end{split}
\end{equation*}
\end{proposition}
\begin{proof}
First we compute the error terms $\mathcal{E}_{1}$ and $\mathcal{E}_{2}$. An explicit calculation using the null Yang-Mills equations yields 
\begin{align} \nonumber 
\mathcal{E}_{1}\sim &\left(\nabla\omega-\chibar\cdot(\eta+\etabar)+\omegabar(\eta+\etabar)-\betabar\nonumber+\epsilon\sigma^{F}\cdot\alphabar^{F}-\rho^{F}\cdot\alphabar^{F}\right)\alphabar^{F} +(\eta-\etabar)\cdot\hnabla_{3}\alphabar^{F}\\\nonumber +&(\rho^{F},\sigma^{F})\nabla_{3}\tr\chibar +(\widetilde{\betabar}+\alphabar^{F}\cdot (\rho^{F}+\sigma^{F}))\alphabar^{F}
+\alphabar^{F}\alphabar^{F}+(\eta+\etabar)\hnabla_{3}\alphabar^{F}-\chibar\cdot\hnabla\alphabar^{F}+\chibar\eta\alphabar^{F},
\end{align}
\begin{align} \nonumber
\mathcal{E}_{2}\sim & \alphabar^{F}\rho^{F}+(\eta,\etabar)(-\tr\chibar (\rho^F,\sigma^{F})-\nonumber \widehat{\text{div}}\alphabar^F +( \eta,\etabar)\cdot \alphabar^F)-\chibar\hnabla\rho^{F}+\chibar\eta\rho^{F}\\\nonumber +&\alphabar^{F}\sigma^{F}-\chibar\hnabla\sigma^{F}+\chibar\eta\sigma^{F}
+\omega\hnabla_{3}\alphabar^{F}-\omegabar\hnabla\rho^{F}+\omegabar~^{*}\hnabla\sigma^{F}+(\eta-\etabar)\hnabla\alphabar^{F}\\\nonumber +&\sigma\alphabar^{F}+\rho^{F}\alphabar^{F}+(|\rho^{F}|^{2}+|\sigma^{F}|^{2})\alphabar^{F} +\omegabar(\tr\chi\alphabar^F-2 \Hodge{\etabar} \sigma^F- 2 \etabar \rho^F + 2 \omega \alphabar^F\\\nonumber -& \chibarhat \cdot \alpha^F)-\frac{1}{2}\tr\chi\hnabla_{3}\alphabar^{F}.
\end{align}
Note that the expression of $\mathcal{E}_{1}$ contains $\nabla_{3}\tr\chibar$, which contains $(\tr\chibar)^{2}$, a double anomaly that might cause problems in closing the estimates later. Therefore, we work with the re-normalized entity $\widetilde{\tr\chibar}:=\tr\chibar+\frac{2}{|u|}$. We write $\nabla_{3}\tr\chibar$ as follows:
\begin{align}\nonumber
\nabla_{3}(\tr\chibar)=& \nabla_{3}(\tr\chibar+\frac{2}{|u|}-\frac{2}{|u|})=\nabla_{3}\widetilde{\tr\chibar}+\frac{2\Omega^{-1}}{|u|^{2}}\\\nonumber 
\sim& \frac{2}{|u|^{2}}(\Omega^{-1}-1)+|\widetilde{\tr\chibar}|^{2}+2\omegabar\tr\chibar-|\chibarhat|^{2}-|\alphabar^{F}|^{2}+\frac{2\Omega^{-1}}{|u|^{2}}.
\end{align}
Therefore $\mathcal{E}_{1}$ reads 
\begin{align} \nonumber
\mathcal{E}_{1}\sim &\left(\nabla\omega-\chibar\cdot(\eta+\etabar)+\omegabar(\eta+\etabar)-\betabar\nonumber+\epsilon\sigma^{F}\cdot\alphabar^{F}-\rho^{F}\cdot\alphabar^{F}\right)\alphabar^{F}\\\nonumber 
+&(\eta-\etabar)\cdot\hnabla_{3}\alphabar^{F}+(\rho^{F},\sigma^{F})(\frac{2}{|u|^{2}}(\Omega^{-1}-1)+|\widetilde{\tr\chibar}|^{2}+2\omegabar\tr\chibar-|\chibarhat|^{2}-|\alphabar^{F}|^{2}+\frac{2\Omega^{-1}}{|u|^{2}})\\\nonumber +&(\widetilde{\betabar}+\alphabar^{F}\cdot (\rho^{F}+\sigma^{F}))\alphabar^{F}
+\alphabar^{F}\alphabar^{F}+(\eta+\etabar)\hnabla_{3}\alphabar^{F}-\chibar\cdot\hnabla\alphabar^{F}+\chibar\eta\alphabar^{F},
\end{align}
Now using the commutation formulae \eqref{c1}-\eqref{c2} as appropriate, we obtain 
\begin{align}  \nonumber
\mathcal{E}^{I}_{1}\sim &\sum_{J_{1}+J_{2}=I-1}\nabla^{J_{1}+1}\tr\underline{\chi}\hnabla^{J_{2}}\Psi_{1}\\\nonumber+& \sum_{J_{1}+J_{2}+J_{3}+J_{4}=I-1}\nabla^{J_{1}}(\eta+\underline{\eta})^{J_{2}+1}\nabla^{J_{3}}(\widehat{\underline{\chi}},\tr\underline{\chi})\hnabla^{J_{4}}\Psi_{1}
\\ \nonumber+&\sum_{J_1+J_2+J_3+J_4=I-1}\nabla^{J_1}(\eta+\etabar)^{J_2}\nabla^{J_3}(\chibarhat,\hsp \tildetr)\hnabla^{J_4}\Psi_{1}\\\nonumber +&\sum_{J_{1}+J_{2}+J_{3} +J_{4}=J-1}\nabla^{J_{1}}(\eta+\underline{\eta})^{J_{2} +1}\hat{\nabla}^{J_{3}}\tr\underline{\chi}\hnabla^{J_{4}}\Psi_{1}\\\nonumber+& \sum_{J_1+J_2+J_3+J_4=J-1}\nabla^{J_1} (\eta+\etabar)^{J_2+1}\nonumber\nabla^{J_3}(\chibarhat,\tr\chibar)\hnabla^{J_4}\Psi_{1}\\\nonumber+& \sum_{J_1+J_2+J_3+J_4=J-1}\nabla^{J_1}(\eta+\etabar)^{J_2}\hnabla^{J_3}\alphabar^F \hnabla^{J_4}\Psi_{1}\\\nonumber+& \sum_{J_{1}+J_{2}+J_{3}+J_{4}+J_{5}=J-1}\nabla^{J_{1}}(\eta+\underline{\eta})^{J_{2}}\hnabla^{J_{3}}\alphabar^{F}\hnabla^{J_{4}}(\rho^{F},\sigma^{F})\hnabla^{J_{5}}\Psi_{1}\\\nonumber 
+&\sum_{J_1+J_2+J_3+J_{4}=J}\nabla^{J_1}(\eta+\etabar)^{J_2}\nabla^{J_{3}+1}\omega\hnabla^{J_{4}}\alphabar^{F}\\\nonumber +&\sum_{J_1+J_2+J_3+J_{4}+J_{5}=J}\nabla^{J_1}(\eta+\etabar)^{J_2}\nabla^{J_{3}}\chibar\nabla^{J_{4}}(\eta,\etabar)\hnabla^{J_{5}}\alphabar^{F}\\\nonumber+& \sum_{J_1+J_2+J_3+J_{4}+J_{5}=J}\nabla^{J_1}(\eta+\etabar)^{J_2}\nabla^{J_{3}}\omegabar\nabla^{J_{4}}(\eta,\etabar)\hnabla^{J_{5}}\alphabar^{F}\\\nonumber 
+& \sum_{J_1+J_2+J_3+J_{4}=J}\nabla^{J_1}(\eta+\etabar)^{J_2}\nabla^{J_{3}}\widetilde{\betabar}\hnabla^{J_{4}}\alphabar^{F}\\\nonumber 
+& \sum_{J_1+J_2+J_3+J_{4}=J}\nabla^{J_1}(\eta+\etabar)^{J_2}\nabla^{J_{3}}(\eta,\etabar)\hnabla^{J_{4}}\hnabla_{3}\alphabar^{F}\\\nonumber +&\underbrace{\sum_{J_1+J_2+J_3+J_{4}+J_{5}=J}\nabla^{J_1}(\eta+\etabar)^{J_2}\nabla^{J_{3}}\widetilde{\tr\chibar}\nabla^{J_{4}}\widetilde{\tr\chibar}\hnabla^{J_{5}}\rho^{F}}_{I}\\\nonumber
+&\underbrace{\frac{1}{|u|^{2}}\sum_{J_1+J_2+J_3=I-1}\nabla^{J_1}(\eta+\etabar)^{J_2}\hnabla^{J_{3}}\rho^{F}}_{II}\\\nonumber
+&\frac{1}{|u|^{2}}\sum_{J_1+J_2+J_3+J_4=I-1}\nabla^{J_1}(\eta+\etabar)^{J_2}\nabla^{J_{3}}(\Omega^{-1}-1)\hnabla^{J_{4}}\rho^{F}\\\nonumber
+&\sum_{J_1+J_2+J_3+J_{4}+J_{5}=J}\nabla^{J_1}(\eta+\etabar)^{J_2}\nabla^{J_{3}}\chibarhat\nabla^{J_{4}}\chibarhat\hnabla^{J_{5}}\rho^{F}\\\nonumber 
+&\sum_{J_1+J_2+J_3+J_{4}+J_{5}=J}\nabla^{J_1}(\eta+\etabar)^{J_2}\nabla^{J_{3}}\omegabar\nabla^{J_{4}}\tr\chibar\hnabla^{J_{5}}\rho^{F}\\\nonumber 
+&\sum_{J_1+J_2+J_3+J_{4}+J_{5}=J}\nabla^{J_1}(\eta+\etabar)^{J_2}\hnabla^{J_{3}}\alphabar^{F}\hnabla^{J_{4}}\alphabar^{F}\hnabla^{J_{5}}\rho^{F}\\\nonumber 
+&\sum_{J_1+J_2+J_3+J_{4}+J_{5}=J}\nabla^{J_1}(\eta+\etabar)^{J_2}\hnabla^{J_{3}}\alphabar^{F}\hnabla^{J_{4}}\alphabar^{F}\hnabla^{J_{5}}\rho^{F}\\\nonumber 
+&\sum_{J_1+J_2+J_3+J_{4}+J_{5}=J}\nabla^{J_1}(\eta+\etabar)^{J_2}\hnabla^{J_{3}}\alphabar^{F}\hnabla^{J_{4}}(\rho^{F},\sigma^{F})\hnabla^{J_{5}}\alphabar^{F}\\\nonumber 
+&\sum_{J_1+J_2+J_3+J_{4}=J}\nabla^{J_1}(\eta+\etabar)^{J_2}\hnabla^{J_{3}}\alphabar^{F}\hnabla^{J_{4}}\alphabar^{F}\\\nonumber 
+&\sum_{J_1+J_2+J_3+J_{4}=J}\nabla^{J_1}(\eta+\etabar)^{J_2}\nabla^{J_{3}}(\eta,\etabar)\hnabla^{J_{4}}\hnabla_{3}\alphabar^{F}\\\nonumber 
+&\sum_{J_1+J_2+J_3+J_{4}+J_{5}=J}\nabla^{J_1}(\eta+\etabar)^{J_2}\nabla^{J_{3}}\chibar\nabla^{J_{4}}\eta\hnabla^{J_{5}}\alphabar^{F}\\ 
+&\underbrace{\sum_{J_1+J_2+J_3+J_{4}=J}\nabla^{J_1}(\eta+\etabar)^{J_2}\nabla^{J_{3}}\chibar\hnabla^{J_{4}+1}\alphabar^{F}}_{III}
\end{align}
and similarly
\begin{align}  \nonumber
\mathcal{E}^{I}_{2}\sim& \sum_{J_{1}+J_{2}+J_{3}+J_{4}=J-1}\nabla^{J_{1}}(\eta+\underline{\eta})^{J_{2}}\nabla^{J_{3}}\beta\hnabla^{J_{4}}\Psi_{2}\\\nonumber +&\sum_{J_{1}+J_{2}+J_{3}+J_{4}+J_{5}=J-1}\nabla^{J_{1}}(\eta+\underline{\eta})^{J_{2}}\hnabla^{J_{3}}\alpha^{F}\hnabla^{J_{4}}(\rho^{F},\sigma^{F})\hnabla^{J_{5}}\Psi_{2}\\  \nonumber
+&\sum_{J_{1}+J_{2}+J_{3}+J_{4}=I-1}\nabla^{J_{1}}(\eta+\underline{\eta})^{J_{2}}\hnabla^{J_{3}}\alpha^{F}\hnabla^{J_{4}}\Psi_{2}\\\nonumber+& \sum_{J_{1}+J_{2}+J_{3} +J_{4}=I-1}\nabla^{J_{1}}(\eta+\underline{\eta})^{J_{2}}\nabla^{J_{3}}\chi\hnabla^{J_{4}}\Psi_{2}\\\nonumber +&\sum_{J_{1}+J_{2}+J_{3}=I-1}\nabla^{J_{1}}(\eta+\underline{\eta})^{J_{2}}\hnabla^{J_{3}}\alphabar^{F}\hnabla^{J_{4}}\rho^{F}\\\nonumber 
+&\sum_{J_{1}+J_{2}+J_{3}=I-1}\nabla^{J_{1}}(\eta+\underline{\eta})^{J_{2}}\hnabla^{J_{3}}\alphabar^{F}\hnabla^{J_{4}}\rho^{F}\\\nonumber 
+&\sum_{J_{1}+J_{2}+J_{3}+J_{4}+J_{5}=I-1}\nabla^{J_{1}}(\eta+\underline{\eta})^{J_{2}}\nabla^{J_{3}}(\eta+\etabar)\nabla^{J_{4}}\tr\chibar\hnabla^{J_{5}}(\rho^{F},\sigma^{F})\\\nonumber 
+&\sum_{J_{1}+J_{2}+J_{3}+J_{4}=I-1}\nabla^{J_{1}}(\eta+\underline{\eta})^{J_{2}}\nabla^{J_{3}}(\eta+\etabar)\hnabla^{J_{4}+1}\alphabar^{F}\\\nonumber 
+&\sum_{J_{1}+J_{2}+J_{3}+J_{4}+J_{5}=I-1}\nabla^{J_{1}}(\eta+\underline{\eta})^{J_{2}}\nabla^{J_{3}}(\eta+\etabar)\nabla^{J_{4}}(\eta+\etabar)\hnabla^{J_{5}}\alphabar^{F}\\\nonumber 
+&\underbrace{\sum_{J_{1}+J_{2}+J_{3}+J_{4}=I-1}\nabla^{J_{1}}(\eta+\underline{\eta})^{J_{2}}\nabla^{J_{3}}\chibar\hnabla^{J_{4}+1}(\rho^{F},\sigma^{F})}_{potentially~dangerous~term}\\\nonumber 
+&\sum_{J_{1}+J_{2}+J_{3}+J_{4}+J_{5}=I-1}\nabla^{J_{1}}(\eta+\underline{\eta})^{J_{2}}\hnabla^{J_{3}}(\eta+\etabar)\hnabla^{J_{4}}\chibar \hnabla^{J_{5}}(\rho^{F},\sigma^{F})\\\nonumber 
+&\sum_{J_{1}+J_{2}+J_{3}+J_{4}+J_{5}=I-1}\nabla^{J_{1}}(\eta+\underline{\eta})^{J_{2}}\nabla^{J_{3}}(\eta+\etabar)\nabla^{J_{4}}\chibar\hnabla^{J_{5}}\sigma^{F}\\\nonumber
+&\sum_{J_{1}+J_{2}+J_{3}+J_{4}=I-1}\nabla^{J_{1}}(\eta+\underline{\eta})^{J_{2}}\nabla^{J_{3}}\omega\nabla^{J_{4}}\hnabla_{3}\alphabar^{F}\\\nonumber 
+&\sum_{J_{1}+J_{2}+J_{3}+J_{4}=I-1}\nabla^{J_{1}}(\eta+\underline{\eta})^{J_{2}}\nabla^{J_{3}}\omegabar\nabla^{J_{4}+1}(\rho^{F},\sigma^{F})\\\nonumber 
+&\sum_{J_{1}+J_{2}+J_{3}+J_{4}=I-1}\nabla^{J_{1}}(\eta+\underline{\eta})^{J_{2}}\nabla^{J_{3}}(\eta+\etabar)\nabla^{J_{4}+1}\alphabar^{F}\\\nonumber 
+&\sum_{J_{1}+J_{2}+J_{3}+J_{4}=I-1}\nabla^{J_{1}}(\eta+\underline{\eta})^{J_{2}}\nabla^{J_{3}}\alphabar^{F}\nabla^{J_{4}}(\rho^{F},\sigma^{F})\\\nonumber 
+&\sum_{J_{1}+J_{2}+J_{3}+J_{4}+J_{5}=I-1}\nabla^{J_{1}}(\eta+\underline{\eta})^{J_{2}}\nabla^{J_{3}}(\rho^{F},\sigma^{F})\nabla^{J_{4}}(\rho^{F},\sigma^{F})\hnabla^{J_{5}}\alphabar^{F}\\\nonumber 
+&\sum_{J_{1}+J_{2}+J_{3}+J_{4}+J_{5}=I-1}\nabla^{J_{1}}(\eta+\underline{\eta})^{J_{2}}\nabla^{J_{3}}\omegabar\nabla^{J_{4}}\tr\chibar\hnabla^{J_{5}}\alphabar^{F}\\\nonumber 
+&\sum_{J_{1}+J_{2}+J_{3}+J_{4}+J_{5}=I-1}\nabla^{J_{1}}(\eta+\underline{\eta})^{J_{2}}\nabla^{J_{3}}\omegabar\nabla^{J_{4}}\etabar\hnabla^{J_{5}}(\rho^{F},\sigma^{F})\\\nonumber 
+&\sum_{J_{1}+J_{2}+J_{3}+J_{4}+J_{5}=I-1}\nabla^{J_{1}}(\eta+\underline{\eta})^{J_{2}}\nabla^{J_{3}}\omegabar\nabla^{J_{4}}\omega\hnabla^{J_{5}}\alphabar^{F}\\\nonumber 
+&\sum_{J_{1}+J_{2}+J_{3}+J_{4}+J_{5}=I-1}\nabla^{J_{1}}(\eta+\underline{\eta})^{J_{2}}\nabla^{J_{3}}\omegabar\nabla^{J_{4}}\chibarhat\hnabla^{J_{5}}\alpha^{F}\\
+&\sum_{J_{1}+J_{2}+J_{3}+J_{4}=I-1}\nabla^{J_{1}}(\eta+\underline{\eta})^{J_{2}}\nabla^{J_{3}}\tr\chi\nabla^{J_{4}}\hnabla_{3}\alphabar^{F}.
\end{align} 
In the error term, we will only estimate the most dangerous and borderline terms. The most dangerous error term is estimated as follows:
\begin{align}  \nonumber
&\iint_{u,\ubar}\frac{a^{2}}{\upr^2} \scaleoneSuprimeubarprime{\langle a^{\frac{I-1}{2}}\hnabla^{I-1} \hnabla_{3}\alphabar^{F}, \frac{a}{|u^{'}|}a^{\frac{I-1}{2}}\tr\chibar\cdot \hnabla^{I}\rho^{F}\rangle}\duprime \dubarprime\\\nonumber 
\lesssim &\iint_{u,\ubar}\frac{a^{3}}{|u^{'}|^{5}}\frac{|u^{'}|^{2}}{a}||a^{\frac{I-1}{2}}\hnabla^{I-1} \hnabla_{3}\alphabar^{F}||_{L^{2}_{sc}}\frac{a}{|u^{'}|^{2}}||\tr\chibar||_{L^{\infty}_{sc}}||a^{\frac{I-1}{2}}\hnabla^{I}\rho^{F}||_{L^{2}_{sc}}\duprime \dubarprime\\\nonumber 
 \lesssim &\iint_{u,\ubar}\frac{a^{2}}{|u^{'}|^{3}}||a^{\frac{I-1}{2}}\hnabla^{I-1} \hnabla_{3}\alphabar^{F}||_{L^{2}_{sc}}||a^{\frac{I-1}{2}}\hnabla^{I}\rho^{F}||_{L^{2}_{sc}}\duprime \dubarprime\\\nonumber
 \lesssim &\iint_{u,\ubar}\frac{a}{|u^{'}|^{2}}(\frac{a}{|u^{'}|}||a^{\frac{I-1}{2}}\hnabla^{I-1} \hnabla_{3}\alphabar^{F}||_{L^{2}_{sc}})||a^{\frac{I-1}{2}}\hnabla^{I}\rho^{F}||_{L^{2}_{sc}}\duprime \dubarprime\\\nonumber
\lesssim & \iint_{u,\ubar}\frac{a}{|u^{'}|^{2}}\left(\frac{a^{2}}{|u^{'}|^{2}}||a^{\frac{I-1}{2}}\hnabla^{I-1} \hnabla_{3}\alphabar^{F}||^{2}_{L^{2}_{sc}}+||a^{\frac{I-1}{2}}\hnabla^{I}\rho^{F}||^{2}_{L^{2}_{sc}}\right)\duprime \dubarprime\\\nonumber 
\lesssim &\int_{\ubar}\int_{u}\frac{a^{3}}{|u^{'}|^{4}}||a^{\frac{I-1}{2}}\hnabla^{I-1} \hnabla_{3}\alphabar^{F}||^{2}_{L^{2}_{sc}}\duprime \dubarprime+\int_{u}\frac{a}{|u^{'}|^{2}}(\int_{\ubar}||a^{\frac{I-1}{2}}\hnabla^{I}\rho^{F}||^{2}_{L^{2}_{sc}})\duprime \dubarprime\\ 
\lesssim &1+\int_{\ubar}\int_{u}\frac{a^{3}}{|u^{'}|^{4}}||a^{\frac{I-1}{2}}\hnabla^{I-1} \hnabla_{3}\alphabar^{F}||^{2}_{L^{2}_{sc}}\duprime \dubarprime.
\end{align}
Other ill-behaved (but not necessarily borderline) terms include $\mathcal{E}^{I}_{2}$, for example, $\frac{a}{|u|}\chibarhat\cdot \hnabla^{I+1}\rho^{F}\subset \mathcal{E}^{I}_{2}$ and $\frac{1}{|u|}\hnabla^{I-1}\Psi_{1}\subset \mathcal{E}^{I}_{1}$ among others. These are estimated as follows: 
\begin{align}
\nonumber &\iint_{u,\ubar}\frac{a^{2}}{\upr^2} \scaleoneSuprimeubarprime{\langle a^{\frac{I-1}{2}}\hnabla^{I-1} \hnabla_{3}\alphabar^{F}, \frac{a}{|u^{'}|}a^{\frac{I-1}{2}}\chibarhat\cdot \hnabla^{I}\rho^{F}\rangle}\duprime \dubarprime\\\nonumber 
\lesssim &\iint_{u,\ubar}\frac{a^{\frac{5}{2}}}{|u^{'}|^{4}}||a^{\frac{I-1}{2}}\hnabla^{I-1} \hnabla_{3}\alphabar^{F}||_{L^{2}_{sc}(S_{u^{'},\ubar^{'})}}\frac{a^{\frac{1}{2}}}{|u^{'}|}||\chibarhat||_{L^{\infty}_{sc}(S_{u^{'},\ubar^{'}})}||a^{\frac{I-1}{2}}\hnabla^{I}\rho^{F}||_{L^{2}_{sc}(S_{u^{'},\ubar^{'})}}\duprime \dubarprime\\\nonumber 
\lesssim &\frac{1}{a^{\frac{1}{2}}}\iint_{u,\ubar}\frac{a^{3}}{|u^{'}|^{4}}||a^{\frac{I-1}{2}}\hnabla^{I-1} \hnabla_{3}\alphabar^{F}||^{2}_{L^{2}_{sc}(S_{u^{'},\ubar^{'})}}\duprime \dubarprime+\iint_{u,\ubar}\frac{a^{\frac{5}{2}}}{|u^{'}|^{4}}||a^{\frac{I-1}{2}}\hnabla^{I}\rho^{F}||^{2}_{L^{2}_{sc}(S_{u^{'},\ubar^{'})}}\duprime \dubarprime \\
\lesssim  &\left(\frac{1}{a^{\frac{1}{2}}}\mathbb{YM}^{2}+\frac{a^{\frac{5}{2}}}{|u|^{3}}\mathbb{YM}^{2}\right)
\lesssim 1,
\end{align}
\begin{align} \nonumber
&\iint_{u,\ubar}\frac{a^{2}}{\upr^2} \scaleoneSuprimeubarprime{\langle a^{\frac{I-1}{2}}\hnabla^{I-1} \hnabla_{3}(\rho^{F},\sigma^{F}) , \frac{1}{|u|}a^{\frac{I-1}{2}}\hnabla^{I-1}\hnabla_{3}(\rho^{F},\sigma^{F})\rangle}\duprime \dubarprime\\\nonumber 
\lesssim& \iint_{u,\ubar}\frac{a^{2}}{|u^{'}|^{4}}||a^{\frac{I-1}{2}}\hnabla^{I-1} \hnabla_{3}(\rho^{F},\sigma^{F})||^{2}_{L^{2}_{sc}(S_{u^{'},\ubar^{'}})}\duprime \dubarprime\\
\lesssim& \frac{1}{|u|}\sup_{u^{'}\in[u,u_{\infty}]}\int_{H^{(0,\ubar)}_{u^{'}}}\frac{a^{2}}{|u^{'}|^{2}}\scaletwoSuubarprime{a^{\frac{I-1}{2}}\hnabla^{I-1} (\hnabla_{3}(\rho^{F},\sigma^{F})}^2 d\ubar^{'}
\lesssim \frac{1}{|u|}\mathbb{YM}\lesssim 1,
\end{align}

\begin{align} 
\nonumber &\int_{u,\ubar}\frac{a}{\upr} \\  \nonumber&\scaleoneSuprimeubarprime{\langle a^{\frac{I-1}{2}}\hnabla^{I-1} \Psi_{1},\frac{a}{|u^{'}|}a^{\frac{I-1}{2}} \sum_{J_1+J_2+J_3+J_{4}=I-1}\nabla^{J_1}(\eta+\etabar)^{J_2}\nabla^{J_{3}}(\tr\chibar,\chibarhat)\hnabla^{J_{4}+1}\alphabar^{F}\rangle}\duprime \dubarprime\\\nonumber 
\lesssim &1+\int_{u,\ubar}\frac{a}{|u^{'}|^{2}}||a^{\frac{I-1}{2}}\hnabla^{I-1}\Psi_{1}||_{L^{2}_{sc}(S_{u^{'},\ubar})}||a^{\frac{I-1}{2}}\hnabla^{I}\alphabar^{F}||_{L^{2}_{sc}(S_{u^{'},\ubar})}\duprime \dubarprime \\ 
\lesssim &1+\mathbb{YM}^{2}[\alphabar^{F}]+\int_{\ubar}\int_{u_{\infty}}^{u}\frac{a}{|u^{'}|^{2}}||a^{\frac{I-1}{2}}\hnabla^{I-1}\Psi_{1}||^{2}_{L^{2}_{sc}(S_{u^{'},\ubar})}\duprime d\ubar^{'},
\end{align}
 
\begin{align}
\int_{u,\ubar}&\frac{a}{|u^{'}|} ||\langle\nonumber a^{\frac{I-1}{2}}\hnabla^{I-1} \Psi_{2},\frac{a}{|u^{'}|}a^{\frac{I-1}{2}}\sum_{J_{1}+J_{2}+J_{3}+J_{4}=J}\nabla^{J_{1}}(\eta+\underline{\eta})^{J_{2}}\nabla^{J_{3}}\tr\chi\nabla^{J_{4}}\hnabla_{3}\alphabar^{F}\rangle||_{L^{1}_{sc}(S_{u^{'},\ubar^{'}})}\duprime \dubarprime\\\nonumber 
\lesssim &1+\int_{u,\ubar}\frac{a^{2}}{|u^{'}|^{4}}\left(||a^{\frac{I-1}{2}}\hnabla^{I-1} \Psi_{2}||^{2}_{L^{2}_{sc}(S_{u^{'},\ubar})}+||a^{\frac{I-1}{2}}\hnabla^{I-1}\hnabla_{3}\alphabar^{F}||^{2}_{L^{2}_{sc}(S_{u^{'},\ubar^{'}})}\right)\\
\lesssim &1+\int_{\ubar}\int_{u}\frac{a^{3}}{|u^{'}|^{5}}||a^{\frac{I-1}{2}}\hnabla^{I-1} \hnabla_{3}\alphabar^{F}||^{2}_{L^{2}_{sc}(S_{u,\ubar})}du^{'}d\ubar^{'}+\frac{1}{a}\int_{\ubar}\int_{u}||a^{\frac{I-1}{2}}\hnabla^{I-1}\hnabla_{3}\alphabar^{F}||^{2}_{L^{2}_{sc}(S_{u^{'},\ubar})}du^{'}d\ubar^{'}\\\nonumber 
\lesssim& 1,
\end{align} 

\begin{align}
\int_{u,\ubar}&\frac{a}{|u^{'}|} \\  \nonumber &||\langle\nonumber a^{\frac{I-1}{2}}\hnabla^{I-1} \Psi_{1},\frac{a}{|u^{'}|}\sum_{J_1+J_2+J_3+J_{4}+J_{5}=J}\nabla^{J_1}(\eta+\etabar)^{J_2}\nabla^{J_{3}}\widetilde{\tr\chibar}\nabla^{J_{4}}\widetilde{\tr\chibar}\hnabla^{J_{5}}\rho^{F}\rangle||_{L^{1}_{sc}(S_{u^{'},\ubar^{'}})}\duprime \dubarprime\\\nonumber 
\lesssim& 1+\int_{u,\ubar}\frac{1}{|u^{'}|^{3}}\left(||\frac{a}{|u^{'}|}a^{\frac{I-1}{2}}\hnabla^{I-1}\hnabla_{3}\rho^{F}||^{2}_{L^{2}_{sc}(S_{u^{'},\ubar^{'}})}+||a^{\frac{I-1}{2}}\hnabla^{I-1}\rho^{F}||_{L^{2}_{sc}(S_{u^{'},\ubar^{'}})}\right)du^{'}d\ubar^{'}\\ 
\lesssim& 1+\int_{u_{\infty}}^{u}\frac{1}{|u^{'}|^{3}}\int_{0}^{\ubar}\frac{a^{2}}{|u^{'}|^{2}}||a^{\frac{I-1}{2}}\hnabla^{I-1}\hnabla_{3}\rho^{F}||^{2}_{L^{2}_{sc}(S_{u^{'},\ubar^{'}})}+\frac{1}{|u|^{2}}\mathbb{YM}[\alphabar^{F}]^{2},
\end{align} 
where we have used the $L^{2}$ estimate for the Yang-Mills curvature $\rho^{F}$. We estimate some of the remaining terms as follows: 
\begin{align}  \nonumber
&\int_{u,\ubar}\frac{a}{|u^{'}|} ||\langle\nonumber a^{\frac{I-1}{2}}\hnabla^{I-1} \Psi_{1},a^{\frac{I-1}{2}}\frac{a}{|u^{'}|}\frac{1}{|u^{'}|^{2}}\sum_{J_1+J_2+J_3=I-1}\nabla^{J_1}(\eta+\etabar)^{J_2}\hnabla^{J_{3}}\rho^{F}\rangle||_{L^{1}_{sc}(S_{u^{'},\ubar^{'}})}du^{'}d\ubar^{'}\\\nonumber 
\lesssim &1+\int_{u,\ubar}\frac{a^{2}}{|u^{'}|^{5}}\left(||a^{\frac{I-1}{2}}\hnabla^{I-1}\Psi||^{2}_{L^{2}_{sc}(S_{u^{'},\ubar^{'}})}+||a^{\frac{I-1}{2}}\hnabla^{I-1}\rho^{F}||_{L^{2}_{sc}(S_{u^{'},\ubar^{'}})}\right)du^{'}d\ubar^{'}\\\nonumber 
\lesssim & 1+\int_{u_{\infty}}^{u}\frac{a^{2}}{|u^{'}|^{5}}\int_{0}^{\ubar}||a^{\frac{I-1}{2}}\hnabla^{I-1}\Psi||^{2}_{L^{2}_{sc}(S_{u^{'},\ubar^{'}})}\dubarprime \duprime+\frac{a^{2}}{|u|^{4}}\mathbb{YM}[\alphabar^{F}]^{2}\\
\lesssim & 1+\int_{u_{\infty}}^{u}\frac{a^{2}}{|u^{'}|^{5}}\int_{0}^{\ubar}||a^{\frac{I-1}{2}}\hnabla^{I-1}\Psi||^{2}_{L^{2}_{sc}(S_{u^{'},\ubar^{'}})} \dubarprime \duprime,
\end{align}

\begin{align} \nonumber
&\int_{u,\ubar}\frac{a^I}{|u^{'}|}  ||\langle \frac{a}{|u^{'}|}\hnabla^{I-1}\hnabla_{3}\alphabar^{F},\frac{a}{|u^{'}|}\sum_{J_{1}+J_{2}+J_{3}+J_{4}=I-1}\nabla^{J_{1}}\nonumber(\eta+\underline{\eta})^{J_{2}}\nabla^{J_{3}}\chi\hnabla^{J_{4}}\hnabla_{3}\alphabar^{F}\rangle||_{L^{1}_{sc}(S_{u^{'},\ubar^{'}})}du^{'}d\ubar^{'}\\\nonumber 
\lesssim &1+\int_{u,\ubar}\frac{a^{3}}{|u^{'}|^{5}}||\hnabla^{I-1}\hnabla_{3}\alphabar^{F}||^{2}_{L^{2}_{sc}(S_{u^{'},\ubar^{'}})}||\chihat||_{L^{\infty}_{sc}(S_{u^{'},\ubar^{'}})}du^{'}d\ubar^{'}
\\ 
\lesssim &1+\frac{a^\frac{1}{2}}{|u|}\sup_{\ubar}\int_{u}\frac{a^{3}}{|u^{'}|^{4}}||\hnabla^{I-1}\hnabla_{3}\alphabar^{F}||^{2}_{L^{2}_{sc}(S_{u^{'},\ubar^{'}})}||\chihat||_{L^{\infty}_{sc}(S_{u^{'},\ubar^{'}})}du^{'},
\end{align}
\begin{align} \nonumber
\int_{u,\ubar} &\frac{a^I}{|u^{'}|}||\langle \frac{a}{|u^{'}|}\hnabla^{I-1}\hnabla_{3}\alphabar^{F},\frac{a}{|u^{'}|}\sum_{J_{1}+J_{2}+J_{3}+J_{4}=I-1}\nabla^{J_{1}}(\eta+\underline{\eta})^{J_{2}}\hnabla^{J_{3}}\alpha^{F}\hnabla^{J_{4}}\hnabla_{3}\alphabar^{F}\rangle||_{L^{1}_{sc}(S_{u^{'},\ubar^{'}})}du^{'}d\ubar^{'}\\\nonumber 
\lesssim & 1+\int_{u,\ubar}\frac{a^{3}}{|u^{'}|^{5}}||\hnabla^{I-1}\hnabla_{3}\alphabar^{F}||^{2}_{L^{2}_{sc}(S_{u^{'},\ubar^{'}})}||\alpha^{F}||_{L^{\infty}_{sc}(S_{u^{'},\ubar^{'}})}du^{'}d\ubar^{'}
\\\nonumber 
\lesssim &1+\frac{a^\frac{1}{2}}{|u|}\sup_{\ubar}\int_{u}\frac{a^{3}}{|u^{'}|^{4}}||\hnabla^{I-1}\hnabla_{3}\alphabar^{F}||^{2}_{L^{2}_{sc}(S_{u^{'},\ubar^{'}})}||\chihat||_{L^{\infty}_{sc}(S_{u^{'},\ubar^{'}})}du^{'}\\
\lesssim &
1+\frac{a^\frac{1}{2}}{|u|}\sup_{\ubar}\int_{u}\frac{a^{3}}{|u^{'}|^{4}}||\hnabla^{I-1}\hnabla_{3}\alphabar^{F}||^{2}_{L^{2}_{sc}(S_{u^{'},\ubar^{'}})}||\chihat||_{L^{\infty}_{sc}(S_{u^{'},\ubar^{'}})}du^{'}.
\end{align}
Collecting all the terms and using Gr\"onwall's inequality, we obtain 
\begin{equation*}
\begin{split}
&\int_{\Hu}\scaletwoSuubarprime{\hnabla^{I-1} (\frac{a}{|u|}\hnabla_{3}\rho^{F},\frac{a}{|u|}\hnabla_{3}\sigma^{F})}^2 \dubarprime + \int_{u_{\infty}}^{u}\frac{a}{\upr^2} \scaletwoSuprime{\hnabla^{I-1}(\frac{a}{|u^{'}|} \hnabla_{3}\alphabar^{F})}^2 \duprime \\ \lesssim &\int_{H_{u_{\infty}}^{(0,\ubar)}}\scaletwoSuzubarprime{\hnabla^{I-1} (\frac{a}{|u|_{\infty}}\hnabla_{3}\rho^{F},\frac{a}{|u|_{\infty}}\hnabla_{3}\sigma^{F})}^2 \dubarprime\\ +& \int_{u_{\infty}}^{u} \frac{a}{\upr^2} \scaletwoSuzprime{\hnabla^{I-1} (\frac{a}{|u^{'}|}\hnabla_{3}\alphabar^{F})}^2 \duprime +1.
\end{split}
\end{equation*}
This completes the energy estimates for $\hnabla_{3}\alphabar^{F}$.  
 
\end{proof}

\subsection{Energy Estimates for the Weyl/Riemann Curvature components}

\par \noindent For $(\Psi_1, \Psi_2) \in \begin{Bmatrix} (\alpha, \tbeta) , (\tbeta, (\rho,\sigma)), ((\rho,\sigma),\tbetabar), (\tbetabar, \alphabar) \end{Bmatrix}$ the energy estimates are carried out in Bianchi pairs, via the aid of the following proposition:

\begin{proposition}\label{curvprop}
Under the assumptions of Theorem \ref{main1} and the bootstrap assumptions \eqref{bootstrap}, for a Bianchi pair $(\Psi_1, \Psi_2)$ satisfying

\[ \nabla_3 \nabla^i \Psi_1 + \left( \frac{i+1}{2} + s_2(\Psi_1) \right) \tr\chibar \nabla^i \Psi_1 - \mathcal{D}\nabla^i \Psi_2 = P_i,\]\[\nabla_4 \nabla^i \Psi_2 - \Hodge{\mathcal{D}} \nabla^i \Psi_1 = Q_i, \]the following holds true:

\begin{equation}
\begin{split}
&\int_{\Hu}\scaletwoSuubarprime{\nabla^i \Psi_1}^2 \dubarprime + \int_{\Hbu}\frac{a}{\upr^2} \scaletwoSuprime{\nabla^i \Psi_2}^2 \duprime \\ \lesssim &\int_{H_{u_{\infty}}^{(0,\ubar)}}\scaletwoSuzubarprime{\nabla^i \Psi_1}^2 \dubarprime + \int_{\Hbar_0^{(u_{\infty},u)}} \frac{a}{\upr^2} \scaletwoSuzprime{\nabla^i \Psi_2}^2 \duprime \\&+  \iint_{\mathcal{D}_{u,\ubar}}\frac{a}{\upr} \scaleoneSuprimeubarprime{\nabla^i \Psi_1 \cdot P_i}\duprime \dubarprime + \iint_{\mathcal{D}_{u,\ubar}}\frac{a}{\upr} \scaleoneSuprimeubarprime{\nabla^i \Psi_2 \cdot Q_i}\duprime \dubarprime.
  \end{split}
\end{equation}
\end{proposition}
With this Proposition as the main tool, we begin with $(\alpha, \tbeta)$.
\begin{proposition}
\label{alphaen} Under the assumptions of Theorem \ref{main1} and the bootstrap assumptions \eqref{bootstrap}, there holds, for all $0\leq i \leq N+4$:

\begin{equation}
    \begin{split}
        &\frac{1}{\al} \scaletwoHu{\aln \alpha} + \frac{1}{\al}\scaletwoHbaru{\aln \tbeta}\\& \lesssim  \frac{1}{\al} \scaletwoHzero{\aln \alpha} + \frac{1}{\al}\scaletwoHbarzero{\aln \tbeta} + \frac{1}{a^{\frac{1}{3}}}. 
    \end{split}
\end{equation}
\end{proposition}

\begin{proof}
We recall the (schematic) Bianchi equations for $\nabla^i \alpha, \nabla^i \tbeta$:

    \begin{align}  \nonumber
       & \nabla_3 \nabla^i \alpha + \frac{i+1}{2}\tr\chibar \nabla^i\alpha - \mathcal{D}\nabla^{i}\tbeta \\ \nonumber =& \sumitm \nablapp \nabla^{i_3+1}\tbeta +\sumitm \nablapp \nablat \alpha\\  \nonumber +& \sumif \nablap \nablat (\psi_g,\chihat) \nablaf (\rho,\sigma, \tbeta) \\ \nonumber &+ \sumif \nablap \hnablat \Yub \hnabla^{i_4+1}\alphaF \\  \nonumber &+ \sum_{i_1+i_2+i_3+i_4=i}\nablap \hnablat (\alphaF, \Yub) \hnabla^{i_4+1}\Yub \\ \nonumber &+\sumiF \nablap \nablat(\psi_g,\chihat,\chibarhat,\tr\chibar)\hnablaf(\alphaF, \rhoF,\sigmaF)\hnablaF(\alphaF,\Yub)\\ \nonumber &+ \sumif \nablap \hnablat\Yub \hnablaf \hnabla_4\alphaF \\ \nonumber &+ \sumif \nablap \nablat (\chibarhat,\tildetr) \nablaf \alpha \\ \nonumber &+ \sumifi \nablapp \nablat \tr\chibar \nablaf \alpha \\ &+\sumifim \nablapp \nablat(\chibarhat,\tr\chibar)\nablaf \alpha:=P_i^1 +\dots + P_i^{10}.
    \end{align}Similarly, we have 
    
    \begin{align}
         \nonumber    \nabla_4 \nabla^i \tbeta - \Hodge{\mathcal{D}}\nabla^i \alpha= &\sumif \nablap \nablat(\psig, \chihat)\nablaf(\tbeta,\alpha) \\  \nonumber &+ \sumif \nablap \hnablat \alphaF \hnabla^{i_4+1}\alphaF \\ \nonumber &+ \sumiF \nablap \nablat(\psig,\chihat)\hnablaf \Yub \hnablaF \alphaF \\ \nonumber &+ \sumiF \nablap \nablat \psig \hnablaf \alphaF \hnablaF \alphaF \\ \nonumber &+ \sumif \nablap \nablat(\chihat,\tr\chi)\nablaf \tbeta \\ \nonumber &+ \sumifim \nablapp \nablat (\chihat,\tr\chi)\nablaf \tbeta \\  &+ \sumiFi \nablap \hnablat \alphaF \hnablaf(\rhoF,\sigmaF)\nablaF \tbeta := Q_i^1 + \dots + Q_i^7.
        \end{align}
Applying Proposition \ref{curvprop}, we have

\begin{equation}
    \begin{split}
        &\scaletwoHu{\aln \alpha} + \scaletwoHbaru{\aln \tbeta} \\ \lesssim & \scaletwoHzero{\aln \alpha} + \scaletwoHbarzero{\aln \tbeta} \\&+  \intubar \intu  \frac{a}{\upr} \scaleoneSuprimeubarprime{\ali P_i \cdot \aln \alpha }  \duprime \dubarprime
        \\ &+\intubar \intu  \frac{a}{\upr} \scaleoneSuprimeubarprime{\ali Q_i \cdot \aln \tbeta }  \duprime \dubarprime.
        \end{split}
\end{equation}By H\"older's inequality, one has

\begin{equation}
    \begin{split}
         &\intubar \intu  \frac{a}{\upr} \scaleoneSuprimeubarprime{\ali P_i \cdot \aln \alpha }  \duprime \dubarprime \\\leq&\intu \frac{a}{\upr^2} \sum_{j=1}^{10}\left(  \intubar \scaletwoSuprimeubarprime{\ali P_i^j }^2\dubarprime \right)^{\frac{1}{2}}\duprime
        \cdot \sup_{u^{\prime}} \lVert \aln \alpha \rVert_{L^2_{(sc)}(H_{u^{\prime}}^{(0,\ubar)})},
    \end{split}
\end{equation}Let us focus on the sum in the above line. For the first three terms, there holds

\[ \sum_{j=1}^3 \left( \intubar \scaletwoSuprimeubarprime{\ali P_i^j}^2 \dubarprime \right)^{\frac{1}{2}}\lesssim \frac{\al \Gamma \cdot R}{\upr}. \]For the fourth and fifth terms, there holds

\[ \left( \intubar \scaletwoSuprimeubarprime{\ali P_i^4}^2 \dubarprime \right)^{\frac{1}{2}} + \left( \intubar \scaletwoSuprimeubarprime{\ali P_i^5}^2 \dubarprime \right)^{\frac{1}{2}}\lesssim \frac{\al \Gamma \cdot M}{\upr}. \] For the sixth term, there holds

\[ \left( \intubar \scaletwoSuprimeubarprime{\ali P_i^6}^2 \dubarprime \right)^{\frac{1}{2}}\lesssim \Gamma^3. \]For the seventh term, there holds 

\[ \left( \intubar \scaletwoSuprimeubarprime{\ali P_i^7}^2 \dubarprime \right)^{\frac{1}{2}}\lesssim \frac{\al \Gamma \cdot M}{\upr}.\] For the eighth term, there holds

\[ \left( \intubar \scaletwoSuprimeubarprime{\ali P_i^8}^2 \dubarprime \right)^{\frac{1}{2}}\lesssim \Gamma^2 + \Gamma R.\]For the ninth and tenth terms, there holds

\[ \left( \intubar \scaletwoSuprimeubarprime{\ali P_i^{9}}^2 \dubarprime \right)^{\frac{1}{2}} + \left( \intubar \scaletwoSuprimeubarprime{\ali P_i^{10}}^2 \dubarprime \right)^{\frac{1}{2}}\lesssim \Gamma^3 + \Gamma^2 R. \]
Putting everything together, there holds

\begin{equation} \label{alphabeta1} \intubar \intu \frac{a}{\upr} \scaleoneSuprimeubarprime{\ali P_i \aln \alpha} \duprime \dubarprime\lesssim \Gamma^3 + \Gamma^2 R + \Gamma R+1.\end{equation}Similarly, for the analogous term involving $\tbeta$, there holds

\begin{equation}
    \begin{split}
         &\intubar \intu  \frac{a}{\upr} \scaleoneSuprimeubarprime{\ali Q_i \cdot \aln \tbeta }  \duprime \dubarprime \\\leq&\sum_{j=1}^{7} \intu \frac{a}{\upr^2} \left(  \intubar \scaletwoSuprimeubarprime{\ali Q_i^j }^2\dubarprime \right)^{\frac{1}{2}}\duprime
        \cdot \sup_{u^{\prime}} \lVert \aln \tbeta \rVert_{L^2_{(sc)}(\Hb_{\ubar^{\prime}}^{(u_{\infty},u)})}.
    \end{split}
\end{equation}We estimate term by term. For the first term, there holds

\begin{equation}
    \intu \frac{a}{\upr^2}\left( \intubar \scaletwoSuprimeubarprime{\ali Q_i^1}^2 \dubarprime \right)^{\frac{1}{2}} \duprime \lesssim \frac{a\Gamma (R+\Gamma)}{\lvert u \rvert }
\end{equation}
For the second term, there holds

\begin{equation}
    \intu \frac{a}{\upr^2}\left( \intubar \scaletwoSuprimeubarprime{\ali Q_i^2}^2 \dubarprime \right)^{\frac{1}{2}} \duprime \lesssim \frac{a\Gamma (M+\Gamma)}{\lvert u \rvert }
\end{equation}For the third and fourth terms, there holds 
\begin{equation}
\begin{split}
    &\intu \frac{a}{\upr^2}\left( \intubar \scaletwoSuprimeubarprime{\ali Q_i^3}^2 \dubarprime \right)^{\frac{1}{2}} \duprime +    \intu \frac{a}{\upr^2}\left( \intubar \scaletwoSuprimeubarprime{\ali Q_i^4}^2 \dubarprime \right)^{\frac{1}{2}} \duprime  \\\lesssim& \frac{a\Gamma (M+\Gamma)}{\lvert u \rvert }
    \end{split}
\end{equation}
For the fifth term there holds 

\begin{equation}
    \intu \frac{a}{\upr^2}\left( \intubar \scaletwoSuprimeubarprime{\ali Q_i^5}^2 \dubarprime \right)^{\frac{1}{2}} \duprime \lesssim \frac{\al \Gamma (R+\Gamma)}{\lvert u \rvert}
\end{equation} For the sixth term, we have 
\begin{equation}
    \intu \frac{a}{\upr^2}\left( \intubar \scaletwoSuprimeubarprime{\ali Q_i^6}^2 \dubarprime \right)^{\frac{1}{2}} \duprime \lesssim \frac{a\Gamma^2 (R+\Gamma)}{\lvert u \rvert^2}
\end{equation}For the seventh term, we have
\begin{equation}
    \intu \frac{a}{\upr^2}\left( \intubar \scaletwoSuprimeubarprime{\ali Q_i^7}^2 \dubarprime \right)^{\frac{1}{2}} \duprime \lesssim \frac{\Gamma^2 (R+\Gamma)}{\lvert u \rvert^3}
\end{equation}

\par \noindent Putting everything together, we have

\begin{equation}\label{alphabeta2}
    \sum_{j=1}^7 \intu \frac{a}{\upr^2} \left( \intubar \scaletwoSuprimeubarprime{\ali Q_i^j}^2 \dubarprime\right)^{\frac{1}{2}} \duprime \lesssim \frac{a \Gamma(R+M+\Gamma)}{\lvert u \rvert}+1.
\end{equation}Combining \eqref{alphabeta1} and \eqref{alphabeta2}, we have 

\begin{equation}
    \begin{split}
        & \frac{1}{\al} \scaletwoHu{\aln \alpha} + \frac{1}{\al}\scaletwoHbaru{\aln \tbeta} \\ \lesssim & \frac{1}{\al} \scaletwoHzero{\aln \alpha} + \frac{1}{\al} \scaletwoHbarzero{\aln \tbeta} \\&+  \frac{1}{\al} \intubar \intu  \frac{a}{\upr} \scaleoneSuprimeubarprime{\ali P_i \cdot \aln \alpha }  \duprime \dubarprime
        \\ &+ \frac{1}{\al} \intubar \intu  \frac{a}{\upr} \scaleoneSuprimeubarprime{\ali Q_i \cdot \aln \tbeta }  \duprime \dubarprime \\ \lesssim &\frac{1}{\al} \scaletwoHzero{\aln \alpha} + \frac{1}{\al} \scaletwoHbarzero{\aln \tbeta} + \frac{1}{a^{\frac{1}{3}}}.
        \end{split}
\end{equation}The claim follows.
\end{proof}\noindent We now move on to energy estimates for the remaining pairs $(\tbeta, (\rho,\sigma)), ((\rho,\sigma),\tbetabar)$ and $(\tbetabar, \alphabar)$.

\begin{proposition}
Let $(\Psi_1, \Psi_2) \in \begin{Bmatrix} (\tbeta, (\rho,\sigma)), ((\rho,\sigma),\tbetabar), (\tbetabar, \alphabar) \end{Bmatrix}$. For $0\leq i \leq N+4$, there holds:

\begin{align}
    & \scaletwoHu{\aln \Psi_1} + \scaletwoHbaru{\aln \Psi_2} \\ \lesssim &  \scaletwoHzero{\aln \Psi_1} +  \scaletwoHbarzero{\aln \Psi_2}+  \frac{1}{a^{\frac{1}{3}}}.
\end{align}
\end{proposition}

\begin{proof}
The schematic equations for $\Psi_1, \Psi_2$ are:

\begin{align}
       \nonumber  \nabla_3 \Psi_1 + \left( \frac{1}{2} + s_2(\Psi_1)\right) \tr\chibar \Psi_1 - \mathcal{D} \Psi_2 =& (\psi,\chihat)\Psi + \alphaF \hnabla\Yub + \Yub \hnabla (\alphaF, \Yub) \\ +&(\psi,\chihat,\chibarhat,\tr\chibar)(\alphaF,\Yub)(\alphaF,\Yub),
    \end{align}

\begin{equation}
    \begin{split}
        \nabla_4 \Psi_2- \Hodge{\mathcal{D}}\Psi_1 = (\psi,\chibarhat)(\Psi_u, \alpha) +(\alphaF, \Yub)\hnabla (\alphaF,\Yub) +(\psi,\chihat,\chibarhat,\tr\chibar)(\alphaF,\Yub)(\alphaF,\Yub).
    \end{split}
\end{equation}Commuting with $i$ angular derivatives, for $\Psi_1$, we have:

\begin{equation}
    \begin{split}
        &\nabla_3 \nabla^i \Psi_1 + \left(\frac{i+1}{2}+s_2(\Psi_1)\right)\tr\chibar \nabla^i \Psi_1 - \mathcal{D}\nabla^i \Psi_2 \\= &\sumitm \nablapp \nabla^{i_3+1}\Psi_2 +\sumif\nablap \nablat(\psi_g,\chihat)\nablaf\Psi \\ &+\sumif \nablap \hnablat \alphaF \hnabla^{i_4+1}\Yub + \sumif \nablap \hnablat \Yub \hnabla^{i_4+1}(\alphaF, \Yub)\\ &+\sumiF \nablap \nablat (\psi_g,\chihat,\chibarhat,\tr\chibar) \hnablaf (\alphaF, \Yub)\hnablaF (\alphaF,\Yub) \\&+ \sumif\nablap \nablat(\chibarhat,\tildetr) \nablaf \Psi_1 + \sumifi \nablapp \nablat \tr\chibar \nablaf \Psi_1 \\ &+\sumifim \nablapp \nablat (\chibarhat,\tr\chibar) \nablaf \Psi_1 \\&+ \sumiFi \nablap \hnablat \alphabarF \hnablaf(\rhoF,\sigmaF)\nablaF \Psi_1 := P_i.
    \end{split}
\end{equation}Analogously, for $\Psi_2$, we have 

\begin{equation}
    \begin{split}
        &\nabla_4 \nabla^i \Psi_2 -\Hodge{ \mathcal{D}} \nabla^i \Psi_1 \\= &\sumitm \nablapp \nabla^{i_3+1}\Psi_1 + \sumif \nablap \nablat(\psi_g,\chibarhat) \nablaf(\Psi, \alpha) \\&+ \sumif \nablap \hnablat( \alphaF, \Yub) \hnabla^{i_4+1}(\alphaF,\Yub) \\&+ \sumiF \nablap \nablat (\psi_g,\chihat,\chibarhat,\tr\chibar) \hnablaf (\alphaF, \Yub)\hnablaF (\alphaF,\Yub) \\ &+\sumif \nablap \nablat(\chihat,\tr\chi) \nablaf \Psi_2 \\ &+ \sumifim \nablapp \nablat (\chihat,\tr\chi) \nablaf \Psi_2 \\&+ \sumiFi \nablapp \hnablat \alphaF \hnablaf (\rhoF,\sigmaF)\nablaF \Psi_2 := Q_i.
    \end{split}
\end{equation}Making use of Proposition \ref{curvprop} once again, we arrive at 

\begin{equation}\label{mainpsu}
    \begin{split}
        &\scaletwoHu{\aln \Psi_1}^2 + \scaletwoHbaru{\aln \Psi_2}^2 \\ \lesssim & \scaletwoHzero{\aln \Psi_1}^2 + \scaletwoHbarzero{\aln \Psi_2}^2 \\&+  \intubar \intu  \frac{a}{\upr} \scaleoneSuprimeubarprime{\ali P_i \cdot \aln \Psi_1 }  \duprime \dubarprime
        \\ &+\intubar \intu  \frac{a}{\upr} \scaleoneSuprimeubarprime{\ali Q_i \cdot \aln \Psi_2 }  \duprime \dubarprime .
        \end{split}
\end{equation}For the first spacetime integral in the above, we estimate

\begin{equation}\label{comb1}
    \begin{split}
       & \intubar \intu  \frac{a}{\upr} \scaleoneSuprimeubarprime{\ali P_i \cdot \aln \Psi_1 }  \duprime \dubarprime \\\lesssim& \intu \frac{a}{\upr^2} \left( \intubar \scaletwoSuprimeubarprime{\ali P_i}^2 \dubarprime \right)^{\frac{1}{2}} \duprime \cdot \left( \intubar \scaletwoSuprimeubarprime{\aln \Psi_1}^2 \dubarprime \right)^{\frac{1}{2}} \duprime.
    \end{split}
\end{equation}For the first term:

\begin{equation}
    \begin{split}
        \intu \frac{a}{\upr^2} \left( \intubar \scaletwoSuprimeubarprime{\ali \sumitm \nablapp \nabla^{i_3+1}\Psi_2 }^2 \dubarprime \right)^{\frac{1}{2}} \duprime,
    \end{split}
\end{equation}if $i_3+1 \geq N+3$, we estimate

\begin{equation}
    \begin{split}
        &\intu \frac{a}{\upr^2} \left( \intubar \scaletwoSuprimeubarprime{\ali \sumitm \nablapp \nabla^{i_3+1}\Psi_2 }^2 \dubarprime \right)^{\frac{1}{2}} \duprime \\ \lesssim & \sup_{0 \leq \ubar^{\prime}\leq \ubar} \intu \frac{a}{\upr^3} \scaletwoSuprimeubarprime{     (\al \nabla)^{i_3+1} \Psi_2} \scaleinfinitySuprimeubarprime{(\al)^{i_1+i_2}\nablapp}  \duprime   
    \end{split}
\end{equation}and we can estimate $\scaleinfinitySuprimeubarprime{(\al)^{i_1+i_2}\nablapp} $ by $\frac{\upr}{\al}$ using the bootstrap assumption \eqref{bootstrap}, to obtain

\begin{equation}
    \begin{split}
         &\intu \frac{a}{\upr^2} \left( \intubar \scaletwoSuprimeubarprime{\ali \sumitm \nablapp \nabla^{i_3+1}\Psi_2 }^2 \dubarprime \right)^{\frac{1}{2}} \duprime \\ \lesssim & \sup_{0 \leq \ubar^{\prime}\leq \ubar}\intu \frac{\al}{\upr^2}\scaletwoSuprimeubarprime{     (\al \nabla)^{i_3+1} \Psi_2}  \duprime \lesssim \frac{R}{\lvert u \rvert^{\frac{1}{2}}}\lesssim 1. 
    \end{split}
\end{equation}If, however, $i_3+1 \leq N+2$, we can control the corresponding $L^2_{(sc)}(S)$ norm just by the bootstrap assumption \eqref{bootstrap} to get the bound 

\begin{equation}
    \begin{split}
         &\intu \frac{a}{\upr^2} \left( \intubar \scaletwoSuprimeubarprime{\ali \sumitm \nablapp \nabla^{i_3+1}\Psi_2 }^2 \dubarprime \right)^{\frac{1}{2}} \duprime \lesssim \frac{a \Gamma^2}{\lvert u \rvert^2} \lesssim 1.
         \end{split}
\end{equation}For the rest of the terms, we estimate using the same philosophy as appropriate. There holds  
\begin{equation}
    \begin{split}
        &\intu \frac{a}{\upr^2} \left( \intubar \scaletwoSuprimeubarprime{\ali \sumif \nablap \nabla^{i_3}(\psi_g,\chihat) \nablaf\Psi }^2 \dubarprime \right)^{\frac{1}{2}} \duprime \\ \lesssim & \frac{\al \Gamma (R+\Gamma)}{\lvert u \rvert}\lesssim 1.
    \end{split}
\end{equation}Here in particular, when $i_4\geq N+3$, we treat the cases $\Psi=\Psi_u$ and $\Psi= \Psi_{\ubar}$ separately. For the third and fourth terms, there holds  

\begin{equation}
    \begin{split}
         &\intu \frac{a}{\upr^2} \left( \intubar \scaletwoSuprimeubarprime{\ali \sumif \nablap \hnabla^{i_3}\alphaF \hnabla^{i_4+1}\Yub}^2 \dubarprime \right)^{\frac{1}{2}} \duprime \\ +   &\intu \frac{a}{\upr^2} \left( \intubar \scaletwoSuprimeubarprime{\ali \sumif \nablap \hnabla^{i_3}\Yub \hnabla^{i_4+1}(\alphaF,\Yub)}^2 \dubarprime \right)^{\frac{1}{2}} \duprime  \\\lesssim & \frac{ \Gamma (M+\Gamma)}{\lvert u \rvert}\lesssim 1.
    \end{split}
\end{equation}For the fifth term, we have 

\begin{equation}
    \begin{split}
         &\intu \frac{a}{\upr^2} \left( \intubar \scaletwoSuprimeubarprime{\ali \sum_{i} \nablap \nabla^{i_3}(\psi,\chihat,\chibarhat,\tr\chibar) \hnabla^{i_4}(\alphaF,\Yub)\hnablaF(\alphaF,\Yub)}^2 \dubarprime \right)^{\frac{1}{2}} \duprime \\ \lesssim& \frac{a}{\lvert u \rvert} \Gamma[\tr\chibar]\Gamma[\alphaF]^2\lesssim \frac{a}{\lvert u \rvert}\lesssim 1, 
         \end{split}
\end{equation}where we have used the improvements $\Gamma[\tr\chibar] +\Gamma[\alphaF]\lesssim 1$ from Propositions \ref{trchibarbound} and \ref{alphaFprop}. For the sixth and seventh terms, we can bound them by one as in previous calculations. For the eighth term, using the fact that $i-2 \leq N+2$ and the improvements from Proposition \ref{psiuprop}  and Proposition \ref{alphaen}, we arrive at

\begin{equation}
    \begin{split}
         &\intu \frac{a}{\upr^2} \left( \intubar \scaletwoSuprimeubarprime{\ali \sumifim \nablapp \nablat (\chibarhat,\tr\chibar) \nablaf \Psi_1}^2 \dubarprime \right)^{\frac{1}{2}} \duprime \\ \lesssim& \frac{a}{\lvert u \rvert} \Gamma[\tr\chibar]\Gamma[\alphaF]^2\lesssim \frac{a}{\lvert u \rvert}\Gamma[\eta,\etabar] \Gamma[\tr\chibar]\Gamma[\Psi_1] \lesssim \frac{a}{\lvert u \rvert}\scaletwoSu{\aln(\tbeta,\tbetabar)} \cdot 1 \cdot (\mathcal{R}[\alpha]+1)\\ \lesssim & (\mathcal{R}[\alpha]+1)^2 \lesssim 1,
         \end{split}
\end{equation}where in the last line we made use of Proposition \ref{alphaen}. For the final term in $P_i$, we have 

\begin{equation}
    \begin{split}
         &\intu \frac{a}{\upr^2} \left( \intubar \scaletwoSuprimeubarprime{\ali \sumiFi \nablap \hnablat \alphabarF \hnablaf (\rhoF,\sigmaF) \nablaF \Psi_1 }^2 \dubarprime \right)^{\frac{1}{2}} \duprime \\ \lesssim& \frac{a}{\lvert u \rvert} \cdot \frac{\al \Gamma^2 (R+\Gamma)}{\lvert u\rvert^2} \lesssim 1.
         \end{split}
\end{equation}This completes the estimates for the first spacetime integral in \eqref{mainpsu}. For the second and last one, a double application of H\"older's inequality yields
\begin{equation}
    \begin{split}
        &\intubar \intu  \frac{a}{\upr} \scaleoneSuprimeubarprime{\ali Q_i \cdot \aln \Psi_2 }  \duprime \dubarprime \\ \lesssim & \intubar \left( \intu \frac{a}{\upr^2} \scaletwoSuprimeubarprime{\ali Q_i}^2 \duprime \right)^{\frac{1}{2}} \lVert \aln \Psi_2 \rVert_{L^2_{(sc)}(\Hbar_{\ubar^{\prime}}^{(u_{\infty}, u)})} \\ \lesssim& \left(\intubar \intu \frac{a}{\upr^2} \scaletwoSuprimeubarprime{\ali Q_i}^2 \duprime \dubarprime \right)^{\frac{1}{2}} \left( \intubar \lVert \aln \Psi_2 \rVert^2_{L^2_{(sc)}(\Hbar_{\ubar^{\prime}}^{(u_{\infty}, u)})} \dubarprime \right)^{\frac{1}{2}} \\ \lesssim &\intubar \intu \frac{a}{\upr^2} \scaletwoSuprimeubarprime{\ali Q_i}^2 \duprime \dubarprime + \frac{1}{4} \intubar \lVert \aln \Psi_2 \rVert^2_{L^2_{(sc)}(\Hbar_{\ubar^{\prime}}^{(u_{\infty}, u)})} \dubarprime
    \end{split}
\end{equation}Define $B:= \intubar \intu \frac{a}{\upr^2} \scaletwoSuprimeubarprime{\ali Q_i}^2 \duprime \dubarprime$. We can then  estimate $B$ as follows:

\begin{equation}
\begin{split}
    &\intubar \intu \frac{a}{\upr^2}\scaletwoSuprimeubarprime{\ali \sumitm \nablapp \nabla^{i_3+1}\Psi_1}^2 \duprime \dubarprime \\ \lesssim &\intubar \intu \frac{a}{\upr^2}\scaletwoSuprimeubarprime{\ali \psi_g \nabla^i \Psi_1}^2 \duprime \dubarprime \\+& \intubar \intu \frac{a}{\upr^2}\scaletwoSuprimeubarprime{\ali \nabla \psi_g \nabla^{i-1}\Psi_1 }^2 \duprime \dubarprime \\+& \intubar \intu \frac{a}{\upr^2}\scaletwoSuprimeubarprime{\ali \psi_g \hsp  \psi_g  \nabla^{i-1}\Psi_1}^2 \duprime \dubarprime \\+&\intubar \intu \frac{a}{\upr^2}\scaletwoSuprimeubarprime{\ali \sum_{\substack{i_1+i_2+i_3=i-1\\ i_3< i-2}} \nablapp \nabla^{i_3+1}\Psi_1}^2 \duprime \dubarprime \\ \lesssim &1, 
\end{split}
\end{equation}where in the first three integrals we estimate $\psi_g, \nabla \psi_g$ in $L^{\infty}_{(sc)}(S_{u,\ubar})$ and $\nabla^i \Psi_1$ in the hypersurface norm $L^2_{(sc)}(H_u^{(0,\ubar)})$ and in the last integral, since $i-2\leq N+2$, we can estimate $\nabla^i \Psi_1$ in $L^2_{(sc)}(S_{u,\ubar})$ using the bootstrap assumption on the norm $\bbGamma$. For the second term, we similarly have

\begin{align}  \nonumber
    &\intubar \intu \frac{a}{\upr^2}\scaletwoSuprimeubarprime{\ali \sumif \nablap \nablat (\psi_g, \chihat) \nablaf(\Psi, \alpha) }^2 \duprime \dubarprime \\ \lesssim & \left(\mathcal{R}[\alpha]+1 \right)^2 \lesssim 1,
\end{align}where we have used the improvements on $\chihat$ from Proposition \ref{chihats} and the energy estimate from Proposition \ref{alphaen}. For the third term, we notice that the worst case is when the term is $\ali \sumif \nablap \hnablat \alphaF \hnabla^{i_4+1}\alphaF$. We therefore estimate

\begin{equation}
    \begin{split}
         &\intubar \intu \frac{a}{\upr^2}\scaletwoSuprimeubarprime{\ali \sumif \nablap \hnablat\alphaF \hnabla^{i_4+1}\alphaF}^2 \duprime \dubarprime\\ \lesssim & \intu \frac{a}{\upr^2} \intubar \scaletwoSuprimeubarprime{\ali \alphaF \hnabla^{i+1}\alphaF}^2 \dubarprime \duprime \\+&\intu \frac{a}{\upr^2} \intubar \scaletwoSuprimeubarprime{\ali \sum_{\substack{i_1+i_2+i_3+i_4=i\\ i_4 < i}}\nablap \hnablat\alphaF \hnabla^{i_4+1}\alphaF}^2 \dubarprime \duprime. \label{doublealphaf}
    \end{split}
\end{equation}When $i_4<i$, we can control $\hnabla^{i_4+1}\alphaF$ using the bootstrap assumption on the $\bbGamma$ norm, hence the second integral of \eqref{doublealphaf} is bounded above by $1$. For the first, we have

\begin{equation}
    \begin{split}
         & \intu \frac{a}{\upr^2} \intubar \scaletwoSuprimeubarprime{\ali \alphaF \hnabla^{i+1}\alphaF}^2 \dubarprime \duprime \\ \lesssim & \intu \frac{a}{\upr^2} \frac{a^2 \Gamma[\alphaF]^2}{\upr^2} \intubar \frac{1}{a}\scaletwoSuprimeubarprime{\ali \hnabla^{i+1}\alphaF}^2\dubarprime \duprime\\ \lesssim & \frac{a^3 \Gamma[\alphaF]^2 M[\alphaF]^2}{\lvert u \rvert^3}.
    \end{split}
\end{equation}We now make use of the improvements given by Proposition \ref{alphaFprop} and the Yang-Mills energy estimate on $(\alphaF,(\rhoF,\sigmaF))$ to deduce that 

\begin{equation}
    \Gamma[\alphaF]^2 \lesssim (\underline{\mathbb{Y}\mathbb{M}}[\rhoF,\sigmaF]+1)^2 \lesssim 1,
\end{equation}
\begin{equation}
    M[\alphaF]^2 \lesssim 1,
\end{equation}so that, finally,
\begin{equation}
    \intubar \intu \frac{a}{\upr^2}\scaletwoSuprimeubarprime{\ali \sumif \nablap \hnablat\alphaF \hnabla^{i_4+1}\alphaF}^2 \duprime \dubarprime \lesssim 1.
\end{equation}The rest of the terms can also be bounded above by $1$, using the same approach. We finally arrive at an estimate of the form 

\begin{equation}\label{comb2}
    \begin{split}
        &\intubar \intu  \frac{a}{\upr} \scaleoneSuprimeubarprime{\ali Q_i \cdot \aln \Psi_2 }  \duprime \dubarprime \\ \lesssim & 1 + \frac{1}{4} \intubar \lVert \aln \Psi_2 \rVert^2_{L^2_{(sc)}(\Hbar_{\ubar^{\prime}}^{(u_{\infty}, u)})} \dubarprime.
        \end{split}
\end{equation}From here, collecting all the terms together and using Proposition \ref{prop54}, we arrive at the desired result.
\end{proof}
\section{The formation of a trapped surface}

\noindent As in all trapped surface formation results in the absence of symmetry, the formation argument is an ODE argument. It relies heavily, as we shall see, on the estimates obtained during the proof of semi-global existence. In this section we prove Theorem \ref{main2}.

\begin{proof}
The idea is to obtain estimates for $\lvert \chihat \rvert_{\gamma}^2$ and $\lvert \alphaF \rvert_{\gamma,\delta}^2$. We begin with $\chihat$. Recall the null structure equation \begin{equation}
\nabla_3 \chihat + \frac{1}{2}\tr\chibar \chihat - 2\omegabar\chihat = \nabla \hat{\otimes}\eta - \frac{1}{2}\tr\chi \chibarhat + \eta \hat{\otimes}\eta. \end{equation} Contracting once with another $\chihat$ tensor, we obtain 

\begin{equation}
\frac{1}{2}\nabla_3 \lvert \chihat \rvert_{\gamma}^2 + \frac{1}{2}\tr\chibar \lvert \chihat \rvert_{\gamma}^2 - 2 \omega \lvert \chihat \rvert_{\gamma}^2 = \chihat ( \nabla \hat{\otimes}\eta - \frac{1}{2}\tr\chi \chibarhat + \eta \hat{\otimes}\eta).
\end{equation}Using the fact that $\omegabar = -\frac{1}{2} \Omega^{-1}\nabla_3 \Omega$, we can rewrite the  above as 

\begin{equation}
\nabla_3 (\Omega^2 \lvert \chihat \rvert_{\gamma}^2) + \tr\chibar \Omega^2 \lvert \chihat \rvert_{\gamma}^2 = 2 \Omega^2 \chihat ( \nabla \hat{\otimes}\eta - \frac{1}{2}\tr\chi \chibarhat + \eta \hat{\otimes}\eta).
\end{equation}Using $\nabla_3 = \frac{1}{\Omega}\left( \frac{\partial}{\partial u} + b^A \frac{\partial}{\partial \theta^A}\right)$, we have 

\begin{equation}
\frac{\partial}{\partial u} (\Omega^2 \lvert \chihat \rvert_{\gamma}^2) + \Omega \tr\chibar \cdot \Omega^2 \lvert \chihat \rvert_{\gamma}^2 =  2 \Omega^3 \chihat ( \nabla \hat{\otimes}\eta - \frac{1}{2}\tr\chi \chibarhat + \eta \hat{\otimes}\eta) - b^A \frac{\partial}{\partial \theta^A} (\Omega^2 \lvert \chihat \rvert_{\gamma}^2).
\end{equation}Substituting \[ \Omega \tr\chibar = \Omega \left( \tr\chibar + \frac{2}{\lvert u \rvert} \right) - (\Omega-1) \frac{2}{\lvert u \rvert} - \frac{2}{\lvert u \rvert},  \]we have 

\begin{equation}
\begin{split}
\frac{\partial}{\partial u} (\Omega^2 \lvert \chihat \rvert_{\gamma}^2) - \frac{2}{\lvert u \rvert} \Omega^2 \lvert \chihat \rvert_{\gamma}^2 = &2 \Omega^3 \chihat ( \nabla \hat{\otimes}\eta - \frac{1}{2}\tr\chi \chibarhat + \eta \hat{\otimes}\eta) - b^A \frac{\partial}{\partial \theta^A} (\Omega^2 \lvert \chihat \rvert_{\gamma}^2)\\ - &\Omega(\tr\chibar + \frac{2}{\lvert u \rvert})(\Omega^2 \lvert \chihat \rvert_{\gamma}^2 ) + (\Omega-1) \cdot \frac{2}{\lvert u \rvert}\cdot (\Omega^2 \lvert \chihat \rvert_{\gamma}^2).
\end{split}
\end{equation}We therefore have the following equation for $\chihat$:

\begin{equation}\label{er}
\begin{split}
\frac{\partial}{\partial u}\left( \lvert u \rvert^2 \Omega^2 \lvert \chihat \rvert_{\gamma}^2 \right) = & 2 \lvert u \rvert^2 \Omega^3 \chihat ( \nabla \hat{\otimes}\eta - \frac{1}{2}\tr\chi \chibarhat + \eta \hat{\otimes}\eta) - \lvert u \rvert^2 b^A \frac{\partial}{\partial \theta^A}(\Omega^2 \lvert \chihat \rvert_{\gamma}^2) \\ & - \lvert u \rvert^2 \Omega (\tr\chibar + \frac{2}{\lvert u \rvert}) (\Omega^2 \lvert \chihat \rvert_{\gamma}^2 ) + \lvert u \rvert^2 (\Omega-1) \cdot \frac{2}{\lvert u \rvert} (\Omega^2 \lvert \chihat \rvert_{\gamma}^2).
\end{split}
\end{equation}For $b^A$, there holds

\begin{equation}
\frac{\partial b^A}{\partial \ubar} = -4 \Omega^2 \zeta^A, 
\end{equation}where $\zeta = \frac{1}{2}(\eta- \etabar)$. Using the derived estimates on $\eta, \etabar$, one can then conclude that \begin{equation}
\lVert b^A \rVert_{L^{\infty}(S_{u,\ubar})} \lesssim \frac{\al}{\lvert u \rvert^2},
\end{equation}in the slab of existence. For the right-hand side of \eqref{er}, we can obtain, using the various estimates on the Ricci coefficients obtained during the proof of Theorem \ref{main1}, the following bounds:

\begin{equation}
\scaleinfinitySu{ 2 \lvert u \rvert^2 \Omega^3 \chihat ( \nabla \hat{\otimes}\eta - \frac{1}{2}\tr\chi \chibarhat + \eta \hat{\otimes}\eta) }\leq \frac{a}{\upr^2},
\end{equation}
\begin{equation}
\scaleinfinitySu{\lvert u \rvert^2 b^A \frac{\partial}{\partial \theta^A} (\Omega^2 \lvert \chihat \rvert_{\gamma}^2 ) } \leq \frac{a^{\frac{3}{2}}}{\lvert u \rvert^2},
\end{equation}
\begin{equation}
\scaleinfinitySu{\lvert u \rvert^2 \Omega (\tr\chibar + \frac{2}{\lvert u \rvert}) (\Omega^2 \lvert \chihat \rvert_{\gamma}^2 )}\leq \frac{a}{\upr^2},
\end{equation}

\begin{equation}
\scaleinfinitySu{\lvert u \rvert^2 (\Omega-1) \cdot \frac{2}{\lvert u \rvert} (\Omega^2 \lvert \chihat \rvert_{\gamma}^2)}\leq \frac{a}{\lvert u \rvert^2}.
\end{equation}Putting everything together, we have

\begin{equation}
\big \lvert \frac{\partial}{\partial u}\left( \lvert u \rvert^2 \Omega^2 \lvert \chihat \rvert_{\gamma}^2 \right) \big \rvert \lesssim \frac{a^{\frac{3}{2}}}{\lvert u \rvert^2}  \ll \frac{a^{\frac{7}{4}}}{\lvert u \rvert^2}.
\end{equation}As a consequence, one has

\begin{equation}
\lvert u \rvert^2 \Omega^2 \lvert \chihat \rvert_{\gamma}^2 (u, \ubar, \theta^1, \theta^2) -\lvert u_{\infty} \rvert^2 \Omega^2 \lvert \chihat \rvert_{\gamma}^2(u_{\infty},\ubar,\theta^1, \theta^2) \geq -\frac{a^{\frac{7}{4}}}{\lvert u \rvert} + \frac{a^{\frac{7}{4}}}{\lvert u_{\infty}\rvert}.
\end{equation}However, our gauge choice was $\restri{\Omega}{H_{u_{\infty}}} \equiv 1$, therefore 
\begin{equation}
\lvert u \rvert^2 \Omega^2 \lvert \chihat \rvert_{\gamma}^2 (u, \ubar, \theta^1, \theta^2) \geq \lvert u_{\infty} \rvert^2 \lvert \chihat \rvert_{\gamma}^2(u_{\infty},\ubar,\theta^1, \theta^2) - \frac{a^{\frac{7}{4}}}{\lvert u \rvert}.
\end{equation}Integrating from 0 to 1 with respect to $\ubar$, we have

\begin{equation}\label{616}
\int_{0}^{1} \lvert u \rvert^2 \Omega^2 \lvert \chihat \rvert_{\gamma}^2 (u, \ubar^{\prime}, \theta^1, \theta^2) \dubarprime \geq \int_0^1 \lvert u_{\infty} \rvert^2   \lvert \chihat \rvert_{\gamma}^2(u_{\infty},\ubar^{\prime},\theta^1, \theta^2) \dubarprime - \frac{a^{\frac{7}{4}}}{\lvert u \rvert}.
\end{equation}An analogous procedure using the null Yang-Mills equation for $\alphaF$ yields

\begin{equation}\label{617}
\big \lvert \frac{\partial}{\partial u}\left( \lvert u \rvert^2 \Omega^2 \lvert \alphaF \rvert_{\gamma,\delta}^2 \right) \big \rvert \lesssim  \frac{a^{\frac{3}{2}}}{\lvert u \rvert^2}  \ll \frac{a^{\frac{7}{4}}}{\lvert u \rvert^2},
\end{equation}whence we also get \begin{equation}\label{618}
\int_{0}^{1} \lvert u \rvert^2 \Omega^2 \lvert \alphaF \rvert_{\gamma, \delta}^2 (u, \ubar^{\prime}, \theta^1, \theta^2) \dubarprime \geq \int_0^1 \lvert u_{\infty} \rvert^2   \lvert \alphaF \rvert_{\gamma, \delta}^2(u_{\infty},\ubar^{\prime},\theta^1, \theta^2) \dubarprime - \frac{a^{\frac{7}{4}}}{\lvert u \rvert}.
\end{equation}Indeed, consider the equation 

\begin{equation}
    \hnabla_3 \alpha^F + \frac{1}{2}\tr\chibar\alpha^F = \hnabla \rho^F + \Hodge{\hnabla}\sigma^F -2 \Hodge{\eta}\sigma^F+ 2 \eta \rho^F +2 \omegabar \alpha^F - \chihat \cdot \alphabar^F
\end{equation}and contract with $\alphaF$ to get

\begin{equation}
\frac{1}{2}\hnabla_3 \lvert \alphaF \rvert_{\gamma, \delta}^2 + \frac{1}{2} \tr\chibar \lvert \alphaF \rvert_{\gamma, \delta}^2 - 2 \omegabar \lvert \alphaF \rvert_{\gamma, \delta}^2 = \alphaF (\hnabla \rho^F + \Hodge{\hnabla}\sigma^F -2 \Hodge{\eta}\sigma^F+ 2 \eta \rho^F +2 \omegabar \alpha^F - \chihat \cdot \alphabar^F).
\end{equation}
As before, we obtain

\begin{equation}
\begin{split}
&\frac{\partial}{\partial u} \left( \lvert u \rvert^2 \Omega^2 \lvert \alphaF \rvert_{\gamma, \delta}^2 \right)= \\ &2 \lvert u \rvert^2 \Omega^3 \alphaF ( \hnabla \rho^F + \Hodge{\hnabla}\sigma^F -2 \Hodge{\eta}\sigma^F+ 2 \eta \rho^F +2 \omegabar \alpha^F - \chihat \cdot \alphabar^F) -\lvert u \rvert^2 b^A \frac{\partial}{\partial \theta^A} \left( \lvert u \rvert^2 \Omega^2 \lvert \alphaF \rvert_{\gamma, \delta}^2 \right)\\ - &\lvert u \rvert^2 \Omega \left(\tr\chibar+ \frac{2}{\lvert u \rvert}\right) (\Omega^2  \lvert \alphaF \rvert_{\gamma, \delta}^2) + \lvert u \rvert^2 (\Omega-1)\frac{2}{\lvert u \rvert}(\Omega^2  \lvert \alphaF \rvert_{\gamma, \delta}^2).
\end{split}
\end{equation}Using the bounds obtained in the preceding sections, there holds

\begin{equation}\label{622}
\scaleinfinitySu{2 \lvert u \rvert^2 \Omega^3 \alphaF ( \hnabla \rho^F + \Hodge{\hnabla}\sigma^F -2 \Hodge{\eta}\sigma^F+ 2 \eta \rho^F +2 \omegabar \alpha^F - \chihat \cdot \alphabar^F)} \leq \frac{a}{\lvert u \rvert^2} + \frac{a^2}{\lvert u \rvert^3},
\end{equation}
\begin{equation}
\scaleinfinitySu{\lvert u \rvert^2 b^A \frac{\partial}{\partial \theta^A} \left( \lvert u \rvert^2 \Omega^2 \lvert \alphaF \rvert_{\gamma, \delta}^2 \right)} \leq \frac{a^{\frac{3}{2}}}{\lvert u \rvert^2},
\end{equation}
\begin{equation}
\scaleinfinitySu{\lvert u \rvert^2 \Omega \left(\tr\chibar+ \frac{2}{\lvert u \rvert}\right) (\Omega^2  \lvert \alphaF \rvert_{\gamma, \delta}^2) }\leq \frac{a}{\lvert u \rvert^2},
\end{equation}

\begin{equation}\label{625}
\scaleinfinitySu{\lvert u \rvert^2 (\Omega-1)\frac{2}{\lvert u \rvert}(\Omega^2  \lvert \alphaF \rvert_{\gamma, \delta}^2)} \leq \frac{a}{\lvert u \rvert^2}.
\end{equation}
Equations \eqref{622}-\eqref{625} yield \eqref{617} and \eqref{618}.
Combining \eqref{616} and \eqref{618},we conclude that 

\begin{equation}
\begin{split}
&\int_0^1 \lvert u \rvert^2 \Omega^2 \left(\lvert \chihat \rvert_{\gamma}^2 + \lvert \alphaF \rvert_{\gamma, \delta}^2 \right) (u, \ubar^{\prime},\theta^1, \theta^2) \dubarprime\\ \geq &\int_0^1 \lvert u_{\infty} \rvert^2 \Omega^2 \left(\lvert \chihat \rvert_{\gamma}^2 + \lvert \alphaF \rvert_{\gamma, \delta}^2 \right) (u_{\infty}, \ubar^{\prime},\theta^1, \theta^2) \dubarprime - \frac{2 a^{\frac{7}{4}}}{\lvert u \rvert} \geq a -  \frac{2 a^{\frac{7}{4}}}{\lvert u \rvert} \geq a - 8 \frac{a^{\frac{7}{4}}}{a} \geq \frac{7a}{8},
\end{split}\label{626}
\end{equation} uniformly for $u \in [u_{\infty}, - a/4]$. Pick $u = -a/4$. Then \eqref{626} yields 

\begin{equation}
\begin{split}
&(-\frac{a}{4})^2 \int_0^1  \left(\lvert \chihat \rvert_{\gamma}^2 + \lvert \alphaF \rvert_{\gamma, \delta}^2 \right) (-\frac{a}{4}, \ubar^{\prime},\theta^1, \theta^2) \dubarprime \\ \geq &\frac{6}{7}\int_0^1(-\frac{a}{4})^2 \Omega^2 \left(\lvert \chihat \rvert_{\gamma}^2 + \lvert \alphaF \rvert_{\gamma, \delta}^2 \right) (u, \ubar^{\prime},\theta^1, \theta^2) \dubarprime \geq \frac{6}{7}\cdot \frac{7a}{8} \geq \frac{3a}{4},
\end{split}
\end{equation}where we have made use of the bound $\lVert \Omega-1 \rVert_{L^{\infty}(S_{u,\ubar})} \lesssim \frac{1}{a}$, for $a$ sufficiently large. This implies 

\begin{equation} \label{628}
 \int_0^1  \left(\lvert \chihat \rvert_{\gamma}^2 + \lvert \alphaF \rvert_{\gamma, \delta}^2 \right) (-\frac{a}{4}, \ubar^{\prime},\theta^1, \theta^2) \dubarprime \geq \frac{3a}{4} \cdot \frac{16}{a^2} = \frac{12}{a}.
\end{equation}Consider now the equation

\begin{equation}
\nabla_4 \tr\chi + \frac{1}{2}(\tr\chi)^2 = -\lvert \chihat \rvert_{\gamma}^2 - 2 \omega\tr\chi - \lvert \alphaF \rvert_{\gamma, \delta}^2.
\end{equation}This can be rewritten, using $\omega =-\frac{1}{2}\nabla_4 (\text{log}\hsp \Omega)$, as 

\begin{equation}
\nabla_4 \tr\chi + \frac{1}{2}(\tr\chi)^2 = -\lvert \chihat \rvert_{\gamma}^2 + \frac{1}{\Omega}\nabla_4 \Omega \cdot \tr\chi  - \lvert \alphaF \rvert_{\gamma, \delta}^2,
\end{equation}which rewrites as

\begin{equation}
\nabla_4 (\Omega^{-1}\tr\chi) = \Omega^{-1}\left( \nabla_4 \tr\chi - \Omega^{-1}\nabla_4 \Omega \tr\chi\right) =  \Omega^{-1} \left( - \frac{1}{2}(\tr\chi)^2 - \lvert \chihat \rvert_{\gamma}^2- \lvert \alphaF \rvert_{\gamma, \delta}^2 \right).
\end{equation}Finally, using the fact that $e_4 = \Omega^{-1}\frac{\partial}{\partial \ubar}$, we arrive at the equation

\begin{equation}\label{632}
\frac{\partial}{\partial \ubar} (\Omega^{-1}\tr\chi) =  \frac{1}{2}(\tr\chi)^2 - \lvert \chihat \rvert_{\gamma}^2- \lvert \alphaF \rvert_{\gamma, \delta}^2.
\end{equation}Since at $\ubar =0$ the initial data are Minkowskian, we have 

\begin{equation}
(\Omega^{-1}\tr\chi)(-\frac{a}{4},0,\theta^1,\theta^2) = 1^{-1} \cdot \frac{2}{\frac{a}{4}}= \frac{8}{a}.
\end{equation}Integrating \eqref{632}, we obtain

\begin{equation}
\begin{split}
(\Omega^{-1}\tr\chi)(-\frac{a}{4},1,\theta^1,\theta^2) & \leq (\Omega^{-1}\tr\chi)(-\frac{a}{4},0,\theta^1,\theta^2) - \int_0^1 \lvert \chihat \rvert_{\gamma}^2+ \lvert \alphaF \rvert_{\gamma, \delta}^2(- \frac{a}{4},\ubar^{\prime},\theta^1, \theta^2)\dubarprime \\& \leq \frac{8}{a}-\frac{12}{a} <0. \label{634}
\end{split}
\end{equation}Notice that the above bound holds pointwise on the sphere $S_{-\frac{a}{4},1}$. Finally, recall the bound

\begin{equation}
\lVert \tr\chibar + \frac{2}{\lvert u \rvert} \rVert_{L^{\infty}(S_{u,\ubar})} \leq \frac{1}{\lvert u \rvert^2}.
\end{equation}As a result, 
\begin{equation}\label{636}
\tr\chibar(-\frac{a}{4},1, \theta^1,\theta^2) < 0, 
\end{equation} for all $(\theta^1,\theta^2)\in \mathbb{S}^2$. The bounds \eqref{634} and \eqref{636} finally imply that the sphere $S_{-\frac{a}{4},1}$ is trapped, which was what we wanted.
\end{proof}

\section{Concluding Remarks}
\noindent We conclude with a few brief remarks. The first one we wish to make is that, given the fact that all preceding estimates are independent of $u_{\infty}$, one can take (carefully) the limit $u_{\infty} \to \infty$ and thus obtain a trapped surface formation criterion from past null infinity. In that case, the present result could be interpreted as a formation of trapped surfaces due to the focusing of an amalgamation of Yang-Mills waves and gravitational waves. Secondly, going back to the statement of Theorem \ref{main1},  we see that conditions (1) and (2) are formally independent. This means that, in principle, both $\chihat$ and ${(\alphaF)^{P}}_Q$ can contribute equally to the largeness of the data. As a consequence, our result implies that in a spherically symmetric spacetime ($\chihat =0$) the mere presence of Yang-Mills "radiation" is enough to form a trapped surface, as long as this radiation is, of course, suitably large. Also, since we noticed in the work of \cite{smoller1991smooth} that Yang-Mills fields can counterbalance the gravitational attraction to yield regular solutions, it would be natural to investigate the precise mathematical condition on the Yang-Mills fields that would lead to a `non-trapping' scenario. Finally, we note that many subsequent problems naturally arise, for example the formation of apparent horizons for the Einstein-Yang-Mills system as well as trapped surface formation questions for different matter models. Important examples of such models are the Einstein-Klein-Gordon system and more ambitiously the Einstein-Euler system, which promises to be the most mathematically difficult, due to the possible formation of shocks and the breakdown of spacetime regularity before trapped surfaces have a chance to form. We intend to investigate these issues in the future.

\end{document}